%% file: convex-boundary-1.tex
\documentclass[12pt,reqno]{amsart}
\usepackage{amsmath,a4wide}
\usepackage{xfrac}
\usepackage{stmaryrd,mathrsfs,bm, bbm,amsthm,mathtools,yfonts,amssymb,color}
\usepackage{xcolor}
\usepackage[pagebackref,citecolor=black,linkcolor=blue,colorlinks=true]{hyperref}

\begingroup
\newtheorem{theorem}{Theorem}[section]
\newtheorem{lemma}[theorem]{Lemma}
\newtheorem{proposition}[theorem]{Proposition}
\newtheorem{corollary}[theorem]{Corollary}
\endgroup

\theoremstyle{definition}
\newtheorem{definition}[theorem]{Definition}
\newtheorem{remark}[theorem]{Remark}
\newtheorem{ipotesi}[theorem]{Assumption}

\numberwithin{equation}{section}

\newcommand\supp{{\rm spt}}

\newcommand\res{\mathop{\hbox{\vrule height 7pt width .3pt depth 0pt
\vrule height .3pt width 5pt depth 0pt}}\nolimits}

\newcommand{\bT}{\mathbf{T}}
\newcommand{\bG}{\mathbf{G}}
\newcommand{\cH}{{\mathcal{H}}}

\newcommand{\bp}{{\mathbf{p}}}
\newcommand{\bmm}{{\mathbf{m}}}
\newcommand{\me}{{\mathbf{me}}}

\newcommand{\p}{{\mathbf{p}}}

\newcommand{\sW}{{\mathscr{W}}}
\newcommand{\sC}{{\mathscr{C}}}
\newcommand\sS{{\mathscr S}}
\newcommand\sP{{\mathscr P}}

\newcommand{\bGam}{{\bm \Gamma}}
\newcommand{\bDel}{{\bm \Delta}}
\newcommand\bmo{{\bm m}_0}
\newcommand{\cM}{{\mathcal{M}}}
\newcommand{\bU}{{\mathbf{U}}}

\newcommand{\phii}{{\bm{\varphi}}}
\newcommand{\Phii}{{\bm{\Phi}}}

\newcommand{\cU}{{\mathcal{U}}}
\newcommand{\cV}{{\mathcal{V}}}

\newcommand{\cL}{{\mathcal{L}}}
\newcommand{\cK}{{\mathcal{K}}}

\newcommand{\bE}{{\mathbf{E}}}

\newcommand{\bB}{{\mathbf{B}}}
\newcommand{\bC}{{\mathbf{C}}}

\newcommand{\be}{{\mathbf{e}}}
\newcommand{\bd}{{\mathbf{d}}}

\newcommand\Z{{\mathbb Z}}

\newcommand\N{{\mathbb N}}

\newcommand\R{{\mathbb R}}

\newcommand{\eps}{{\varepsilon}}
\newcommand{\bA}{\mathbf{A}}

\def\Xint#1{\mathchoice
{\XXint\displaystyle\textstyle{#1}}%
{\XXint\textstyle\scriptstyle{#1}}%
{\XXint\scriptstyle\scriptscriptstyle{#1}}%
{\XXint\scriptscriptstyle\scriptscriptstyle{#1}}%
\!\int}
\def\XXint#1#2#3{{\setbox0=\hbox{$#1{#2#3}{\int}$ }
\vcenter{\hbox{$#2#3$ }}\kern-.6\wd0}}
\def\mint{\Xint-}

\newcommand{\Lip}{{\rm {Lip}}}

\newcommand{\dist}{{\rm {dist}}}
\newcommand{\dv}{{\text {div}}}

\newcommand\Id{{\rm Id}\,}

\newcommand{\cA}{{\mathcal{A}}}
\newcommand{\cB}{{\mathcal{B}}}
\newcommand{\cG}{{\mathcal{G}}}

\newcommand{\cT}{{\mathcal{T}}}

\newcommand{\mass}{{\mathbf{M}}}

\newcommand\e{\mathbf{e}}

\def\Is#1{{\mathcal{A}}_{#1} (\R^{n})}
\newcommand{\Iqs}{{\mathcal{A}}_Q(\R^{n})}
\newcommand{\Iq}{{\mathcal{A}}_Q}
\def\a#1{\left\llbracket{#1}\right\rrbracket}

\newcommand{\D}{\textup{Dir}}
\newcommand{\de}{\partial}
\newcommand{\xii}{{\bm{\xi}}}
\newcommand{\ro}{{\bm{\rho}}}

\newcommand{\etaa}{{\bm{\eta}}}

\newcommand{\ex}{{\mathbf{ex}}}

\newcommand\B{{\mathbf{B}}}
\newcommand{\bh}{\mathbf{h}}
\newcommand{\breg}{\mathrm{Reg}_b}
\newcommand{\bsing}{\mathrm{Sing}_b}

\title[An Allard-type boundary regularity theorem]{An Allard-type boundary regularity theorem for $2d$ minimizing currents at smooth curves with arbitrary multiplicity}

\author[C. De Lellis]{Camillo De Lellis}
\author[S. Nardulli]{Stefano Nardulli}
\author[S. Steinbr\"uchel]{Simone Steinbr\"uchel}

\begin{document}

\maketitle

\begin{abstract}
We consider integral area-minimizing $2$-dimensional currents $T$ in $U\subset \mathbb R^{2+n}$ with $\partial T = Q\a{\Gamma}$, where $Q\in \mathbb N \setminus \{0\}$ and $\Gamma$ is sufficiently smooth. We prove that, if $q\in \Gamma$ is a point where the density of $T$ is strictly below $\frac{Q+1}{2}$, then the current is regular at $q$. The regularity is understood in the following sense: there is a neighborhood of $q$ in which $T$ consists of a finite number of regular minimal submanifolds meeting transversally at $\Gamma$ (and counted with the appropriate integer multiplicity). In view of well-known examples, our result is optimal, and it is the first nontrivial generalization of a classical theorem of Allard for $Q=1$. As a corollary, if $\Omega\subset \mathbb R^{2+n}$ is a bounded uniformly convex set and $\Gamma\subset \partial \Omega$ a smooth $1$-dimensional closed submanifold, then any area-minimizing current $T$ with $\partial T = Q \a{\Gamma}$ is regular in a neighborhood of $\Gamma$. 
\end{abstract}

\tableofcontents

\input{introduction}
\input{outline}

\input{convex-hull}

\input{tangent-cones}

\input{first-decomposition}

\input{multi-valued}
\input{Lipschitz-approximation}
\input{center-manifold}

\input{cm-proof}

\input{blow-up-1}
\input{blow-up-2}



\bibliographystyle{plain}
\bibliography{references}

\end{document}

%% file: introduction.tex
\section{Introduction} 

Consider an area-minimizing integral current $T$ of dimension $m\geq 2$ in $\mathbb R^{m+n}$ and assume that $\partial T$ is a smooth submanifold, namely $\partial T = \sum_i Q_i \a{\Gamma_i}$, where $Q_i$ are (positive) integer multiplicites and $\Gamma_i$ finitely many pairwise disjoint oriented smooth and connected submanifolds of dimension $m-1$. The present paper is focused on understanding how regular $T$ can be at points $p\in \cup_i \Gamma_i$ and our primary interest is that the integer multiplicities are allowed to be larger than $1$ and the codimension $n$ is at least $2$. Indeed, when the codimension is $1$ the situation is completely understood (cf. \cite[Problem 4.19]{Collection}): first of all the coarea formula for functions of bounded variation allows to decompose, locally, the current $T$ into a sum of area minimizing integral currents which take the boundary with multiplicity $1$; hence we can apply to each piece of the decomposition the celebrated theorem by Hardt and Simon \cite{HS}, which guarantees full regularity at the boundary, namely the absence of any singularity. 

A quite general boundary regularity theory was developed by Allard in the pioneering fundamental work \cite{AllB}, which covers any dimension and codimension and is valid for more general objects, namely stationary varifolds. In \cite{AllB} Allard restricts his attention to boundary points where the density, namely the limit of the mass ratio
\[
\Theta (T,q):=\lim_{r\downarrow 0} \frac{\|T\| (\bB_\rho (q))}{\rho^m}\, ,
\]
is sufficiently close to $\frac{1}{2}$. His Boundary Regularity Theorem guarantees then that, under such assumption, $q$ is always a regular point. Indeed this generalizes a similar statement in his PhD thesis \cite{AllPhD}, which covered the case of area minimizing currents in codimension $1$.  

In the introduction to \cite{AllPhD} Allard points out that when the multiplicity of the boundary $\Gamma$ is allowed to be an arbitrary natural number $Q>1$, the assumption $\Theta (T,q) < \frac{1}{2}+\varepsilon$ is empty and should be replaced by $\Theta (T,q)<\frac{Q}{2}+ \varepsilon$. However he quotes a possible extension of his theorem as a very challenging problem. This basic question was raised again by White in the collection of open problems \cite{Collection}, cf. Problem 4.19, where he also explains that the nontrivial situation is in higher codimension, given the decomposition through the coarea formula already explained a few paragraphs above. Our paper gives the very first result in that direction and solves Allard's ``higher multiplicity'' question for $2$-dimensional integral currents.
Before stating it we wish to discuss what we mean by ``regularity at the boundary''. 

\begin{definition}\label{d:regular_and_singular}
Assume $T$ is an area minimizing $2$-dimensional integral current in $U\subset \mathbb R^{2+n}$ such that $\partial T \res U = Q \a{\Gamma}$ for some integer $Q\geq 1$ and some $C^1$ embedded arc $\Gamma$. $p$ is called a \textbf{\emph{regular boundary point}} if $T$ consists, in a neighborhood of $p$, of the union of finitely many smooth submanifolds with boundary $\Gamma$, counted with appropriated integer multiplicities, which meet at $\Gamma$ transversally. More precisely, if there are:
\begin{itemize}
    \item[(i)] a neighborhood $U$ of $p$;
    \item[(ii)] a finite number $\Lambda_1, \ldots , \Lambda_J$ of $C^1$ oriented embedded $2$-dimensional surfaces in $U$;
    \item[(iii)] and a finite number of positive integers $k_1, \ldots , k_J$ 
\end{itemize}
such that:
\begin{itemize}
\item[(a)] $\partial \Lambda_j \cap U = \Gamma\cap U = \Gamma_i \cap U$ (in the sense of differential topology) for every $j$;
\item[(b)] $\Lambda_j \cap \Lambda_l = \Gamma\cap U$ for every $j\neq l$;
\item[(c)] for all $j\neq l$ and at each $q\in \Gamma$ the tangent planes to $\Lambda_j$ and $\Lambda_l$ are distinct;
\item[(d)] $T \res U = \sum_j k_j \a{\Lambda_j}$ (hence $\sum_j k_j = Q_i$).
\end{itemize}
The set $\breg (T)$ of boundary regular points is a relatively open subset of $\Gamma$ and its complement in $\Gamma$ will be denoted by $\bsing (T)$. 
\end{definition}

Our main Theorem reads as follows.

\begin{theorem}\label{t:Allard-general}
Let $U\subset \mathbb R^{n+2}$ be an open set, $\Gamma\subset U$ be a $C^{3, \alpha_0}$ embedded arc for some $\alpha_0>0$, and $T$ be a $2$-dimensional area-minimizing integral current such that $\partial T = Q \a{\Gamma}$. If $q\in \Gamma$ and $\Theta (T,q)< \frac{Q+1}{2}$, then $T$ is regular at $q$ in the sense of Definition \ref{d:regular_and_singular}.
\end{theorem}

\begin{remark}\label{r:optimality}
Note that it is well known that there are smooth curves (counted with multiplicity $1$) in the Euclidean space, even in $\mathbb R^3$, which span more than one area-minimizing current. In particular, if $\Gamma\subset \mathbb R^3$ is such a curve and $T_1$, $T_2$ two area minimizing currents with $\partial T_i = \a{\Gamma}$, $i=1,2$, then $T:=T_1+T_2$ is an area minimizing current with $\partial T = 2\a{\Gamma}$ (this follows because any area-minimizing current $S$ with boundary $\partial S = 2 \a{\Gamma}$ must have mass which doubles that of $T_i$, and hence equals that of $T$). Let us analyze the above example more accurately. In view of the interior and boundary regularity theory, both $T_1$ and $T_2$ are smooth submanifolds up to the boundary, i.e. a standard argument using Allard's boundary regularity theorem \cite{AllB} (cf. \cite[Section 5.23]{Almgren}) implies that $T_i = \a{\Lambda_i}$ for two connected smooth submanifolds such that $\partial \Lambda_i = \Gamma$ in the classical sense of differential topology. Since any integral area-minimizing $2$-dimensional current in $\mathbb R^3$ is an embedded submanifold (with integer multiplicity) away from the boundary, we also conclude that $\Lambda_1$ and $\Lambda_2$ do not intersect except at their common boundary $\Gamma$. The Hopf boundary lemma then implies that at every point $p\in \Gamma$ the two currents have distinct tangents, i.e. $\Lambda_1$ and $\Lambda_2$ meet at their common boundary {\em transversally}. 
\end{remark}

In view of the above remark we cannot expect, in general, a ``better'' conclusion than the one of Theorem \ref{t:Allard-general} or, in other words, we cannot expect that the number $J$ in Definition \ref{d:regular_and_singular} is $1$. However, an obvious corollary of Theorem \ref{t:Allard-general} is the following. 

\begin{theorem}\label{t:Allard-flat}
Let $U, T, \Gamma$ and $q$ be as in Theorem \ref{t:Allard-general}. Then there is a neighborhood $U'$ of $q$ in which $T = Q \a{\Lambda}$ for some smooth minimal surface $\Lambda$ if and only if one tangent cone to $T$ at $q$ is ``flat'', i.e. contained in a $2$-dimensional linear subspace of $\mathbb R^{2+n}$. 
\end{theorem}

Even though the assumption that $\Theta (T, q)$ is sufficiently close to $\frac{Q}{2}$ seems, at a first glance, very restrictive, we can either follow a lemma of Allard in \cite{AllB} (valid in any dimension and codimension) or a simple classificaton of the boundary tangent cones (cf. \cite{DNS}) to show that it holds when $\supp (\partial T)$ is contained in the boundary of a bounded $C^2$ uniformly convex set $\Omega$. 
For this reason, complete regularity can be achieved when there is a ``convex barrier''. Since this is an assumption which will be used often in some sections of the work, we wish to isolate its statement. 

\begin{ipotesi}\label{a:main}
$\Omega\subset \mathbb R^{2+n}$ is a bounded $C^{3,\alpha_0}$ uniformly convex set for some $\alpha_0 >0$, $\Gamma\subset \partial \Omega$ is the disjoint union of finitely many $C^{3,\alpha_0}$ simple closed curves  $\{\Gamma_i\}_{i=1, \ldots , N}$. $T$ is a $2$-dimensional area-minimizing integral current in $\mathbb R^{2+n}$ such that $\partial T = \sum_i Q_i \a{\Gamma_i}$.
\end{ipotesi}

\begin{theorem}\label{t:main}
Let $\Gamma$, $\Omega$ and $T$ be as in Assumption \ref{a:main}. Then $\bsing (T)$ is empty.
\end{theorem}

In fact we can give a suitable local version of the above statement from which Theorem \ref{t:main} can be easily concluded, cf. Theorem \ref{t:main-local}. 

\medskip

In the next section we will outline the arguments to prove Theorem \ref{t:Allard-general}, \ref{t:Allard-flat}, and \ref{t:main}. Before coming to it we wish to point two things. We are confident that the methods used in this work generalize to cover the same statement as in Theorem \ref{t:Allard-general} in an arbitrary smooth (i.e. $C^{3, \alpha_0}$) complete Riemannian manifold, but in order to keep the technicalities at bay we have decided to restrict our attention to Euclidean ambient spaces. Even though the basic ideas behind this work are quite simple, the overall proof of the theorems is quite lengthy. For instance before the recent paper \cite{DDHM} of the first author, joint with De Philippis, Hirsch, and Massaccesi, not even the existence of a single boundary regular point was known, without some convex barrier assumption and in a general Riemannian manifold. Part of the challenge is that 
several crucial PDE ingredients are absent in codimension higher than $1$. Let us in particular mention three facts:
\begin{itemize}
\item[(a)] There is no ``soft'' decomposition theorem which allows to reduce the general case to that of multiplicity $1$ boundaries;
\item[(b)] Boundary singularities occur even in the case of multipliciy $1$ smooth boundaries;
\item[(c)] There is no maximum principle (and in particular no Hopf boundary lemma) available even if we knew apriori that the minimizing currents are completely smooth.
\end{itemize}

\subsection{Acknowledgments} C.D.L. acknowledges support from the National Science Foundation through the grant  FRG-1854147. The second author would like to thank Fapesp for financial support via the grant ``Bolsa de Pesquisa no Exterior'' number 2018/22938-4. 

%% file: outline.tex
\section{Outline of the proof}

In the first step (cf. Section \ref{s:wedges}), we use the classical convex hull property to reduce the statement of Theorem \ref{t:main} to a local version, cf. Theorem \ref{t:main-local}. The latter statement will then focus only on a portion of the boundary, but under the assumption that the support of the current is contained in a suitable convex region, cf. Assumption \ref{a:main-local}. The crucial point is that this convex region forms a ``wedge'' at each point of the boundary, cf. Definition \ref{d:wedge}.  

In the second step (cf. Section \ref{s:cones}) we recall the classical Allard's monotonicity formula and we appeal to a classification result for $2$-dimensional area-minimizing integral cones with a straight boundary (see \cite{DNS}) to conclude that, in all the cases we are dealing we can assume, without loss of generality, that all the tangent cones to $T$ at every boundary point $p$ consist of a finite number of halfplanes with common boundary $T_p \Gamma$, counted with a positive integer multiplicity, cf. Theorem \ref{t:classification_cones}.  

At this point, taking advantage of pioneering ideas of White, cf. \cite{Wh}, and of a recent paper by Hirsch and Marini, cf. \cite{HM}, the tangent cone can be shown to be unique at each point $p\in \Gamma$. We need, strictly speaking, a suitable generalization of \cite{HM}, but the simple technical details are given in the shorter paper \cite{DNS}. This uniqueness result has two important outcomes:
\begin{itemize}
    \item[(a)] At any point $p\in \Gamma$ where the tangent cone is not {\em flat} (i.e. it is not contained in a {\em single} half-plane) we can decompose the current into simpler pieces, cf. Theorem \ref{t:decomposition};
    \item[(b)] the convergence rate of the current to the cone is polynomial (cf. also Corollary \ref{c:cylindrical_excess_decay}.
\end{itemize}
Point (a) reduces all our regularity statement to Theorem \ref{t:Allard-flat}. In fact we will focus on a slightly more technical version of it, cf. Theorem \ref{t:flat-points} Point (b) gives one crucial piece of information which will allow us to conclude Theorem \ref{t:flat-points}. The remaining part of this work will in fact be spent to argue for Theorem \ref{t:flat-points} by contradiction: if a flat boundary point $p$ is singular, then the convergence rate to the flat tangent cone at $p$ must be {\em slower} than polynomial, contradicting thus (b).

\medskip

We first address a suitable linearized version of Theorem \ref{t:flat-points}: we introduce multivalued functions and define the counterpart of flat boundary points in that context, which are called {\em contact points}. In Theorem \ref{t:contact-points}, we then prove an analog of Theorem \ref{t:flat-points} in the case of multivalued functions minimizing the Dirichlet energy using a version of the frequency function (see Definition \ref{Def:distance}) first introduced by Almgren. However, while the proof of Theorem \ref{t:contact-points} might be instructive to the reader because it illustrates, in a very simplified setting, the idea behind the ``slow decay'' at singular points, the crucial fact which will be used to show Theorem \ref{t:flat-points} is contained in Theorem \ref{t:I=1}: the latter states that, if a multi-function vanishes identically at a straight line and it is $I$-homogeneous, either it is a multiple copy of a single classical harmonic function, or the homogeneity equals $1$. 

The overall idea is that, if $p$ is a singular flat point, then it can be efficiently approximated at small scales by an homogeneous harmonic (i.e. Dirichlet minimizing) multivalued function as above (not necessarily unique), which however cannot be a multiple copy of a single classical harmonic function. Since the homogeneity of the latter will be forced to be $1$, we will infer from it the slow decay of the ``cylindrical excess'' (cf. Definition \ref{d:excess}). However, the work to accomplish the latter approximation proves to be quite laborious and it will pass through a {\em series} of more and more refined approximations.

\medskip

First of all, in the Sections \ref{s:Lip-approx}, \ref{s:harmonic-approx}, \ref{s:higher-int-estimate}, 
and \ref{s:strong-Lipschitz} we prove that the current can be efficiently approximated by multivalued Lipschitz functions when sufficiently flat (cf. Theorem \ref{t:strong_Lipschitz}) and that the latter approximation almost minimizes the Dirichlet energy (cf. Theorem \ref{t:o(E)}). These sections take heavily advantage of the tools introduced in \cite{DS2,DS3} and of some ideas in \cite{DDHM}. However these approximations are not sufficient to carry on our program. 

A new refined approximation is then devised in Section \ref{s:center-manifold}. At every sufficiently small scale we can construct a ``center manifold'' (i.e. a classical $C^3$ surface with boundary $\Gamma$) and a multivalued Lipschitz approximation over its normal bundle (called {\em normal approximation}), which approximates the current as efficiently as the ``straight'' approximation in Theorem \ref{t:strong_Lipschitz}, cf. Theorem \ref{t:cm} and Theorem \ref{t:normal-approx} for the relevant statements. This new normal approximation has however two important features:
\begin{itemize}
    \item[(i)] It approximates the current well not only at the ``starting scale'' but also across smaller scales as long as certain decay conditions are ensured.
    \item[(ii)] At all such scales the normal approximation has average close to $0$ (namely it is never close to a multiple copy of a single harmonic function, compared to its own Dirichlet energy).
\end{itemize}
The Sections \ref{s:tilting}, \ref{s:interpolating}, and \ref{s:cm-final} provide a proof of Theorem \ref{t:cm} and Theorem \ref{t:normal-approx}. While the first center manifold was introduced in the monograph \cite{Almgren} by Almgren, our constructions borrows from the ideas and tools introduced in \cite{DS4} and \cite{DDHM}. 

Our proof would be at this point much easier if the validity of (ii) above would hold, around the given singular flat point $p$, at {\em all} scales smaller than the one where we start the construction of the center manifold. Unfortunately we do not know how to achieve this. We are therefore forced to construct a sequence of center manifolds which cover different sets of scales, cf. again Section \ref{s:flattening}. At certain particular scales we need therefore to change approximating maps, i.e. to pass from one center manifold to the next. Section \ref{s:stopping-scales} provides then important information about the latter ``exchange scales''. Both sections are heavily influenced by similar considerations made in the papers \cite{DS4,DS5}. 

The remaining parts of the paper are thus focused to show that, at a sufficiently small scale around the flat point $p$, all these normal approximations are close to some homogeneous Dir-minimizing function (not necessarily the same across all scales), which by Theorem \ref{t:I=1} will then result to be $1$-homogeneous. The key ingredient to show this homogeneity is the almost monotonicity of the frequency function of the normal approximation (a celebrated quantity introduced by Almgren in his pioneering work \cite{Almgren}). In order to deal with the boundary we resort to an important variant introduced in \cite{DDHM}. The key point is to show that, as $r\downarrow 0$, the frequency function $I (r)$ of the approximation at scale $r$ converges to a limit. However, since our approximation might change at {\em some particular} scales, the function $I$ undergoes a possibly infinite number of jump discontinuities, while it is almost monotone in the complement of these discontinuities. In order to show that the limit exists we thus need:
\begin{itemize}
    \item[(1)] a suitable quantification of the monotonicity on each interval delimited by two consecutive discontinuities;
    \item[(2)] a suitable bound on the series of the absolute values of such jumps.
\end{itemize}
The relevant estimates, namely \eqref{e:Gronwall} and \eqref{e:jump_estimate}, are contained in Theorem \ref{t:frequency}. While the proof of \eqref{e:Gronwall} takes advantage of similar cases handled in \cite{DS5} and \cite{DDHM}, \eqref{e:jump_estimate} is entirely new and we expect that the underlying ideas behind it will prove useful in other contexts. The Sections \ref{s:proof_monotonicity_frequency} and \ref{s:proof_jump_estimate} are dedicated to prove the respective estimates.

Finally, in Section \ref{s:blowup_and_conclusion} we carry on the (relatively simple) argument which, building upon all the work of the previous sections, shows that the rate of convergence to the tangent cone at a singular flat point must to be {\em slower} than any polynomial rate. As already mentioned, since the convergence rate has to be polynomial at {\em every} point, this shows that a singular flat point cannot exist.

%% file: convex-hull.tex
\newpage
\section{Convex hull property and local statement}\label{s:wedges}

We start recalling the following well known fact:

\begin{proposition}\label{p:convex-hull}
Assume $T$ is an area minimizing $m$-dimensional current in $\mathbb R^{m+n}$ with $\supp (\partial T)$ compact. Then $\supp (T)$ is contained in the convex hull of $\supp (\partial T)$.
\end{proposition}

\begin{proof}
The statement can be concluded from much stronger ones, for instance we can use that $\|T\|$ is an integral stationary varifold in $\mathbb R^{m+n}\setminus \supp (T)$ and invoke \cite[Theorem 19.2]{Sim}.
\end{proof}

We then take advantage of a simple and elementary fact which combines the regularity of $\Gamma$ with the uniform convexity of the barrier $\Omega$. We will state this fact in higher generality than we actually need in this manuscript.

\begin{definition}\label{d:wedge}
First of all, given an $(m-1)$-dimensional plane $V\subset \mathbb R^{m+n}$ we denote by $\p_V$ the orthonogonal projection onto $V$.
Given additionally a unit vector $\nu$ normal to $V$ and an angle $\vartheta \in (0, \frac{\pi}{2})$ we then define the \textbf{\emph{wedge with spine $V$, axis $\nu$ and opening angle $\vartheta$}} as the set
\begin{equation}\label{e:wedge}
W (V, \nu, \vartheta):= \big\{y: |y-\p_V (y) - (y\cdot \nu) \nu | \leq (\tan \vartheta) y \cdot \nu\big\}\, .
\end{equation}
\end{definition}

\begin{figure}[htp]
    \centering
    \includegraphics[width=10cm]{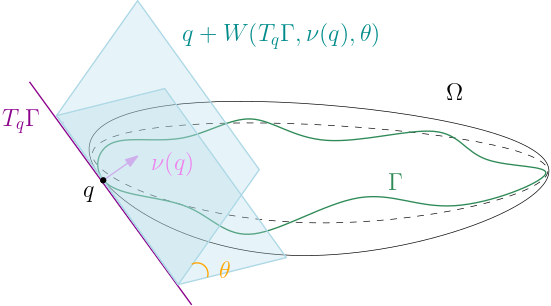}
    \caption{An illustration of the wedge where $V$ is the tangent $T_q \Gamma$ to $\Gamma$ at some boundary point $q$, whereas $\nu$  the interior unit normal $\nu (q)$ to the convex barrier $\Omega$ at $q$.}
\end{figure}

In particular we have the following lemma.

\begin{lemma}\label{l:wedge-lemma}
Let $\Omega\subset \mathbb R^{m+n}$ be a $C^2$ bounded open set with uniformly convex boundary and $\Gamma$ a $C^2$ $(m-1)$-dimensional submanifold of $\Omega$ without boundary. Then there is a $0<\vartheta < \frac{\pi}{2}$ (which depends only on $\Gamma$ and $\Omega$) such that the convex hull of $\Gamma$ satisfies
\[
{\rm ch}\, (\Gamma) \subset \bigcap_{q\in \Gamma} (q+ W (T_q \Gamma, \nu (q), \vartheta))\, .
\]
\end{lemma}
 
 We postpone the proof of the lemma to the end of the section
Using Proposition \ref{p:convex-hull} and Lemma \ref{l:wedge-lemma} we can reduce Theorem \ref{t:main} to a suitable local statement. In particular we will replace Assumption \ref{a:main} with the following one:

\begin{ipotesi}\label{a:main-local}
$Q\geq 1$ is an arbitrary integer and $\vartheta$ a given positive real number smaller than $\frac{\pi}{2}$. $\Gamma$ is a $C^{3,\alpha}$ arc in $U = \bB_1 (0) \subset \mathbb R^{2+n}$ with endpoints lying in $\partial \bB_1 (0)$\footnote{I.e. $\Gamma=\hat{\gamma}([0,1])$ where $\hat{\gamma}:[0,1]\to\overline{\bB_1(0)}$ is a $C^{3,\alpha}$ diffeomorphism onto its image.}. Moreover $\nu: \Gamma \to \mathbb S^{n+1}$ is a $C^{2,\alpha}$ map such that $\nu(q)\perp T_q\Gamma$. $T$ is a $2$-dimensional area-minimizing integral current in $U$ such that:
\begin{align} 
&(\partial T) \res U= Q \a{\Gamma},\\
&\supp (T) \subset \bigcap_{q\in \Gamma} (q+W ( T_q \Gamma, \nu (q), \vartheta))\, .\label{e:inside-wedges}
\end{align}
Moreover, 
\begin{equation}
\bA := \|\kappa\|_{L^\infty} + \|\dot\nu\|_{L^\infty}\leq 1\, , 
\end{equation}
where $\kappa$ denotes the curvature of $\Gamma$ and $\dot \nu$ is the derivative, in the arclength parametrization, of $\nu$. 
\end{ipotesi}

\begin{theorem}\label{t:main-local}
Let $\Gamma$ and $T$ be as in Assumption \ref{a:main-local}. Then $\bsing (T)$ is empty.
\end{theorem}

\begin{proof}[Proof of Lemma \ref{l:wedge-lemma}] Since $q+ W (V, \nu, \vartheta)$ is a convex set, we just need to show the existence of a $0<\vartheta<\frac{\pi}{2}$ such that $\Gamma \subset (q+W (T_q \Gamma, \nu, \vartheta))$ for every $q\in \Gamma$. 
The latter is equivalent to show the existence of a constant $C>0$ such that
\begin{equation}\label{e:convexity-1}
 |(p-q) - ((p-q) \cdot \nu (q)) \nu (q) - \p_V (p-q)| \leq C ((p-q) \cdot \nu (q))
\qquad \forall p,q\in \Gamma\, .
\end{equation}
The strict convexity of $\partial \Omega$ ensures that for every $\varepsilon>0$ there is a constant $C$ such that \eqref{e:convexity-1} holds if additionally $|p-q|\geq \varepsilon$. Thus we just have to show the inequality for a sufficiently small $\varepsilon$. In order to do that, fix $q$ and assume w.l.o.g. that it is the origin, while at the same time we assume that $T_q \Gamma = T_0 \Gamma = \{x_m= \ldots = x_{m+n} =0\}$ and $\nu = \frac{\partial}{\partial x_{m+n}}$. We will use accordingly the coordinates $(y,z,w)$, with $y\in \mathbb R^{m-1}$, $z\in \mathbb R^n$, and $w\in \mathbb R$. By the $C^2$ regularity of $\Omega$ and $\Gamma$, in a sufficiently small ball $\bB_\varepsilon (q) = \bB_\varepsilon (0)$ the points $p$ in $\Gamma$ are described by
\begin{equation}\label{e:formula}
p = (y,z,w) = (y, f (y), g (y,f(y)))
\end{equation}
for some $f$ and $g$ which are $C^2$ functions. Observe that $f(0)=0$, $Df (0) =0$, $g(0)=0$, and $Dg (0) =0$. Moreover $\|D^2f\|_{C^0} \leq C_0$ and $D^2 g \geq c_0 {\rm Id}$ for constants $c_0>0$ and $C_0$, which depend only on $\Gamma$ and $\Omega$. Similarly, the size of the radius $\varepsilon$ in which the formula \eqref{e:formula} and the estimates are valid depends only on $\Omega$ and $\Gamma$ and not on the choice of the point $q$. Next, compute
\[
((p-q) \cdot \nu (q)) = g (y, f(y)) \geq c_0 (|y|^2 + |f(y)|^2) \geq c_0 |y|^2
\]
and
\[
|(p-q) - ((p-q) \cdot \nu (q)) \nu (q) - \p_V (p-q) | = |f(y)| \leq C_0 |y|^2\, .
\]
The desired inequality is then valid for $C:= \frac{C_0}{c_0}$. 
\end{proof}

%% file: tangent-cones.tex
\section{Tangent cones}\label{s:cones}

We start recalling Allard's boundary monotonicity formula. More specifically, we first define

\begin{definition}\label{def:density}
For every point $p\in \bB_1$, we define the \textbf{\emph{density of $T$ at the point $p$}}\index{Density of $T$ at some point}
\[
\Theta(T,p):=\lim_{r\downarrow 0}\frac{\|T\|(\bB_r(p))}{\pi r^2}\, ,
\]
whenever the latter limit exists.
\end{definition}

Next, we introduce the notation $\kappa$ for the curvature of $\Gamma$ and we consider the functions $\Theta_{\rm i} (T,p,r)$ and $\Theta_{\rm b} (T,p,r)$ given by
\begin{align}
\Theta_{\rm i} (T,p,r):=\; & \frac{\|T\|(\bB_r(p))}{\pi r^2}\,, \\
\Theta_{\rm b} (T, p,r):=\; &\exp \left( C_0 \|\kappa\|_0 r\right)\frac{\|T\|(\bB_r(p))}{\pi r^2}\, ,
\end{align}
where $C_0=C_0 (n)$ is a suitably large constant.

\begin{theorem}\label{thm:allard}
Let $T$ be as in Assumption \ref{a:main-local}.
\begin{itemize}
\item[(a)] If $p\in \bB_1 \setminus \Gamma$, then $r\mapsto \Theta_{\rm i} (T,p,r)$ is monotone on $(0,\min \{\dist (p, \Gamma), 1-|p|\})$,
\item[(b)] if $p\in \bB_1 \cap \Gamma$, then $r\mapsto\Theta_{\rm b} (T,p,r)$ is monotone on $(0,1-|p|)$.
\end{itemize}
Thus  the density exists at every point of $\bB_1$. Moreover, the restrictions of the map $p\mapsto \Theta (T,p)$ to $\Gamma\cap \bB_1$ and to $\bB_1\setminus \Gamma$ are both upper semicontinuous. 

If $X\in C^1_c (\bB_1, \R^{2+n})$, then the first variation of $T$ with respect to $X$ satisfies 
\begin{equation}\label{e:first_var}
\delta T (X) = Q \int_\Gamma X\cdot \vec{n} (x)\, d\cH^1 (x)\,  
\end{equation}
where $\vec n$ is a Borel vector field with $|\vec n|\leq 1$. 

Moreover, if $p\in \Gamma$ and $0<s<r < 1-|p|$, we then have the following precise monotonicity identity
\begin{align}
&\;r^{-2} \|T\| (\bB_r (p)) - s^{-2} \|T\| (\bB_s (p)) - \int_{\bB_r (p)\setminus \bB_s (p)} \frac{|(x-p)^\perp|^2}{|x-p|^4}\, d\|T\| (x)\nonumber\\
= &\; Q \int_s^r \int_{\Gamma \cap \bB_\rho (p)} (x-p)\cdot \vec{n} (x)\, d\cH^{1} (x)\, d\rho\, ,
\label{e:monot_identity}
\end{align}
where $Y^\perp (x)$ denotes the component of the vector $Y (x)$ orthogonal to the tangent plane of $T$ at $x$ (which is oriented by $\vec{T} (x)$). 
\end{theorem}

Note that $\delta T (X) =0$ for $X\in C^1_c (\bB_1\setminus \Gamma)$ follows in a straightforward way from the minimality property of $T$. In particular $\|T\|$ is a stationary integral varifold in $\bB_1\setminus \Gamma$ and (a) and (b) are consequences of the celebrated works of Allard, cf. \cite{All} and \cite{AllB}. Next note  that \eqref{e:first_var} follows from \eqref{e:monot_identity} arguing, for instance, as in \cite{Theodora} for \cite[Eq. (31)]{Theodora} (see \cite{All,AllB} as well). Coming to \eqref{e:first_var}, note first that the derivation of \cite[(3.8)]{DDHM} is valid under our assumptions, with the additional information $\delta T = \delta T_s$ (following the terminology and notation of \cite[Section 3]{DDHM}). We then just need to show that $\|\delta T_s\|\leq Q \cdot \mathcal{H}^1 \res \Gamma$. The latter follows easily arguing as in \cite[Section 3.4]{DDHM} once we have shown that $\Theta (T, p) = \frac{Q}{2}$ at every $p\in \Gamma$, see below. 

As in \cite[Section 3]{DDHM} we introduce the following notation and terminology.

\begin{definition}\label{def:tc}
Fix a point $p\in\supp(T)$ and define for all $r>0$
\[
\iota_{p,r}(q):=\frac{q-p}{r}\,.
\]
We denote by $T_{p,r}$ the currents
\[
T_{p,r}:=(\iota_{p,r})_\sharp T\,.
\]
We call the current $T_{p,r}$ the \textbf{\emph{blow up at the point $p$ and scale $r$ of $T$}}. Let $T_0$ be a current such that there exists a sequence $r_k\to0$ of radii such that $T_{p,r_k}\to T_{0}$, we say that \textbf{\emph{$T_0$ is a tangent cone to $T$ at $p$}}.  
\end{definition}
We recall the following consequence of the Allard's monotonicity formula, cf. \cite{AllB}.

\begin{theorem}\label{thm:ex_tc}
Let $T$ be as in Assumption \ref{a:main-local} or as in Theorem \ref{t:Allard-general}. Fix $p\in\supp(T)$ and take any sequence $r_k\downarrow 0$. Up to subsequences $T_{p,r_k}$ is converging locally in the sense of currents to an area-minimizing integral current $T_0$ 
\begin{itemize}
\item[(a)] $T_0$ is a cone with vertex $0$ and $\|T_0\|(\bB_1(0))=\pi \Theta(T,p)$;
\item[(b)] if $p\in\supp\, (T)\setminus\Gamma$, then $\partial T_0=0$;
\item[(c)] if $p\in\Gamma$, then $\partial T_0=Q \a{T_p \Gamma}$.
\end{itemize}
Moreover $\|T_{p,r_k}\|$ converges, in the sense of measures, to $\|T_0\|$. 
\end{theorem}

We next show the following elementary fact:

\begin{theorem}\label{t:classification_cones}
Let $T$ be as in Assumption \ref{a:main-local} and $p\in \Gamma$. Any tangent cone $T_0$ at $p\in \Gamma$ has then the following properties:
\begin{itemize}
\item[(a)] $\supp (T_0)$ is contained in $W(T_p \Gamma, \nu (p), \vartheta)$ (where $\nu (p)$ and $\vartheta$ are the vector and the constant given in Assumption \ref{a:main-local});
\item[(b)] There are $k_1, \ldots k_N\in \mathbb N\setminus \{0\}$ and $2$-dimensional distinct oriented half-planes $V_1, \ldots , V_N$ with $\partial \a{V_i} = \a{T_p \Gamma}$ such that
\begin{equation}\label{e:description_cone}
T_0 = \sum_i k_i \a{V_i}\, .
\end{equation}
\end{itemize}
Note in particular that $2 \Theta (T, p) = Q = \sum_i k_i$, and thus $1\leq N \leq Q$. 

Conclusion (b) holds under the assumptions of Theorem \ref{t:Allard-general} provided we choose $p$ sufficiently close to $q$.
\end{theorem}

The first part of the theorem is in fact at the same time a particular case of a more general theorem of Allard in higher dimensions (under Assumption \ref{e:inside-wedges}) and of a general classification of all $2$-dimensional area-minimizing cones with $\partial T_0 = Q \a{\ell}$, where $\ell$ is a straight line, given \cite{DNS}. In particular since point (a) is obvious, point (b) is a direct corollary of \cite[Proposition 4.1]{DNS} and of (a). As for the second part of the statement, observe that, by \cite[Proposition 4.1]{DNS}, $2 \Theta (T,p)$ is always an integer no smaller than $Q$. Recalling that $\Gamma \ni p \mapsto \Theta (T,p)$ is upper semicontinuous, under the assumptions of Theorem \ref{t:Allard-general} we must necessarily have $\Theta (T,P)=\frac{Q}{2}$ for every $p$ sufficiently close to $q$. Then conclusion (b) follows again from \cite[Proposition 4.1]{DNS}. Since it will be useful later, we introduce a notation for the cones as in \eqref{e:description_cone}.

\begin{definition}
Let $\ell\subset \mathbb R^{2+n}$ be a $1$- dimensional line passing through the origin and let $Q\in \mathbb N\setminus \{0\}$. We denote by $\mathscr{B}_{Q} (\ell)$ the set of area minimizing cones of the form $T = \sum_{i=1}^N k_i \a{V_i}$, for any finite collection of distinct half-planes $V_i$ such that $\partial \a{V_i}= \a{\ell}$ and any finite collection of positive integers $\{k_i\}_{i=1}^N$ such that $\sum_{i=1}^N k_i = Q$. Moreover we will call such cones \textbf{\emph open books}.
\end{definition}

%% file: first-decomposition.tex
\section{Uniqueness of tangent cones and first decomposition}\label{s:uniqueness}

In this section we appeal to \cite[Theorem 1.1]{DNS}, which follows the ideas of Hirsch and Marini in \cite{HM}, in order to claim that the tangent cone to $T$ at $p\in \Gamma$ is unique. 

\begin{theorem}\label{t:uniqueness}
Let $T$ and $\Gamma$ be as in Assumption \ref{a:main-local}. Then the tangent cone at each $p\in \Gamma$ is unique and from now on will be denoted by $T_{p,0}$. The same conclusion holds under the assumptions of Theorem \ref{t:Allard-general} provided $q$ is sufficiently close to $p$.
\end{theorem}

In fact such a uniqueness theorem comes with a power-law decay (cf. \cite[Theorem 2.1]{DNS}), which in turn allows us to decompose the current at any point $p\in \Gamma$ where the tangent cone is not contained in a {\em single} half-plane. Before coming to its statement, we introduce the following terminology.

\begin{definition}\label{d:flat}
Let $T$ and $\Gamma$ be as in Assumption \ref{a:main-local}. If the tangent cone $T_{p,0}$ to $T$ at $p\in \Gamma$ is of the form $Q\a{V}$ for some $2$-dimensional half-plane $V$, then $p$ is called a \textbf{{\em flat boundary point}}.
\end{definition}

\begin{theorem}[Decomposition]\label{t:decomposition}
Let $T$ and $\Gamma$ be:
\begin{itemize}
    \item either as in Assumption \ref{a:main-local},
    \item or either as in Theorem \ref{t:Allard-general}.
\end{itemize}
Assume that $p\in \Gamma$ is {\em not a } flat boundary point and in the second case assume further that $p$ is sufficiently close to $q$. Then there is $\rho>0$ with the following property. There are two positive integers $Q_1$ and $Q_2$ and two area-minimizing currents $T_1$ and $T_2$ in $\bB_\rho (p)$ such that:
\begin{itemize}
\item[(a)] $T_1+T_2 = T\res \bB_\rho (p)$ (thus $Q_1+Q_2=Q$),
\item[(b)] $\partial T_i \res \bB_\rho (p) = Q_i \a{\Gamma\cap \bB_\rho (p)}$,
\item[(c)] $\supp (T_1) \cap \supp (T_2) = \Gamma \cap \bB_\rho (p) $,
\item[(d)] at each point $q\in \bB_\rho (p)$ the tangent cones to $T_1$ and $T_2$ have only the line $T_q \Gamma$ in common, i.e., $(T_1)_{q,0}\in\mathscr{C}_{min,Q_1}(T_q \Gamma)$ and to $(T_2)_{q,0}\in\mathscr{C}_{min,Q_2}(T_q \Gamma)$.
\end{itemize} 
\end{theorem}

At flat points we are not able to decompose the current further and in fact the final byproduct of the regularity theory of this paper is that in a neighborhood of each flat point, the current is supported in a single smooth minimal sheet. For the moment the uniqueness of the tangent cones (and the corresponding decay from which we derive it) allows us to draw the following conclusion.

\begin{theorem}\label{t:sandwich}
Let $T$ and $\Gamma$ be as in Assumption \ref{a:main-local} or as in Theorem \ref{t:Allard-general}. Assume that $p\in \Gamma$ is a flat boundary point, that $Q \a{V}$ is the unique tangent cone of $T$ at $p$, and, in the case of Theorem \ref{t:Allard-general} that $p$ is sufficiently close to $q$. 
Let $n(p) \in V$ be the unit normal to $\Gamma$ at $p$ and define in a neighborhood of $p$
\begin{equation}\label{e:new-n}
n (q) = \frac{n(p) - n(p)\cdot \tau (q) \tau (q)}{|n(p) - n(p)\cdot \tau (q) \tau (q)|}    
\end{equation}
where $\tau$ is the unit tangent vector to $\Gamma$ orienting it.

Then, for every $\theta>0$ there is a $\rho>0$ such that
\begin{equation}
\supp (T) \cap \bB_\rho (p) \, \subset \bigcap_{q\in \bB_\rho (p)\cap \Gamma} ( q+ W (T_q \Gamma, n(q), \theta))\, .
\end{equation}
\end{theorem}

The previous two theorems allow us to reduce both Theorem \ref{t:main-local} and Theorem \ref{t:Allard-general} to the following simpler statement. We postpone the proof to Section \ref{s:from-decomposition-to-flat-points}.

\begin{ipotesi}\label{a:main-local-2}
$Q\geq 1$ is an arbitrary integer and $\vartheta$ a given positive real number smaller than $\frac{\pi}{2}$. $\Gamma$ is a $C^{3,\alpha}$ arc in $\bB_1 (0) \subset \mathbb R^{2+n}$ with endpoints lying in $\partial \bB_1 (0)$. $T$ is a $2$-dimensional area-minimizing integral current in $U$ such that $(\partial T) \res U= Q \a{\Gamma}$.
$0\in \Gamma$ is a flat point, $Q\a{V}$ is the unique tangent cone to $T$ at $0$ and we let $n$ be as in \eqref{e:new-n}. Moreover
\begin{equation}\label{e:new-wedge}
\supp (T) \subset \bigcap_{q\in \bB_1 (0)\cap \Gamma} ( q+ W (T_q \Gamma, n(q), \vartheta))\, ,
\end{equation}
where $\vartheta$ is a small constant.
\end{ipotesi}

\begin{theorem}\label{t:flat-points}
Let $T$ and $\Gamma$ be as in Assumption \ref{a:main-local-2}. Then there is a neighborhood $U$ of $0$ and a smooth minimal surface $\Sigma$ in $U$ with boundary $\Gamma$ such that $T\res U = Q \a{\Sigma}$.
\end{theorem}

Obviously the latter theorem implies as well Theorem \ref{t:Allard-flat}.

\subsection{Decay towards the cone} We first state a more precise version of Theorem \ref{t:uniqueness}. To that end we recall the flat norm $\mathcal{F}$ and the definition of spherical excess. Given an integral $2$-dimensional current $S$ we set
\[
\mathcal{F} (S) := \inf \{ \mass (P) + \mass (R): S= \partial P + R, \, R \in \mathbf{I}_2, \, P\in \mathbf{I}_3\}\, .
\]
Moreover, for $T$ as in Assumption \ref{a:main-local} and $p\in \Gamma$ we define the spherical excess $e (p,r)$ at the point $p$ and with radius $r$ by
\begin{equation}\label{e:excess}
e (p,r) := \frac{\|T\| (\bB_r (p))}{\pi r^2} - \Theta (T, p)
= \frac{\|T\| (\bB_r (p))}{\pi r^2} - \frac{Q}{2}\, .
\end{equation}

We are now ready to state the main decay theorem. Its proof follows the ideas of \cite{HM}, but it is in fact a consequence of a more general result, which is proved separately in our work \cite{DNS}, cf. \cite[Theorem 2.1]{DNS}.

\begin{theorem}\label{t:decay}
Let $T$ and $\Gamma$ be as in Theorem \ref{t:uniqueness}. Then there are positive constants $\varepsilon_0$, $C$ and $\alpha$ with the following property. If $p\in \Gamma$ and $e (p, r) \leq \varepsilon_0^2$ for some $r\leq \dist (p, \partial \bB_1)$, then: 
\begin{itemize}
\item[(a)] $|e (p,\rho)| \leq C |e(p, r)| \left(\frac{\rho}{r}\right)^{2\alpha} + C \rho^{2\alpha}$ for every $\rho \leq r$,
\item[(b)] There is a unique tangent cone $T_{p,0}$ to $T$ at $p$,
\item[(c)] The following estimates hold for every $\rho \leq r$
\begin{align}
&\mathcal{F} (T_{p,\rho}\res\bB_1, T_{p,0}\res \bB_1) \leq C(r) |e (p, r)|^{\sfrac{1}{2}} \left({\textstyle{\frac{\rho}{r}}}\right)^\alpha + C \rho^\alpha,\, \label{e:decay}\\
& \dist_H (\supp (T_{p, \rho})\cap \overline \bB_1, \supp (T_{p,0})\cap \overline \bB_1) \leq C \left({\textstyle{\frac{\rho}{r}}}\right)^\alpha.\label{e:Hausdorff-distance-estimate}
\end{align}
\end{itemize}
\end{theorem}

\subsection{From Theorem \ref{t:decay} to Theorem \ref{t:decomposition}} We fix a point $p$ as in the statement of Theorem \ref{t:decomposition}, we choose a radius $r_0$ so that $\bB_{2r_0} (p) \subset \bB_1 (0)$. We fix thus $\varepsilon_0$, $\alpha$ and $C$ given by Theorem \ref{t:decay}. Moreover, in order to simplify the notation, we write $T_p$ rather than $T_{p,0}$ for the unique tangent cone to $T$ and $p$.

First of all we observe that
\begin{align*}
e (q, r_0) &=\frac{\|T\| (B_{r_0} (q))}{\pi r_0^2} - 
\frac{Q}{2} \leq \frac{\|T\| (B_{r_0+|p-q|} (p))}{\pi r_0^2} - \frac{Q}{2}\\
&= \left(\frac{r_0+|p-q|}{r_0}\right)^2 e (p, r_0+|p-q|)
+ \left(\left(\frac{r_0+|p-q|}{r_0}\right)^2 -1\right)\frac{Q}{2}
\end{align*}
In particular, if $r_0$ is chosen sufficiently small, we can assume that $e (q, r_0)\leq 5\varepsilon_0^2$ for every point $q\in \Gamma \cap \bB_{r_0} (p)$. The rest of the proof is divided into three steps

In a first step we compare tangent cones between different points and prove
\begin{equation}\label{e:comparison}
\mathcal{F} (T_q\res \bB_1, T_p\res \bB_1) \leq C |q-p|^{\alpha} \qquad \forall q\in \bB_{r_0} (p)\, .
\end{equation}
Next, since $T_p$ is not flat by assumption and because of the classification of tangent cones, we can find half-planes $V$ and $V_1, \ldots V_N$ all distinct, such that 
\begin{equation}
T_p = Q_1 \a{V} + \sum_i \bar{Q}_i \a{ V_i}\, ,
\end{equation}
where $Q_1< Q$ and $Q_2 := Q-Q_1 = \sum_i \bar{Q}_i >0$.
Let $n$ be the unit vector in $V$ which is orthogonal to $T_p \Gamma$. We then infer the existence of a positive $\vartheta_0$ with the property that
\begin{equation}\label{e:cone-in-p-split}
\bigcup_i V_i \subset \overline{\mathbb R^{2+n} \setminus W (T_p \Gamma, n, 8 \vartheta_0)} =:
W^c (T_p \Gamma, n, 8\vartheta_0)\, .
\end{equation}
For every point $q\in \Gamma$ sufficiently close to $p$ we project $n$ onto the orthogonal complement of $T_q \Gamma$ and normalize it to a unit vector $n(q)$. \eqref{e:comparison} will then be used to show the existence of $r>0$ such that
\begin{equation}\label{e:cone-in-q-split}
\supp (T_q) \subset W (T_q \Gamma, n(q), 2\vartheta_0) \cup W^c (T_q \Gamma, n(q), 7 \vartheta_0)
\qquad \forall q\in \Gamma \cap \bB_r (p)\, .
\end{equation}
Hence we use \eqref{e:decay} to show the existence of $\bar r>0$ such that
\begin{equation}\label{e:conical-split}
\supp (T)\cap \bB_{\bar r} (q)\subset (q+ W (T_q \Gamma, n(q), 3\vartheta_0)) \cup (q+W^c (T_q \Gamma, n (q), 6\vartheta_0))\, .
\end{equation}
\eqref{e:conical-split} allows us to define
\begin{align}
T_1 &:= T \res \left( \bB_{\bar r} (p) \cap \bigcap_q (q+ W (T_q \Gamma, n(q), 3\vartheta_0)) \right),\\
T_2 &:= T \res \left( \bB_{\bar r} (p) \cap \bigcap_q (q+ W^c (T_q \Gamma, n(q), 6\vartheta_0)) \right),
\end{align}
and to show that $T_1+T_2 = T \res \bB_{\bar r} (p)$ and that each of the $T_i$ is area-minimizing. The final step is then to prove that 
\begin{equation}\label{e:bordo-T1}
\partial T_1 \res \bB_{\bar r} (p) = Q_1 \a{\Gamma\cap \bB_{\bar r} (p)}.\, 
\end{equation}

\medskip

{\bf Step 1. Proof of \eqref{e:comparison}} In order to prove \eqref{e:comparison} set $\rho_0:= |p-q|$ and observe that, it suffices to show the estimate 
\[
\mathcal{F} (T_p \res \bB_1, T_{q, \rho}\res \bB_1) \leq C \rho^{\alpha}\, 
\]
for some $\rho\in [\rho_0, 2\rho_0]$, whose choice will be specified later.
For $v \in \mathbb R^{2+n}$, denote by $\tau_v$ the translation by the vector $v$. If we choose $v:= (q-p)/\rho$ it is easy to see that
$T_{q,\rho}\res \bB_1 = (\tau_{-v})\sharp (T_{p,\rho}\res \bB_1 (v))$ and since the flat norm is invariant under translations, we get
\[
\mathcal{F} (T_p \res \bB_1, T_{q, \rho}\res \bB_1)
= \mathcal{F} ((\tau_v)_\sharp (T_p\res \bB_1 (0)), T_{p, \rho} \res \bB_1 (v)) \,.
\]
On the other hand, observe that $T_p$ is invariant by translation along $T_p \Gamma$ and that, if we write $v = w + \p_{T_p \Gamma} (v) =: w+z$, then $|w|\leq C\rho$. Hence we have
\begin{align*}
\mathcal{F} (T_p \res \bB_1, T_{q, \rho}\res \bB_1)
&= \mathcal{F} ((\tau_w)_\sharp (T_p \res \bB_1 (z)), T_{p,\rho} \res \bB_1 (v))\\
&\leq \mathcal{F} ((\tau_w)_\sharp (T_p \res \bB_1 (z)), T_p \res \bB_1 (z)) + \mathcal{F} (T_p \res \bB_1 (z), T_p \res \bB_1 (v))\\ &\qquad+ \mathcal{F} (T_p \res \bB_1 (v), T_{p,\rho} \res \bB_1 (v))\, .
\end{align*}
The first two summands can be easily estimated with $C\rho$. Indeed for the first term we write 
\[
(\tau_w)_\sharp (T_p \res \bB_1 (z)) -
T_p \res \bB_1 (z) = \partial ((T_p \res \bB_1 (z))\times \a{[0,w]}) =: \partial Z
\]
and we estimate $\mass (Z) \leq C |w| \leq C \rho$,
whereas for the second term we can estimate directly
\[
\mass (T_p \res \bB_1 (z) - T_p \res \bB_1 (v))\leq C |w|\, .
\]
It remains to bound the third summand. To that end we employ the fact that we are free to choose $\rho\in [\rho_0, 2\rho_0]$ appropriately. Note that the point $v$ depends on $\rho$: we will therefore write $v (\rho)$ from now on and use $v_0$ for $v (\rho_0)$, while we define $\sigma:= \frac{\rho}{\rho_0}$. By a simple rescaling argument we observe that
\[
\mathcal{F} (T_p \res \bB_1 (v (\rho)), T_{p,\rho} \res \bB_1 (v (\rho)) \leq C
\mathcal{F} (T_p \res \bB_\sigma (v_0), T_{p, \rho_0} \res \bB_\sigma (v_0)) \qquad \text{ for all } \sigma \in [1,2]\, . 
\]
We complete the proof by showing that, if $\sigma$ is chosen appropriately, then 
\begin{equation}\label{e:Fubini}
\mathcal{F} (T_p \res \bB_\sigma (v_0), T_{p, \rho_0} \res \bB_\sigma (v_0)) \leq C \mathcal{F} (T_p \res \bB_3 (0), T_{p, \rho_0} \res \bB_3 (0))\, ,
\end{equation}
since, again using a simple scaling argument, we can estimate
$\mathcal{F} (T_p \res \bB_3 (0), T_{p, \rho_0} \res \bB_3 (0))
\leq C \mathcal{F} (T_p \res \bB_1 (0), T_{p, 3\rho_0} \res \bB_1 (0))$ and take advantage of \eqref{e:decay}. In order to show \eqref{e:Fubini}, fix currents $R$ and $S$ such that $(T_p - T_{p, \rho_0}) \res \bB_3 (0) = R+ \partial S$ with 
\[
\mass (R) + \mass (S) \leq 2 \mathcal{F} (T_p \res \bB_3 (0), T_{p, \rho_0} \res \bB_3 (0))\, .
\]
Let now $d (x):= |x-v_0|$ and for every $\sigma$ we can then use the slicing formula \cite[Lemma 28.5]{Sim} to write
\[
(T_p - T_{p, \rho_0}) \res \bB_\sigma (v_0) = R\res \bB_\sigma (v)+ \partial (S\res \bB_\sigma (v_0)) - \langle S, d, \sigma\rangle\, .
\]
Since
\[
\int_1^2 \mass (\langle S, d, \sigma\rangle)\, d\sigma 
\leq \mass (S\res \bB_2 (v_0)) \leq \mass (S)\, ,
\]
it suffices to choose a $\sigma$ for which  $\mass (\langle S, d, \sigma\rangle)\leq 2 \mass (S)$.

\medskip

{\bf Step 2. Proof of \eqref{e:conical-split}} The latter is a simple consequence of the estimates proved in the previous two steps and of \eqref{e:Hausdorff-distance-estimate} and is left to the reader.

\medskip

{\bf Step 3. Proof of \eqref{e:bordo-T1}} Observe that $\partial T_1 \res \bB_{\bar r} (p)$ is supported in $\Gamma \cap \bB_{\bar r} (p)$ and is a flat chain without boundary in $\bB_{\bar r} (p)$. By the Constancy Lemma of Federer \cite[4.1.7]{Fed}, it follows that $\partial T_1 \res \bB_{\bar r} (p) = \Theta \a{\Gamma\cap \bB_{\bar r} (p)}$ for some constant $\Theta$. In particular $T_1$ is integral and thus $\Theta$ is an integer. Since it is area minimizing, it follows from our analysis that $T_1$ has a unique tangent cone $(T_1)_p$ at $p$ and that $\pi \Theta$ equals twice the mass of $(T_1)_p$ in $\bB_1 (0)$. On the other hand the latter cone is the restricion of $T_p$ to $W (T_p \Gamma, n(p), 3\vartheta_0)$, which by assumption is $Q_1 \a{V}$ for a fixed half-plane $V$ with boundary $T_p \Gamma$. Thus $\Theta = Q_1$, which completes the proof.

\subsection{From Theorem \ref{t:flat-points} to Theorem \ref{t:main-local}}\label{s:from-decomposition-to-flat-points} In this subsection we show how to conclude Theorem \ref{t:main-local} from Theorem \ref{t:flat-points} and Theorem \ref{t:decomposition}. We argue by induction on $Q$. We start observing that for $Q=1$ there are no boundary singular points, as it can be concluded by \cite{AllB}. Assume therefore that Theorem \ref{t:main-local} holds for all $Q$ strictly smaller than some fixed positive integer $\bar{Q}$: our aim is to show that it holds for $Q= \bar Q$. First of all observe that by Theorem \ref{t:decomposition} we know that the set $F:=\{p\in \Gamma: \mbox{$p$ is a flat boundary point}\}$ is closed in $\Gamma$. If $F=\Gamma$, then $T$ has no boundary singularities. Otherwise, by Theorem \ref{t:flat-points}(a), it suffices to show that the dimension of $\bsing (T)\setminus F$ is $0$. It then suffices to show that for every $p\in \Gamma \setminus F$ there is a radius $\rho$ such that $\bsing (T) \cap \bB_\rho (p)$ has dimension $0$. Fix $\rho$ as in Theorem \ref{t:decomposition} and let $T_1$ and $T_2$ satisfy the conclusion of that theorem. We claim that 
\begin{equation}\label{e:inclusion}
\bsing (T) \cap \bB_\rho (p) \ \subset \ \bsing (T_1) \cup \bsing (T_2)\, .
\end{equation} 
Since by the induction hypothesis each $\bsing (T_i)$ has dimension $0$, the latter claim would conclude the proof. In order to show \eqref{e:inclusion}, consider a point $q$ which is a boundary regular point for both $T_1$ and $T_2$: we aim to prove that $q$ is a regular point for $T$ as well. By the very definition of boundary regular point, for each $i$ there is a neighborhood $U_i\subset \bB_\rho (p)$ of $p$, minimal surfaces $\Lambda^i_j$, and integer coefficients $k^i_j$ such that:
\begin{itemize}
\item $T_i \res U_i = \sum_j k^i_j \a{\Lambda^i_j}$;
\item $\Lambda^i_j \cap \Lambda^i_k \subset \Gamma$ for every $j\neq k$;
\item the tangents of $\Lambda^i_j$ at every point $\bar{q}\in \Gamma \cap U$ are all distinct. 
\end{itemize}
Now, in $U := U_1 \cap U_2$ we clearly have
\[
T\res U = \sum_{i=1}^2  \sum_j k^i_j \a{\Lambda^i_j\cap U}\, .
\]
Note that, by Theorem \ref{t:decomposition}(c) $\Lambda^1_j \cap \Lambda^2_k \subset \supp (T_1) \cap \supp (T_2) \subset \Gamma$ for every $j\neq k$. Moreover, if $\bar{q}\in \Gamma \cap U$, then $(T_1)_{\bar q,0} = \sum_j k^1_j \a{T_{\bar q} \Lambda^1_j}$ and $(T_2)_{\bar{q}, 0} = \sum_k k^2_k \a{T_{\bar{q}} \Lambda^2_k}$. We conclude from Theorem \ref{t:decomposition}(d) that for every $j$ and $k$ the half planes $T_{\bar{q}} \Lambda^1_j$ and $T_{\bar{q}} \Lambda^2_k$ are distinct, i.e. intersect only in $T_{\bar q} \Gamma$. This shows that $q$ is then a boundary regular point of $T$.

%% file: multi-valued.tex
\section{Multi-valued functions}\label{s:multi-valued}

The next step of our proof is a detailed study of the boundary behaviour of ${\rm Dir}$-minimizing multi-valued functions. In this section we consider maps $u:B_\rho (x) \cap D\to \Iqs$ where $D\subset \mathbb R^2$ is a planar domain such that $\partial D$ is $C^2$. We will be interested in maps which take a preassigned value $Q\a{f}$ at $\partial D \cap B_\rho (x)$. Since by subtracting the average $\etaa \circ u$ we still get a ${\rm Dir}$-minimizer, we can without loss of generality, assume that $f$ vanishes identically. We summarize the relevant assumptions in the following 

\begin{ipotesi}\label{a:multi-valued}
$D\subset \mathbb R^2$ is a $C^2$ open set, $U$ is a bounded open set and $u\in W^{1,2} (D\cap U, \Iqs)$ a multivalued function such that $u|_{\partial D \cap U} \equiv Q\a{0}$ and $\etaa\circ u \equiv 0$. $u$ is {\rm Dir} minimizing in the sense that, for every $K\subset U$ compact and for every $v\in W^{1,2} (D\cap U, \Iqs)$ which coincides with $u$ on $(U\setminus K)\cap D$ and vanishes on $\partial D \cap U$, we have
\[
\D\, (u) \leq \D \, (v)\, .
\]
\end{ipotesi}

Observe that under our assumptions, we can apply the regularity theory of \cite{DS1} and \cite{Jonas} to conclude that $u$ is H\"older continuous in $K\cap \overline{D}$ for every compact set $K\subset U$. More precisely we have the following

\begin{theorem} There is a geometric constant $\alpha (Q)>0$ and a constant $C$ which depends only on $Q$ and $D$ such that, if $u$ and $D$ are as in Assumption \ref{a:multi-valued}, then
\[
[u]_{0,\alpha, B_\rho (x)\cap D} \leq C \rho^{-\alpha} \left(\D (u, B_{2\rho} (x)\cap D)\right)^{\frac{1}{2}}
\]
for every $B_{2\rho} (x)\subset U$.
\end{theorem}

In the final blow-up in Section \ref{s:blowup_and_conclusion}, we will prove that the limit of a suitable approximating sequence is a homogeneous Dir-minimizer. The following theorem will then exclude the existence of singular boundary points. It is a consequence of the classification of tangent functions (Theorem \ref{t:classification}).
\begin{theorem}\label{t:I=1}
Assume $D = \{x_2>0\}$, $U = B_1 (0)$ and $u: D\cap U\to \mathcal{A}_Q (\mathbb R^n)$ is a Dir-minimizing $I$-homogeneous map such that $u|_{\partial D} = Q \a{0}$. Either $u$ is a single harmonic function with multiplicity $Q$ (i.e. $u= Q \a{\etaa\circ u}$) or $I=1$.
\end{theorem}

Observe that under the additional information that $\etaa\circ u \equiv 0$, the first alternative would imply that $u$ vanishes identically. 

In case that the approximating sequence consisted of Dir-minimizers (which it does not in our case), we mention for completeness here the analouge definition of singular boundary points for Dir-minimizers (i.e. points at the boundary where the order of ``vanishing'' of the Dir-minimizer is larger than $1$) and prove its absence. Even though we will not need Definition \ref{d:contact-points} nor Theorem \ref{t:contact-points} for our analysis, it illustrates the ideas of our argument.

\begin{definition}\label{d:contact-points}
Let $D$, $u$ and $U$ be as in Assumption \ref{a:multi-valued}. $x\in \partial D$ will be called a \textbf{{\em contact point}} if there is a positive $\delta>0$ such that
\begin{equation}
\liminf_{\rho \downarrow 0} \frac{1}{\rho^{2+\delta}} \int_{B_\rho (x) \cap D} |Du|^2 = 0\, .    
\end{equation}
\end{definition}

In section \ref{s:contact-points} we will show the following multi-valued counterpart of Theorem \ref{t:flat-points}.

\begin{theorem}\label{t:contact-points}
Let $D$, $u$ and $U$ be as in Assumption \ref{a:multi-valued}. If $x\in \partial D$ is a contact point, then $u$ vanishes identically on the connected component of $D\cap U$ whose boundary contains $x$.
\end{theorem}

\subsection{Monotonicity of the frequency function} We introduce here the basic tool of our analysis, the frequency function, pioneered by Almgren. The version of the Almgren's frequency function used here is an extension introduced for the first time in the literature in \cite{DDHM} to deal with boundary regularity.
One of the outcomes of our analysis is that the limit of the frequency function exists at every boundary point $x$ unless $u$ vanishes identically in a neighborhood of it. 

We recall the definition of the frequency function as in \cite[Definition 4.13]{DDHM}.
\begin{definition}\label{Def:distance} 
Consider $u \in W_{l o c}^{1,2}\left(D, \Is{Q}\right)$ and fix any cut-off $\phi:[0, \infty[\rightarrow[0, \infty]$
which equals 1 in a neighborhood of $0$, it is non increasing and equals $0$ on $[1, \infty[.\text { We next }$ fix a function $d:\R ^2 \rightarrow\R ^{+}$ which is $C^{2}$ on the punctured space $\R ^2 \backslash\{0\}$ and satisfies the following properties:
\begin{enumerate}
\item [(i)] $d(x)=|x|+O\left(|x|^{2}\right)$,
\item [(ii)] $\nabla d(x)=\frac{x}{|x|}+O(|x|)$,
\item [(iii)] $D^{2} d(x) =|x|^{-1}\left( \Id -|x|^{-2} x \otimes x\right)+O(1)$.
\end{enumerate}
By \cite[Lemma 4.25]{DDHM}, we deduce the existence of such a $d$ satisfying also that $\nabla d$ is tangent to $\partial D$.
We define the following quantities:
\begin{align*}
D_{\phi, d}(u, r)&:=\int_{D} \phi\left(\frac{d(x)}{r}\right)|D u|^{2}(x)dx, \\
H_{\phi, d}(u, r)&:=-\int_{D} \phi^{\prime}\left(\frac{d(x)}{r}\right)|\nabla d(x)|^{2} \frac{|u(x)|^{2}}{d(x)}dx.
\end{align*}
The \textbf{\em{frequency function}} is then the ratio
$$
I_{\phi, d}(u, r):=\frac{r D_{\phi, d}(u, r)}{H_{\phi, d}(u, r)}.
$$
\end{definition}

This quantity is essentially monotone.

\begin{theorem}\label{Thm:monotoneFrequency} 
Let $D$, $U$ and $u$ be as in Assumption \ref{a:multi-valued}. Then there is a function $d$ satisfying the requirements of Definition \ref{Def:distance} such that the following holds for
every $\phi$ as in the same definition. Either $u \equiv Q\a{0}$ in a neighborhood of $0$, or $D_{\phi,d}(u, r)$ is positive for every $r$ (hence $I_{\phi,d} (u, r)$ is well defined) and the limit 
\[
0<\lim _{r \downarrow 0} I_{\phi, d}(u, r)<+\infty
\]
exists and it is a positive finite number. In fact, there is an $r_0>0$ and $C$ such that $r \mapsto e^{C r} I_{\phi, d}(u, r)$ is monotone for all $0<r<r_0$.
\end{theorem}

We first recall the following identities (compare \cite[Proposition 4.18]{DDHM}).
\begin{proposition}\label{Prop:H'D'} 
 Let $\phi$ and $d$ be as in Definition \ref{Def:distance} and assume in addition that $\phi$ is Lipschitz. Let $\Omega$, $D$, $U$ and $u$ be as in Assumption \ref{a:multi-valued}. Then, for every $0<r<1$, we have
\begin{equation}\label{Eq:D'}
D^{\prime}(r)=-\int_D \phi^{\prime}\left(\frac{|d(x)|}{r}\right) \frac{|d(x)|}{r^{2}}|D u|^{2} d x, 
\end{equation}
\begin{equation}\label{Eq:H'}
H^{\prime}(r)=\left(\frac{1}{r}+O(1)\right) H(r)+2 E(r),
\end{equation}
where
\begin{equation}\label{Eq:E}
E(r):=-\frac{1}{r} \int_D \phi^{\prime}\left(\frac{d(x)}{r}\right) \sum_{i} u_{i}(x) \cdot\left(D u_{i}(x) \cdot \nabla d(x)\right) d x,
\end{equation}
and the constant $O(1)$ appearing in \eqref{Eq:H'} depends on the function d but not on $\phi$.
\end{proposition}

Theorem \ref{Thm:monotoneFrequency} follows as in \cite{DDHM}, as soon as we  can show the validity of the above identities. In turn the latter can be proved following also the computations in \cite{DDHM}, provided we prove that both the outer variations
$g_\eps(x):= \sum_i \a{u_i(x) +\eps \varphi\left(\frac{d(x)}{r} \right)u_i(x)}$ and the inner variations
$u \circ \psi_t$, with $\psi_t$ being the flow of $Y(x):= \varphi \left( \frac{d(x)}{r} \right) \frac{d(x) \nabla d(x)}{|\nabla d(x)|^2}$,
are competitors to our problem. This is however obvious. Clearly the outer variations are well defined and preserve the condition that $u|_{\partial D \cap U} \equiv Q \a{0}$. As for the inner variations note that, since $\nabla d$ is tangent to $\partial D$, so is $Y$ and thus its flow maps $\partial D$ onto itself and $D$ into itself. This shows that the inner variations are well defined and provide admissible competitors too.

\subsection{Classification of tangent functions} Following a common path which started with Almgren's monumental work (see \cite{DDHM}, but also \cite{DS1,DS2,DS3,DS4,DSS1,DSS2,DSS3,DSS4}) we use the monotonocity of the frequency function to define tangent functions to $u$. 
Let $D$, $u$, $U$ and $f$ be as in Assumption \ref{a:multi-valued}. Let $x\in \partial D$ and denote by $n(x)$ the interior unit normal to $\partial D$. If we denote by $V^+$ the half space $\{y: n(x) \cdot y >0\}$, the tangent functions to $u$ at $x$ are multivalued functions defined on $V^+$, which turn out to be locally ${\rm Dir}$-minimizing and in fact satisfy Assumption \ref{a:multi-valued} with $D=V^+$ for any bounded open set $U$.

The central result is the following theorem of which Theorem \ref{t:I=1} is a direct corollary.

\begin{theorem}\label{t:classification}
Let $D$, $U$ and $u$ be as in Assumption \ref{a:multi-valued}. Let $x\in \partial D$ and assume that, for some $\rho>0$, $D\cap B_\rho (x)$ is connected and $u$ does not vanish identically on $B_\rho (x)\cap D$. Define 
\[
u_{x,\rho} (y) := \sum_i \a{\frac{u_i (x+\rho y)}{\D (u, B_\rho (x))^{\sfrac{1}{2}}}}\, .
\]
Then $I_0 (x):= \lim_{r\to 0} I(u(\cdot-x), r) =1$ and, for every sequence $\rho_k\downarrow 0$, there is a subsequence (not relabeled) such that $u_{x, \rho_k}$ converges locally uniformly on $V^+$ to a Dir-minimizer $u_{x,0}= \sum_i \a{v_i}$ satisfying the following properties:
\begin{itemize}
    \item[(a)] each $v_i: V^+\to \mathbb R^n$ is a linear function that vanishes at $\partial V^+$;
    \item[(b)] for every $i\neq j$, either $v_i \equiv v_j$, or $v_i (y)\neq v_j (y)$ for every $y\in V^+$;
    \item[(c)] $\D (u_{x,0}, B_1)=1$ and $\etaa\circ u_{x,0} =0$.
\end{itemize}
\end{theorem}

\begin{proof} First of all we let $I:= I_0 (x)$. It follows from the same arguments of \cite[Lemma 4.28]{DDHM} that a subsequence, not relabeled, of $u_{x,\rho_k}$ converges to a Dir-minimizer $u_{x,0} = \sum_i \a{v_i}$ which has the property (c) and which is $I$-homogeneous. Up to a rotation of the system of coordinates we can assume that $V^+= \{x_1>0\}$ (and hence $\partial V^+$ is the $x_2$-axis). From now on we use polar coordinates on $V^+$ and in particular we identify $\partial B_1 \cap V^+$ with $(-\frac{\pi}{2},\frac{\pi}{2})$. Let $g = \sum_i \a{g_i}$ be the restriction of $u_{x,0}$ on $\partial B_1 \cap V^+$. We can then use \cite[Proposition 1.2]{DS1} to conclude the existence of H\"older maps $g_1, \ldots, g_Q: (-\pi,\pi)\to \mathbb R^n$ such that 
\[
g (\theta) = \sum_i \a{g_i (\theta)}.
\]
In particular
\[
u_{x,0} (\theta, r) = \sum_i \a{r^I g_i (\theta)},
\]
and each $u_i (\theta, r)= r^I g_i (\theta)$ is an harmonic polynomial. In particular $I$ must be an integer. Since however $u_{x.0}\equiv Q\a{0}$ on $\{x_1=0\}$ and $\D (u_{x,0}, B_1) >0$, it must be a positive integer. 

Observe that, if $i\neq j$ and $\theta_0\in (-\frac{\pi}{2},\frac{\pi}{2})$ is a point where $g_i (\theta_0) = g_j (\theta_0)$, then $g_i$ and $g_j$ must coincide in a neighborhood of $\theta_0$, otherwise the whole halfline $\{(r\cos \theta_0, r\sin \theta_0)\}$ consists of singularities of $u_{x,0}$, contradicting \cite[Theorem 0.11]{DS1}. In particular by the unique continuation principle for harmonic functions we have 
\begin{itemize}
\item[(Alt)'] either $u_i (r,\theta) \neq u_j (r,\theta)$ for every $(r,\theta)\in ]0,+\infty[\times(\frac{\pi}{2}, \frac{\pi}{2})$, or $u_i (r,\theta)=u_j (r,\theta)$ for every $(r,\theta)\in ]0,+\infty[\times(\frac{\pi}{2}, \frac{\pi}{2})$,
\end{itemize}
so
\begin{itemize}
\item[(Alt)] either $g_i (\theta) \neq g_j (\theta)$ for every $\theta\in (-\frac{\pi}{2}, \frac{\pi}{2})$, or $g_i (\theta)= g_j (\theta)$ for every $\theta\in (-\frac{\pi}{2}, \frac{\pi}{2})$.
\end{itemize}
Next, using the classification of $2$-dimensional harmonic polynomials, we know that there are coefficients $a_i,b_i \in \mathbb R^n$ such that
\[
g_i (\theta) = a_i \cos (I\theta) + b_i \sin (I \theta)\, .
\]
If $I$ were even, since $g_i (\frac{\pi}{2}) = g_i (-\frac{\pi}{2})=0$, we conclude that $a_i =0$. But then all the $g_i$'s would vanish at $\theta=0$ and (Alt) would imply that they all coincide everywhere. This would however contradict (c). Likewise, if $I$ were odd and larger than $1$, then we would have $b_i =0$ and all the $g_i$'s would vanish at $\theta = \frac{\pi}{2I}$. We thus conclude that $I$ is necessarily equal to $1$. This proves then (a), while (Alt) shows (b). 
\end{proof}

\subsection{Proof of Theorem \ref{t:contact-points}}\label{s:contact-points}
Fix a point $x\in \partial D$ and assume that $u$ does not vanish in any neighborhood of $x$. Then Theorem \ref{t:classification} implies that the frequency function $I_0(x)$ is $1$. Arguing as in \cite[Corollary 4.27]{DDHM} we conclude however that, for every $\delta>0$, there is a radius $\rho>0$ such that 
\[
\frac{D(r)}{r^{2+\delta}} \geq (1-\delta) \frac{D (\rho)}{\rho^{2+\delta}} >0  \qquad
\forall r<\rho\, .
\]
This shows that $x$ cannot be a contact point.

%% file: Lipschitz-approximation.tex
\section{First Lipschitz approximation}\label{s:Lip-approx}

In this section we consider a neighborhood of a flat point and we introduce the cylindrical excess $\bE (T, \bC_r (p, V))$ as in \cite[Definition 5.1]{DDHM}. Then, under the assumption that $\bE (T, \bC_r (p, V))$ is sufficiently small, we produce an efficient approximation of the current with a multivalued graph. One important point is that the graph of such approximation, considered as an integral current, will also have boundary $Q\a{\Gamma}$. From now on, given a point $p$ and a plane $V$ through the origin, $B_r (p,V)$ will denote the disk $ \bB_r (p) \cap (p+V)$, $V^\perp$ the orthogonal complement of $V$ and $\bC_r (p, V)$ the cylinder $B_r (p,V)+V^\perp$. We then denote by  $\bp_V$ and $\bp_V^\perp$ the orthogonal projections respectively on $V$ and its orthogonal complement.

\begin{definition}\label{d:excess}
For a current $T$ in a cylinder $\bC_r (p, V)$ we define the \textbf{\emph{cylindrical excess}} 
$\bE (T, \bC_r (p, V))$ and the \textbf{\emph{excess measure}} $\be_T$ of a set $F\subset B_{4r}(\bp_V (p),V)$ as
\begin{align*}
\bE (T, \bC_r (p, V)) := \; &\frac{1}{2\pi r^2} \int_{\bC_r (p, V)} |\vec{T}-\vec{V}|^2\, d\|T\|,\\
\be_T(F):=\; & \frac{1}{2} \int_{F + V^\perp} |\vec{T}-\vec{V}|^2\, d\|T\|\,.
\end{align*}
The \textbf{\emph{height}} in a set $G\subset \mathbb R^{2+n}$ with respect to a plane $V$ is defined as
\begin{equation}
\bh (T, G, V) := \sup \{|\bp^\perp_V (q-p)|: q,p\in \supp (T)\cap G\}\, .
\end{equation}
\end{definition}

If $p$ and $V$ are omitted, then we understand that $V=\mathbb R^2\times \{0\}$ and $\bC_r = \bC_r(0,\mathbb R^2\times \{0\})$. 

\begin{ipotesi}\label{Ass:app} Let $\Gamma$ and $T$ be as in Assumption \ref{a:main-local-2}. $q$ is a fixed point, which without loss of generality we assume to be the origin, $r$ an arbitrary radius such that $(\partial T) \res \bC_{4r}= Q \a{\Gamma}\res \bC_{4r}$ and
\begin{itemize}
\item[(i)] $q = (0,0)\in \Gamma$ and $T_q \Gamma = \R \times \{0\}\subset V_0 = \mathbb R^2\times \{0\}$;
\item[(ii)] $\gamma=\bp(\Gamma)$ divides $B_{4r}$ in two disjoint open sets $D$ and $B_{4r}\setminus \overline{D}$;
\item[(iii)] $\bp_{\#} T\res \bC_{4r} = Q \a{ D}$.
\end{itemize}
\end{ipotesi}

Observe that thanks to  (iii) we have the identities
\begin{align}
\bE (T, \bC_{4r}) =\; &\frac{1}{2\pi (4r)^2} \left( \|T\| (\bC_{4r}) - Q |D|\right)\label{e:excess_2},\\
\be_T (F) =\; & \|T\| (F\times\R^n) - Q |D\cap F|\,.\label{e:excess_3}
\end{align}

Following a classical terminology we define noncentered maximal functions for Radon measures $\mu$ and (Lebesgue) integrable functions $f: U \to \R_+$ by setting
\begin{align*}
\bmm f(z) & := \sup_{z\in B_s(y)\subset U} \frac{1}{\pi s^2} \int_{B_s(y)} f\,, \\
 \bmm\mu (z) &:= \sup_{z\in B_s(y)\subset U} \frac{\mu(B_s(y))}{\pi s^2}. 
\end{align*}

\begin{remark}\label{r:little-psi} Observe that by our assumptions there is an interval $I\subset \mathbb R$ containing $(-5r, 5r)$ and function $\psi: I\to \mathbb R^{n+1}$ with the property that $\bC_{5r} \cap \Gamma = \{(t, \psi (t)): t\in I\}$. Moreover $\psi (0) =0$, $\dot \psi (0) =0$ and $\|\ddot\psi\|_{C^0} \leq C \bA$ for a geometric constant $C (n)$. In particular $|\psi (t)| \leq C \bA t^2$ and $|\dot\psi (t)|\leq C \bA t$. 
Finally observe that, if we write $\psi = (\psi_1,\bar \psi)$, then $\partial D = (t,\psi_1 (t))$ and $\Gamma$ can be written as the graph of a function $g$ on $\partial D$ defined by $g (t, \psi_1 (t)) = \bar\psi (t)$.
\end{remark}

\begin{figure}[htp]
    \centering
    \includegraphics[width=10cm]{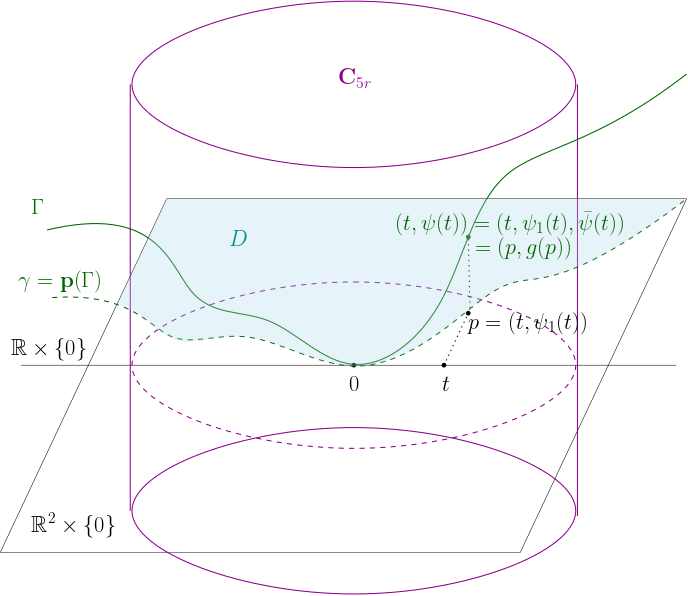}
    \caption{An illustration of the maps describing the boundary.}
\end{figure}

\newpage
\begin{proposition}[First Lipschitz approximation]\label{p:Lipschitz_1}
There are positive constants $C$ and $c_0$ (depending only on $Q$ and $n$) with the following properties. Assume $T$ satisfies Assumption \ref{Ass:app},
$E := \bE (T, \bC_{4r}) \leq c_0$. Then, for any $\delta_*\in (0,1)$, there are a closed set $K\subset D\cap B_{3r}$ and a $Q$-valued function $u$
on $D\cap B_{3r}$ with the following properties:
\begin{align}
u|_{\partial D \cap B_{3r}} =\;& Q \a{g}\label{e:boundary}\\
\Lip (u)\leq\;& C  (\delta_*^{\sfrac{1}{2}} + r \bA)\label{prop:Lipschitz_11}\\
{\rm osc} (u) \leq\; & C \bh (T, \bC_{4r}) + C r E^{\sfrac{1}{2}} + C r^2 \bA\label{e:oscillation}\\
K \subset \;& B_{3r} \cap \{\me_T\leq \delta_*\}\label{prop:Lipschitz_16}\\
 \bG_{u} \res[K \times \R^n] =\; & T \res [K\times \R^n] \label{e:differenza}\\
|(D\cap B_{s}) \setminus K| \leq\; & \frac{C}{\delta_*}\:\be_T\left(\{\me_T> 4^{-1} \delta_*\}\cap B_{s+ r_1 r}\right) + C \frac{\bA^2}{\delta_*} s^2 \quad \forall s \le 3r+ r_1 r\label{e:stimaK}\\
 \frac{\|  T- \bG_u\|(\bC_{2r})}{r^2}\le\; &\frac{C}{\delta_*}(E+\bA^2r^2) \quad  \label{e:graphmass}
\end{align}
where $r_1 = c \sqrt{\frac{E +\bA^2 r^2}{\delta_*}}$ and $c$ is a geometric constant.
\end{proposition}

\begin{proof} Since the statement is invariant under dilations we assume w.l.o.g. that $r=1$. Consider the extension $\hat{g}$ of the function $g$ defined in Remark \ref{r:little-psi} which is simply given by $\hat g (x_1, x_2) = \bar\psi (x_1)$. In order to simplify our notation, we drop the hat symbol and denote the extension by $g$ as well.
Consider next the current $\hat{T}\in\mathbf{I}_2(\mathbf{C}_{4})$ which consists of $\hat{T} = T\res\bC_4 + Q \bG_g \res ((B_4\setminus D)\times \mathbb R^n)$, where we use notation $\bG_g$ for the integer rectifiable current naturally associated to the graph of a function $g:B_4 \to \mathbb R^n$. More formally, if $\bar g (x) = (x, g (x))$, then
\begin{equation}\label{e:pushforward}
\bG_g \res ((B_4\setminus D)\times \mathbb R^n) = \bar g_\sharp (\a{B_4\setminus D}).
\end{equation}
In particular from \eqref{e:pushforward} and the classical theory of currents we see that
\begin{align}
(\partial \hat{T}) \res \bC_4 = &Q \a{\Gamma} \res \bC_4 - Q \bar g_\sharp (\a{\partial D\cap B_4}) =
Q\a{\Gamma} \res \bC_4 - Q \a{\Gamma} \res \bC_4 =0\,,\\
\bp_\sharp \hat{T} =&Q \a{D} + Q \a{B_4\setminus D} = Q \a{B_4}\, .
\end{align}
Moreover, we can use \cite[Corollary 3.3]{DS2} to estimate
\begin{align}
\|\hat{T}\| (\bC_4) - Q\pi 4^2 &= \bE (T, \bC_4) + Q (\|\bG_g\| ((B_4\setminus D)\times \mathbb R^n) - |D|)\nonumber\\
&\leq  \bE (T, \bC_4) + Q \int_{B_4 \setminus D} |Dg|^2 \leq E + C \bA^2\, .
\end{align}
Similarly, we can define for $F \subset B_4$
\[
\be_{\hat T} (F)= \|\hat{T}\| (F\times \mathbb R^n) - Q|F| 
\]
and the same considerations give
\[
\be_{\hat T} (F) \leq \be_T (F\cap D) + C \bA^2 |F\setminus D|\, .
\]
Moreover, we can apply \cite[Proposition 3.2]{DS3} to $\hat{T}$ to obtain a closed set $\hat{K}\subset B_{3}$ and $\hat{u} \in \operatorname{Lip}\left(B_{3}, \mathcal{A}_{Q}\left(\mathbb{R}^{n}\right)\right)$
which satisfy all the estimates \eqref{prop:Lipschitz_11}-\eqref{e:graphmass}, with the only relevant differences in \eqref{e:stimaK}, which becomes
\begin{equation}
|B_s \setminus \hat K| \leq\; \frac{C}{\delta_*}\:\be_T\left(\{\me_T> 4^{-1}\delta_*\}\cap B_{s+r_1 r} (x)\right) + C \frac{\bA^2}{\delta_*} s^2 \quad \text{ for every } s \le 3r\, . \label{e:stima-hat-K}
\end{equation}
In order to show \eqref{e:boundary}, we define an ``almost reflection'' $h$ on the boundary $\partial D$ in the following way:
\[
h (x_1, x_2) = (x_1,  2\psi_1 (x_1) - x_2)\, 
\]
and set $K:= h (\hat{K})\cap \hat{K}$. We now take the map $\hat{u}$, restrict it to $\hat{K}$ and then extend it again to a Lipschitz map $u$ with the additional property that \eqref{e:boundary} holds. 
In fact we first define $u:K\cup(\partial D\cap B_2)\to\Is{Q}$ as
\[
u(y) = \begin{cases} Q \a{g(y)} &, \text{ if } y \in \partial D\\ \hat u(y)&, \text{ else. }
\end{cases}
\]
Note that in principle a point $y$ could belong to both $K$ and $\partial D$: in that case we are ignoring the value given by $\hat{u}$ and force such value to be the one given by $Q\a{g}$. However a byproduct of the next elementary argument is that in fact $\hat{u} (y) = Q \a{g(y)}$ for every $y\in \partial D$.

We now wish to show that the bound on $\Lip (u)$ and ${\rm osc}(u)$ becomes worse only by a geometric factor. In fact, since the oscillation of $Q\a{g}$ is controlled by $\bA$, we just need to focus on the Lipschitz bound. 
Consider $p\in\partial D$, $q\in K$. By construction of $h$, let $\sigma$ be the vertical segment joining $q$ and $h(q)$ and let $\tilde{q}$ be the only intersection of $\sigma$ with $\partial D$. Thus 
\begin{align*}
\mathcal{G}(u(p),u(q)) & \le \mathcal{G}(u(q),u(h(q)))+\mathcal{G}(u(h(q)),u(p))\\
  & \le  \mathcal{G}(u(q),u(h(q)))+ C\mathcal{G}(u(\tilde{q}),u(p))\\
  & \le \mathcal{G}(u(q),u(h(q)))+CQ|g(p)-g(\tilde{q})|\\
  & \le 2|q-p|\Lip(\hat{u})+CQ\mathbf{A}|p-q|.
\end{align*} 
Now we can use the Lipschitz Extension Theorem \cite[Theorem 1.7]{DS1} to extend $u$ to the whole domain $B_2$, while enlarging the Lipschitz constant and the oscillation by a geometric factor.

So far our map satisfies \eqref{e:boundary}, \eqref{prop:Lipschitz_11},
and \eqref{e:oscillation}. However, \eqref{prop:Lipschitz_16} and \eqref{e:differenza} are obvious because $K\subset \hat{K}$.

Next we show \eqref{e:stimaK} holds with a slightly larger constant. First of all notice that, provided $\bA$ is sufficiently small, $h$ is a diffeomorphism and that $ h^{-1} (B_s) \subset B_{s+C\bA s^2}$, because $h(0)=0$ and $\|Dh - {\rm Id}\|_{C (B_s)} =\|Dh - Dh (0)\|_{C (B_s)}\leq C \bA s$. In particular we can estimate
\begin{align*}
|(B_s\cap D) \setminus K| &\leq |B_s\setminus \hat{K}| +|B_s \setminus h (\hat{K})|\\
&\leq |B_s\setminus \hat{K}| + C |h^{-1} (B_s) \setminus \hat{K}| \leq C |h (B_{s+ C \bA s^2}\setminus \hat{K})|\, .
\end{align*}
Finally we conclude
\begin{align*}
\|T-\bG_u\| (\bC_{2}) &\leq \|T - \bG_{\hat{u}}\| (\bC_2) + \|\bG_u - \bG_{\hat{u}}\| (\bC_2).
\end{align*}
For the first summand, we already have the desired estimate from \cite[Proposition 3.2]{DS3}. For the second we observe
\begin{align*}
\|\bG_u - \bG_{\hat{u}}\| (\bC_2) = \|\bG_u - \bG_{\hat{u}}\| ((B_2\setminus K)\times \mathbb R^n)
\leq C |B_2\setminus K|\,,
\end{align*}
and we then use \eqref{e:stimaK}. This shows \eqref{e:graphmass}.

The proof would be complete, except that our approximation and estimates hold on slightly smaller balls than claimed. It can however easily be checked that in \cite[Proposition 3.2]{DS3}, we just need to reduce slightly the size of the radius from $4$ to a fixed smaller one, while the argument is literally the same: the price to pay are just worse constants in the estimates.  
\end{proof}

\section{Harmonic approximation}\label{s:harmonic-approx}

\begin{definition}[$E^\beta$-Lipschitz approximation]\label{d:Lip-approx}
Let $\beta \in \left(0, 1\right)$ and $T$ be as in Proposition \ref{p:Lipschitz_1}. After setting $\delta_* = (E+\bA^2)^{2\beta}$, the corresponding map $u$ given by the proposition will be called the
{\emph{$E^{\beta}$-Lipschitz approximation of $T$ in $\bC_{3r}$}} and will be denoted by $f$.
\end{definition}

In this section we use the minimimizing assumption on $T$ to show that
the $E^{\beta}$-Lipschitz approximation is close to a $\D$-minimizing function $w$. 
We first introduce some notation.
\begin{ipotesi}\label{a:multi-valuednonhomogeneous}
$D\subset \mathbb R^2$ is a $C^2$ open set, $U$ is a bounded open set and $u\in W^{1,2} (D\cap U, \Iqs)$ a multivalued function such that $u|_{\partial D \cap U} \equiv Q\a{g}$, where $g$ is as in Remark \ref{r:little-psi}. $u$ is {\rm Dir}-minimizing in the sense that, for every $K\subset U$ compact and for every $v\in W^{1,2} (D\cap U, \Iqs)$ which coincides with $u$ on $(U\setminus K)\cap D$ and $v|_{\partial D \cap U} \equiv Q\a{g}$ we have
\[
{\rm Dir}\, (u) \leq {\rm Dir}\, (v)\, .
\]
\end{ipotesi}

\begin{theorem}[First harmonic approximation]\label{t:o(E)}
For every $\eta>0$ and every $\beta\in (0, 1)$, there exist a constant $\varepsilon=\varepsilon(\eta,\beta)>0$ with the following property. 
Let $T$ and $\Gamma$ be as in Assumption \ref{Ass:app} in $\bC_{4r}$ (in particular $T$ is area minimizing in $\bC_{4r}$).
If $E =\bE(T,\bC_{4r})\leq \varepsilon$ and $r\bA \leq \varepsilon E^{\frac{1}{2}}$, then the
$E^{\beta}$-Lipschitz approximation $f$ in $\bC_{3r}$ satisfies
\begin{equation}\label{e:few energy}
\int_{B_{2r}\cap D\setminus K}|Df|^2\leq \eta E \pi (4r)^2 = \eta\,\e_T(B_{4r}).
\end{equation}
Moreover, there exists a Dir-minimizing function $w$
such that $w|_{\partial D \cap B_{2r}} = Q\a{g}$ and
\begin{gather}
r^{-2} \int_{B_{2r}\cap D}\cG(f,w)^2+
\int_{B_{2r}\cap D}\cG (Df, Dw)^2 \leq \eta E \,\pi \,(4\,r)^2 =
\eta\, \e_T(B_{4r})\, ,\label{e:quasiarm}
\end{gather}
\begin{equation}\label{e:quasiarmaver}
\int_{B_{2r}\cap D} |D(\boldsymbol{\eta} \circ f)-D(\boldsymbol{\eta} \circ w)|^{2} \leq \eta E\pi(4 r)^2=\eta \mathbf{e}_{T}\left(B_{4r}\right).
\end{equation}
\end{theorem}
The following proposition provides a Taylor expansion of the mass of the current associated to the graph of a $Q$-valued function. It is proven in \cite[Corollary 3.3]{DS2} (although the corollary is stated for $V$ open, the proof works obviously when $V$ is merely measurable).
\begin{proposition}\label{p:TaylorExpansionMassGraph}$($Taylor expansion of the mass, see \cite[Corollary 3.3]{DS2}$)$. There are dimensional constants $c, C>0$ such that the following holds. Let $V \subset \mathbb{R}^2$ be a bounded measurable set and let $u: V \rightarrow \mathcal{A}_{Q}\left(\mathbb{R}^{n}\right)$ be a Lipschitz function with $\operatorname{Lip}(u) \leq c .$ Denote by $\mathbf{G}_{u}$ the integer rectifiable current associated to the graph of $u$ as in \cite[Definition 1.10]{DS2}. Then, the following Taylor expansion of the mass of $\mathbf{G}_{u}$ holds:
$$
\mathbf{M}\left(\mathbf{G}_{u}\right)=Q|V|+\int_{V} \frac{|D u|^{2}}{2}+\int_{V} \sum_{i} R\left(D u_{i}\right),
$$
where $R: \mathbb{R}^{n \times 2} \rightarrow \mathbb R$ is a $C^{1}$ function satisfying $|R(D)|=|D|^{3} L(D)$ for some positive function $L$ such that $L(0)=0$ and $\operatorname{Lip}(L) \leq C$. 
\end{proposition}
\begin{remark}\label{r:5.5DS3} We write here the analog of (\cite[Remark 5.5]{DS3}). There exists a dimensional constant $c>0$ such that, if $E \leq c,$ then the $E^{\beta}$ -Lipschitz approximation satisfies the following estimates:
\begin{equation}\label{e:LipBeforeMinimizing}
\operatorname{Lip}(f)\leq C (E+C\mathbf{A}^2)^\beta,
\end{equation}
\begin{equation}\label{e:DirEnergyBeforeMinimizing}
\int_{B_{3 s}(x)\cap D}|D f|^{2} \stackrel{\eqref{e:stimaK}}{\le} C (E + \mathbf{A}^2) s^2.
\end{equation}
Indeed \eqref{e:LipBeforeMinimizing} follows from Proposition \ref{p:Lipschitz_1}, by the choice of $\beta$ and the scaling of $\mathbf{A} .$ While \eqref{e:DirEnergyBeforeMinimizing} follows from Proposition \ref{p:TaylorExpansionMassGraph} since for $E$ sufficiently small
$$
\int_{B_{3 s}(x)\cap D} \sum_{i} R\left(D f_{i}\right) \leq C E^{2 \beta} \int_{B_{3 s}(x)\cap D}|D f|^{2}<\frac{1}{4} \int_{B_{3 s}(x)\cap D}|D f|^{2},
$$
and therefore
$$
\begin{aligned}
\int_{B_{3 s}(x)\cap D}|D f|^2 & \le C\left(\mathbf{M}\left(\mathbf{G}_{f}\res\mathbf{C}_{3 s}(x)\cap (D\times\mathbb{R}^{n}\right)-Q |D|\right)\\
& \leq C\left(\mathbf{M}\left(T\res\mathbf{C}_{3 s}(x)\right)-Q |D|\right)+C \mathbf{M}\left(\mathbf{G}_{f} \res\left(B_{3 s}(x) \cap D\setminus K\right) \times \mathbb{R}^{n}\right)\\
& \leq C E s^2+C (E+\bA^2)^{2 \beta}\left|B_{3 s}(x) \cap D\setminus K\right| \leq C (E + \mathbf{A}^2) s^2.
\end{aligned}
$$
\end{remark}

\begin{proof}[Proof of Theorem \ref{t:o(E)}.] By rescaling, it is not restrictive to assume that $r=1$. The proof of \eqref{e:few energy} is by contradiction. Assume there exist a constant $c_{1}>0,$ a sequence of currents $\left(T_{k}\right)_{k \in \mathbb{N}}$ satisfying Assumption \ref{Ass:app} and corresponding $E_{k}^{\beta}$ -Lipschitz approximations $\left(f_{k}\right)_{k \in \mathbb{N}}$ which violate \eqref{e:few energy} for $\eta= c_1>0$.
At the same time $\partial T \res \bC_4(0) = Q\a{\Gamma_k}$, where $\Gamma_k$ is a sequence of $C^2$ curves. For the latter we have $T_0 \Gamma_k =\mathbb R \times \{0\}$ and a parametrization $\psi^k: \mathbb R \to \mathbb R^{n+1}$ of the form
\[
\psi^k (t) = (\psi^k_1 (t), \bar\psi^k (t))\, .
\]
Moreover we assume $\|\psi^k\|_{C^2} \leq C \bA_k \leq C \varepsilon_k E_k^{\frac{1}{2}}$. 
The domain of definition of the map $f_k$ is a set $D_k$ which can be explicitly written as 
\[
D_k = \{(x_1,x_2)\in B_3: x_2> \psi_1^k (x_1)\}\, .
\]
Summarizing, our currents satisfy the following:
\begin{eqnarray}\label{e:5.6}
\mathbf{E}\left(T_{k}, \mathbf{C}_{4}\right)\le\varepsilon_k \rightarrow 0, \quad \mathbf{A}_{k} \leq\varepsilon_kE^{\frac{1}{2}}_{k}\quad \text { and } \quad \int_{D_k \setminus K_{k}}\left|D f_{k}\right|^{2} \geq c_{1} E_{k}
\end{eqnarray}
where $K_{k}:=\left\{x \in B_{3}: \mathbf{m e}_{T_{k}}(x)<E_{k}^{2 \beta}\right\}$. 
Set $\Lambda_k:=\left\{x \in D_k: \mathbf{m e}_{T_{k}}(x) \leq 2^{-2} E_{k}^{2 \beta}\right\}$
and observe that $\Lambda_{k} \cap B_{3}\subset K_{k}$. From Proposition \ref{p:Lipschitz_1} it follows that
\begin{eqnarray}\label{e:5.7DS3}
\operatorname{Lip}\left(f_{k}\right) & \leq & C E_{k}^{\beta} \\ \label{e:5.8DS3}
\left|B_{r}\cap D_k \setminus K_{k}\right| 
& \leq & C E_{k}^{-2 \beta} \mathbf{e}_{T_k}\left(B_{r+r_{0}(k)} \setminus \Lambda_{k}\right)+C\varepsilon_k^2E_k^{2(1-\beta)} 
\quad \text { for every } r \leq 3
\end{eqnarray}
where $r_{0}(k)=16 E_{k}^{(1-2 \beta) /2}<\frac{1}{2}$. Then, \eqref{e:5.6}, \eqref{e:5.7DS3}, and \eqref{e:5.8DS3} give 
\begin{equation}
c_{1} E_{k} \leq \int_{B_2\cap D_k \setminus K_{k}}\left|D f_{k}\right|^{2} 
\leq C \mathbf{e}_{T_{k}}\left(B_{s} \setminus \Lambda_{k}\right) +C\varepsilon_k^2E_k^2
\quad \text { for every } s \in\left[\frac{5}{2}, 3\right].
\end{equation}
Setting $c_{2}:=c_{1} /(2 C),$ we have 
\[
2 c_{2} E_{k} \leq \mathbf{e}_{T_{k}}\left(B_{s}\cap D_k \setminus \Lambda_{k}\right)=\mathbf{e}_{T_{k}}\left(B_{s}\cap D_k\right)-\mathbf{e}_{T_{k}}\left(B_{s} \cap \Lambda_{k}\right)\, ,
\]
implying
\begin{equation}\label{e:5.10}
\mathbf{e}_{T_{k}}\left(\Lambda_{k} \cap B_{s}\right) \leq \mathbf{e}_{T_{k}}\left(D_k \cap B_{s}\right)-2 c_{2} E_{k}\,.
\end{equation}
Next observe that $2\pi 4^2 E_{k}=\mathbf{e}_{T_{k}}\left(B_{4}\cap D_k\right) \geq \mathbf{e}_{T_{k}}\left(B_{s}\cap D_k\right)$. Therefore, by the Taylor expansion in \cite[Remark 5.4]{DS3}, \eqref{e:5.10} and the fact that $E_{k} \downarrow 0$, it follows that for every $s \in[5 / 2,3]$ and $k$ large enough so that $C E^{2\beta_k} \leq c_2$, we have
\begin{eqnarray}\nonumber
\int_{\Lambda_{k} \cap B_{s}} \frac{\left|D f_{k}\right|^{2}}{2} & \stackrel{\text{Taylor}}{\le} & \left(1+C E_{k}^{2 \beta}\right) \mathbf{e}_{T_{k}}\left(\Lambda_{k} \cap B_{s}\right) \\ 
&\stackrel{\eqref{e:5.10}}{\leq} & \left(1+C E_{k}^{2 \beta}\right)\left(\mathbf{e}_{T_{k}}\left(B_{s}\cap D_k\right)-2 c_{2} E_{k}\right)\nonumber\\
&\leq& \mathbf{e}_{T_{k}}\left(B_{s}\cap D_k\right)-c_{2} E_{k}\, .\label{e:5.11DS3}
\end{eqnarray}

Our aim is to show that \eqref{e:5.11DS3} contradicts the minimality of $T_{k}$. To construct a competitor, we write $f_{k}(x)=\sum_{i} \llbracket f_{k}^{i}(x) \rrbracket \in \mathcal{A}_{Q}\left(\mathbb{R}^{n}\right)$. We consider $h_{k}:=E_{k}^{-1 / 2} f_{k}$. Observe that $h_k|_{\partial D_k} =Q\llbracket E_k^{-1/2}\bar\psi^k\rrbracket$ and that in turn $\|\bar\psi^k\|_{C^2} \leq C \varepsilon_k E_k^{\frac{1}{2}}$. In particular $E_k^{-1/2} \bar\psi^k$ converges strongly to $0$ in $C^2$. Extend $\bar\psi^k$ to $B_3\cap D_k$ by keeping it constant in the variable $x_2$. Thus $\cG (h_k, Q\llbracket E_k^{-1/2} \bar\psi^k\rrbracket)$ is a classical $W^{1,2}$ function that vanishes on $\partial D_k$. Since by \cite[Remark 5.5(5.5)]{DS3} we have $\sup _{k} \D\left(h_{k}, B_{3}\cap D\right)<\infty$, the Poincar\'e inequality gives
\[
\|\cG (h_k, Q\llbracket E_k^{-1/2} \bar\psi^k\rrbracket)\|_{L^2 (D_k\cap B_3)} \leq C\, ,
\]
which in turn implies $\|\cG (h_k, Q\a{0})\|_{L^2 (D_k\cap B_3)}\leq C$. Hence $\{h_k\}$ is bounded in $W^{1,2}$. Even though the domains of the $h_k$ depend on $k$, we can extend the maps identically equal to $Q\{\bar\psi^k\}$ on their complement, and thus treat them as maps on $B_3$. Up to a subsequence, not relabeled, we can thus assume that the maps converge to some $h\in W^{1,2}$. Observe that $h$ vanishes identically on the lower half disk $B_3^-:=\{(x_1, x_2)\in B_3:x_2<0\}$ and thus we will also consider it as a map defined on the upper half disk $B_3^+$, taking the value $Q\a{0}$ on the $x_1$-axis. 

Since 
\begin{equation}\label{e:L2-convergence}
\left\|\mathcal{G}\left(h_{k}, h\right)\right\|_{L^{2}\left(B_{3}\right)} \rightarrow 0
\end{equation} 
and the following inequalities hold for every
open $\Omega^{\prime} \subset B_3$ and any sequence of measurable sets $J_{k}$ with $\left|J_{k}\right| \rightarrow 0$,
\begin{eqnarray}\label{e:cc(4.2)}
\liminf_{k \rightarrow+\infty}\left(\int_{\Omega^{\prime} \setminus J_{k}}\left|D h_{k}\right|^{2}-\int_{\Omega^{\prime}}\left|D h\right|^{2}\right) & \geq & 0, \\\label{e:cc(4.3)}
\limsup _{k \rightarrow+\infty} \int_{\Omega}\left(\left|D h_{k}\right|-\left|D h\right|\right)^{2} & \leq & \limsup_{k\to+\infty} \int_{\Omega}\left(\left|D h_{k}\right|^{2}-\left|D h\right|^{2}\right).
\end{eqnarray}
Applying the first inequality with $J_k$ being the complement of $\Lambda_k$ we reach the following inequality
\begin{equation}\label{e:energy-loss}
\frac{1}{2} \int_{B_s^+} |Dh|^2 \leq \liminf_{k\to\infty} E_k^{- 1} \be_{T_k} (B_s\cap D_k) - c_2 \qquad \mbox{for every $s<3$.}
\end{equation}
Now we wish to find a radius $r\in [\frac{5}{2},3]$ and a competitor function $H_k$ such that
\begin{itemize}
    \item $\left.H_k\right|_{(B_{3} \setminus B_{r})\cap D_k}=\left.h_{k}\right|_{(B_{3} \setminus B_{r})\cap D_k}$;
    \item $\left. H_k\right|_{\partial D_k\cap B_3} = \left. h_k\right|_{\partial D_k\cap B_3}$; 
    \item The following estimates hold for a subsequence (not relabeled)
\begin{align}\label{e:5.13DS3}
\lim_{k \to \infty} \operatorname{Dir}\left(H_k, B_{r}\right) & \leq \operatorname{Dir}\left(h, B_{r}\right)+\frac{c_{2}}{4}, \\ \label{e:5.14DS3}
\operatorname{Lip}\left(H_k\right) & \leq C^{*} E_{k}^{\beta-1 / 2}\, ,\\
\|\cG(H_{k}, h_{k})\|_{L^2(B^+_{r})} & \leq C \D(h_{k}, B^+_{r}) + C \D(H_{k}, B^+_{r})\le M<+\infty,
\end{align}
where $C^*$ is a constant independent of $k$.
\end{itemize}
After proving that such a function exists, we can then follow the proof of \cite[Theorem 5.2]{DS3} mutatis mutandis.

In order to show our claim we will use \eqref{e:L2-convergence}, the Lipschitz bound $\Lip (h_k) \leq C E_k^{\beta-1/2}$, the bound $\sup_k \D (h_k, B_3) \leq C$, and \eqref{e:energy-loss}. Note next that, since $\|\bar \psi^k/E^{1/2}\|_{C^2}\downarrow 0$, all these facts remain true if we replace $h_k$ with the map
\[
\bar{h}_k (x) := \sum_i \a{(h_k)_i - \bar\psi^k}\, .
\]
The advantage of the latter is that $\bar{h}_k|_{\partial D_k} = Q\a{0}$. Assuming that we find corresponding maps $\bar{H}_k$ satisfying all the properties above, we can then simply get $H_k$ by adding back $\bar\psi^k$:
\[
H_k (x) = \sum_i \a{(\bar{H}_k)_i + \bar\psi^k}\,  
\]
(because the difference in the Dirichlet energies of $H_k$ and $\bar{H}_k$ and the difference in the Lipschitz constants are both infinitesimal). 

The next issue is that the domains $D_k\cap B_s$ are curved compared to $B_s^+$. To resolve this, we invoke Lemma \ref{l:change-of-variable} below. For each $k$ we apply the lemma to $\psi^k_1$ and get a corresponding diffeomorphism $\Phi_k$ which maps each $B_s \cap D_k$ diffeomorphically onto $B_s^+$. Observe that 
\begin{equation}\label{e:almost-identity}
\lim_{k\to\infty} \left(\|\Phi_k- {\rm Id}\|_{C^1} + \|\Phi_k^{-1} - {\rm Id}\|_{C^1}\right) = 0
\end{equation}
because $\|\psi^k_1\|_{C^1} \to 0$. For this reason the maps $\tilde{h}_k := \bar{h}_k\circ \Phi_k^{-1}$ satisfy the same assumptions as $\bar h_k$ (and hence as $h_k$). Indeed, after having built the corresponding competitors $\tilde{H}_k$, we can then define $\bar H_k := \tilde{H}_k \circ \Phi_k$. Again the desired conclusion follows 
because the difference of the Lipschitz constants and Dirichlet energies are infinitesimal. 

Summarizing, we have reduced the proof of the proposition to showing that the competitor $H_k$ can be constructed, without loss of generality, under the additional assumptions that all $h_k$'s are defined on the same domain $B_3^+$ and that they all vanish on $\{(x_1,x_2)\in B_3^+: x_2 =0\}$. This is accomplished in Proposition \ref{p:competitor} below. Now that we have illustrated how to construct suitable competitors we can proceed with the proof of the theorem. We restart observing that, when $k$ is large enough, \eqref{e:cc(4.2)}  implies the following inequalities
\begin{equation}\label{e:5.16DS3}
\operatorname{Dir}\left(h, B_{r}\right) \leq \operatorname{Dir}\left(h_{k}, B_{r} \cap \Gamma_{k}\right)+\frac{c_{2}}{4} \stackrel{(5.11)}{\leq} \frac{\mathbf{e}_{T_{k}}\left(B_{r}\right)}{E_{k}}-\frac{3 c_{2}}{4} E_{k}\,.
\end{equation}
Note that \eqref{e:5.14DS3} follows from \eqref{e:raccordo1} as $E_{k}^{\beta-1 / 2} \uparrow \infty$. Thus $C^{*}$ depends on $c_{2}$ and on the choice of the two sequences, but not on $k$. From now on, although this and similar constants are not dimensional, we will keep denoting them by $C,$ with the understanding that they do not depend on $k$. Note that, from \eqref{e:5.7DS3} and \eqref{e:5.8DS3}, one gets
$$
\begin{aligned}
\left\|T_{k}-\mathbf{G}_{f_{k}}\right\|\left(\mathbf{C}_{3}\right) & \leq\left\|T_{k}\right\|\left(\left(B_{3} \setminus K_{k}\right) \times \mathbb{R}^{n}\right)+\left\|\mathbf{G}_{f_{k}}\right\|\left(\left(B_{3} \setminus K_{k}\right) \times \mathbb{R}^{n}\right) \\
& \leq Q\left|B_{3} \setminus K_{k}\right|+E_{k}+Q\left|B_{3} \setminus K_{k}\right|+C\left|B_{3} \setminus K_{k}\right| \operatorname{Lip}\left(f_{k}\right) \\
& \leq E_{k}+C E_{k}^{1-2 \beta} \leq C E_{k}^{1-2 \beta}.
\end{aligned}
$$
Let $(z, y)$ denote the coordinates on $\mathbb{R}^2 \times \mathbb{R}^{n}$ and consider the function $\varphi(z, y)=|z|$ and the slice $\left\langle T_{k}-\mathbf{G}_{f_{k}}, \varphi, r\right\rangle .$ Observe that, by the coarea formula and Fatou's lemma,
$$
\int_{r}^{3} \liminf _{k} E_{k}^{2 \beta-1} \mathbf{M}\left(\left\langle T_{k}-\mathbf{G}_{f_{k}}, \varphi, s\right\rangle\right) d s \leq \liminf _{k} E_{k}^{2 \beta-1}\left\|T_{k}-\mathbf{G}_{f_{k}}\right\|\left(\mathbf{C}_{3}\right) \leq C.
$$
Therefore, for some $\bar{r} \in(r, 3)$, up to subsequences (not relabeled) $\mathbf{M}\left(\left\langle T_{k}-\mathbf{G}_{f_{k}}, \varphi, \bar{r}\right\rangle\right)$
$\leq C E_{k}^{1-2 \beta}$.
Let now $v_{k}:= E_{k}^{1 / 2} H_{k}|_{B_{\bar{r}}}$ and consider the current $Z_{k}:=\mathbf{G}_{v_{k}}\res\mathbf{C}_{\bar{r}}$. Since $(v_{k})|_{\partial B_{\bar{r}}}=\left.f_{k}\right|_{\partial B_{\bar{r}}},$ one gets $\partial Z_{k}=\left\langle\mathbf{G}_{f_{k}}, \varphi, \bar{r}\right\rangle$ and hence, $\mathbf{M}\left(\partial\left(T_{k}\res\mathbf{C}_{\bar{r}}\right.\right.$
$\left.\left.-Z_{k}\right)\right) \leq C E_{k}^{1-2 \beta} .$ We define
\begin{equation}\label{e:5.18DS3}
S_{k}=T_{k}\res\left(\mathbf{C}_{4} \setminus \mathbf{C}_{\bar{r}}\right)+Z_{k}+R_{k}\,,
\end{equation}
where (cp. \cite[Remark 5.3]{DS3}) $R_{k}$ is an integral current such that
$$
\partial R_{k}=\partial\left(T_{k}\res\mathbf{C}_{\bar{r}}-Z_{k}\right) \quad\text { and } \quad \mathbf{M}\left(R_{k}\right) \leq C E_{k}^{(1-2 \beta)2}\,.
$$
In particular, we have $\partial S_{k}=\partial\left(T_{k}\res\mathbf{C}_{4}\right)$. We now show that, since $\beta<\frac{1}{4},$ for $k$ large enough, the mass of $S_{k}$ is strictly smaller than the one of $T_{k} .$ To this aim we write
\begin{align*}
\operatorname{Dir}\left(v_{k}, B_{\bar{r}}\right)-\operatorname{Dir}\left(f_{k}, B_{\bar{r}} \cap \Lambda_{k}\right)=& \int_{B_{\bar{r}}}\left|D v_{k}\right|^{2}-\int_{B_{\bar{r}} \cap \Lambda_{k}}\left|D f_{k}\right|^{2} =: I_{1}\,. 
\end{align*}
The first term is estimated by \eqref{e:5.13DS3} and \eqref{e:cc(4.2)}. Indeed, recall that $v_{k}=E_{k}^{1 / 2} H_{k}$ and $f_{k}=E_{k}^{1 / 2} h_{k}$ (but also that the two functions coincide on $B_{\bar{r}} \setminus B_{r}$ ). We thus deduce that $I_{1} \leq \frac{c_{2}}{2} E_{k}$ for $k$ large enough. Hence,
\begin{align}
\mathbf{M}\left(S_{k}\right)-\mathbf{M}\left(T_{k}\right) 
&\ \leq \mathbf{M}\left(Z_{k}\right)+C \mathbf{M}\left(R_{k}\right)-\mathbf{M}\left(T_{k}\res\mathbf{C}_{\bar{r}}\right) \notag\\
&\ \leq Q\left|B_{\bar{r}}\right|+\int_{B_{\bar{r}}} \frac{\left|D v_{k}\right|^{2}}{2}+C E_{k}^{1+2 \beta}+C E_{k}^{(1-2 \beta)2}-Q\left|B_{\bar{r}}\right|-\mathbf{e}_{T_{k}}\left(B_{\bar{r}}\right) \notag\\
&\ \leq \int_{B_{\bar{r}} \cap \Lambda_{k}} \frac{\left|D f_{k}\right|^{2}}{2}+\frac{1}{2} c_{2} E_{k}+C E_{k}^{1+2 \beta}+C E_{k}^{(1-2 \beta)2}-\mathbf{e}_{T_{k}}\left(B_{\bar{r}}\right) \notag\\ \label{e:5.20DS3}
& \stackrel{\eqref{e:5.11DS3}}{\leq}-\frac{c_{2} E_{k}}{2}+C E_{k}^{1+\beta}+C E_{k}^{(1-2 \beta)2}<0,
\end{align}
as soon as $E_{k}$ is small enough, i.e., $k$ large enough. This gives the desired contradiction and proves \eqref{e:few energy}.

Now, we come to the proof of \eqref{e:quasiarm} and \eqref{e:quasiarmaver}. To this aim, we argue again by contradiction using similar constructions of competitors. Without loss of generality, we assume $x=0$ and $s=1$. Suppose $\left(T_{k}\right)_{k}$ is a sequence with $E_{k}:=\mathbf{E}\left(T_{k}, \mathbf{C}_{4}\right)$ satisfying 
\begin{eqnarray}\label{e:5.6bis}
\mathbf{E}\left(T_{k}, \mathbf{C}_{4}\right)\le\varepsilon_k \rightarrow 0, \qquad \mathbf{A}_{k} \leq\varepsilon_kE^{\frac{1}{2}}_{k},
\end{eqnarray}
but contradicting \eqref{e:quasiarm} or \eqref{e:quasiarmaver}. Let us denote by $f_{k}$ the $E_{k}^{\beta}$ -Lipschitz approximation of $T_k$. We know that, for any sequence of Dir-minimizing functions $\bar{u}_{k}$ which we might choose, we will have by the contradiction assumption that 
\begin{equation}\label{e:5.21DS3}
\liminf _{k} \underbrace{E_{k}^{-1} \int_{B_{2}}\left(\mathcal{G}\left(f_{k}, \bar{u}_{k}\right)^{2}+\left(\left|D f_{k}\right|-\left|D \bar{u}_{k}\right|\right)^{2}+\left|D\left(\boldsymbol{\eta} \circ f_{k}-\boldsymbol{\eta} \circ \bar{u}_{k}\right)\right|^{2}\right)}_{=: I(k)}>0.
\end{equation}
As in the previous argument, we introduce the auxiliary normalized functions $h_{k}=E_{k}^{-1 / 2} f_{k}$ and, after extraction of a subsequence, the function $h$ satisfies \eqref{e:cc(4.2)} and \eqref{e:cc(4.3)}. Moreover $\left\|\mathcal{G}\left(h_{k}, h\right)\right\|_{L^{2}\left(B_{3}\right)} \rightarrow 0 .$ We next claim (and prove)
\begin{enumerate}
\item[(i)] $\lim _{k} \int_{B_{2}}\left|D h_{k}\right|^{2}=\int_{B_{2}}\left|D h\right|^{2}$,
\item[(ii)] $h$ is Dir-minimizing in $B_{2}$.
\end{enumerate}

Indeed, if $(i)$ were false, then there is a positive constant $c_{2}$ such that, for any $r \in[5 / 2,3]$,
\begin{equation}\label{e:5.22DS3}
\int_{B_{r}} \frac{\left|D h\right|^{2}}{2} 
\leq \int_{B_{r}} \frac{\left|D h_{k}\right|^{2}}{2}-c_{2} 
\leq \frac{\mathbf{e}_{T_{k}}\left(B_{r}\right)}{E_{k}}-\frac{c_{2}}{2},
\end{equation}
provided $k$ is large enough (where the last inequality is again an effect of the Taylor expansion of \cite[Remark 5.4]{DS3}). We next define the competitor currents $S_{k}$ as in the argument leading to \eqref{e:5.20DS3}. Replacing in the argument above \eqref{e:5.11DS3} and \eqref{e:5.16DS3} by \eqref{e:5.22DS3}, we deduce again \eqref{e:5.20DS3}. On the other hand \eqref{e:5.20DS3} contradicts the minimality of $T_{k}$. So we conclude that $(i)$ is true. 

If $(ii)$ were false, then $h$ is not Dir-minimizing in $B_{2}$. Thus, we can find a competitor $\tilde{h}\in W^{1,2}(B_3, \Is{Q})$ with less energy in the ball $B_{2}$ than $h$ and such that $\tilde{h}=h$ on $B_{3} \setminus B_{5 / 2}$. So for any $r \in[5 / 2,3]$, the function $\tilde{h}$ satisfies
\begin{equation}\label{e:5.23DS3}
\int_{B_{r}} \frac{\big|D \tilde{h}\big|^{2}}{2} 
\leq \int_{B_{r}} \frac{\left|D h\right|^{2}}{2}-c_{2}
=\lim _{k \to \infty} \int_{B_{r}} \frac{\left|D h_{k}\right|^{2}}{2}-c_{2} 
\leq \frac{\mathbf{e}_{T}\left(B_{r}\right)}{E_{k}}-\frac{c_{2}}{2},
\end{equation}
provided $k$ is large enough (here $c_{2}>0$ is some constant independent of $r$ and $k$). On the other hand, $\tilde{h}=h$ on $B_{3} \setminus B_{5 / 2}$ and therefore $\left\|\mathcal{G}\left(\tilde{h}, h_{k}\right)\right\|_{L^{2}\left(B_{3} \setminus B_{5 / 2}\right)} \rightarrow 0$.
We then construct the competitor current $S_{k}$ of \eqref{e:5.18DS3}. This time however, we use the map $\tilde{h}$ in place of $h$ to construct $H_{k}$ via Proposition \ref{p:competitor} and we reach the contradiction \eqref{e:5.20DS3} using \eqref{e:5.23DS3} in place of \eqref{e:5.11DS3} and \eqref{e:5.16DS3}. We next set $\bar{u}_k:=E_{k}^{1 / 2} h$ and we will show that $I(k) \rightarrow 0$, violating \eqref{e:5.21DS3}. Observe first that as $\left\|\mathcal{G}\left(h_{k}, h\right)\right\|_{L^{2}} \rightarrow 0,$ we have $D\left(\boldsymbol{\xi} \circ h_k\right)-D\left(\boldsymbol{\xi} \circ h\right) \rightarrow 0$ weakly in $L^{2}$ (recall the definition of $\boldsymbol{\xi}=\boldsymbol{\xi}_{BW}$ in \cite[Section 2.5]{DS3}). So, (i) and the identities $\left|D\left(\boldsymbol{\xi} \circ h_{k}\right)\right|=\left|D h_{k}\right|$, $\left|D\left(\boldsymbol{\xi} \circ h\right)\right|=\left|D h\right|$ imply that $D\left(\boldsymbol{\xi} \circ h_{k}\right)-D\left(\boldsymbol{\xi} \circ h\right)$ converges strongly to $0$ in $L^{2}$.
If we next set $\hat{h}=\sum_{i} \llbracket h^{i}-\boldsymbol{\eta} \circ h\rrbracket$ and $\hat{h}_{k}=\sum_{i} \llbracket h_{k}^{i}-\boldsymbol{\eta} \circ h_{k} \rrbracket,$ we obviously have
$\left\|\mathcal{G}\left(\hat{h}, \hat{h}_{k}\right)\right\|_{L^{2}}+\left\|\boldsymbol{\eta} \circ h-\boldsymbol{\eta} \circ h_{k}\right\|_{L^{2}} \rightarrow 0 .$ Recall however that the Dirichlet energy
enjoys the splitting
$$
\operatorname{Dir}\left(h_{k}\right)=Q \int\left|D\left(\boldsymbol{\eta} \circ h_{k}\right)\right|^{2}+\operatorname{Dir}\left(\hat{h}_{k}\right), \quad \operatorname{Dir}\left(h\right)=Q \int\left|D\left(\boldsymbol{\eta} \circ h\right)\right|^{2}+\operatorname{Dir}\left(\hat{h}\right).
$$
So (i) implies that the Dirichlet energy of $\boldsymbol{\eta} \circ h_{k}$ and $\hat{h}_{k}$ converge, respectively, to the one of $\eta \circ h$ and $\hat{h}$ (which, we recall again, are independent of $k$ because the $h_{k}$ 's are translating sheets). We thus infer that $D\left(\boldsymbol{\eta} \circ h\right)-D\left(\boldsymbol{\eta} \circ h_{k}\right)$ converges to 0 strongly in $L^{2}$. Coming back to $\bar{u}_{k}$ we observe that $\bar{u}_k$ is Dir-minimizing and
$$
E_{k}^{-1} \int_{B_{2}} \mathcal{G}\left(\bar{u}_{k}, f_{k}\right)^{2}=\int_{B_{2}} \mathcal{G}\left(h, h_{k}\right)^{2} \rightarrow 0.
$$
So,
$$
\begin{aligned}
\limsup _{k} I(k) \leq &\limsup _{k} \int_{B_{2}}\left(\left|D h_{k}\right|-\left|D h\right|\right)^{2}+\left|D\left(\boldsymbol{\eta} \circ h_{k}-\etaa \circ h\right)\right|^{2}.
\end{aligned}
$$
Thus $I(k) \rightarrow 0,$ which contradicts \eqref{e:5.21DS3}.
\end{proof}

\subsection{Technical lemmas} 
\begin{lemma}\label{l:change-of-variable}
There is a positive geometric constant $c>0$ with the following property. Consider a $C^1$ function $\psi_1 :[0,4]\to \mathbb R$ such that $\psi_1 (0) = \psi_1' (0) =0$ and $\|\psi_1\|_{C^1}\leq c$. Then there is a map $\Phi: B_4\to B_4$ such that
\begin{itemize}
    \item $\Phi$ maps $B_s$ diffeomorphically onto itself for every $s\in (0,4]$;
    \item if we set $D:= \{(x_1, x_2): |x_1|\leq 4, x_2 > \psi_1 (x_1)\}$ then $\Phi$ maps $D\cap B_s$ diffeomorphically onto $B_s^+$ for every $s\in (0,4]$;
    \item $\|\Phi^{-1} - {\rm Id}\|_{C^1} + \|\Phi-{\rm Id}\|_{C^1} \leq C \|\psi_1\|_{C^1}$.
\end{itemize}
\end{lemma}

\begin{proof} We use polar coordinates $(\theta,r)$ and let the angle $\theta$ vary from $-\frac{\pi}{2}$ (included) to $\frac{3\pi}{2}$ (excluded). It is in fact easier to define the map $\Phi^{-1}$.
If $c$ is sufficiently small, each circle $\partial B_s$ intersects the graph of $\psi_1$ in exactly two points, given in polar coordinates by $(\theta_r (s), s)$ and $(\theta_l (s),s)$, with $\theta_l (s) > \theta_r (s)$. Furthermore, again assuming $c$ is sufficiently small, $|\theta_r (s)|\leq \frac{\pi}{4}$ and $|\theta_l (s) - \pi|\leq \frac{\pi}{4}$. In polar coordinates the map $\Phi^{-1}$ is then defined on $B^+_4$ by the formula
\[
\Phi^{-1} (\theta, s) = \left(\frac{\theta_r (s) (\pi-\theta) + \theta_l (s) \theta}{\pi}, s\right)\, .
\]
The verification that $\|\Phi^{-1} - {\rm Id}\|_{C^1}\leq C \|\psi_1\|_{C^1}$ is left to the reader.

We then need to extend the map to the lower half disk keeping the same estimate. This could be reached for instance by the formula
\[ \Phi^{-1} (\theta,s) = 
\left(\frac{2\pi-(\theta_l -\theta_r)}{\pi} \theta e^{a(\theta-\pi)(\theta-2\pi)} + 2\theta_l-\theta_r, s\right) 
\qquad \mbox{for $\pi < \theta < 2\pi$,}
\]
where $a= a (s):= \pi^{-2}(1-\frac{\theta_l (s)-\theta_r (s)}{2\pi-(\theta_l (s)-\theta_r (s))})$. 
\end{proof}

In the next proposition we want to ``patch'' functions defined on the upper half disk $B_s^+$ which vanish on the $x_1$-axis. For convenience we introduce the notation $\cH_s$ horizontal boundary for $\cH_s = \{(x_1, 0) : |x_1|<s\}$.

\begin{proposition}\label{p:competitor}
Consider two radii $1\leq r_0<r_1 < 4$ and maps 
$h_k,h \in W^{1,2}(B^+_{r_1}, \Iqs)$ satisfying
\[
\sup_k \D(h_k,B^+_{r_1})< + \infty \quad\text{and}\quad
\|\cG(h_k,h)\|_{L^2(B^+_{r_1}\setminus B_{r_0})} \to 0
\]
and $h_k|_{\cH_{r_1}} = h|_{\cH_{r_1}} = Q \a{0}$. 
Then for every $\eta>0$, there exist $r\in ]r_0, r_1[$, a subsequence of $\{h_k\}_k$
(not relabeled)
and functions $H_k\in W^{1,2} (B^+_{r_1}, \Iqs)$ such that:
\begin{itemize}
\item $H_{k}\vert_{B^+_{r_1}\setminus B^+_r} = h_{k}\vert_{B^+_{r_1}\setminus B^+_r}$;
\item $H_k|_{\cH_s} =Q \a{0}$ and
\item $\D(H_{k}, B^+_{r_1}) \leq \D(h , B^+_{r_1}) + \eta$.
\end{itemize}
Moreover, there is a dimensional constant $C$ and a constant $C^*$ (depending on $\eta$ and the
two sequences, but not on $k$) such that
\begin{align}
\Lip(H_{k}) &\leq C^*\, (\Lip(h_{k}) + 1) \,,\label{e:raccordo1}\\
\|\cG(H_{k}, h_{k})\|_{L^2(B^+_{r})} &\leq C \D(h_{k}, B^+_{r}) + 
C \D(H_{k}, B^+_{r})\, ,\label{e:raccordo2}\\
\|\etaa\circ H_{k}\|_{L^1(B^+_{r_1})} &\leq C^*\, \|\etaa\circ h_{k}\|_{L^1(B^+_{r_1})} +
C \|\etaa\circ h\|_{L^1 (B^+_{r_1})}\, .\label{e:raccordo3}
\end{align}
\end{proposition}

Before coming to the proof of the proposition we state the following variant of the Lipschitz approximation in \cite[Lemma 4.5]{DS3}. Observe that the only difference is that our functions are defined on the upper half disks and vanish on the horizontal boundary. We need the Lipschitz approximation $f_\varepsilon$ to satisfy the same requirement.

\begin{lemma}[Lusin type Lipschitz approximation]\label{l:approx}
Let $f\in W^{1,2}(B^+_r,\Iq)$ be such that $f|_{\cH_r} = Q \a{0}$.
Then for every $\varepsilon >0$ there exists $f_\eps\in \Lip (B^+_r,\Iq)$ satisfying $f_\eps|_{\cH_r} = Q \a{0}$ and
\begin{align}
\int_{B^+_r}\cG(f,f_\eps)^2+\int_{B^+_r}\big(|Df|-|Df_\eps|\big)^2
+ \int_{B^+_r}\big(|D(\etaa\circ f)|-|D(\etaa\circ f_\eps)|\big)^2
\leq\eps\, .\label{e:approx interior}
\end{align}
If in addition $f\vert_{\de B^+_r\setminus \cH_r}\in W^{1,2}(\de B_r,\Iq)$,
then $f_\eps$ can be chosen to satisfy also
\begin{equation}\label{e:approx bordo}
\int_{\de B^+_r\setminus \cH_r}\cG(f,f_\eps)^2+\int_{\de B^+_r\setminus \cH_r}\big(|Df|-|Df_\eps|\big)^2 \leq \eps.
\end{equation}
\end{lemma}
Now we need the following interpolation lemma.
\begin{lemma}[Interpolation]\label{l:InterpolationSimpler} There exists a constant $C_{0}=C_{0}(n, Q)>0$ with the following property. Assume $r \in ] 1,3\left[, f \in W^{1,2}\left(B_{r}, \mathcal{A}_{Q}\right)\right.$ satisfies $f|_{\cH_r} = Q \a{0}$ and $\left.f\right|_{\partial B_{r}} \in W^{1,2}\left(\partial B_{r}, \mathcal{A}_{Q}\right)$, and $g \in W^{1,2}\left(\partial B_{r}^+, \mathcal{A}_{Q}\right)$ is such that $g|_{\cH_r\cap\partial B_r^+} = Q \a{0}$. Then, for every $\left.\varepsilon \in\right] 0, r[$, there exists a function $h_\varepsilon \in W^{1,2}\left(B_{r}, \mathcal{A}_{Q}\right)$ such that $\left.h_\varepsilon\right|_{\partial B_{r}}=g$, $h_\varepsilon|_{\cH_r} = Q \a{0}$ and
\begin{align}
    \int_{B_{r}^+}|D h_\varepsilon|^{2} 
    &\leq \int_{B_{r}^+}|D f|^{2}+\varepsilon \int_{\partial B_{r}^+}\left(\left|D_{\tau} f\right|^{2}+\left|D_{\tau} g\right|^{2}\right)+\frac{C_{0}}{\varepsilon} \int_{\partial B_{r}^+} \mathcal{G}(f, g)^{2}\,,\label{e:4.14GAFA} \\
    \operatorname{Lip}(h_\varepsilon) 
    &\leq C_{0}\left\{\operatorname{Lip}(f)+\operatorname{Lip}(g)+\varepsilon^{-1} \sup _{\partial B_{r}^+} \mathcal{G}(f, g)\right\}\,,\label{e:4.15GAFA}\\
    \int_{B_{r}^+}|\boldsymbol{\eta} \circ h_\varepsilon| 
    &\leq C_{0} \int_{\partial B_{r}^+}|\boldsymbol{\eta} \circ g|+C_{0} \int_{B_{r}^+}|\boldsymbol{\eta} \circ f|\,, \label{e:4.16GAFA}
\end{align}
where $D_{\tau}$ denotes the tangential derivative.
\end{lemma}

\begin{proof}
The proof is the same as in \cite[Lemma 4.6]{DS3}, because the map constructed there by the linear interpolation on the annulus and taking $f$ in the interior disk vanishes on $\cH_{r_1}$.
\end{proof}

\begin{proof}[Proof of Lemma \ref{l:approx}] We can apply directly \cite[Lemma 5.5]{DDHM} to obtain a Lipschitz function $\tilde{f}_\varepsilon$ satisfying $(\tilde{f}_\varepsilon)_{|\cH_r} = Q \a{0}$ and \eqref{e:approx interior}. 
\end{proof}

\begin{proof}[Proof of Proposition \ref{p:competitor}] The proof goes along the same lines as the proof of \cite[Proposition 4.4]{DS3} using Lemmas \ref{l:approx} and \ref{l:InterpolationSimpler} instead of \cite[Lemma 4.5, Lemma 4.6]{DS3}, taking into account that the situation here is simpler because we do not have translating sheets. For the sake of completeness we report here the details. Set for simplicity $A_{k}:=\left\|\mathcal{G}\left(h_{k}, h\right)\right\|_{L^{2}\left(B_{r_{1}} ^+\setminus B_{r_{0}}^+\right)}$ and $B_{k}:=$ $\left\|\etaa \circ h_{k}\right\|_{L^{1}\left(B_{r_{1}}^+\right)} .$ If for any $k$ large enough $A_{k} \equiv 0$, then there is nothing to prove and so we can assume that, for a subsequence (not relabeled) $A_{k}>0 .$ In case that for yet another subsequence (not relabeled) $B_{k}>0,$ we consider the function
\begin{equation}\label{e:4.17GAFA}
\psi_{k}(r):=\int_{\partial B_{r}}\left(\left|D h_{k}\right|^{2}+\left|D h\right|^{2}\right)+A_{k}^{-2} \int_{\partial B_{r}} \mathcal{G}\left(h_k, h\right)^{2}+B_{k}^{-1} \int_{\partial B_{r}}\left|\boldsymbol{\eta} \circ h_{k}\right|.
\end{equation}
By assumption $\lim \inf _{k} \int_{r_{0}}^{r_{1}} \psi_{k}(r) d r<\infty .$ Hence by Fatou's Lemma, there is an $\left.r \in\right] r_{0}, r_{1}[$
and a subsequence (not relabeled) such that $\lim _{k} \psi_{k}(r)<\infty .$ Thus, for some $M>0$ we have
\begin{align}
    \int_{\partial B_{r}^+} \mathcal{G}\left(h_{k}, h\right)^{2} 
    &\rightarrow 0\,,\label{e:4.18GAFADS3}\\
    \operatorname{Dir}\left(h, \partial B_{r}^+\right)+\operatorname{Dir}\left(h_{k}, \partial B_{r}^+\right) 
    &\leq M\,, \label{e:4.19GAFADS3}\\
    \int_{\partial B_{r}^+}\left|\boldsymbol{\eta} \circ h_{k}\right| 
    &\leq M\left\|\boldsymbol{\eta} \circ h_{k}\right\|_{L^{1}\left(B_{r_{1}}\right)}\,.\label{e:4.20GAFADS3}
\end{align}
In case $B_{k}=0$ for all $k$ large enough, we define $\psi_{k}$ by dropping the last summand in \eqref{e:4.17GAFA} and reach the same conclusion. We apply Lemma \ref{l:approx} with $f=h$, $r=r_1$ and find a Lipschitz function $h_{\bar{\varepsilon}_1}$ satisfying the conclusion of the lemma with $\bar{\varepsilon}_{1}=\bar{\varepsilon}_{1}(\eta, M)>0$ (which will be chosen later). In particular we have
\begin{align*}
    \left\|\mathcal{G}\left(h_{k}, h_{\bar{\varepsilon}_1}\right)\right\|_{L^{2}(B_{r_1}^+\setminus B_{r_0}^+)} 
    &\le \left\|\mathcal{G}\left(h_{k}, h\right)\right\|_{L^{2}(B_{r_1}^+\setminus B_{r_0}^+)}+\left\|\mathcal{G}\left(h, h_{\bar{\varepsilon}_1}\right)\right\|_{L^{2}(B_{r_1}^+\setminus B_{r_0}^+)}
    \le o(1)+\bar{\varepsilon}_{1}\,,\\
    \operatorname{Dir}\left(h_{\bar{\varepsilon}_{1}}, \partial B_{r}^+\right) 
    &\leq \operatorname{Dir}\left(h, \partial B_{r}^+\right)
    \le M+\bar{\varepsilon}_{1}\,.
\end{align*}
To obtain also the estimate \eqref{e:raccordo3}, which will be required in the construction of the center manifold, we argue along the same lines of \cite[Proposition 4.4]{DS3}. For $h_{\bar{\varepsilon}_{1}}=\sum_{i=1}^Q\a{(h_{\bar{\varepsilon}_{1}})_i}$ we set $\bar{h}_{\bar{\varepsilon}_{1}}:=\sum_{i=1}^Q\a{(h_{\bar{\varepsilon}_{1}})_i-\etaa\circ h_{\bar{\varepsilon}_{1}}+(\etaa\circ h)*\varphi_{\rho}}$, where $\varphi_\rho(x):=\frac1{\rho^n}\varphi(\frac x\rho)$, and $\varphi(x)=\bar{\varphi}(x-z_0)$ with $\bar{\varphi}$ being the standard bump function with support in $B_1 (0)$, $z_0:=(0,-2)$ and $\rho$ will be chosen small enough later. Observe that $\supp(\varphi_\rho)=B_\rho(\rho z_0)\subseteq B_r^-$ for every $\rho$ small enough and $\supp(\varphi)=B_1(z_0)$. The reason to introduce this convolution kernel $\varphi_\rho$ with support contained in $B_r^-$ is that we need to preserve the zero boundary condition on $\cH_r$. Indeed, we claim that such an $\bar{h}_{\bar{\varepsilon}_{1}}$ satisfies $(\bar{h}_{\bar{\varepsilon}_{1}})|_{\cH_r} = Q\a{0}$ in addition to all the other conclusion of the proposition. The fact that $(\bar{h}_\varepsilon)|_{\cH_r} = Q\a{0}$ is a simple consequence of the definitions and we leave it to the reader. Observe that the standard approximation properties of mollifiers reinterpreted suitably extends to this new kind of kernel. In particular, we can choose $\rho$ small enough to have 
\begin{align}
    Q^2\|\etaa\circ h-(\etaa\circ h)*\varphi_\rho\|_{L^2}^2
    &\le\bar{\varepsilon}_{1}\,, \label{e:convolution}\\
    \|D(\etaa\circ h)-D((\etaa\circ h)*\varphi_\rho)\|_{L^2}^2
    &\le\bar{\varepsilon}_{1}\,,\label{e:convolution1}
\end{align}
for some small $\bar{\varepsilon}_{1}$.
These last two inequalities combined with \eqref{e:4.18GAFADS3}, \eqref{e:4.19GAFADS3}, \eqref{e:4.20GAFADS3} imply
\begin{itemize}
    \item $\displaystyle \left\|\mathcal{G}\left(h_{k}, \bar{h}_{\bar{\varepsilon}_{1}}\right)\right\|_{L^{2}} 
        \stackrel{\eqref{e:convolution}}{\le} \left\|\mathcal{G}\left(h_{k}, h\right)\right\|_{L^{2}}+2\left\|\mathcal{G}\left(h, \bar{h}_{\bar{\varepsilon}_{1}}\right)\right\|_{L^{2}}+\bar{\varepsilon}_{1} \leq o(1)+3 \bar{\varepsilon}_{1}\,,
    $
    \item $\displaystyle \operatorname{Dir}\left(\bar{h}_{\bar{\varepsilon}_{1}}, \partial B_{r}\right) 
        \leq 2 M+2 \bar{\varepsilon}_{1}\,,
    $
    \item $\displaystyle \operatorname{Dir}\left(\bar{h}_{\bar{\varepsilon}_{1}}, B_{r}\right)
        = \sum_{i} \int_{B_{r}}\left|D\left(\bar{h}_{\bar{\varepsilon}_{1}}\right)_{i}-D\left(\boldsymbol{\eta} \circ \bar{h}_{\bar{\varepsilon}_{1}}\right)+D\left((\etaa \circ h)* \varphi_{\bar{\rho}}\right)\right|^{2}
    $
    \vspace*{-0.5cm}
    \begin{align*}
        \ &= \int_{B_{r}}\left(\left|D \bar{h}_{\bar{\varepsilon}_{1}}\right|^{2}-Q\left|D\left(\boldsymbol{\eta} \circ \bar{h}_{\bar{\varepsilon}_{1}}\right)\right|^{2}+Q\left|D\left((\etaa \circ h)* \varphi_{\bar{\rho}}\right)\right|^{2}\right) \\
        &= Q \int_{B_{r}}\left(\left|D\left(\boldsymbol{\eta} \circ h\right)\right|^{2}-\left|D\left(\boldsymbol{\eta} \circ \bar{h}_{\bar{\varepsilon}_{1}}\right)\right|^{2} + \left|D\left(\boldsymbol{\eta} \circ h * \varphi_{\bar{\rho}}\right)\right|^{2}-\left|D\left(\boldsymbol{\eta} \circ h\right)\right|^{2}\right) \\
        &\quad +\operatorname{Dir}\left(\bar{h}_{\bar{\varepsilon}_{1}}, B_{r}\right)\\
        &\leq \operatorname{Dir}\left(h_{\bar{\varepsilon}_{1}}, B_{r}\right)+2 Q \bar{\varepsilon}_{1}\,,
    \end{align*}
\end{itemize}
where we used \eqref{e:approx interior},\eqref{e:convolution1} in the last inequality.
We can then apply the interpolation Lemma \ref{l:InterpolationSimpler} with $f=\bar{h}_{\bar{\varepsilon}_{1}}$ and $g={h_{k}}_{|\partial B_{r}^+}$, and $\varepsilon=\bar{\varepsilon}_{2}=\bar{\varepsilon}_{2}(\eta, M)>0$ to get maps $H_{k}$ satisfying $\left.H_{k}\right|_{\partial B_{r}^+}=\left.h_{k}\right|_{\partial B_{r}^+}$, $H_{k}\vert_{B^+_{r_1}\setminus B^+_r} = h_{k}\vert_{B^+_{r_1}\setminus B^+_r}$. Now, we use \eqref{e:4.18GAFADS3}, \eqref{e:4.19GAFADS3}, \eqref{e:4.20GAFADS3} \eqref{e:approx interior} and \eqref{e:approx bordo} to deduce
\begin{eqnarray}
\operatorname{Dir}\left(H_{k}, B_{r}^+\right) & \stackrel{\eqref{e:4.14GAFA}}{\leq} & \operatorname{Dir}\left(\bar{h}_{\bar{\varepsilon}_{1}}, B_{r}^+\right)+\bar{\varepsilon}_{2} \operatorname{Dir}\left(\bar{h}_{\bar{\varepsilon}_{1}}, \partial B_{r}^+\right)+\bar{\varepsilon}_{2} \operatorname{Dir}\left(h_{k}, \partial B_{r}^+\right) \notag \\
&& + \frac{C_{0}}{\bar{\varepsilon}_{2}} \int_{\partial B_{r}^+} \mathcal{G}\left(\bar{h}_{\bar{\varepsilon}_{1}}, h_{k}\right)^{2} \notag \\
& \stackrel{\eqref{e:approx bordo}}{\le} & \operatorname{Dir}\left(h, B_{r}^+\right)+\bar{\varepsilon}_{1}+2 Q \bar{\varepsilon}_{1}+3\bar{\varepsilon}_{2} \left[\operatorname{Dir}\left(h, \partial B_{r}^+\right)+\bar{\varepsilon}_{1}\right]+\bar{\varepsilon}_{2}M \notag \\
&& + \frac{C_{0}}{\bar{\varepsilon}_{2}}\left[\int_{\partial B_{r}^+} \mathcal{G}\left(h, h_{k}\right)^{2}+\int_{\partial B_{r}^+} \mathcal{G}\left(h_{\bar{\varepsilon}_{1}}, h\right)^{2}\right]\notag \\
 & \leq & \operatorname{Dir}\left(h, B_{r}^+\right)+\bar{\varepsilon}_{1}(1+2 Q) +\bar{\varepsilon}_{2}(4M+3\bar{\varepsilon}_{1}) +C_{0} \bar{\varepsilon}_{2}^{-1}\left[o(1)+\bar{\varepsilon}_{1}\right].\notag
\end{eqnarray}
 An appropriate choice of the parameters $\bar{\varepsilon}_{1}$ and $\bar{\varepsilon}_{2}$ gives the desired bound $\operatorname{Dir}\left(H_{k}, B_{r}\right) \leq$ $\operatorname{Dir}\left(h, B_{r}\right)+\eta$ for $k$ large enough.
Observe next that, by construction, $\operatorname{Lip}\left(\bar{h}_{\bar{\varepsilon}_{1}}\right)$ depends on $\eta$ and $h$, but not on $k$.  Moreover, we have
$$
\left\|\mathcal{G}\left(\bar{h}_{\bar{\varepsilon}_{1}}, h_{k}\right)\right\|_{L^{\infty}\left(\partial B_{r}\right)} \le C\left\|\mathcal{G}\left(\bar{h}_{\bar{\varepsilon}_{1}}, h_{k}\right)\right\|_{L^{2}\left(\partial B_{r}\right)}+C \operatorname{Lip}\left(h_{k}\right)+C \operatorname{Lip}\left(\bar{h}_{\bar{\varepsilon}_{1}}\right).
$$
To prove the last inequality put $F(x):=\mathcal{G}\left(\bar{h}_{\bar{\varepsilon}_{1}}(x), h_{k}(x)\right)$ and observe that $F(x)\le F(y)+\Lip(F)|x-y|$, then integrate in $y$ and use the Cauchy-Schwarz inequality combined with the fact that $\Lip(F)\le C(\Lip(\bar{h}_{\bar{\varepsilon}_{1}})+ \Lip(h_{k}))$.
Thus \eqref{e:raccordo1} follows from  \eqref{e:4.15GAFA}. Finally, \eqref{e:raccordo2} follows from the Poincaré inequality applied to $\mathcal{G}\left(H_{k}, h_{k}\right)$ (which vanishes identically on $\left.\partial B_{r}^+\right)$, in fact we have
\begin{eqnarray*}
\|\mathcal{G}\left(H_{k}, h_{k}\right)\|^2_{L^2(B_{r_1}^+)}
\le C\|\nabla\mathcal{G}\left(H_{k}, h_{k}\right)\|^2_{L^2(B_{r_1}^+)}
\le C \D(h_{k}, B^+_{r_1}) + C \D(H_{k}, B^+_{r_1}).
\end{eqnarray*}
\eqref{e:raccordo3} follows from \eqref{e:4.16GAFA}, because of \eqref{e:4.20GAFADS3} and $\left\|\boldsymbol{\eta} \circ \bar{h}_{\bar{\varepsilon}_{1}}\right\|_{L^{1}\left(B_{r}\right)}=\left\|\left(\boldsymbol{\eta} \circ h\right) * \varphi_{\bar{\rho}}\right\|_{L^{1}\left(B_{r}\right)} \leq\left\|\boldsymbol{\eta} \circ h\right\|_{L^{1}\left(B_{r_{1}}\right)}$ if $\bar{\rho}$ is also chosen small
enough such that $r+\bar{\rho}<r_{1}$. Indeed, observe that 
\begin{eqnarray*}
\|\etaa\circ H_k\|_{L^1(B_{r_1}^+)} & = & \|\etaa\circ H_k\|_{L^1(B_{r}^+)}+\|\etaa\circ h_k\|_{L^1(B_{r_1}^+\setminus B_{r}^+)}\\
& \stackrel{\eqref{e:4.16GAFA}}{\le} & C_{0} \int_{\partial B_{r}^+}|\boldsymbol{\eta} \circ h_k|+C_{0} \int_{B_{r}^+}|\etaa\circ \bar{h}_{\bar{\varepsilon}_{1}}|+\|\etaa\circ h_k\|_{L^1(B_{r_1}^+\setminus B_{r}^+)}\\
& \stackrel{\eqref{e:4.20GAFADS3}}{\le} & C_{0} \|\etaa\circ h_k\|_{L^1(B_r^+)}+C_{0} \int_{B_{r}^+}|(\etaa\circ h)*\varphi_{\rho}|+\|\etaa\circ h_k\|_{L^1(B_{r_1}^+\setminus B_{r}^+)}\\
& \stackrel{\eqref{e:convolution}}{\le} & C_{0} \|\etaa\circ h_k\|_{L^1(B_r^+)}+ C \|\etaa\circ h\|_{L^1(B_r^+)}+\|\etaa\circ h_k\|_{L^1(B_{r_1}^+\setminus B_{r}^+)}\\
& \le & C \|\etaa\circ h_k\|_{L^1(B_{r_1}^+)}+C \|\etaa\circ h\|_{L^1(B_{r_1}^+)},
\end{eqnarray*}
provided $\rho$ is chosen so small that $\bar{r}+\rho<r$.
\end{proof}

\section{Higher integrability estimate}\label{s:higher-int-estimate}

We consider the density $\mathbf{d}_T$ of the measure $\be_T$ with respect to the Lebesgue measure $|\cdot|$, i.e.
\[
\mathbf{d}_T (y) = \limsup_{s\to 0} \frac{\be_T (B_s (y))}{\pi s^2}\, .
\]
We will drop the subscript $T$ when the current in question is clear from the context.
Clearly, under the assumptions of Proposition \ref{p:Lipschitz_1}, $\|\mathbf{d}_T\|_{L^1} \leq C E$. Now, following the approach of \cite{DS3}, we wish to prove an $L^p$ estimate for a $p>1$, which is just a geometric constant. 

\begin{theorem}\label{t:L-p}
There exist constants $p >1$, $C$, and $\varepsilon>0$ (depending on $n$ and $Q$) such that, if $T$ is as in Proposition \ref{p:Lipschitz_1}, then
\begin{equation}\label{e:higher1}
\int_{\{\bd\leq1\}\cap B_2} \bd^{p} \leq C\, \left(E + \bA^{2}\right)^p.
\end{equation}
\end{theorem}

\subsection{Higher integrability for Dir-minimizers}
We start with an analogous estimate for the gradient of Dir-minimizers. 

\begin{proposition}\label{p:L-p}
There are constants $q>1,\delta>0$ and $C$ (depending only on $Q$ and $n$) with the following property. Consider a connected domain $D$ in $\mathbb R^2$ such that:
\begin{itemize}
    \item the curvature $\kappa$ of $\partial D$ enjoys the bound $\|\kappa\|_\infty \leq \delta$;
    \item $\partial D \cap B_{16} (x)$ is connected for every $x$.
\end{itemize}
Let $0<\rho\leq 1$ and $u: B_{8\rho} (x)\cap D \to \Iqs$ be a Dir-minimizing function such that $u|_{\partial  D\cap B_\rho (x)} = Q \a{g}$ for some $C^1$ function $g$. Then
\begin{equation}\label{e:L-p}
\left(\mint_{B_\rho (x)\cap D} |Du|^{2q}\right)^{\frac{1}{q}} \leq
C \mint_{B_{8\rho} (x)\cap D} |Du|^2 + C \|Dg\|_\infty^2\, .
\end{equation} 
\end{proposition}
\begin{proof}
First of all, the claim follows from \cite[Theorem 6.1]{DS3} when $B_{2\rho} (x) \subset D$, while it is trivial if $B_{2\rho} (x) \subset {\rm int}\, (D^c)$. We can thus assume, without loss of generality, that $B_{2\rho} (x)$ intersects $\partial D$. Let $y$ be a point in such intersection and observe that $B_\rho (x) \subset B_{4\rho} (y)$. The claim thus follows if we can show
\begin{equation}\label{e:L-p-bordo}
\left(\mint_{B_r (y)\cap D} |Du|^{2q}\right)^{\frac{1}{q}} \leq
C \mint_{B_{2r} (y)\cap D} |Du|^2 + C \|Dg\|_\infty^2\, ,
\end{equation} 
for every $y\in \partial D$ and every $r\leq 4$. We now define
\[
\bar{u} (z) = \sum_i \a{u_i (z) - \etaa\circ u (z)}\, ,
\]
and observe that $|Du|\leq |D\bar u| + Q |D \etaa\circ u|$, while $\etaa\circ u$ is a classical harmonic function such that $\etaa\circ u|_{\partial D\cap B_2} = g$, and $\bar u$ is a Dir-minimizing function such that $\bar u|_{\partial D \cap B_2} = Q \a{0}$. Observe that 
\[
\left(\mint_{B_r (y)\cap D} |D\etaa \circ u|^{2q}\right)^{\frac{1}{q}} \leq
C \mint_{B_{2r} (y)\cap D} |D\etaa\circ u|^2 + C \|Dg\|_\infty^2
\]
is a classical estimate for (single-valued) harmonic functions and that $|D\etaa\circ u|\leq |Du|$. Hence, it suffices to prove \eqref{e:L-p-bordo} when $g= Q\a{0}$. Moreover without loss of generality we can assume that $y=0$ and $r=1$. Our goal is thus to show
\[
\||Du|\|_{L^{2q} (B_1\cap D)}\leq C \||Du|\|_{L^2 (B_2\cap D)}\, ,
\]
under the assumption that $u|_{\partial D \cap B_2}= Q \a{0}$. 
If we extend $|Du|$ trivially to the complement of $D$, by setting it identically equal to $0$, the inequality is just an higher integrability estimate for the function $|Du|$ on $B_1$. By Gehring's lemma, it suffices to prove the existence of a constant $C$ such that
\begin{equation}\label{e:Gehring}
\||Du|\|_{L^2 (B_\rho (x))} \leq C \||Du|\|_{L^1 (B_{8\rho} (x))}
\end{equation}
whenever $B_{8\rho} (x) \subset B_2$. However, in the ``interior case'' $B_{2\rho} (x)\subset D$, the stronger 
\[
\||Du|\|_{L^2 (B_\rho (x))} \leq C \||Du|\|_{L^1 (B_{2\rho} (x))}
\]
is already proved in \cite[Proposition 6.2]{DS3}. Hence, arguing as above, it suffices to prove \eqref{e:Gehring}, with the ball $B_{4\rho} (x)$ replacing $B_\rho (x)$ in the left hand side, under the additional assumption $x\in \partial D$. Again by scaling, we are reduced to prove the following estimate
\begin{equation}\label{e:Gehring-2}
\||Du|\|_{L^2 (B_1\cap D)} \leq C \||Du|\|_{L^1 (B_2\cap D)}
\qquad \mbox{if $0\in \partial D$.} 
\end{equation}
First of all observe that, by our assumptions, if $\delta$ is sufficiently small, for every $r\in (1,2)$ the domain $D\cap B_r$ is biLipschitz equivalent to the half disk $B_r \cap\{(x_1,x_2): x_2>0\}$, with uniform bounds on the Lipschitz constants of the homeomorphism and its inverse. In particular, we recall that, by classical Sobolev space theory, we have
\[
\min_{c\in \mathbb R} \|f-c\|_{H^{1/2} (\partial (B_r\cap D))} 
\leq C \|Df\|_{L^1 (\partial (B_r\cap D))}
\]
for every classical function $f\in W^{1,1} (\partial B_r, \mathbb R)$.
Moreover there is an extension $F\in W^{1,2} (B_r\cap D)$ of $f$ such that
\begin{equation}\label{e:H1-2}
\|DF\|_{L^2 (B_r\cap D)} \leq C \|f-c\|_{H^{1/2} (\partial (B_r\cap D))}
\leq C \|Df\|_{L^1 (\partial (B_r\cap D))}\, .
\end{equation}
Thus, using Fubini and \eqref{e:H1-2}, under our assumptions on $u$, we find a radius $r\in (1,2)$ and an extension $v$ of the classical function $\xii\circ u|_{\partial (B_r\cap D)}$ to $B_r\cap D$ such that
\begin{equation}\label{e:intermedia}
\|D \xii\circ u\|_{L^2 (B_r\cap D)} \leq C
\|D\xii \circ u\|_{L^1 (\partial (B_r\cap D))} \leq 
C \|D\xii \circ u\|_{L^1 (B_2\cap D)}\leq C \\|Du|\|_{L^1 (D\cap B_2)}\, .
\end{equation}
If we consider the multivalued function $\xii^{-1} \circ \ro \circ v$, the latter has trace $w:= \xii^{-1} \circ \xii \circ u$ on $\partial (B_r\cap D)$. Therefore, by minimality of $u$,
\[
\||Du|\|_{L^2 (B_r\cap D)} \leq \|Dw\|_{L^2 (B_r\cap D)}
\leq C \|Dv\|_{L^2 (B_r\cap D)}\, .
\]
Combining the latter inequality with \eqref{e:intermedia} we achieve \eqref{e:Gehring-2}.
\end{proof}
\subsection{Improved excess estimates}
\begin{proposition}[Weak excess estimate]\label{p:WEE} For every $\eta >0,$ there exists $\varepsilon >$ 0 with the following property. Let $T$ be area minimizing and assume it satisfies Assumption \ref{Ass:app} in $\mathbf{C}_{4 s}(x) .$ If $E=\mathbf{E}\left(T, \mathbf{C}_{4 s}(x)\right) \leq \varepsilon,$ then
\begin{equation}\label{e:6.4}
\mathbf{e}_{T}(A) \leq \eta_{10} E s^2+C \mathbf{A}^{2} s^{4}
\end{equation}
for every $A \subset B_{s}(x)\cap D$ Borel with $|A| \leq \varepsilon \left|B_{s}(x)\right|$.
\end{proposition}
\begin{proof}
Without loss of generality, we can assume $s=1$ and $x=0 .$ We distinguish the two regimes: $E \leq \mathbf{A}^{2}$ and $\mathbf{A}^{2} \leq E .$ In the former, clearly $\mathbf{e}_{T}(A) \leq C E \leq C \mathbf{A}^{2} .$ In the latter, we let $f$ be the $E^{\frac{1}{8}}$ -Lipschitz approximation of $T$ in $\mathbf{C}_{3}$ and, arguing as for the proof of \cite[Theorem 5.2]{DS3} we find a radius $r \in(1,2)$ and a current $R$ such that
$$
\partial R=\left\langle T-\mathbf{G}_{f}, \varphi, r\right\rangle
$$
and
$$
\mathbf{M}(R) \leq \left(\frac{C}{\delta_*}(E+\bA^2r^2)\right)^2\le CE^{2-\frac{1}{2}}.
$$
Therefore, by the Taylor expansion in Remark $5.4,$ we have:
\begin{align}
\|T\|\left(\mathbf{C}_{r}\right) 
&\stackrel{minimality}{\leq} \mathbf{M}\left(\mathbf{G}_{f}\res\mathbf{C}_{r}+R\right) 
\stackrel{triangular}{\leq}\left\|\mathbf{G}_{f}\right\|\left(\mathbf{C}_{r}\right)+C E^{\frac{3}{2}}\nonumber\\ 
&\quad \stackrel{Taylor}{\leq} Q\left|B_{r}\right|+\int_{B_{r}} \frac{|D f|^{2}}{2}+C E^{\frac{5}{4}}\, .\label{e:6.5}
\end{align}
On the other hand, using again the Taylor expansion for the part of the current which coincides with the graph of $f,$ we deduce as well that
\begin{equation}\label{e:6.6}
\|T\|\left(\left(B_{r} \cap K\right) \times \mathbb{R}^{n}\right) \geq Q\left|B_{r} \cap K\right|+\int_{B_{r} \cap K} \frac{|D f|^{2}}{2}-C E^{\frac{5}{4}}.
\end{equation}
Subtracting \eqref{e:6.6} from \eqref{e:6.5}, we deduce
\begin{equation}\label{e:6.7}
\mathbf{e}_{T}\left(B_{r}\cap D \setminus K\right) \leq \int_{B_{r}\cap D \setminus K} \frac{|D f|^{2}}{2}+C E^{\frac{5}{4}}.
\end{equation}
If $\varepsilon$ is chosen small enough, we infer from \eqref{e:6.7} and \eqref{e:few energy} in Theorem \ref{t:o(E)} that
$$
\mathbf{e}_{T}\left(B_{r}\cap D \setminus K\right) \leq \bar \eta E+C E^{1+\gamma},
$$
for a suitable $\bar \eta>0$ to be chosen. Let now $A \subset B_{1}$ be such that $|A| \leq \varepsilon \pi$. If $\varepsilon$ is small enough, we can again apply Theorem \ref{t:o(E)} and so by \eqref{e:L-p} there is a Dir-minimizing $w$ such that $|D f|$ is close in $L^{2}$ (with an error $\bar \eta E)$ to $|D w|$ and by  \cite[Remark 5.5]{DS3} $\operatorname{Dir}(w) \leq C E .$ By Proposition \ref{p:L-p} we have
$\||D w|\|_{L^q\left(B_{1}\right)} \leq C E^{\frac12}$. Therefore we can deduce
\begin{eqnarray}\label{e:6.10}
\mathbf{e}_{T}(A) &\stackrel{\eqref{e:few energy}, \eqref{e:quasiarm}}{\leq} & \int_{A}|D w|^{2}+3 \eta E+C E^{1+\gamma} \notag\\ 
&\leq & C\|Dg\|_{\infty}^2|A|^{1-2/ q}+C\left(|A|^{1-2/q}+\bar \eta\right) E+C E^{1+\gamma} \notag\\
& \le & C\left(|A|^{1-2/q} +\bar \eta\right) E+C E^{\frac{5}{4}}.
\end{eqnarray}
Hence, if $\varepsilon$ and $\eta$ are suitably chosen, \eqref{e:6.4} follows from \eqref{e:6.10}.
\end{proof}

\subsection{Proof of Theorem \ref{t:L-p}} The proof follows from Proposition \ref{p:WEE} arguing exactly as in \cite[Section 6.3]{DS3}.

\section{Strong Lipschitz approximation}\label{s:strong-Lipschitz}
In this section we show how Theorem \ref{t:L-p} gives a simple proof of the following approximation result analogous to  \cite[Theorem 2.4]{DS3}. 
\begin{theorem}[Boundary Almgren Strong Approximation]\label{t:strong_Lipschitz}
There are geometric constants $\gamma_1>0$, $\varepsilon_A>0$, and $C>0$ with the following properties. 
Let $T$ and $\Gamma$ be as in Assumption \ref{Ass:app} with $\varepsilon = \varepsilon_A$, let $f$ be the $E^\gamma$-Lipschitz approximation and $K\subset B_{3r}$ the corresponding set where $\bG_f$ and $T$ coincide. Then:
\begin{align}\label{e:strong_Lip}
\Lip (f) &\leq C (E + r^2 \bA^2)^{\gamma_1}\\ \label{e:strong_Lip_Osc}
{\rm osc}\, (f) &\leq C \bh (T, \bC_{4r}) + C r (E+ r^2 \bA^2)^{\frac{1}{2}}\\ \label{e:strongKestimate}
|B_r\setminus K| + \be_T (B_r\setminus K) &\leq C r^2 (E+r^2 \bA^2)^{1+\gamma_1}\\
\left|\|T\| (A\times \mathbb R^n) - Q |A\cap D| - \frac{1}{2} \int_A |Df|^2\right|
&\leq C r^2 (E+r^2 \bA^2)^{1+\gamma_1}
\end{align}
for every closed set $A\subset B_r$.
\end{theorem}
We postpone the proof till the end of this section however we anticipate that it goes along the same line of \cite[Theorem 2.4]{DS3} using Theorems \ref{t:StrongExcEst} and \ref{t:regconv} below instead of \cite[Theorem 7.1]{DS3} and \cite[Theorem 7.3]{DS3} respectively. The substantial changes necessary to adapt the argument of the interior case, i.e., \cite[Theorem 2.4]{DS3} concerns mainly the proof of Theorem \ref{t:regconv} while the proof of Theorem \ref{t:StrongExcEst} is essentially the same as that of  \cite[Theorem 7.1]{DS3}. So we start by stating the Almgren's boundary strong excess estimate.

\begin{theorem}[Almgren's boundary strong excess estimate]\label{t:StrongExcEst} There are constants $\varepsilon_{11}, \gamma_{11}, C>0$ (depending on $n, Q)$ with the following property. Assume $T$ satisfies Assumption \ref{Ass:app} in $\bC_{4}$ and is area minimizing. If $E=\mathbf{E}\left(T, \bC_{4}\right)<\varepsilon_{11},$ then
\begin{equation}\label{e:StrongExcEst}
\mathbf{e}_{T}(A) \leq C\left((E+ \bA^2)^{\gamma_{11}}+|A|^{\gamma_{11}}\right)\left(E+\mathbf{A}^{2}\right),
\end{equation} for every Borel set $A \subset B_{\frac{9}{8}}$.
\end{theorem}
This estimate complements \eqref{e:higher1} enabling to control the excess also in the region where $\mathbf{d}>1$. We call it boundary strong Almgren's estimate because a similar formula in the interior case can be found in the big regularity paper (cf. \cite[Sections 3.24-3.26 and 3.30(8)]{Alm}) and is a strengthened version of Proposition \ref{p:WEE} that we called weak excess estimate. To prove \eqref{e:StrongExcEst} we construct a suitable competitor to estimate the size of the set $\tilde{K}$ where the graph of the $E^{\beta}$ Lipschitz approximation $f$ differs from $T$. Following Almgren, we embed $\Is{Q}$ in a large Euclidean space, via a bilipschitz embedding $\xii$. We then regularize $\xii \circ f$ by convolution and project it back onto $\mathcal{Q}=\xii\left(\Is{Q}\right) .$ To avoid loss of energy we need a rather special "almost projection" $\boldsymbol{\rho}_{\delta}^{\star}$ that preserves zero boundary data, i.e., $\boldsymbol{\rho}_\delta^\star(0)=0$.

\begin{proposition}$($\cite[Proposition 7.2]{DS3}$)$\label{p:rho*} For every $n, Q \in \mathbb{N} \setminus\{0\}$ there are geometric constants $\delta_{0}, C>$ 0 with the following property. For every $\delta \in] 0, \delta_{0}\left[\right.$ there is $\boldsymbol{\rho}_{\delta}^{\star}: \mathbb{R}^{N(Q, n)} \rightarrow \mathcal{Q}=$
$\xii\left(\Is{Q}\right)$ such that $\boldsymbol{\rho}_{\delta}^{\star}(0)=0$, $\left|\boldsymbol{\rho}_{\delta}^{\star}(P)-P\right| \leq C \delta^{8^{-nQ}}$ for all $P \in \mathcal{Q}$ and, for every $u \in W^{1,2}\left(\Omega, \mathbb{R}^{N}\right)$, the following holds:
\begin{equation}\label{e:EnergyEstimaterho*}
\int\left|D\left(\boldsymbol{\rho}_{\delta}^{\star} \circ u\right)\right|^{2} \leq\left(1+C \delta^{8^{-nQ-1}}\right) \int_{\left\{\operatorname{dist}(u, Q) \leq \delta^{nQ+1}\right\}}|D u|^{2}+C \int_{\left\{\operatorname{dist}(u, Q)>\delta^{nQ+1}\right\}}|D u|^{2}.
\end{equation}
\end{proposition}
\begin{proof} $\rho_\delta^\star$ is the projection obtained in \cite[Proposition 7.2]{DS3}.\end{proof}

Here we show the Strong Excess Approximation of Almgren in our version that takes into account the non-homogeneous boundary value problem,  concluding in this way the proof of Theorem \ref{t:strong_Lipschitz}. Theorem \ref{t:L-p} enters crucially in the argument when estimating the second summand of \eqref{e:EnergyEstimaterho*} for the regularization of $\xii \circ f$.
\subsection{Regularization by convolution with a non centered kernel}
Here we construct the competitor preserving the boundary conditions. 

\begin{proposition}\label{t:regconv} Let $\beta_{1} \in\left(0, \frac{1}{4}\right)$ and $T$ be an area minimizing current satisfying Assumption \ref{Ass:app} in $\mathbf{C}_{4} .$ Let $f$ be its $E^{\beta_1}$-Lipschitz approximation. Then, there exist constants $\bar{\varepsilon}_{12}, \gamma_{12}, C>0$ and a subset of radii $B \subset[9 / 8,2]$ with $|B|>1 / 2$ with the following properties. If $\mathbf{E}\left(T, \mathbf{C}_{4}\right)\leq \bar{\varepsilon}_{12},$ for every $\sigma \in B,$ there exists a $Q$-valued function $h \in \operatorname{Lip}\left(B_{\sigma}\cap D, \Is{Q}\right)$ such that
\begin{align*}
h |_{B_{\sigma}\cap\partial D}&=g,\\
h|_{\partial B_{\sigma}\cap D}&=f|_{\partial B_{\sigma}\cap D},\\
\operatorname{Lip}(h) &\leq C (E+ \mathbf{A}^2)^{\beta_1}, 
\end{align*}
and
\begin{equation}\label{e:regconvDirichletEnergyEstimate}
\int_{B_{\sigma}\cap D}|D h|^{2} \leq \int_{B_{\sigma} \cap K\cap D}|D f|^{2}+C \left(E+\mathbf{A}^{2}\right)^{1+\gamma_{12}}.
\end{equation}
\end{proposition}
\begin{proof} Since $|D f|^{2} \leq C \mathbf{d}_{T} \leq C E^{2\beta_1} \leq 1$ on $K,$ by Theorem \ref{t:L-p} there
is $q_{1}=2 p_{1}>2$ such that
\begin{equation}\label{e:7.4}
\| |D f|\|_{L^{q_{1}}\left(K \cap B_{2}\right)}^{2} \leq C\left(E+\mathbf{A}^{2}\right).
\end{equation}
Given two (vector-valued) functions $h_{1}$ and $h_{2}$ and two radii $0<\bar{r}<r,$ we denote by $\operatorname{lin}\left(h_{1}, h_{2}\right)$ the linear interpolation in $B_{r} \setminus \bar{B}_{\bar{r}}$ between $\left.h_{1}\right|_{\partial B_{r}}$ and $\left.h_{2}\right|_{\partial B_{\bar{r}}} .$ More
precisely, if $(\theta, t) \in \mathbb{S}_+^{m-1} \times[0, \infty)$ are spherical coordinates, then
$$
\operatorname{lin}\left(h_{1}, h_{2}\right)(\theta, t)=\frac{r-t}{r-\bar{r}} h_{2}(\theta, t)+\frac{t-\bar{r}}{r-\bar{r}} h_{1}(\theta, t).
$$
Next, let $\delta>0$ and $\varepsilon>0$ be two parameters and let $1<r_{1}<r_{2}<r_{3}<2$ be three radii, all to be chosen suitably later. 
First of all extend the function $g$ to the whole disk $B_3$ by making it coinstant in the direction $x_2$, i.e. $g (x_1, x_2) = g (x_1, \psi_1 (x_1))$. 
We then extend the $E^{\beta_1}$-Lipschitz approximation to a function $f^*$ defined on the entire $B_3$ by setting 
\[
f^* (x) =\left\{
\begin{array}{ll}
f(x) \qquad &\mbox{if $x\in B_3\cap D$}\\
Q\a{g(x)} \qquad &\mbox{if $x\in B_3\cap D^-$}\, .
\end{array}\right.
\]
From now to keep our notation simpler we denote $f^*$ as well by $f$. Observe moreover that
\[
\left.(\boldsymbol{\eta}\circ f)\right|_{D^-} = g\, .
\]
We next define a translation operator
$\oplus: A_{Q}\left(\mathbb{R}^N\right) \times \mathbb{R}^N\rightarrow \mathcal{A}_{Q} \left(\mathbb{R}^N\right)$ setting
\[
\quad T \oplus t=\sum_{i=1}^{Q} \a{t_{i}+t} \qquad\mbox{for }\;\; T=\sum_{i=1}^{Q} \a{t_i}\, .
\]
We then introduce
$\tilde{f}:=f\oplus (-\etaa\circ f)$, so that
$\tilde{f}|_{D^-}=Q\a{0}$ and $\boldsymbol{\eta}\circ\tilde{f}=0$.

Next we define, as in the proof of Proposition \ref{p:competitor},  $\varphi_\varepsilon(x):=\frac1{\varepsilon^n}\varphi(\frac x\varepsilon)$, and $\varphi(x)=\bar{\varphi}(x-z_0)$ with $\bar{\varphi}$ being the standard bump function with support in $B_1 (0)$ and $z_0:=(0,-2)$. 
We therefore set
\[
 \tilde{h}_\varepsilon:=(\boldsymbol{\eta}\circ f)*\varphi_\varepsilon- g *\varphi_\varepsilon+ g\, .
\]
We easily see that $(\tilde{h}_\varepsilon)|_{\partial D\cap B_{r_3}}=g|_{\partial D\cap B_{r_3}}$, and 
\[
\Lip(\tilde{h}_{\varepsilon})\le C(E+\bA^2)^{\beta_1}\,. 
\]
Recall the maps $\boldsymbol{\rho}_\delta^\star$ and $\xii$ of $[\mathrm{LS} 11 \mathrm{~b},$ Theorem 2.1$]$ and observe that $\xii(Q\a{0})=0$ and $\boldsymbol{\rho}_{\delta}^{\star}(0_{\mathbb{R}^n})=0_{\mathbb{R}^n}$.
We then set $\tilde{f}'_1:= \xii \circ \tilde{f}$
$$
\tilde{g}^{\prime}_{\delta, \varepsilon, s}:=\left\{\begin{array}{ll}
\sqrt{E+\bA^2} \boldsymbol{\rho}\circ\Phi\circ\operatorname{lin}\left(\frac{\tilde{f}_1^\prime\circ\Phi^{-1}}{\sqrt{E+\bA^2}}, \boldsymbol{\rho}^{\star}_\delta\left(\frac{\tilde{f}_1^\prime\circ\Phi^{-1}}{\sqrt{E+\bA^2}}\right)\right), & \text { in } (B_{r_{3}}\setminus B_{r_{2}})\cap D, \\
\sqrt{E+\bA^2} \boldsymbol{\rho}\circ\Phi\circ \operatorname{lin}\left(\boldsymbol{\rho}_{\delta}^{\star}\left(\frac{\tilde{f}_1^\prime\circ\Phi^{-1}}{\sqrt{E+\bA^2}}\right), \boldsymbol{\rho}_{\delta}^{\star}\left((\frac{\tilde{f}_1^\prime* \varphi_{\varepsilon}) \circ \Phi^{-1}}{\sqrt{E+\bA^2}}\right) \right), & \text { in } (B_{r_{2}}\setminus B_{r_{1}})\cap D,\\
\sqrt{E+\bA^2} \boldsymbol{\rho}_\delta^{\star}\left(\frac{\tilde{f}_1^{\prime}*\varphi_{\varepsilon}}{\sqrt{E+\bA^2}} \right), & \text { in } B_{r_1}\cap D,
\end{array}\right.
$$ 
where $\Phi$ is the diffeomorphism constructed in Proposition \ref{l:change-of-variable}.
Now, we define 
\begin{equation}
\hat{h}_{\delta, \varepsilon, s}:=\sum_{i=1}^Q\a{\left(\xii^{-1}\circ\tilde{g}^{\prime}_{\delta, \varepsilon, s}\right)_i-\etaa\circ\left(\xii^{-1}\circ\tilde{g}^{\prime}_{\delta, \varepsilon, s}\right)}, \text { in } B_{r_{3}}\cap D,
\end{equation}
and 
\begin{equation}\label{e:h_delta}
h_{\delta, \varepsilon, s}:=\sum_{i=1}^Q\a{\left(\xii^{-1}\circ\tilde{g}^{\prime}_{\delta, \varepsilon, s}\right)_i-\etaa\circ\left(\xii^{-1}\circ\tilde{g}^{\prime}_{\delta, \varepsilon, s}\right)+\tilde{h}_{\varepsilon}}, \text { in } B_{r_{3}}\cap D.
\end{equation}
Notice that the convolution of any function $u$ satisfying $u_{|B_3\setminus D}\equiv 0$ with $\varphi_\varepsilon$ for $\varepsilon$ small enough always produces smooth function $u*\varphi_\varepsilon$ satisfying $(u*\varphi_\varepsilon)_{|B_3\setminus D}\equiv 0$, because we have assumed that $\partial D$ is the graph of a Lipschitz function and so it stays inside a cone with fixed angles. With this last fact in mind it is easy to see that $(\tilde{g}^{\prime}_\delta)_{|\partial D}=0$, and $(h_\delta)_{|\partial D}=g$, $\etaa\circ\hat{h}_{\delta, \varepsilon, s}=0$. We will prove that, for $\sigma:=r_{3}$ in a suitable set $B \subset[9 / 8,2]$ with $|B|>1 / 2,$ we can choose $r_{2}=r_{3}-s$ and $r_{1}=r_{2}-s$ so that $h$ satisfies the conclusion of the proposition. Our choice of the parameters will imply the following inequalities:
\begin{equation}\label{e:preliminarybounds}
\delta^{2 \cdot 8^{-nQ}} \leq s, \quad \varepsilon \leq s, \quad \text { and } \quad E^{1-2 \beta_{1}} \leq \varepsilon^{2}.
\end{equation}
We estimate the Lipschitz constant of $\tilde{g}^{\prime}_\delta$. This can be easily done observing that
\begin{itemize}
    \item in $B_{r_{1}}\cap D$, we have  
    \[\operatorname{Lip}\left(\tilde{g}_\delta^\prime\right) \leq C \operatorname{Lip}\left(\tilde{f}_1^{\prime} * \varphi_\varepsilon\right) \leq C \operatorname{Lip}\left(\tilde{f}_1^{\prime}\right) \leq C (E+ \mathbf{A}^2)^{\beta_1},\]
    \item in $(B_{r_{2}}\setminus B_{r_{1}})\cap D$, we have  
    \[\operatorname{Lip}\left(\tilde{g}_\delta^\prime\right) \leq C \operatorname{Lip}\left(\tilde{f}_1^{\prime}\right)+C \frac{\left\|\tilde{f}_1^{\prime}-\tilde{f}_1^{\prime} * \varphi_\varepsilon\right\|_{L^\infty}}{s} \leq C\left(1+\frac{\varepsilon}{s}\right) \operatorname{Lip}\left(\tilde{f}_1^{\prime}\right) \leq C (E+\mathbf{A}^2)^{\beta_1},\]
    \item in $(B_{r_{3}}\setminus B_{r_{2}})\cap D$, we have 
        \begin{equation}\label{e:LipschitzEstimateInEuclideanSpace}
            \operatorname{Lip}\left(\tilde{g}_\delta^{\prime}\right) 
            \ \leq \ C \operatorname{Lip}\left(\tilde{f}_1^{\prime}\right)+C (E+\bA^2)^{\sfrac{1}{2}} \frac{\delta^{8^{-nQ}}}{s} 
            \ \leq \ C E^{\beta_{1}}+C (E+\bA^2)^{\sfrac{1}{2}} 
            \ \leq \ C (E+\bA^2)^{\beta_1}\,.
        \end{equation}
\end{itemize}
In the first inequality of the last line we have used that, since $\mathcal{Q}$ is a cone, $(E+\bA^2)^{-\sfrac{1}{2}}\tilde{f}_1^{\prime}(x) \in$ $\mathcal{Q}$ for every $x$, hence \[
\left|\boldsymbol{\rho}_{\delta}^{\star}\left(\frac{\tilde{f}_1^{\prime}}{\sqrt{E+\bA^2}}\right)-\frac{\tilde{f}_1^{\prime}}{\sqrt{E+\bA^2}}\right| \leq C \delta^{8^{-nQ}}\, .
\]
From \eqref{e:LipschitzEstimateInEuclideanSpace} and \eqref{e:preliminarybounds} we deduce easily that $\tilde{g}_\delta^\prime$ is continuous and piecewise Lipschitz and so globally Lipschitz and furthermore that
\begin{equation}\label{e:LipschitzPreliminaryEstimates}
\Lip(h_{\delta, \varepsilon, s})\le C(E+\bA^2)^{\beta_1}.
\end{equation} 

In the following Steps $1$-$3$ we estimate the Dirichlet energy of $h_{\delta,\varepsilon,s}$ and finally in Step $4$ we obtain the desired estimate \eqref{e:regconvDirichletEnergyEstimate} of Theorem \ref{t:regconv} for a suitable choice of $\delta,\varepsilon,s$ depending on some powers of the infinitesimal quantity $E$ (see \eqref{e:SuitablePowersOfTheExcess} below). Before we realize this program, we recall that for every $f\in W^{1,2}(\Omega,\Is{Q})$ we have
\begin{equation}\label{e:DirEnergyForMeanTranslated} 
0\le\D(f\oplus(-\etaa\circ f))=\D(f)-Q\D(\etaa\circ f).
\end{equation}
We write here the estimate of the Dirichlet energy of $\tilde{h}_\varepsilon$ which will be useful in combination with \eqref{e:DirEnergyForMeanTranslated}.
\begin{eqnarray}
\int\left|Dg * \varphi_{\varepsilon}-Dg\right|^{2} 
&\le &C \mathbf{A}^{2} \varepsilon^{2},\label{e:Reminder0}\\
\left\|Dg * \varphi_{\varepsilon}-Dg\right\|_{\infty}
&\le& C \| D^{2}g\|_{\infty} \varepsilon \ \leq \ C \mathbf{A}\varepsilon,\notag\\
\left|\int\left(Dg * \varphi_{\varepsilon}-D g\right)\left(D(\etaa\circ f) * \varphi_{\varepsilon}\right)\right| & \le &
C \mathbf{A}\varepsilon\int|D(\etaa\circ f) * \varphi_{\varepsilon}| \nonumber \\
& \stackrel{\text{Rem.} \ref{r:5.5DS3}}{\le} & C \mathbf{A}\varepsilon(E+\mathbf{A}^2)^{\frac12} \notag \\\label{e:Reminder}
& \stackrel{\text{Young}}{\le} & C \varepsilon (E+\bA^2)\, .
\end{eqnarray}
Summing \eqref{e:Reminder}, \eqref{e:Reminder0}, we obtain
\begin{align*}
    \int|D\tilde{h}_\varepsilon|^2 & = \int|D(\etaa\circ f)*\varphi_\varepsilon|^2+\int\left|Dg * \varphi_{\varepsilon}-Dg\right|^{2}
 - 2\int\left(Dg*\varphi_\varepsilon-Dg\right)\left(D(\etaa\circ f) * \varphi_{\varepsilon}\right)\\
& \le \int|D(\etaa\circ f)|^2+C \mathbf{A}^{2} \varepsilon^{2}+C \varepsilon (E+\bA^2)\\
& \le C\int|Df|^2+C\varepsilon(E+ \mathbf{A}^{2})\, .
\end{align*}

\textbf{Step 1. Energy in} $B_{r_{3}} \setminus B_{r_{2}}$.  By Proposition \ref{p:rho*}, we have $\left|\boldsymbol{\rho}_{\delta}^{\star}(P)-P\right| \leq$ $C \delta^{8^{-nQ}}$ for all $P \in \mathcal{Q}:=\xii(\Is{Q})$. Thus, elementary estimates on the linear interpolation give
\begin{eqnarray}
\int_{(B_{r_{3}} \setminus B_{r_{2}})\cap D}\left|D\tilde{g}_{\delta}^{\prime}\right|^{2} 
&\leq & \frac{C (E+\bA^2)}{\left(r_{3}-r_{2}\right)^{2}} \int_{(B_{r_{3}} \setminus B_{r_{2}})\cap D}\left|\frac{\tilde{f}_1^{\prime}}{\sqrt{E+\bA^2}}-\boldsymbol{\rho}_{\delta}^{\star}\left(\frac{\tilde{f}_1^{\prime}}{\sqrt{E+\bA^2}}\right)\right|^{2}\notag \\
&& +C \int_{(B_{r_{3}} \setminus B_{r_{2}})\cap D}\big|D \tilde{f}_1^{\prime}\big|^{2} + C \int_{(B_{r_{3}} \setminus B_{r_{2}})\cap D}\left|D\left(\boldsymbol{\rho}_{\delta}^{\star}\circ \tilde{f}_1^{\prime}\right)\right|^{2} \notag \\
&\leq & C \int_{(B_{r_{3}}\setminus B_{r_{2}})\cap D}\big|D \tilde{f}_1^{\prime}\big|^{2}+C (E+\bA^2) s^{-1} \delta^{2 \cdot 8^{-nQ}}.
\end{eqnarray}
Hence, using that $\Lip(\xii) \leq 1$ and \eqref{e:DirEnergyForMeanTranslated}, we estimate
\begin{align}\nonumber
\int_{(B_{r_{3}} \setminus B_{r_{2}})\cap D}\left|Dh_{\delta,\varepsilon,s}\right|^{2} & = \int_{(B_{r_{3}} \setminus B_{r_{2}})\cap D}\left|D\hat{h}_{\delta,\varepsilon,s}\right|^{2}+Q\int_{(B_{r_{3}} \setminus B_{r_{2}})\cap D}\left|D\tilde{h}_\varepsilon\right|^{2}\\ \nonumber
& \le \int_{(B_{r_3} \setminus B_{r_2})\cap D}\left|D\tilde{g}_{\delta}^{\prime}\right|^{2}-Q\int\etaa+C\int_{(B_{r_3} \setminus B_{r_2})\cap D}\left|D\tilde{h}_\varepsilon\right|^2\\ 
 \label{e:EnergyEstimateExteriorLayer}
& \le C \int_{(B_{r_{3}} \setminus B_{r_{2}})\cap D}\left|D f\right|^{2}+C (E+\bA^2) \left(\varepsilon + s^{-1} \delta^{2 \cdot 8^{-nQ}} \right). 
\end{align}

\textbf{Step 2. Energy in} $B_{r_{2}} \setminus B_{r_{1}}$. Here, using the same interpolation inequality and a standard estimate on convolutions of $W^{1,2}$ functions, we get
$$
\begin{aligned}
\int_{(B_{r_{2}} \setminus B_{r_{1}})\cap D}\left|D\tilde{g}_{\delta}^{\prime}\right|^{2} & \leq C \int_{(B_{r_{2}+\varepsilon} \setminus B_{r_{1}-\varepsilon})\cap D}\left|D\tilde{f}_1^{\prime}\right|^{2}+\frac{CC_\Phi}{\left(r_{2}-r_{1}\right)^{2}} \int_{B_{r_{2}} \setminus B_{r_{1}}}\left|\tilde{f}_1^{\prime}-\varphi_{\varepsilon} * \tilde{f}_1^{\prime}\right|^{2} \\
& \leq CC_\Phi \int_{(B_{r_{2}+\varepsilon} \setminus B_{r_{1}-\varepsilon})\cap D}\big|D \tilde{f}_1^{\prime}\big|^{2}+C C_\Phi\varepsilon^{2} s^{-2} \int_{B_{3}\cap D}\big|D \tilde{f}_1^{\prime}\big|^{2} \\
& \leq C \int_{(B_{r_{2}+\varepsilon} \setminus B_{r_{1}-\varepsilon})\cap D}\big|D \tilde{f}_1^{\prime}\big|^{2}+C \varepsilon^{2} (E+\bA^2) s^{-2}\\
& \le C \int_{(B_{r_{2}+\varepsilon} \setminus B_{r_{1}-\varepsilon})\cap D}\left|D f\right|^{2}+C \varepsilon^{2} (E+\bA^2) s^{-2}.
\end{aligned}
$$
So coming back to the energy estimate on $h_{\delta,\varepsilon,s}$ we get 
\begin{align}\nonumber
\int_{(B_{r_{2}} \setminus B_{r_{1}})\cap D}\left|Dh_{\delta,\varepsilon,s}\right|^{2} & = \int_{(B_{r_2} \setminus B_{r_1})\cap D}\left|D\hat{h}_{\delta,\varepsilon,s}\right|^{2}+Q\int_{(B_{r_2} \setminus B_{r_1})\cap D}\left|D\tilde{h}_\varepsilon\right|^{2}\\\nonumber
 & \le \int_{(B_{r_2} \setminus B_{r_1})\cap D}\left|D\tilde{g}_{\delta}^{\prime}\right|^{2}+C\int_{(B_{r_2} \setminus B_{r_1})\cap D}\left|D\tilde{h}_\varepsilon\right|^2\\
& \le C \int_{(B_{r_{2}+\varepsilon} \setminus B_{r_{1}-\varepsilon})\cap D}\left|D f\right|^{2}+ C \varepsilon^{2} (E+\bA^2) s^{-2}+ C \varepsilon (E+\bA^2)\, .
\label{e:energyestimate2}
\end{align}

\textbf{Step 3. Energy in} $B_{r_{1}} .$ Define $Z:=\left\{\operatorname{dist}\left(\frac{\tilde{f}_1^{\prime}}{\sqrt{E}} * \varphi_\varepsilon, \mathcal{Q}\right)>\delta^{nQ+1}\right\}\subseteq D$ and use \eqref{e:EnergyEstimaterho*}
to get
\begin{align}\label{e:7.11}
\int_{B_{r_{1}}\cap D}\left|D \tilde{g}_\delta^{\prime}\right|^{2} 
&\leq\left(1+C \delta^{8^{-\bar{n} Q-1}}\right) \int_{(B_{r_{1}}\cap D) \setminus Z}\left|D\left(\tilde{f}_1^{\prime} * \varphi_{\varepsilon}\right)\right|^{2}+C \int_{Z}\left|D\left(\tilde{f}_1^{\prime} * \varphi_{\varepsilon}\right)\right|^{2}\\
&=: I_{1}+I_{2}. \notag
\end{align}
We consider $I_{1}$ and $I_{2}$ separately. For $I_{1}$ we first observe the elementary inequality
\begin{align}\nonumber
\left\|D\left(\tilde{f}_1^{\prime} * \varphi_\varepsilon\right)\right\|_{L^{2}}^{2} 
&\le \left\|(D \tilde{f}_1^{\prime} ) * \varphi_\varepsilon\right\|_{L^{2}}^{2} \\
&\leq\left\|\left(\big|D \tilde{f}_1^{\prime}\big| \mathbf{1}_{K}\right) * \varphi_\varepsilon\right\|_{L^{2}}^{2}+\left\|\left(\big|D \tilde{f}_1^{\prime}\big| \mathbf{1}_{K^{c}}\right) * \varphi_\varepsilon\right\|_{L^{2}}^{2} \notag \\
&\quad +  2\left\|\left(\big|D \tilde{f}_1^{\prime}\big| \mathbf{1}_{K}\right) * \varphi_\varepsilon\right\|_{L^{2}}\left\|\left(\big| D \tilde{f}_1^{\prime}\big| \mathbf{1}_{K^{c}}\right) * \varphi_\varepsilon\right\|_{L^{2}},\label{e:7.12}
\end{align}
where $K^{c}$ is the complement of $K$ in $D$. Recalling $r_{1}+\varepsilon \leq r_{1}+s=r_{2}$ we estimate the first summand in \eqref{e:7.12} as follows:
\begin{equation}\label{e:7.13}
\left\|\left(\big|D \tilde{f}_1^{\prime}\big| \mathbf{1}_{K}\right) * \varphi_\varepsilon\right\|_{L^{2}\left(B_{r_{1}}\cap D\right)}^{2} \leq \int_{B_{r_{1}+\varepsilon}\cap D}\left(\big|D \tilde{f}_1^{\prime}\big| \mathbf{1}_{K}\right)^{2} \leq \int_{B_{r_{2}} \cap K}\big|D \tilde{f}_1^{\prime}\big|^{2}.
\end{equation}
In order to treat the other terms, recall that $\operatorname{Lip}\left(\tilde{f}_1^{\prime}\right) \leq C (E+\bA^2)^{\beta_{1}}$ and $\left|K^{c}\right| \leq C (E+\bA^2)^{1-2 \beta_{1}}$. Thus, we have
\begin{align}
\left\|\left(\big|D \tilde{f}_1^{\prime}\big| \mathbf{1}_{K^{c}}\right) * \varphi_\varepsilon\right\|_{L^{2}\left(B_{r_{1}}\cap D\right)}^{2} &\leq C (E+\bA^2)^{2 \beta_{1}}\left\|\mathbf{1}_{K^{c} }* \varphi_\varepsilon\right\|_{L^{2}}^{2}\nonumber\\
&\leq C (E+\bA^2)^{2 \beta_{1}}\left\|\mathbf{1}_{K^{c}}\right\|_{L^{1}}^{2}\left\|\varphi_\varepsilon\right\|_{L^{2}}^{2} \leq \frac{C (E+\bA^2)^{2-2 \beta_{1}}}{\varepsilon^{2}}.\label{e:7.14}
\end{align}
Putting \eqref{e:7.13} and \eqref{e:7.14} in \eqref{e:7.12} and recalling $(E+\bA^2)^{1-2 \beta_{1}} \leq \varepsilon^{2}$ and $\int\big|D \tilde{f}_1^{\prime}\big|^{2} \leq C (E+\bA^2)$, we get
\begin{equation}\label{e:7.15}
I_{1} \leq \int_{B_{r_{2}} \cap K}\big|D \tilde{f}_1^{\prime}\big|^{2}+C \delta^{8^{-\bar{n} Q-1}} (E+\bA^2) +C \varepsilon^{-1} (E+\bA^2)^{3 / 2-\beta_{1}}.
\end{equation}
For what concerns $I_{2}$, first we argue as for $I_{1}$, splitting in $K$ and $K^{c}$, to deduce that
\begin{equation}\label{e:7.16}
I_{2} \leq C \int_{Z}\left(\left(\big|D \tilde{f}_1^{\prime}\big| \mathbf{1}_{K}\right) * \varphi_\varepsilon\right)^{2}+C \varepsilon^{-1} (E+\bA^2)^{3 / 2-\beta_{1}}.
\end{equation}
Then, regarding the first summand in \eqref{e:7.16}, we note that
\begin{equation}\label{e:7.17}
|Z| \delta^{2nQ+2} \leq \int_{B_{r_{1}}\cap D}\left|\frac{\tilde{f}_1^{\prime}}{\sqrt{E+\bA^2}} * \varphi_\varepsilon-\frac{\tilde{f}_1^{\prime}}{\sqrt{E+\bA^2}}\right|^{2} \leq C \varepsilon^{2}.
\end{equation}
Next, we recall that $q_{1}=2 p_{1}>2$ and use \eqref{e:7.4} to obtain
\begin{align}\label{e:7.18}\nonumber
\int_{Z}\left(\left(\big|D \tilde{f}_1^{\prime}\big| \mathbf{1}_{K}\right) * \varphi_\varepsilon\right)^{2} 
& \leq|Z|^{\frac{p_{1}-1}{p_{1}}}\left\|\left(\big|D \tilde{f}_1^{\prime}\big| \mathbf{1}_{K}\right) * \varphi_\varepsilon\right\|_{L^{4}}^{2}\\
&\leq C\left(\frac{\varepsilon}{\delta^{nQ+1}}\right)^{\frac{2\left(p_{1}-1\right)}{p_{1}}}\left\|\big|D \tilde{f}_1^{\prime}\big|\right\|_{L^{q_{1}}(K)}^{2} \notag \\
& \leq C\left(\frac{\varepsilon}{\delta^{nQ+1}}\right)^{\frac{2\left(p_{1}-1\right)}{p_{1}}}\left(E+\bA^2\right).
\end{align}
Gathering all the estimates together \eqref{e:7.11}, \eqref{e:7.15}, \eqref{e:7.16} and \eqref{e:7.18} gives
\begin{align}\nonumber
    \int_{B_{r_{1}}\cap D}\left|D \tilde{g}_\delta^{\prime}\right|^{2} & \le \int_{B_{r_1} \cap K}\big|D \tilde{f}_1^{\prime}\big|^{2}+C(E+\bA^2) \delta^{8^{-nQ-1}}+ C\frac{(E+\bA^2)^{3 / 2-\beta_{1}}}{\varepsilon}\\
    &\quad +C(E+\bA^2) \left(\frac{\varepsilon}{\delta^{nQ+1}}\right)^{\frac{2\left(p_{1}-1\right)}{p_{1}}} \notag\\ 
    & =  \int_{B_{r_1} \cap K}\left|D f\right|^{2}-Q\int_{B_{r_1} \cap K}\left|D(\etaa\circ f)\right|^{2} + C(E+\bA^2) \delta^{8^{-nQ-1}} \nonumber \\
    &\quad +C\frac{(E+\bA^2)^{3 / 2-\beta_{1}}}{\varepsilon}+ C(E+\bA^2) \left(\frac{\varepsilon}{\delta^{nQ+1}}\right)^{\frac{2\left(p_{1}-1\right)}{p_{1}}}\label{e:energyestimate3prime}
\end{align}
Define $Z:=\left\{\operatorname{dist}\left((\etaa\circ f) * \varphi_\varepsilon, \mathcal{Q}\right)>\delta^{nQ+1}\right\}$ to get
\begin{align}\label{e:7.11bis}
\int_{B_{r_{1}}\cap D}\left|D (\etaa\circ f) * \varphi_{\varepsilon}\right|^{2} 
&\leq\int_{B_{r_{1}}\cap D\setminus Z}\left|D\left((\etaa\circ f) * \varphi_{\varepsilon}\right)\right|^{2}+\int_{Z}\left|D\left((\etaa\circ f) * \varphi_{\varepsilon}\right)\right|^{2}\\
&=: \hat{I}_{1}+\hat{I}_{2}. \notag
\end{align}
We consider $\hat{I}_{1}$ and $\hat{I}_{2}$ separately. For $\hat{I}_{1}$ we first observe the elementary inequality
\begin{align}\nonumber
\left\|D\left((\etaa\circ f)* \varphi_\varepsilon\right)\right\|_{L^{2}}^{2} & \leq \left\|(D (\etaa\circ f)) * \varphi_\varepsilon\right\|_{L^{2}}^{2}\\ \nonumber
& \le \left\|\left(\left|D (\etaa\circ f)\right| \mathbf{1}_{K}\right) * \varphi_\varepsilon\right\|_{L^{2}}^{2}+\left\|\left(\left|D (\etaa\circ f)\right| \mathbf{1}_{K^{c}}\right) * \varphi_\varepsilon\right\|_{L^{2}}^{2} \\ \label{e:7.12bis}
&\quad +  2\left\|\left(\left|D (\etaa\circ f)\right| \mathbf{1}_{K}\right) * \varphi_\varepsilon\right\|_{L^{2}}\left\|\left(\left|D (\etaa\circ f)\right| \mathbf{1}_{K^{c}}\right) * \varphi_\varepsilon\right\|_{L^{2}}\,.
\end{align}
Recalling $r_{1}+\varepsilon \leq r_{1}+s=r_{2}$, we estimate the first summand in \eqref{e:7.12bis} as follows
\begin{equation}\label{e:7.13bis}
\left\|\left(\left|D (\etaa\circ f)\right| \mathbf{1}_{K}\right) * \varphi_\varepsilon\right\|_{L^{2}\left(B_{r_{1}}\cap D\right)}^{2} \leq \int_{B_{r_{1}+\varepsilon}\cap D}\left(\left|D (\etaa\circ f)\right| \mathbf{1}_{K}\right)^{2} \leq \int_{B_{r_{2}} \cap K}\left|D (\etaa\circ f)\right|^{2}.
\end{equation}
In order to treat the other terms, recall that $\operatorname{Lip}\left(\etaa\circ f\right) \leq C (E+\bA^2)^{\beta_{1}}$ and $\left|K^{c}\right| \leq C (E+\bA^2)^{1-2 \beta_{1}}$. We thus have
\begin{align}
\left\|\left(\left|D (\etaa\circ f)\right| \mathbf{1}_{K^{c}}\right) * \varphi_\varepsilon\right\|_{L^{2}\left(B_{r_{1}}\cap D\right)}^{2} &\leq C (E+\bA^2)^{2 \beta_{1}}\left\|\mathbf{1}_{K^{c} }* \varphi_\varepsilon\right\|_{L^{2}}^{2}\nonumber\\
&\leq C (E+\bA^2)^{2 \beta_{1}}\left\|\mathbf{1}_{K^{c}}\right\|_{L^{1}}^{2}\left\|\varphi_\varepsilon\right\|_{L^{2}}^{2} \notag\\
&\leq \frac{C (E+\bA^2)^{2-2 \beta_{1}}}{\varepsilon}.\label{e:7.14bis}
\end{align}
Putting \eqref{e:7.13bis} and \eqref{e:7.14bis} in \eqref{e:7.12bis}, and recalling $E^{1-2 \beta_{1}} \leq \varepsilon^{2}$ and $\int\left|D (\etaa\circ f)\right|^{2} \leq C E$ we get
\begin{equation}\label{e:7.15bis}
\hat{I}_{1} \leq \int_{B_{r_{2}}\cap D \cap K}\left|D (\etaa\circ f)\right|^{2}+C \varepsilon^{-1} (E+\bA^2)^{3 / 2-\beta_{1}}.
\end{equation}
For what concerns $\hat{I}_{2}$, first we argue as for $\hat{I}_{1}$ (splitting in $K$ and $K^{c}$) to deduce that
\begin{equation}\label{e:7.16bis}
\hat{I}_{2} \leq C \int_{Z}\left(\left(\left|D (\etaa\circ f)\right| \mathbf{1}_{K}\right) * \varphi_\varepsilon\right)^{2}+C \varepsilon^{-1} (E+\bA^2)^{3 / 2-\beta_{1}}.
\end{equation}
Then, regarding the first summand in \eqref{e:7.16bis}, we note that
\begin{equation}\label{e:7.17bis}
|Z| \delta^{2nQ+2} \leq \int_{B_{r_{1}}\cap D}\left|\frac{(\etaa\circ f)}{\sqrt{E+\bA^2}} * \varphi_\varepsilon-\frac{(\etaa\circ f)}{\sqrt{E+\bA^2}}\right|^{2} \leq C \varepsilon^{2}.
\end{equation}
Recalling that $q_{1}=2 p_{1}>2$, we use \eqref{e:7.4} to obtain
\begin{align}\nonumber
\int_{Z}\left(\left(\left|D (\etaa\circ f)\right| \mathbf{1}_{K}\right) * \varphi_\varepsilon\right)^{2} 
& \le |Z|^{\frac{p_{1}-1}{p_{1}}}\left\|\left(\left|D (\etaa\circ f)\right| \mathbf{1}_{K}\right) * \varphi_\varepsilon\right\|_{L^{4}}^{2}\\ \nonumber
& \leq C\left(\frac{\varepsilon}{\delta^{\bar{n}} Q+1}\right)^{\frac{2\left(p_{1}-1\right)}{p_{1}}}\left\|\left|D (\etaa\circ f)\right|\right\|_{L^{q_{1}}(K)}^{2} \\ \label{e:7.18bis}
& \leq C\left(\frac{\varepsilon}{\delta nQ+1}\right)^{\frac{2\left(p_{1}-1\right)}{p_{1}}}\left(E+\bA^2\right).
\end{align}
Gathering all the estimates together, \eqref{e:7.11bis}, \eqref{e:7.15bis}, \eqref{e:7.16bis} and \eqref{e:7.18bis} gives
\begin{align}\label{e:7.19bis}
\int_{B_{r_{1}}\cap D}\left|D (\etaa\circ f)*\varphi_\varepsilon\right|^{2} 
 &\le \int_{B_{r_1}\cap K}\left|D (\etaa\circ f)\right|^{2}+ C \frac{(E+\bA^2)^{3 / 2-\beta_{1}}}{\varepsilon} \notag\\
 &\quad +C (E+\bA^2)\left(\frac{\varepsilon}{\delta^{nQ+1}}\right)^{2-\frac{1}{p_{1}}}.
\end{align}
\normalsize
So combining \eqref{e:energyestimate3prime} and \eqref{e:7.19bis} yields
\begin{align}\nonumber
\int_{B_{r_{1}}\cap D} &\left|Dh_{\delta,\varepsilon,s}\right|^{2} 
 = \int_{B_{r_1}\cap D}\left|D\hat{h}_{\delta,\varepsilon,s}\right|^{2}+Q\int_{ B_{r_1}\cap D}\left|D\tilde{h}_\varepsilon\right|^{2}\\ \nonumber
& \le \int_{B_{r_{1}}\cap D}\left|D \tilde{g}^\prime_\delta\right|^{2}+Q\int_{B_{r_1}\cap D}\left|D\tilde{h}_\varepsilon\right|^{2}\\  \nonumber
& \le \int_{B_{r_1}\cap K}\left|D f\right|^{2}-Q\int_{B_{r_1}\cap K}\left|D(\etaa\circ f)\right|^{2}+Q\int_{ B_{r_1} \cap K}\left|D\etaa\circ f\right|^{2} + C \varepsilon (E+\mathbf{A}^{2})\\ \nonumber
&\quad +  C\left((E+\bA^2) \delta^{8^{-nQ-1}}+\frac{(E+\bA^2)^{3 / 2-\beta_{1}}}{\varepsilon}+(E+\bA^2)\left(\frac{\varepsilon}{\delta^{nQ+1}}\right)^{\frac{2\left(p_{1}-1\right)}{p_{1}}}\right)\nonumber\\ \nonumber
& \le \int_{B_{r_1}\cap K}\left|D f\right|^{2}+C(E+\bA^2) \delta^{8^{-nQ-1}}+ C \frac{(E+\bA^2)^{3 / 2-\beta_{1}}}{\varepsilon}\\ 
&\quad +C(E+\bA^2)\left(\frac{\varepsilon}{\delta^{nQ+1}}\right)^{\frac{2\left(p_{1}-1\right)}{p_{1}}}
+C\varepsilon (E+\bA^2).\label{e:energyestimate3}
\end{align}

\textbf{Step 4. Final estimate}. This part is analogue to \cite[Step 4 of Proposition 7.3]{DS3}. Summing \eqref{e:EnergyEstimateExteriorLayer}, \eqref{e:energyestimate2}, \eqref{e:energyestimate3}, and recalling that $\varepsilon<s$, we conclude
$$
\begin{aligned}
\int_{B_{r_{3}}\cap D}|D h_{\delta,\varepsilon,s}|^{2} \leq & \int_{B_{r_{1}}\cap K}|D f|^{2}+C \int_{(B_{r_{1}+3 s} \setminus B_{r_{1}-s})\cap D}\left|D f^{\prime}\right|^{2} +C (E+\bA^2)\left(\varepsilon + \delta^{8^{-nQ-1}} \right)\\
&+C (E+\bA^2)\left(\frac{\varepsilon^{2}}{s^{2}} + \frac{\delta^{2 \cdot 8^{-nQ}}}{s}+\frac{(E+\bA^2)^{1 / 2-\beta_{1}}}{\varepsilon}+\left(\frac{\varepsilon}{\delta^{nQ+1}}\right)^{\frac{2\left(p_{1}-1\right)}{p_{1}}}\right).
\end{aligned}
$$
We set $\varepsilon=(E+\bA^2)^{a}, \delta=(E+\bA^2)^{b}$ and $s=(E+\bA^2)^{c},$ where
\begin{equation}\label{e:SuitablePowersOfTheExcess}
a=\frac{1-2 \beta_{1}}{4}, \quad b=\frac{1-2 \beta_{1}}{8(nQ+1)}, \quad \text { and } \quad c=\frac{1-2 \beta_{1}}{8^{n} Q 8 (nQ+1)}
\end{equation}
and we finally let $h$ be the corresponding function $h_{\delta,\varepsilon,s}$.
This choice respects \eqref{e:preliminarybounds}. Assume $(E+\bA^2)$ is small enough so that $s \leq \frac{1}{16}$. Now, if $C>0$ is a sufficiently large constant, there is a set $B^{\prime} \subset\left[\frac98, \frac{29}{16}\right]$ with $\left|B^{\prime}\right|>1 / 2$ such that,
$$
\int_{(B_{r_{1}+3 s} \setminus B_{r_{1}-s})\cap D}\left|D f^{\prime}\right|^{2} \leq C s \int_{B_{2}\cap D}\left|D f^{\prime}\right|^{2} \leq C (E+\bA^2)^{1+c} \quad \text { for every } r_{1} \in B^{\prime}.
$$
For $\sigma=r_{3} \in B=2 s+B^{\prime}$ we then conclude the existence of a $\bar \gamma\left(\beta_{1}, n,Q\right)>0$ such that
$$
\int_{B_{\sigma}\cap D}|D h|^{2} \leq \int_{B_{\sigma}\cap K}|D f|^{2}+C \left(E+\mathbf{A}^{2}\right)^{1+\bar\gamma}.
$$
\end{proof}

\begin{proof}[Proof of Theorem \ref{t:StrongExcEst}.] Here we proceed as in the proof of \cite[Theorem 7.1]{DS3}. Choose  $\beta_{1}=\frac{1}{8}$ and consider the set  $B \subset[9 / 8,2]$ given in Proposition \ref{t:regconv}. Using the coarea formula and the isoperimetric inequality (the argument and the map $\varphi$ are the same in the proof of Theorem \ref{t:o(E)} and that of Proposition \ref{p:WEE}), we find $s \in B$ and an integer rectifiable current $R$ such that
$$
\partial R=\left\langle T-\mathbf{G}_{f}, \varphi, s\right\rangle \quad \text { and } \quad \mathbf{M}(R) \leq C E^{\frac{3}{2}}.
$$
Since $h|_{\partial (D\cap B_s)}=f|_{\partial (D\cap B_s)}$ we can use $h$ in place of $f$ in the estimates and, arguing as before (see e.g. the proof of Proposition \ref{p:WEE}), we get, for a suitable $\gamma>0$
\begin{eqnarray}
\|T\|\left(\mathbf{C}_{s}\right) &\leq & Q\left|B_{s}\cap D\right|+\int_{B_{s}\cap D} \frac{|D g|^{2}}{2}+C (E+\bA^2)^{1+\bar \gamma}\nonumber\\ &\stackrel{\eqref{e:regconvDirichletEnergyEstimate}}{\leq} & Q\left|B_{s}\cap D\right|+\int_{B_{s} \cap K} \frac{|D f|^{2}}{2}+C \left(E+\mathbf{A}^{2}\right)^{1+\bar \gamma}.\label{e:7.22}
\end{eqnarray}
On the other hand, by Taylor's expansion in \cite[Remark 5.4]{DS3},
\begin{align}\label{e:7.23}
\|T\|\left(\mathbf{C}_{s}\right) & = \|T\|\left(\left(B_{s}\cap D \setminus K\right) \times \mathbb{R}^{n}\right)+\left\|\mathbf{G}_{f}\right\|\left(\left(B_{s} \cap K\right) \times \mathbb{R}^{n}\right) \notag \\
& \ge \|T\|\left(\left(B_{s}\cap D \setminus K\right) \times \mathbb{R}^{n}\right)+Q\left|K \cap B_{s}\right| \notag\\
&\quad +\int_{K \cap B_{s}} \frac{|D f|^{2}}{2}-C (E+\bA^2)^{1+\bar \gamma}.
\end{align}
Hence, from \eqref{e:7.22} and \eqref{e:7.23}, we get $\mathbf{e}_{T}\left(B_{s}\cap D \setminus K\right) \leq C \left(E+\mathbf{A}^{2}\right)^{1+\bar\gamma}$.
This is enough to conclude the proof. Indeed, let $A \subset B_{9 / 8}\cap D$ be a Borel set. Using the higher integrability of $|D f|$ in $K$ (see \eqref{e:7.4}$)$ and possibly selecting a smaller $\bar \gamma>0,$ we get
$$
\begin{aligned}
\mathbf{e}_{T}(A) & \leq \mathbf{e}_{T}(A \cap K)+\mathbf{e}_{T}(A\setminus K)\\
&\leq \int_{A \cap K} \frac{|D f|^{2}}{2}+C \left(E+\mathbf{A}^{2}\right)^{1+\bar\gamma} \\
& \leq C|A \cap K|^{\frac{p_{1}-1}{p_{1}}}\left(\int_{A \cap K}|D f|^{q_{1}}\right)^{2 / q_{1}}+C \left(E+\mathbf{A}^{2}\right)^{1+\bar\gamma} \\
& \leq C|A|^{\frac{p_1-1}{p_{1}}}\left(E+\mathbf{A}^{2}\right)+C \left(E+\mathbf{A}^{2}\right)^{1+\bar\gamma}.
\end{aligned}
$$
\end{proof}
\begin{proof}[Proof of Theorem \ref{t:strong_Lipschitz}.] Here we proceed exactly as in the proof of \cite[Theorem 2.4]{DS3}. Assume $r=1$ and $x=0 .$ Choose $\beta_{11}<\min \left\{\frac{1}{4}, \frac{\gamma_{11}}{2\left(1+\gamma_{11}\right)}\right\},$ where $\gamma_{11}$ is the constant in Theorem \ref{t:regconv}.
Let $f$ be the $E^{\beta_{11}}$ -Lipschitz approximation of $T$. Clearly \eqref{e:strong_Lip} and \eqref{e:strong_Lip_Osc} follow directly from Proposition \ref{p:Lipschitz_1}, if $\gamma<\beta_{11} .$ Set next $A:=\left\{\mathbf{m e}_{T}>2^{-m} (E+\bA^2)^{2 \beta_{11}}\right\} \cap B_{9 / 8} .$ By Proposition \ref{p:Lipschitz_1} we have $|A| \leq C (E+\bA^2)^{1-2 \beta_{11}}$. If $\varepsilon_A>0$ is sufficiently small, apply \eqref{e:stimaK} and the estimate \eqref{e:StrongExcEst} to $A$ in order to conclude
$$
\left|B_{1} \cap D\setminus K\right| \leq C (E+\bA^2)^{-2 \beta_{11}} \mathbf{e}_{T}(A) \leq C (E+\bA^2)^{\gamma_{11}-2 \beta_{11}\left(1+\gamma_{11}\right)}\left(E+\mathbf{A}^{2}\right).
$$
By our choice of $\gamma_{11}$ and $\beta_{11},$ this last inequality gives \eqref{e:strongKestimate} for some positive $\gamma_1$. Finally, set $S=\mathbf{G}_{f} .$ Recalling the strong Almgren estimate \eqref{e:StrongExcEst} and the Taylor expansion in \cite[Remark 5.4]{DS3} we conclude for every $0<\sigma \leq 1$
\begin{align}
&\left|\|T\|\left(\mathbf{C}_{\sigma}\right)-Q|D|-\int_{B_{\sigma}\cap D} \frac{|D f|^{2}}{2}\right| \\
&\leq \mathbf{e}_{T}\left(B_{\sigma}\cap D \setminus K\right)+ \mathbf{e}_{S}\left(B_{\sigma}\cap D \setminus K\right)+\left|\mathbf{e}_{S}\left(B_{\sigma}\cap D\right)-\int_{B_{\sigma}\cap D} \frac{|D f|^{2}}{2}\right| \\
&\leq C \left(E+\mathbf{A}^{2}\right)^{1+\gamma_{11}}+C\left|B_{\sigma} \cap D\setminus K\right|+C \operatorname{Lip}(f)^{2} \int_{B_{\sigma}\cap D}|D f|^{2}\\ 
&\leq C \left(E+\mathbf{A}^{2}\right)^{1+\gamma_{11}}.
\end{align}
The $L^{\infty}$ bound follows from Proposition \ref{p:Lipschitz_1} and we finish the proof.
\end{proof}

%% file: center-manifold.tex
\section{Center manifold and normal approximation}\label{s:center-manifold}

This section is devoted to prove an analog of \cite[Theorem 8.13]{DDHM}, namely to construct, in a neighborhood of a flat point $p$, a smooth $C^{3,\alpha}$ submanifold with boundary $\Gamma$ and a normal multivalued map $N$ on it. The first is, roughly, an approximation of the average of the sheets lying over the unique tangent plane $V$ to $T$ at $p$. The second is a more accurate approximation of the current $T$, which compared to the one in Section \ref{s:Lip-approx} has the additional property of having (almost) zero average. 

 We start by introducing the spherical excess and the cylindrical excess with respect to a general plane.

\begin{definition}
Given a current $T$ as in Assumption \ref{a:main-local-2} and $2$-dimensional planes  $V, V'$, we define the \textbf{{\em excess of $T$ in balls
and cylinders with respect to planes $V, V'$}} as
\begin{align*}
\bE(T,\B_r (x),V) &:= \left(2 \pi\,r^2\right)^{-1}\int_{\B_r (x)} |\vec T - \vec V|^2 \, d\|T\|,\\
\qquad \bE (T, \bC_r (x, V), V') &:= \left(2\pi\,r^2\right)^{-1} \int_{\bC_r (x, V)} |\vec T - \vec V'|^2 \, d\|T\|\, .
\end{align*}
\end{definition}

\begin{definition}[Optimal planes]\label{d:optimal_planes}
For the case of balls we define the spherical excess as follows. The \textbf{\emph{optimal spherical excess}} at some $x \in \supp(T)\setminus \Gamma$ is given by 
\begin{equation}\label{e:optimal_pi}
\bE(T,\B_r (x)):=\min_V \bE (T, \B_r (x), V), 
\end{equation} 
but in the case of $x\in \Gamma$ we define the \textbf{\emph{optimal boundary spherical excess}} as  
\begin{equation*}
\bE^\flat (T, \B_r (x)) := \min \{\bE (T, \B_r (x), V) : V\supset T_x \Gamma\}. \end{equation*}
The plane $V$ which minimizes $\bE$, resp. $\bE^\flat$, is not unique but since for notational purposes it is convenient to define
a unique ``height'' $\bh (T, \B_r (x))$ we set 
\begin{equation}\label{e:optimal_pi_2}
\bh(T,\B_r(x)) := \min \big\{\bh(T,\B_r (x),V): \mbox{ $V$ optimizes $\bE$ (resp. $\bE^\flat)$}\big\}\, .
\end{equation}
In the case of cylinders we denote by $\bE (T, \bC_r (x, V)) = \bE (T, \bC_r (x, V), V)$ and $\bh(T, \bC_r (x, V))=\bh(T, \bC_r (x, V),V)$.
\end{definition}

We recall that under the above assumptions $\bC_{5R_0} = \bC_{5 R_0} (0, V_0)$ and $\p_\sharp T \res \bC_{5R_0} = Q\a{D}$, where $D\subset B_{5R_0}$ is one of the two connected components in which $B_{5R_0}$ is subdivided by the curve $\gamma = \p (\Gamma)$. Moreover $T_0 \Gamma = \mathbb R\times \{0\}$ and in particular $\Gamma \cap \bC_{5R_0} = \{(t, \psi (t))\}= \{(t, \psi_1 (t), \bar \psi (t))\}$, where $ \psi_1: (-5R_0, 5R_0) \to \R$ and $\bar \psi: (-5R_0, 5R_0) \to \R^n$. In particular $\gamma$ is the graph of $\psi_1$ and without loss of generality we assume that $D=\{(x_1,x_2)\in B_{5R_0}: x_2> \psi_1 (x_1)\}$, namely it is the upper half of $B_{5R_0}\setminus \gamma$. 

In this section we will then work under the following assumptions.

\begin{ipotesi}\label{ass:cm}
$p= q= (0,0)$, $V= V_0= \mathbb R^2\times \{0\}$, $Q$, $T$, and $\Gamma$ are as in Assumption \ref{Ass:app} in the cylinder $\bC_{5R_0}$, where $R_0\geq 1+\sqrt{2}$ is a sufficiently large geometric constant which will be specified later. Moreover $Q \a{V_0}$ is the (unique) tangent cone to $T$ at $0$.

We moreover assume in the sequel that 
\begin{equation}
\bE (T, \bC_{5R_0}(0, V_0)) + \bA^2 \leq \varepsilon_{CM},
\end{equation}
for some small positive parameter $\varepsilon_{CM}= \varepsilon_{CM}(n, Q, R_0)$.
\end{ipotesi}

Under the above assumptions we show now that the height of $T$ in $\bC_{4R_0}$ is also under control.

\begin{lemma}\label{l:hardtSimonHeight}
There are constants $\varepsilon_{CM}, C$ depending on $Q,n$ and $R_0$ such that, if Assumption \ref{ass:cm} holds, then for all $p \in \Gamma$ and $r>0$ such that $\bC_{5r}(p, V_0) \subset \bC_{5R_0}$, we have
\begin{equation}\label{e:height-estimate}
\bh (T, \bC_{4r}(p, V_0)) \leq C r (\bE (T, \bC_{5r}(p, V_0)) + r \bA)^{\frac{1}{2}}\, .
\end{equation}
\end{lemma}

\begin{proof}
We divide the proof into two steps.

\textit{Step 1:} $ \displaystyle \sup_{z \in \supp(T) \cap \bC_{4r}(p, V_0)} | \bp^\perp_{V_0}(z-p)|^2 \leq C r^{-2} \int_{\bC_{9r/2}(p, V_0)} | \bp^\perp_{V_0}(z-p)|^2 d \|T\|(z) + C_0 \bA^2 r^4$.

This is shown in \cite[Lemma 6.6]{DDHM} and carries over word by word to our setting as the only part where the stationarity of the associated integral varifold is needed, is for the harmonicity of the coordinate functions. This however is true, as we test with functions which are supported away from the boundary of $T$. We use this to apply a Moser iteration scheme and estimate the $L^\infty$ norm by the limsup of the $L^p$ norms as $p \to \infty$.

\textit{Step 2:} $ \displaystyle r^{-2} \int_{\bC_{9r/2}(p, V_0)} | \bp^\perp_{V_0}(z-p)|^2 d \|T\|(z) \leq C \ \bE (T, \bC_{5r}(p, V_0)) r^2 + C \bA r^3$.

Also for this, the proof of \cite[Lemma 6.7]{DDHM} carries over as the difference to our situation is a factor $Q$ in the monotonicity formula (Theorem \ref{thm:allard}). From there, we estimate the remainder term by $r^2(\bE (T, \bC_{5r}(0, V_0))+\bA)$.
\end{proof}


\subsection{Whitney decomposition}
We specify next some notation which will be recurrent when dealing with squares inside $V_0$.
For each $j\in \N$, $\sC_j$ denotes the family of closed squares $L$ of $V_0$ of the form 
\begin{equation}\label{e:cube_def}
[a_1, a_1+2\ell] \times [a_2, a_2+ 2\ell] \times \{0\}\subset V_0\, 
\end{equation}
which intersect $D$,
where $2\,\ell = 2^{1-j} =: 2\,\ell (L)$ is the side-length of the square, 
$a_i\in 2^{1-j}\Z$ $\forall i$ and we require in
addition $-4 \leq a_i \leq a_i+2\ell \leq 4$. 
To avoid cumbersome notation, we will usually drop the factor $\{0\}$ in \eqref{e:cube_def} and treat each squares, its subsets and its points as subsets and elements of $\mathbb R^2$. Thus, for the {\em center $x_L$ of $L$} we will use the notation $x_L=(a_1+\ell, a_2+\ell)$, although the precise one is $(a_1+\ell, a_2+\ell, 0, \ldots , 0)$.
Next we set $\sC := \bigcup_{j\in \N} \sC_j$. 
If $H$ and $L$ are two squares in $\sC$ with $H\subset L$, then we call $L$ an \textbf{{\em ancestor}} of $H$ and $H$ a \textbf{{\em descendant}} of $L$. When in addition $\ell (L) = 2\ell (H)$, $H$ is \textbf{{\em a child}} of $L$ and $L$ \textbf{{\em the parent}} of $H$. Moreover, if $H \cap L \neq \emptyset$ but they are not contained in each other, we call them \textbf{{\em neighbours}}.

\begin{definition}\label{e:whitney} A \textbf{\emph{Whitney decomposition of $\overline{D}\cap [-4,4]^2\subset V_0$}} consists of a closed set $\bDel\subset [-4,4]^2\cap \overline{D}$ and a family $\mathscr{W}\subset \sC$ satisfying the following properties:
\begin{itemize}
\item[(w1)] $\bDel \cup \bigcup_{L\in \mathscr{W}} L\cap \overline{D} = [-4,4]^2\cap\overline{D}$ and $\bDel$ does not intersect any element of $\mathscr{W}$;
\item[(w2)] the interiors of any pair of distinct squares $L_1, L_2\in \mathscr{W}$ are disjoint;
\item[(w3)] if $L_1, L_2\in \mathscr{W}$ have nonempty intersection, then $\frac{1}{2}\ell (L_1) \leq \ell (L_2) \leq 2\, \ell (L_1)$.
\end{itemize}
\end{definition}

\begin{remark}
Because of (w1) we will assume that any $L\in \sW$ intersects $\overline{D}$.
\end{remark}

Observe that (w1) - (w3) imply 
\begin{equation}\label{e:separazione}
{\rm sep}\, (\bDel, L) := \inf \{ |x-y|: x\in L, y\in \bDel\} \geq 2\ell (L)  \quad\mbox{for every $L\in \mathscr{W}$,}
\end{equation}
since there is an infinite chain of neighbouring squares $\{L_i\}_{i\in \N}$ with $L_0=L$, $\dist(\bDel, L_i) \to 0$ and $\ell(L_i) \geq 2 \ell(L_{i+1})$ for all $i$.
However, we do {\em not} require any inequality of the form 
${\rm sep}\, (\bDel, L) \leq C \ell (L)$, although this would be customary for what is commonly 
called a Whitney decomposition in the literature.

\begin{ipotesi}\label{ass:hierarchy} 
In the rest of this section we will use several different parameters:
\begin{itemize}
\item[(a)] $\delta_1$ and $\beta_1$ are two small geometric constants which depends only on $Q$, $n$, the constant $\gamma_1$ of Theorem \ref{t:strong_Lipschitz}, in fact they will be chosen smaller than $\frac{\gamma_1}{8}$ and $\delta_1 \leq \frac{\beta_1}{2}$; 
\item[(b)] $M_0$ is a large geometric constant which depends only on $\delta_1$, while $N_0 \geq \frac{ \ln(132 \sqrt 2)}{\ln(2)}$ is a large natural number which will be chosen depending on $\beta_1, \delta_1$, and $M_0$;
\item[(c)] $C^\flat_e$ is a large constant $C^\flat_e (\beta_1, \delta_1, M_0, N_0)$, while $C^\natural_e$ is larger and depends also on $C^\flat_e$;
\item[(d)] $C_h$ is large and depends on $\beta_1, \delta_1, M_0, N_0, C^\flat_e$ and $C^\natural_e$;
\item[(e)] the small threshold $\varepsilon_{CM}$ is the last to be chosen, it depends on all the previous parameters and also on the constant $\varepsilon_A$ of Theorem \ref{t:strong_Lipschitz}.
\end{itemize}
\end{ipotesi}

\begin{definition}\label{d:squares_and_balls}
For each square $L\in \mathscr{C}$ we set $r_L := \sqrt{2} M_0 \ell (L)$ and we say that $L$ is an \textbf{{\em interior square}} if $\dist (x_L, \gamma) \geq 64 r_L$, otherwise we say that $L$ is a \textbf{{\em boundary square}} and we use, respectively, the notation $\mathscr{C}^\natural$ for the interior squares contained in $D$ and $\mathscr{C}^\flat$ for the boundary squares. Next, we define a corresponding $(n+2)$-dimensional balls $\bB_L$, resp. $\bB^\flat_L$, for such $L$'s:
\begin{itemize}
    \item[(a)] If $L\in \mathscr{C}^\natural$, we pick a point $p_L = (x_L, y_L)\in \supp (T) \cap (\{x_L\}\times \mathbb R^n)$ and we set $\bB_L := \bB_{64 r_L} (p_L)$;
    \item[(b)] If $L\in \mathscr{C}^\flat$, we pick $x^\flat_L = (t, \psi_1 (t))\in \gamma$ such that $\dist (x_L, \gamma)= |x^\flat_L-x_L|$, define $p^\flat_L = (t, \psi (t))\in \Gamma\cap (\{x_L^\flat\}\times \mathbb R^n)$ and set $\bB^\flat_L = \bB_{2^7 64 r_L} (p^\flat_L)$.
\end{itemize}
\end{definition}

We are now ready to prescribe $N_0$: we require the inequality
\begin{equation}\label{e:def_N_0}
2^7 64 r_L \le 2^7 64 \sqrt{2} M_0 2^{-{N_0}}\leq 1\, ,
\end{equation}
so that, in particular, all the balls $\bB_L$ and $\bB^\flat_L$ considered above are contained in the cylinder $\bC_{4R_0}$.

The following remark will be useful in the sequel.

\begin{remark}\label{r:good-children}
If $L\in \mathscr{C}^\flat$ and $J$ is the parent of $L$, then $J\in \mathscr{C}^\flat$, while if $L\in \mathscr{C}^\natural$, then every child of $L$ is an element of $\mathscr{C}^\natural$. In fact, if $H$ and $L$ are two squares with nonempty intersection, $\ell (H)<\ell (L)$ and $H$ is a boundary cube, then necessarily $L$ is a boundary cube too.
\end{remark}

\begin{remark}
Fix $L\in \mathscr{C}^\flat$ and subdivide it into the canonical four squares $M$ with half the sidelength. For $M$ any of the following three cases can occur: $M$ might be a boundary square, an interior square, or might simply not belong to $\mathscr{C}^\flat\cup \mathscr{C}^\natural$ (i.e. $M\cap \overline{D} = \emptyset$). However, because of the enlarged radius for boundary squares, it still holds that the ball of a child is contained in the ball of its parent (compare to Proposition  \ref{p:tiltingPlanes}$(i)$). Moreover, $\bB_L^\flat \supset L$ for any boundary square $L$.
\end{remark}

We are now ready to define the refining procedure leading to the desired Whitney decomposition.

\begin{definition}\label{d:refining_procedure}
First of all we set $\bmo:= \bE (T, \bC_{5R_0}) + \|\psi\|^2_{C^{3,\alpha} (]-5R_0, 5R_0[)}$.
We start with all $L\in \sC^\flat\cup \sC^\sharp$ with $\ell (L) = 2^{-N_0}$ and we assign all of them to $\sS$. Next, inductively, for each $j>N_0$ and each $L\in \sC^\flat_j\cup \sC^\natural_j$ such that its parent belongs to $\sS$ we assign to $\sS$ or to $\sW = \sW^e\cup \sW^h\cup \sW^n$ in the following way:
\begin{itemize}
    \item[(EX)] $L\in \sW^e$ if $\bE (T, \bB_L) > C^\natural_e \bmo \ell (L)^{2-2\delta_1}$, resp. if $\bE^\flat (T, \bB^\flat_L) > C^\flat_e \bmo \ell (L)^{2-2\delta_1}$;
    \item[(HT)] $L\in \sW^h$ if $L\not\in \sW^e$ and
    $\bh (T, \bB_L) \geq C_h \bmo^{\frac{1}{4}} \ell (L)^{1+\beta_1}$, resp. $\bh (T, \bB^\flat_L) \geq C_h \bmo^{\frac{1}{4}} \ell (L)^{1+\beta_1}$;
    \item[(NN)] $L\in \sW^n$ if $L\not\in \sW^h\cup \sW^e$ but there is a $J\in \sW$ such that $\ell (J)=2\ell (L)$ and $L\cap J\neq \emptyset$;
    \item[(S)] $L\in \sS$ if none of three conditions above are satisfied. 
\end{itemize}
We denote by $\sC^\flat_j:=\sC^\flat\cap\sC_j$, $\sC^\sharp_j:=\sC^\sharp\cap\sC_j$, $\sS_j := \sS \cap \sC_j$, $\sW_j := \sW \cap \sC_j$, $\sW^e_j := \sW^e \cap \sC_j$, $\sW^h_j := \sW^h \cap \sC_j$ and $\sW^n_j := \sW^n \cap \sC_j$.
Finally, we set
\begin{equation}
\bDel := ([-4,4]^2\cap \overline{D}) \setminus \bigcup_{L\in \sW} L 
= \bigcap_{j\geq N_0} \bigcup_{L\in \sS_j} L\, .
\end{equation}
\end{definition}

A simple consequence of our refining procedure is the following proposition which we will prove in the next section.

\begin{proposition}\label{p:Whitney}
Let $V_0, Q,T$, and $\Gamma$ be as in Assumption \ref{ass:cm} and assume the parameter $N_0$ satisfies \eqref{e:def_N_0}. Then $(\bDel, \sW)$ is a Whitney decomposition of $\overline{D}\cap [-4,4]^2$. Moreover, for any choice of $M_0$ and $N_0$, there is $C^\star (M_0, N_0)$ such that, if $C^\flat_e$, and 
$C^\natural_e/ C^\flat_e$, $C_h/C^\natural_e$,
are larger than $C^\star$, then
\begin{itemize}
\item[(a)] $\sW_{N_0} = \emptyset$;
\item[(b)] if $L\in \sC^\natural\cap \sW^e$ then the parent of $L$ belongs to $\sC^\natural$. 
\end{itemize}
Moreover, the following estimates hold for some geometric constant $C$ depending on $\beta_1$ and $\delta_1$, provided $\varepsilon_{CM}$ is sufficiently small (depending on all the previous parameters as detailed in Assumption \ref{ass:hierarchy}):
\begin{align}
\bE^\flat (T, \bB^\flat_L) \leq C C^\flat_e \bmo \ell (L)^{2-2\delta_1},
\;&\mbox{and}\; \bh (T, \bB^\flat_L) \leq C C_h \bmo^{\frac{1}{4}} \ell (L)^{1+\beta_1},\; \forall L\in \sW\cap \sC^\flat\, ,\label{e:control-in-W}\\
\bE (T, \bB_L) \leq C C^\natural_e \bmo \ell (L)^{2-2\delta_1}
\;&\mbox{and}\; \bh (T, \bB_L) \leq C C_h \bmo^{\frac{1}{4}} \ell (L)^{1+\beta_1},\; \forall L\in \sW\cap \sC^\natural\, .\label{e:control-in-W-2}
\end{align}
\end{proposition}

\subsection{Construction of the center manifold} First of all for each $\bB_L$ and $\bB^\flat_L$, we let $V_L$ be the choice of optimal plane for the excess and the height in the sense of Definition \ref{d:optimal_planes}: note that for boundary squares, namely in $\bB^\flat_L$, the plane $V_L$ optimizes the excess $\bE^\flat$, and thus it is constrained to contain the line $T_{p^\flat_L} \Gamma$. The following key lemma allows us to apply Theorem \ref{t:strong_Lipschitz} (and its interior version \cite[Theorem 2.4]{DS3}) to corresponding cylinders.

\begin{lemma} For any choice of the other parameters, if $\varepsilon_{CM}$ is sufficiently small, the following holds for every $L\in \sS\cup \sW$.
\begin{itemize}
    \item[(a)] If $L\in \sC^\natural$, then $T$ satisfies the assumptions of \cite[Theorem 2.4]{DS3} in $\bC_{32 r_L} (p_L, V_L)$.
    \item[(b)] If $L\in \sC^\flat$, then $T$ satisfies the assumptions of Theorem \ref{t:strong_Lipschitz} in $\bC_{2^7 32 r_L} (p^\flat_L, V_L)$.
\end{itemize}
The corresponding $Q$-valued strong Lipschitz approximations will be denoted by $f_L$ and will be called $V_L$-approximations. 
\end{lemma}

Given a square $L\in \sC^\flat$ which belongs to $\sS\cup \sW$, we denote by $D_L\subset B_{2^7 24 r_L} (p^\flat_L, V^\flat_L)$ the domain of the function $f_L$, which coincides with the orthogonal projection on $p^\flat_L+V^\flat_L$ of $\supp (T) \cap \bC_{2^7 24 r_L} (p^\flat_L, V^\flat_L)$. Note in particular that $\partial D_L \cap B_{2^7 24 r_L} (p^\flat_L, V^\flat_L)$ is the projection of $\Gamma\cap \bC_{2^7 24 r_L} (p^\flat_L, V^\flat_L)$ onto $p^\flat_L + V^\flat_L$, which we will denote by $\gamma_L$. Likewise, we denote by $g_L$ the function over $\gamma_L$ whose graph gives $\Gamma\cap \bC_{2^7 24 r_L} (p^\flat_L, V^\flat_L)$. In particular, Theorem \ref{t:strong_Lipschitz} implies that $f_L|_{\gamma_L} = Q \a{g_L}$.
We now regularize the averages $\etaa\circ f_L$ to suitable harmonic functions $h_L$ in the following fashion.

\begin{definition}\label{d:tilted}
We denote by $h_L$ the harmonic function on $B_{16 r_L} (p_L, V_L)$, resp. $D_L \cap B_{2^7 16 r_L} (p^\flat_L, V_L)$, for $L\in \sC^\natural$, resp. $L\in \sC^\flat$, such that the boundary value of $h_L$ on the respective domain is given by $\etaa\circ f_L$ (in particular it coincides with $g_L$ on $\gamma_L$). $h_L$ will be called \textbf{\em {tilted harmonic interpolating function}}.
\end{definition}

In order to complete the description of our algorithm we need a second important technical lemma.

\begin{lemma}\label{l:u-well-defined}
Consider $L\in \sS\cup \sW$. For every $L\in \sC^\flat$, resp. $L\in \sC^\natural$, there is a smooth function $u_L: D\cap B_{2^7 8r_L} (\p_0 (p^\flat_L), V_0)\to V_0^\perp$, resp. $u_L: B_{8r_L} (\p_0 (p_L), V_0) \to V_0^\perp$, such that
\begin{align}
\bG_{u_L} \res \bC_{8r_L} (p^\flat_L, V_0) &= \bG_{h_L} \res  \bC_{8r_L} (p^\flat_L, V_0),
\qquad \mbox{resp.}\\
\bG_{u_L} \res \bC_{8r_L} (p_L, V_0) &= \bG_{h_L} \res  \bC_{8r_L} (p_L, V_0).
\end{align}
The function $u_L$ will be called \textbf{{\em interpolating function}}. 
\end{lemma}

The center manifold is the result of gluing the interpolating functions appropriately. To that we fix a bump function $\vartheta\in C^\infty_c ((-\frac{3}{2}, \frac{3}{2})^2)$ which is identically $1$ on $[-1,1]^2$ and define
\[
\vartheta_L (x) := \vartheta \left(\frac{x-x_L}{\ell (L)}\right)\, .
\]
Hence, for any fixed $j\geq N_0$ we define 
\begin{equation}\label{e:def-Pj}
\sP^j:= \sS_j \cup \bigcup_{i\leq j} \sW_i
\end{equation}
and the following function $\phii_j$, defined over $D\cap [-4,4]^2\subset V_0$ and taking values in $V_0^\perp$
\begin{equation}
\phii_j (x) := \frac{\sum_{L\in \sP^j} \vartheta_L (x) u_L (x)}{\sum_{H\in \sP^j} \vartheta_H (x)}\, .    
\end{equation}
The center manifold is the graph of the function $\phii$ which is the limit of $\phii_j$ as explained in the statement of the next theorem. 

\begin{theorem}[Center manifold]\label{t:cm}
Let $T$ be as in Assumption \ref{ass:cm} and assume that the parameters satisfy the conditions of Assumption \ref{ass:hierarchy}. Then there is a positive $\omega$ (depending only on $\delta_1$ and $\beta_1$), with the following properties:
\begin{itemize}
    \item[(a)] $\phii_j|_\gamma =g$ for every $j$;
    \item[(b)] $\|\phii_j\|_{C^{3,\omega}}\leq C \bmo^{\frac{1}{2}}$ for some constant $C$ which depends on $\beta_1, \delta_1, M_0, N_0, C^\natural_e,C^\flat_e$, and $C_h$, but not on $\varepsilon_{CM}$;
    \item[(c)] For every $k,k'\geq j+2$, $\phii_k = \phii_{k'}$ on every cube $L\in \sW_j$;
    \item[(d)] $\phii_j$ converges uniformly to a $C^{3,\omega}$ function $\phii$.
\end{itemize}
\end{theorem}

\begin{definition}
The graph of the function $\phii$ will be called \textbf{\em {center manifold}} and denoted by $\cM$. We will define $\Phii (x):= (x, \phii (x))$ as the graphical parametrization of $\cM$ over $[-4,4]^2 \cap \bar D$. The set $\Phii (\bDel)$ will be called the \textbf{\em { contact set}}, while for every $L\in \sW$ the corresponding $\cL := \Phii (L\cap D)$ will be called \textbf{\em {Whitney region}}.
\end{definition}

\subsection{The $\cM$-normal approximation and related estimates}

In what follows we assume that the conclusions of Theorem \ref{t:cm} apply.  
For any Borel set $\cV\subset \cM$ we will denote 
by $|\cV|$ its $\cH^2$-measure and will write $\int_\cV f$ for the integral of $f$
with respect to $\cH^2$. 
$\cB_r (q)$ denotes the geodesic balls in $\cM$. Moreover, we refer to \cite{DS2}
for all the relevant notation pertaining to the differentiation of (multiple valued)
maps defined on $\cM$, induced currents, differential geometric tensors and so on.

\begin{ipotesi}\label{intorno_proiezione}
We fix the following notation and assumptions.
\begin{itemize}
\item[(U)] $\bU :=\big\{x+y: x\in \cM, |y|<1, \mbox{\; and \;}
y\perp \cM\big\}$.
\item[(P)] $\p: \bU \to \cM$ is the map defined by $(x+y)\mapsto x$.
\item[(R)] For any choice of the other parameters, we assume $\eps_{CM}$ to be so small that
$\p$ extends to $C^{2, \kappa}(\bar\bU)$ and
$\p^{-1} (y) = y + \overline{B_1 (0, (T_y \cM)^\perp)}$ for every $y\in \cM$.
\item[(L)] We denote by $\partial_l \bU := \p^{-1} (\de \cM)$ 
the \textbf{\em {lateral boundary}} of $\bU$.
\end{itemize}
\end{ipotesi}

The following is then a corollary of Theorem \ref{t:cm} and the construction algorithm.

\begin{corollary}\label{c:cover}
Under the hypotheses of Theorem \ref{t:cm} and of Assumption \ref{intorno_proiezione}
we have:
\begin{itemize}
\item[(i)] $\supp (\partial (T\res \bU)) \subset \partial_l \bU$, 
$\supp (T\res [-\frac{7}{2}, \frac{7}{2}]^2 \times \R^n) \subset \bU$ 
and $\p_\sharp (T\res \bU) = Q \a{\cM}$;
\item[(ii)] $\supp (\langle T, \p, \Phii (q)\rangle) \subset 
\big\{y\, : |\Phii (q)-y|\leq C \bmo^{\sfrac{1}{4}} 
\ell (L)^{1+\beta_1}\big\}$ for every $q\in L\in \sW$, where $C$ depends on all the parameters except $\eps_{CM}$;
\item[(iii)]  $\langle T, \p, p\rangle = Q \a{p}$ for every $p\in \Phii (\bDel) \cup (\Gamma\cap \partial \cM)$.
\end{itemize}
\end{corollary}

The main reason for introducing the center manifold of Theorem \ref{t:cm} is that we are able to pair it with a good approximating map defined on it.

\begin{definition}[$\cM$-normal approximation]\label{d:app}
An \textbf{\em {$\cM$-normal approximation} of $T$} is given by a pair $(\cK, F)$ such that
\begin{itemize}
\item[(A1)] $F: \cM\to \Iq (\bU)$ is Lipschitz (with respect to the geodesic distance on $\cM$) and takes the special form 
$F (x) = \sum_i \a{x+N_i (x)}$, with $N_i (x)\perp T_x \cM$.
\item[(A2)] $\cK\subset \cM$ is closed and $\bT_F \res \p^{-1} (\cK) = T \res \p^{-1} (\cK)$.
\item[(A3)] $\cK$ contains $\Phii \big(\bDel\cap [-\frac{7}{2}, \frac{7}{2}]^2\big)$ and $\Gamma \cap \Phii (\bar D\cap [-\frac{7}{2}, \frac{7}{2}]^2)$, and on the latter two sets the map $N$ equals $Q\a{0}$.
\end{itemize}
The map $N = \sum_i \a{N_i}:\cM \to \Iq (\R^{2+n})$ is \textbf{\em {the normal part of $F$}}.
\end{definition}

\begin{theorem}[Existence and local estimates for the $\cM$-normal approximation]\label{t:normal-approx}
Let $\gamma_2 := \frac{\gamma_1}{4}$, with $\gamma_1$ the constant
of Theorem \ref{t:strong_Lipschitz}.
Under the hypotheses of Theorem \ref{t:cm} and Assumption~\ref{intorno_proiezione},
if $\eps_{CM}$ is suitably small (depending upon all other parameters but not the current $T$), then
there is an $\cM$-normal approximation $(\cK, F)$ such that
the following estimates hold on every Whitney region $\cL$ associated to
a cube $L\in \sW$, with constants $C = C(\beta_1, \delta_1, M_0, N_0, C^\natural_e, C^\flat_e, C_h)>0:$
\begin{gather}
\Lip (N|
_\cL) \leq C \bmo^{\gamma_2} \ell (L)^{\gamma_2} \quad\mbox{and}\quad  \|N|
_\cL\|_{C^0}\leq C \bmo^{\sfrac{1}{4}} \ell (L)^{1+\beta_1},\label{e:Lip_regional}\\
|\cL\setminus \cK| + \|\bT_F - T\| (\p^{-1} (\cL)) \leq C \bmo^{1+\gamma_2} \ell (L)^{4+\gamma_2},\label{e:err_regional}\\
\int_{\cL} |DN|^2 \leq C \bmo \,\ell (L)^{4-2\delta_1}\, .\label{e:Dir_regional}
\end{gather}
Moreover, for any $a>0$ and any Borel $\cV\subset \cL$, we have $($for $C=C(\beta_1, \delta_1, M_0, N_0, C^\flat_e, C^\natural_e, C_h)$$)$
\begin{equation}\label{e:av_region}
\int_\cV |\etaa\circ N| \leq 
{ C \bmo \left(\ell (L)^{5+\sfrac{\beta_1}{3}} + a\,\ell (L)^{2+\sfrac{\gamma_2}{2}}|\cV|\right)}  + \frac{C}{a} 
\int_\cV \cG \big(N, Q \a{\etaa\circ N}\big)^{2+\gamma_2}\, .
\end{equation} 
\end{theorem}

From \eqref{e:Lip_regional} - \eqref{e:Dir_regional} it is not difficult to infer analogous ``global versions'' of the estimates.

\begin{corollary}[Global estimates for the $\cM$-normal approximation]\label{c:globali} Let $\cM'$ be
the domain $\Phii \big(D \cap [-\frac{7}{2}, \frac{7}{2}]^2\big)$ and $N$ the map of Theorem \ref{t:normal-approx}. Then,  (again with $C = C(\beta_1, \delta_1, M_0, N_0, C^\natural_e, C^\flat_e, C_h)$)
\begin{gather}
\Lip (N|_{\cM'}) \leq C \bmo^{\gamma_2} \quad\mbox{and}\quad \|N|_{\cM'}\|_{C^0}
\leq C \bmo^{\sfrac{1}{4}},\label{e:global_Lip}\\ 
|\cM'\setminus \cK| + \|\bT_F - T\| (\p^{-1} (\cM')) \leq C \bmo^{1+\gamma_2},\label{e:global_masserr}\\
\int_{\cM'} |DN|^2 \leq C \bmo\, .\label{e:global_Dir}
\end{gather}
In addition, since $N= Q\a{0}$ on $\bGam\cap \cM'$, we also get
\begin{equation}\label{e:global_L2}
\int_{\cM'} |N|^2 \leq C \bmo\, .
\end{equation}
\end{corollary}

\subsection{Additional $L^1$ estimate} While the estimates claimed so far have all appropriate counterparts in the papers \cite{DS4} and \cite{DDHM}, we will need an additional important estimate which is noticed here for the first time, even though it is still a consequence of the same arguments leading to Theorem \ref{t:cm} and Theorem \ref{t:normal-approx}. 

\begin{proposition}\label{p:additional}
Consider the function $f:B_3 \to \mathcal{A}_Q (\mathbb R^n)$ with the property that $\bG_f = \bT_F \res \bC_3$. For every $L\in \sW^e$ we then have the estimate
\begin{equation}\label{e:compare-with-average}
\|\phii - \etaa\circ f\|_{L^1 (L)} \leq C \bmo^{3/4} \ell (L)^4
\end{equation}
and in particular, as long as $r\leq 3$ is a radius such that $\ell (L)\leq r$ for every $L\in \sW$ with $L\cap B_r \neq \emptyset$, we have the estimate
\begin{equation}\label{e:compare-with-average-2}
\|\phii - \etaa\circ f\|_{L^1 (B_r)} \leq C \bmo^{3/4} r^4\, .
\end{equation}
\end{proposition}

%% file: cm-proof.tex
\section{Tilting of optimal planes}\label{s:tilting}
We estimate the changes of excess and height when tilting the reference planes of nearby squares.

\begin{proposition}[Tilting of optimal planes]\label{p:tiltingPlanes}
    Let $Q$, $T$ and $\Gamma$ be as in Assumption \ref{ass:cm} and recall the parameters of Assumption \ref{ass:hierarchy}. There are constants $\overline C= \overline C(\beta_1, \delta_1, M_0, N_0,$ $C_e^\natural, C_e^\flat) >0$ and $C= C(\beta_1, \delta_1, M_0, N_0, C_e^\natural, C_e^\flat, C_h)>0$ such that, if $\varepsilon_{CM}=\varepsilon_{CM}(Q, n, R_0, C_h)$ $>0$ is small enough, for any $H, L \in \sS \cup \sW$ with $H$ being equal or a descendant of $L$ we have
    \begin{enumerate}
        \item[$(i)$] $\bB^\Box_H \subset \bB^\Box_L \subset \bB_{4R_0}$,
        \item[$(ii)$] $|V_H-V_L| \leq \overline C \bmo^{\sfrac12} \ell(L)^{1- \delta_1}$,
        \item[$(iii)$] $|V_H-V_0| \leq \overline C \bmo^{\sfrac12}$,
        \item[$(iv)^{\natural}$] if $H \in \sC^\natural$, then\\ $\bh(T, \bC_{36r_H}(p_H, V_0)) \leq C \bmo^{1/4} \ell(H)$ and $\ \supp(T) \cap \bC_{36r_H}(p_H, V_0) \subset \bB_H$,
        \item[$(iv)^{\flat}$] if $H \in \sC^\flat$, then\\ $\bh(T, \bC_{2^7 36r_H}(p^{\flat}_H, V_0)) \leq C \bmo^{1/4} \ell(H)$ and $\ \supp(T) \cap \bC_{2^7 36r_H}(p^{\flat}_H, V_0) \subset \bB_H$,
        \item[$(v)^{\natural}$] if $H, L \in \sC^\natural$, then \\$\bh(T, \bC_{36r_L}(p_L, V_H)) \leq C \bmo^{1/4} \ell(L)^{1+\beta_1}$ and $\ \supp(T) \cap \bC_{36r_L}(p, V_H) \subset \bB_L$,
        \item[$(v)^{\flat}$] if $L \in \sC^\flat$, then \\$\bh(T, \bC_{2^7 36r_L}(p^{\flat}_L, V_H)) \leq C \bmo^{1/4} \ell(L)^{1+\beta_1}$ and $\ \supp(T) \cap \bC_{2^7 36r_L}(p^{\flat}_L, V_H) \subset \bB_L$.
    \end{enumerate}
    where $\Box= \ $ or $\Box= \flat$ depending on whether the square is a boundary square or not.
    Moreover, $(ii)-(v)$ also hold if $H$ and $L$ are neighbours with $\frac12 \ell(L) \leq \ell(H) \leq \ell(L)$.
\end{proposition}

\begin{proof}
    In this proof we will use mainly the following two estimates.
    \begin{align*}
        \bE(T, \bB_r(p),V) &= (2\pi r^2)^{-1} \int_{\bB_r(p)} |\overset{\rightarrow}{T}(x)-\overset{\rightarrow}{ V}|^2 d\|T\| (x)\\
        &\le 2 (2\pi r^2)^{-1} \int_{\bB_r(p)} |\overset{\rightarrow}{T}(x)-\overset{\rightarrow}{ W}|^2 d\|T\| (x) + C|V-W|^2\\
        &= 2\bE(T, \bB_r(p),W) + C|V-W|^2,\\
        \bh(T,\bC_r(p, V),V') &\stackrel{\eqref{e:height-estimate}}{\le} \bh(T,\bC_r(p, V),W) +Cr|V'-W|,
    \end{align*}
    where in the first one we used the monotonicity formula of Theorem \ref{thm:allard} to see that the mass of a ball is comparable to $r^2$ and in the second one we used the height estimate \eqref{e:height-estimate} of Lemma \ref{l:hardtSimonHeight}.
    
    We argue by induction on $i= -\log_2(\ell(H))$. The base step is when $i=N_0$ and $H=L$ while we pass to children squares in the induction step. By the choice of $M_0$ and $N_0$, we notice that there are no squares with side length $2^{-N_0}$ in $\sW$. 
    
    The second inclusion of $(i)$, we already observed in \eqref{e:def_N_0} while the first inclusion of $(i)$ and the inequality in $(ii)$ is redundant for $H=L$. Thus, we show now $(iii)$. We use $(i)$, the optimality of $V_H$, the monotonicity formula of Theorem \ref{thm:allard} and the definition of $\bmo$ to deduce 
    \begin{align}\label{e:baseVH}
        |V _H-V_0|^2 &\leq \overline Cr_H^{-2} \int_{\bB_H^\Box} |\vec{T}- \vec{V}_H|^2 d\|T\| (x) + \overline Cr_H^{-2} \int_{\bB_H^\Box} |\vec{T}-\vec{ V}_0|^2 d\|T\| (x) \notag\\
        &\leq 2 \overline C\bE(T,\bB_H^\Box, V_0) \leq \overline C \bE(T,\bB_{5R_0}, V_0) \leq \overline C \bmo.
    \end{align}
    For $(iv)$ we use the height estimate \eqref{e:height-estimate} of Lemma \ref{l:hardtSimonHeight}. Notice that $\bC_{  36r_H}(p_H^\Box, V_0) \subset \bC_{4R_0}(0,V_0) $ and hence,
    \begin{align*}
        \bh(T, \bC_{  36r_H}(p_H^\Box, V_0))
        \leq \bh (T, \bC_{4R_0}(0, V_0)) 
        \leq \overline C \bmo^{1/4} = \overline C \bmo^{1/4} \ell(H).
    \end{align*}
    Then also the inclusion $\supp(T) \cap \bC_{  36r_H}(p^{\Box}_H, V_0) \subset \bB_H^\Box$ holds, as long as $\varepsilon_{CM}$ is small enough.
    For $(v)$ we observe that as $\bB_H^\Box \subset \bC_{4R_0}(0, V_0)$ we can estimate 
    \[ |p_H^\Box|^2 \leq 9R_0^2 + \bh(T,\bC(4R_0, V_0))^2 \leq 9R_0^2 +C \bmo. \]
    Thus if $\varepsilon_{CM}$ (and thus $\bmo$) is small enough, then $\bC_{  36r_H}(p_H^\Box, V_H ) \cap \bB_{4R_0} \subset \bC_{4R_0}(0, V_0)$. Hence, also $\supp(T) \cap \bC_{  36r_H}(p_H^\Box, V_H ) \subset \bC_{4R_0}(0, V_0)$ and we can estimate
    \begin{align*}
        \bh(T,\bC_{  36r_H}(p_H^\Box, V_H )) 
        &\leq \bh(T,\bC_{4R_0}(0, V_0)) + \overline C|V_H -V_0|\\
        &\leq \overline C \bmo^{1/4}
        = \overline C \bmo^{1/4} \ell(H)^{1+\beta_1}, 
    \end{align*}
    where we used $(iii)$ and $(iv)$.
    
    \begin{figure}[htp]
    \centering
    \includegraphics[width=14cm]{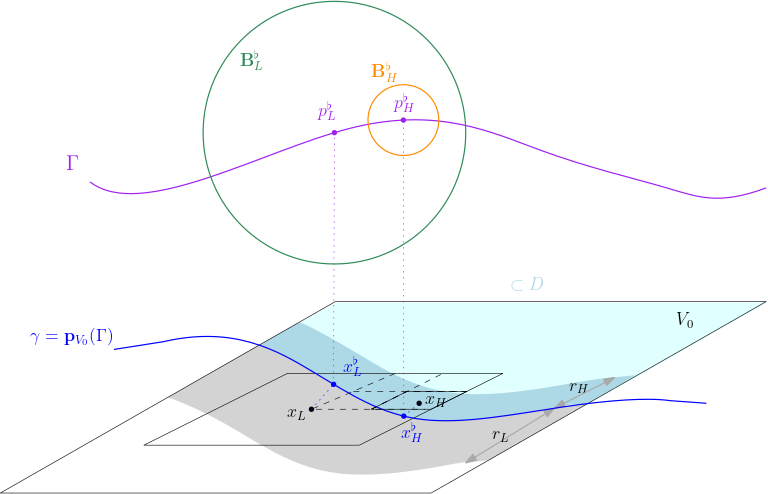}
    \caption{An illustration of the various relevant points in the Whitney square.}
\end{figure}
    
    Induction step: $H \in \sS_{i+1} \cup \sW_{i+1}$ for some $i \geq N_0$. Thus there is a chain of squares such that $H_{i+1} := H \subset H_i \subset \cdots \subset H_{N_0}$ with $H_j \in \sS_j$ for each $j\leq i$. Assume the validity of $(i)-(v)$ for $H_l$ and $H_{k}$ with $N_0 \leq l \leq k \leq i$. We want to show $(i)-(v)$ for $H=H_{i+1}$ and $L=H_j$ with $N_0 \leq j \leq i$. For $(i)$, we notice that it is enough to show the inclusion for $j=i$. Then we have $|x_{H_i}-x_H| \leq \sqrt 2 \ell(H_i)$ and hence, if $\eps_{CM}$ is small enough, we use the induction hypothesis for $(iv)$ to estimate 
    \begin{align*}
        |p^\Box_{H_i}-p_H^\Box|^2 &\leq (\sqrt 2\ell(H_i) + 96r_{H_i})^2 +  \bh(T, \bC_{2r_{H_i}}(p_{H_i}^\Box, V_0))^2\\
        &\leq \ell(H_i)^2 (\sqrt2(1+96M_0))^2 + C \bmo^{1/2} \ell(H_i)^2
        \leq 2^{16} M_0^2 \ell(H_i)^2.
    \end{align*}
    Now we check that $\bB_H^\Box \subset \bB^\Box_{H_i}$. Indeed, we have
    \begin{align*}
        2^7 64r_H +|p^\Box_{H_i}-p_H^\Box| 
    &\leq 2^7 32 \sqrt2 M_0 \ell(H_i) + 2^8 M_0 \ell(H_i) \\
    &\leq 2^7 32 \sqrt2 M_0 \ell(H_i) +2^7 32 \sqrt2 M_0 \ell(H_i)
    = 2^7 64 r_{H_i}.
    \end{align*}
    For $(ii)$, we first show the special case where $j=i$. We notice that by $(i)$, the fact that $2r_H = r_{H_i}$ and $H_i \in \sS_i$, we have
    \begin{align*}
        |V_H  - V _{H_i}|^2 &\leq \overline C \frac{r_H^2}{\|T\|(\bB_H^\Box)} (\bE^\Box(T, \bB_H^\Box) + \bE^\Box(T, \bB^\Box_{H_i})) \qquad \textnormal{(monotonicity formula)}\\
        &\leq \overline C( \bE(T, \bB_H^\Box, V _{H_i}) + \bE^\Box(T, \bB^\Box_{H_i}))
        \leq 2 \overline C \bE^\Box(T, \bB^\Box_{H_i})
        \leq \overline CC_e^\Box \bmo \ell(H)^{2-2\delta_1}.
    \end{align*}
    Now for a general $j \in \{N_0, \dots, i \}$, we use the geometric series to conclude 
    \begin{align*}
        |V_H  - V _{H_j}| \leq \sum_{l=j}^{i}|V_{H_{l+1}}  - V _{H_l}|
        &\leq \overline CC_e^\Box \bmo \sum_{l=j}^{i} \ell(H_l)^{1-\delta_1}\\
        &\leq \overline CC_e^\Box \bmo \sum_{l=j}^{\infty} (2^{-l+j}\ell(H_j))^{1-\delta_1}
        \leq \overline CC_e^\Box \bmo \ell(H_j)^{1-\delta_1}.
    \end{align*}
    $(iii)$ follows by $(ii)$ and \eqref{e:baseVH}.
    To prove $(iv)^\natural$, we observe that by the induction hypothesis, we already know $\supp(T) \cap \bC_{36r_{H_i}}(p^\Box_{H_i}, V _{H_i}) \subset \bB^\Box_{H_i}$. Now we want to see that $\bC_{36r_H}(p_H^\Box, V_0) \subset \bC_{36r_{H_i}}(p^\Box_{H_i}, V_0)$.
    In case where $H_i \in \sC^\natural$, we have $|x_H-x_{H_i}| \leq \sqrt2 \ell(H_i)$, hence
    \[  36 r_H + |x_H-x_{H_i}| \leq 36 r_{H_i} .\]
    On the other hand, if $H_i \in \sC^\flat$, then we recall $|p_H-p^\flat_{H_i}| \leq 2^8 M_0 \ell(H_i)$ which implies 
    \[  36 r_H + |x_H-x^\flat_{H_i}| \leq 36 r_H + |p_H-p^\flat_{H_i}| \leq 2^736 r_{H_i} .\]
    Thus the desired inclusion of the cylinders holds.
    We deduce
    \begin{align*}
        \bh(T, \bC_{36r_H}(p_H, V_0)) \leq \bh(T, \bB^\Box_{H_i}, V_0)
        &\leq \bh(T, \bB^\Box_{H_i}) + \overline C r_{H_i}|V _{H_i} -V_0|\\
        &\leq C_h \bmo^{1/4} \ell(H_i)^{1+\beta_1} +\overline C\ell(H_i)\bmo^{1/2}
        \leq \overline C C_h\ell(H_i)\bmo^{\sfrac14},
    \end{align*}
    where we used the induction hypothesis and that $H_i \in \sS_i$. 
    The previous estimate shows also that $\supp(T) \cap \bC_{36r_H}(p_H, V_0) \subset \bB_H $ assuming that $\varepsilon_{CM}$ is small enough. 
    The proof of $(iv)^\flat$ is analogous because if $H \in \sC^\flat$, then also $H_i \in \sC^\flat$ and so as before
    \[ 2^7 36 r_H + |x_H^\flat-x^\flat_{H_i}| \leq 2^7 36 r_H + |p_H^\flat-p^\flat_{H_i}| \leq 2^736 r_{H_i} .\]
    Now we show $(v)^\natural, (v)^\flat$ for $H=H_{i+1}$ and $L=H_j$ for some $j \in \{N_0, \dots, i\}$ by induction on $j$. For $j=N_0$, we use the estimate on $|V_H -V_{H_{N_0}}|$ to deduce
    \[ \Big( \bC_{2^7 36r_{H_{N_0}}}(p^\Box_{H_{N_0}}, V _{H}) \cap \bB_{4R_0} \Big)
    \subset \Big( \bC_{2^7 36r_{H_{N_0}}}(p^\Box_{H_{N_0}}, V _{H_{N_0}}) \cap  \bB_{5R_0} \Big)
    \subset \bC_{4R_0}(0, V_0) \]
    provided that $\varepsilon_{CM}$ is small enough. Therefore, we have
    \begin{align*}
        \bh(T, \bC_{ 2^7 36r_{H_{N_0}}}(p^\Box_{H_{N_0}}, V_{H} ))
        \leq \bh(T, \bC_{4R_0}(0, V_0)) + C|V_H -V_0|
        \leq \overline C \bmo^{\sfrac12}.
    \end{align*}
    Again if $\varepsilon_{CM}$ is small, this also implies that $\supp(T) \cap \bC_{ 2^7 36r_{H_{N_0}}}(p^\Box_{H_{N_0}}, V_{H} )) \subset \bB^\Box_{H_{N_0}}$.
    Now assume that $(v)^\natural, (v)^\flat$ hold for some $j \geq N_0$ and denote $L=H_{j+1}$. 
    We first consider the case where $L \in \sC^\natural$. Then its parent $H_j$ is still unknown, but in any case, $\bB_L \subset \bB_{H_j}^\Box$ and thus, $\bC_{36r_L}(p_L, V_H) \subset \bC_{36r_{H_{j}}}(p_{H_{j}}, V_H)$ or $\bC_{36r_L}(p_L, V_H) \subset \bC_{2^7 36r_{H_{j}}}(p^\flat_{H_{j}}, V_H)$ respectively. Using the induction hypothesis, we find $\bh(T, \bC_{36 r_{H_{j}}}(p_{H_{j}}, V_{H})) \leq \bh(T, \bB_{H_j}, V_H)$ or $\bh(T, \bC_{2^7 36 r_{H_{j}}}(p^\flat_{H_{j}}, V_{H})) \leq \bh(T, \bB_{H_j}^\flat, V_H)$ respectively. Moreover, using $(ii)$, we deduce
    \begin{align*}
       \bh(T, \bB_{H_j}^\Box, V_H)
        &\leq \bh(T, \bB_{H_j}^\Box)+ \overline C r_{H_{j}} |V_{H} -V_{H_j}|\\
        &\leq \overline C C_h \bmo^{\sfrac14} \ell(H_j)^{1+ \beta_1} + \overline C \bmo^{\sfrac12}\ell(H_j)^{2-\delta_1} 
        \leq \overline C C_h \bmo^{\sfrac14} \ell(H_j).
    \end{align*}
    Thus, we have also  $\supp(T) \cap \bC_{  36r_{L}}(p_{L}, V_{H} )) \subset \bB_{L}$ and finally
    \begin{align*}
        \bh(T, \bC_{  36r_{L}}(p, V_{H}))
        \leq \bh(T, \bB_{L})+ \overline C r_{L} |V_{H} -V_{L}|
        \leq \overline C C_h \bmo^{\sfrac14} \ell(L)^{1+ \beta_1}.
    \end{align*}
    On the other hand, if $L \in \sC^\flat$, then also $H_j \in \sC^\flat$ and we can perform the same argument since $ \bB_L^\flat \subset \bB_{H_j}^\flat$ and $\bC_{2^7 36r_L}(p_L^\flat, V_H) \subset \bC_{2^7 36r_{H_{j}}}(p^\flat_{H_{j}}, V_H)$. This shows both $(v)^\natural$ and $(v)^\flat$.

    For neighbor squares, the argument works exactly the same as everything follows from the smallness of $|p^\Box_L - p^\Box_H|$ and the fact that $\bB_L^\Box \cup \bB_H^\Box \subset \bB_J^\Box$, where $J$ is the parent of $L$.
    
\end{proof}

Very similarly we now prove the excess estimates using the fact, that the parent of any square belongs to $\sS$.

\begin{proof}[Proof of Proposition \ref{p:Whitney}]
    For squares $L$ of side length $2^{-N_0}$, we know by Proposition \ref{p:tiltingPlanes} $(i)$ that $\bB_L^\Box \subset \bB_{4R_0}$ and so we can choose $C_e^\natural$ and $C_e^\flat$ large enough such that
    \[ \bE^\Box(T, \bB_L) \leq C(R_0, N_0) \bE(T, \bB_{4R_0}, V_0)
    \leq C(R_0, N_0)\bmo \leq C_e^\Box \bmo \ell(L)^{2-2\delta_1}.\] Hence, $L \notin \sW^e$. Similarly we see that $L \notin \sW^h$. Indeed, we use Proposition \ref{p:tiltingPlanes} $(ii)$ and the height estimate of Lemma \ref{l:hardtSimonHeight}
    \[ \bh(T, \bB_L^\Box) \leq  \bh(T, \bB_{4R_0}, V_0) + C(R_0,n,Q) |V_L^\Box -V_0| \leq C(R_0,n,Q) \bmo^{1/4}.\]
    Thus, we can choose $C_h$ large enough such that $\bh(T, \bB_L^\Box) \leq C_h \bmo^{1/4} \ell(L)^{1+\beta_1}$. This shows (a).
    
    We claim that (b) holds as long as $C_e^\natural \geq 16 C_e^\flat$. Let $L \in \sC^\natural$ and assume its parent $H \in \sC^\flat$. We want to show that $L \notin \sW^e$. Recall that $|p_L-p_H^\flat| \leq 2^8 M_0\ell(L)$ and thus $\bB_L \subset \bB_H^\flat$. Moreover, as $H$ is a parent, it belongs to $\sS$, thus
    \[ \bE^\flat(T, \bB^\flat_H) \leq C_e^\flat \bmo \ell(H)^{2-2\delta_1}. \]
    This then implies
    \[ \bE(T, \bB_L) \leq \bE(T, \bB_L, V_H ) \leq
    4 \bE^\flat(T, \bB^\flat_H) \leq 16 C_e^\flat \bmo \ell(L)^{2-2\delta_1}. \]
    Now let $L \in \sW \cap \sC^\flat$ and denote by $H \in \sS$ the parent of $L$. As $L$ is a boundary square, so is $H$. By Proposition \ref{p:tiltingPlanes} $(i)$ and $(ii)$, we know that $\bB_L^\flat \subset \bB_H^\flat$ and 
    \begin{align*}
        \bE^\flat(T, \bB_L^\flat) &\leq 4\bE^\flat(T, \bB_H^\flat) \leq C C_e^\flat \bmo \ell(L)^{2-2\delta_1},\\
        \bh(T, \bB_L^\flat) &\leq \bh(T, \bB_H^\flat) + \overline C r_L |V _L-V _H|
        \leq C C_h \bmo^{1/4} \ell(L)^{1+\beta_1}.
    \end{align*}
    On the other hand, for $L \in \sW \cap \sC^\natural$, the parent $H$ of $L$ could be either a  boundary square or an interior square. So we estimate
    \begin{align*}
        \bE(T, \bB_L) &\leq 4\bE^\Box(T, \bB_H^\Box) \leq C (C_e^\flat +C_e^\natural) \bmo \ell(L)^{2-2\delta_1},\\
        \bh(T, \bB_L) &\leq \bh(T, \bB_H^\Box) + \overline C r_L |V_L-V^\Box_H|
        \leq C C_h \bmo^{1/4} \ell(L)^{1+\beta_1}.
    \end{align*}
\end{proof}

\section{Estimates on the interpolating functions}\label{s:interpolating}
We notice that our construction fulfills the estimates needed for the strong Lipschitz approximation.
\begin{proposition}\label{p:f-well-defined}
    Suppose that Assumption \ref{ass:cm} holds true, recall the constants from Assumption \ref{ass:hierarchy} and assume that $\varepsilon_{CM}$ is small enough. Let either $H, L \in \sS \cup \sW$ be neighbors with $\frac{1}{2}\ell (L) \leq \ell (H)\leq \ell (L)$ or let $H$ be a descendant of $L$. Then we have
    \begin{align*}\label{e:TiltedSupportEstimates}
    \supp(T) \cap \bC_{32 r_L}(p_L, V_H) &\subset \bB_L, \quad \textnormal{if } L \in \sC^\natural,\\
    \supp(T) \cap \bC_{2^7 32 r_L}(p_L^\flat, V_H) &\subset \bB_L^\flat, \quad \textnormal{if } L \in \sC^\flat,
    \end{align*}
    and \cite[Theorem 2.4]{DS3} can be applied to $T$ in the cylinder $\bC_{32 r_L}(p_L, V_H)$ and Theorem \ref{t:strong_Lipschitz} in $\bC_{2^7 32 r_L}(p_L^\flat, V_H)$ respectively. The resulting strong Lipschitz approximation we call $f_{HL}$.
    \end{proposition}
\begin{proof}
    The proof of Proposition \ref{e:TiltedSupportEstimates} is completely analogous to \cite[Proposition 4.2]{DS4} for interior squares and to \cite[Proposition 8.25]{DDHM} for boundary squares. 
\end{proof}
\begin{remark}
Observe that if $\ell (H) < \ell (L)$ and $H$ is a boundary square, then $L$ is necessarily also a boundary square, since either $H$ and $L$ are  neighbors or $H\subset L$. When $\ell(H) = \ell (L)$, in case $H$ is a boundary square and $L$ is an interior square, we can simply swap their roles. In particular, without loss of generality, we will in the sequel ignore the case in which $H$ is a boundary square and $L$ is an interior square. 
\end{remark}

\begin{definition} We denote by $f_{HL}$ the strong Lipschitz approximation produced by Proposition \ref{e:TiltedSupportEstimates}. We will however consider the domain of the function $f_{HL}$ a subset of $p_H+ V_H$, resp. $p_H^\flat+V_H$. More precisely, for interior squares the domain is $\bC_{24 r_L} (p_L, V_H) \cap (p_H+V_H)$, while for boundary squares it is $D_{HL}:= D_H \cap \bC_{2^7 24r_L} (p^\flat_L, V_H)$, where we recall that $D_H$ is the projection on $p^\flat_H+ V_H$ of $\supp (T)$. Observe that  $\bC_{24r_L} (p_L, V_H)\cap (p_H+V_H)$ and $\bC_{2^7 24r_L} (p_L^\flat, V_H)\cap (p_H^\Box+V_H)$ are discs, whose centers are given by
  \begin{align}
  p_{HL} &:= p_H + \mathbf{p}_{V_H} (p_L), \qquad \text{resp.}\\
  p^\flat_{HL} &:= p^\Box_H + \mathbf{p}_{V_H} (p_L^\flat).\,
  \end{align}
  (Note that, when $L$ is a boundary square, $H$ might be a boundary square but it might also be an interior square). 
 \end{definition} 
\begin{definition}\label{d:tilted-2}
We then let $h_{HL}$ be the harmonic function on $B_{16 r_L} (p_{HL}, V_H)$, resp. $D_H \cap \bC_{2^7 16 r_L} (p^\flat_{HL}, V _H)$, such that the boundary value of $h_{HL}$ on the respective domain is given by $\etaa\circ f_{HL}$, in particular it coincides with $g_H$ on $\gamma_H$. $h_{HL}$ will be called the \textbf{\em {$(H,L)$-tilted harmonic interpolating function}}.
\end{definition}

Lemma \ref{l:u-well-defined} will then be a particular case of the following more general lemma.

\begin{lemma}\label{l:u-well-defined-2}
Consider $H$ and $L$ as in Proposition \ref{p:f-well-defined}. Then there is a smooth function $u_{HL}: D\cap B_{2^7 8r_L} (\p_0 (p^\flat_L), V_0)\to V_0^\perp$, resp. $u_{HL}: B_{8r_L} (\p_0 (p_L), V_0) \to V_0^\perp$, such that
\begin{align}
\bG_{u_{HL}} \res \bC_{8r_L} (p_L, V_0) &= \bG_{h_{HL}} \res  \bC_{8r_L} (p_L, V_0),\\
\bG_{u_{HL}} \res \bC_{2^7 8r_L} (p^\flat_L, V_0) &= \bG_{h_{HL}} \res  \bC_{2^7 8r_L} (p^\flat_L, V_0), \quad \textnormal{respectively}.
\end{align}
The function $u_{HL}$ will be called \textbf{{\em interpolating function}}. 
\end{lemma}

\subsection{Linearization and first estimates on $h_{HL}$}
\begin{proposition}\label{pr:elliptic_regularization}
Under the Assumptions of Proposition \ref{p:f-well-defined} the following estimates hold for every pair of squares $H$ and $L$ as in Proposition \ref{p:f-well-defined}. First of all
\begin{align}
&\int D (\etaa\circ f_{HL}): D\zeta \leq C \bmo r_L^{4+\beta_1} \|D \zeta\|_0,\label{e:weak_ell_1}
\end{align}
for every function $\zeta$ in $C_c^\infty (B_{8r_L} (p_{HL}, V_H), V_H^\perp)$, resp. $C^\infty_c (D_H\cap B_{2^7 8r_L} (p^\flat_{HL}, V_H), V_H^\perp)$, depending on whether $L\in \sC^\natural$ or $L\in \sC^\flat$. 
Moreover,
\begin{eqnarray}
\|h_{HL} - \etaa\circ f_{HL}\|_{L^1 (B_{8r_L} (p_{HL}, V_H))} & \leq C \bmo r_L^{5+\beta_1},\qquad
&\mbox{if $L\in\mathscr{C}^\natural$;}\label{e:L1_interior}\\
\|h_{HL} - \etaa\circ f_{HL}\|_{L^1 (D_H \cap B_{2^7 8 r_L} (p^\flat_{HL}, V_H))} & \leq C \bmo r_L^{5+\beta_1},\qquad
&\mbox{if $L\in\mathscr{C}^\flat$;}\label{e:L1_boundary}\\
\|D h_{HL}\|_{L^\infty (B_{7r_L} (p_{HL}, V_H))} & \leq C \bmo^{\frac12} r_L^{1-\delta_1},\qquad
&\mbox{if $L\in\mathscr{C}^\natural$;}\label{e:lip_interior}\\
\|D h_{HL}\|_{L^\infty (D_H \cap B_{2^7 7 r_L} (p^\flat_{HL}, V_H))} &\leq C \bmo^{\frac12} r_L^{1-\delta_1}, \qquad
&\mbox{if $L\in\mathscr{C}^\flat$.}
\label{e:lip_boundary}
\end{eqnarray}
\end{proposition}

\begin{proof} {\bf Proof of \eqref{e:weak_ell_1}.} Without loss of generality consider a system of coordinates $(x,y)$ with the property that $p^\Box_{HL}$ is the origin, $(x,0)\in V_H$ and $(0,y)\in V_H^\perp$. Fix $\zeta$ as in the statement of the proposition and in the cylinder $\bC \in \{ \bC_{32r_L} (p_{HL}, V_H), \bC_{2^7 32r_L} (p_{HL}^\flat, V_H)\}$ we consider the vector field $\chi (x,y) = (0, \zeta (x))$. Observe that, by assumption, the vector field vanishes on $\Gamma$. Observe that, though $\chi$ is not compactly supported, since the height of the current in the cylinder $\bC$ is bounded, we can multiply $\chi$ by a cut-off function in the variable $y$ but keeping its values the same on $\supp (T)$. The latter vector field is a valid first variation for the area-minimizing current $T$ and thus we have $\delta T (\chi) =0$. Thus we can use Theorem \ref{t:strong_Lipschitz} and Proposition \ref{p:tiltingPlanes} to estimate
\begin{align*}
|\delta \bG_{f_{HL}} (\chi)| &= |\delta (T-\bG_{f_{HL}}) (\chi)| \leq \|D\zeta\|_0 \|T- \bG_{f_{HL}}\| (\bC)\nonumber\\
&\leq C \|D\zeta\|_0 r_L^2 (\bE^\Box (T, \bC, V_H)+\bA^2r_L^2)^{1+\gamma_1}\\
&\leq C \|D\zeta\|_0 r_L^2 (\bE^\Box (T, \bB^\Box_L)+|V_H-V_L|^2 + \bA^2r_L^2)^{1+\gamma_1}\nonumber\\
&\leq C \|D\zeta\|_0 r_L^2 (\bmo r_L^{2-2\delta_1})^{1+\gamma_1} \leq C \|D\zeta\|_0 \bmo r_L^{4+\beta_1}\, ,
\end{align*}
provided $\delta_1$ and $\beta_1$ are chosen small enough to satisfy $(2-2\delta_1)(1+\gamma_1)\geq 2+\beta_1$. 

Next we use the Taylor expansion \cite[Theorem 4.1]{DS2} to estimate
\begin{align*}
&\left|\delta \bG_{f_{HL}} (\chi) - Q \int \etaa\circ Df_{HL}:D\zeta\right| \leq C\|D\zeta\|_0 \int |Df_{HL}|^3\\
\leq &C \|D\zeta\|_0 \Lip (f_{HL})
\int |Df_{HL}|^2\\
\leq & C \|D\zeta\|_0 (\bE^\Box (T,\bC, V_H) + \bA^2r_L^2)^{\gamma_1} r_L^2 (\bE^\Box (T,\bC, V_H) + \bA^2r_L^2)\\
\leq & C \|D\zeta\|_0 r_L^2 (\bmo r_L^{2-2\delta_1})^{1+\gamma_1}\, .
\end{align*}

\medskip

{\bf Proof of \eqref{e:L1_interior}-\eqref{e:L1_boundary}.} Consider $v:= h_{HL} - \etaa\circ f_{HL}$ on its respective domain $\Omega$ which equals either $B_{8r_L} (p_{HL}, V_H)$ or $D_H\cap B_{2^7 8r_L} (p^\flat_{HL}, V_H)$. Observe that $v$ vanishes on the boundary of $\Omega$. For every $w\in L^2$ we denote by $\zeta= P (w)$ the unique solution of $\Delta \zeta = w$ in $\Omega$ with $\zeta|_{\partial \Omega}=0$, which is an element of the Sobolev space $W^{1,2}_0 (\Omega)$. Next notice that by a simple density argument, the estimate \eqref{e:weak_ell_1} remains valid for any test function $\zeta\in W^{1,2}_0$ and recall also the standard estimate
\[
\|D (P(w))\|_0 \leq C r \|w\|_0\, .
\]
Therefore we can write
\begin{align*}
\|v\|_{L^1} &= \sup_{w:\|w\|_0\leq 1} \int_\Omega v\cdot w = \sup_{w:\|w\|_0\leq 1}      \int_\Omega v\cdot \Delta (P(w))\\
&= \sup_{w:\|w\|_0\leq 1}  \left(- \int_\Omega Dv: D (P(w))\right) = \sup_{w:\|w\|_0\leq 1} \int_{\Omega} D \etaa\circ f_{HL} : D (P(w))\\
&\leq C \sup_{w:\|w\|_0\leq 1} \bmo r_L^{4+\beta_1} \|D P (w)\|_0 \leq C \bmo r_L^{5+\beta_1}\, .
\end{align*}

\medskip

{\bf Proof of \eqref{e:lip_interior}.} Using the mean-value inequality for harmonic functions we simply get
\begin{align*}
\|D h_{HL}\|_{L^\infty (B_{7r_L} (p_{HL}, V_H))} &\leq \frac{C}{r_L^2} \int_{B_{8r_L} (p_{HL}, V_H)} |D h_{HL}|\\
&\leq \frac{C}{r_L} \left(\int_{B_{8r_L} (p_{HL}, V_H)} |D h_{HL}|^2\right)^{\sfrac{1}{2}}\\
&\leq  \frac{C}{r_L} \left(\int_{B_{8r_L} (p_{HL}, V_H)} |D \etaa\circ f_{HL}|^2\right)^{\sfrac{1}{2}}\\
&\leq  \frac{C}{r_L} \left(r_L^2 (\bE (T, \bC, V_H)+\bA^2r_L^2)\right)^{\frac{1}{2}} \leq C\bmo^{\frac{1}{2}} r_L^{1-\delta_1}\, .
\end{align*}

\medskip

{\bf Proof of \eqref{e:lip_boundary}.} Using standard Schauder estimates for harmonic functions, we get
\begin{align*}
\|Dh_{HL}\|_{L^\infty (D_H\cap B_{2^7 7r_L} (p_{HL}^\flat, V_H))} \leq \frac{C}{r_L^2} \int_{D_H\cap B_{2^7 8r_L} (p_{HL}^\flat, V_H)} |Dh_{HL}|
+ C (\|D g_H\|_0 + r_L^{-\alpha} [g_H]_\alpha)\, , 
\end{align*}
where we recall that $g_H: \partial D_H\cap B_{2^7 8r_L} (p_{HL}^\flat, V_H)$ is the graphical parametrization of our boundary curve $\Gamma$ and $\alpha$ is a positive number smaller than $1$, to be chosen later. The first summand on the right hand side is estimated as in the proof above of \eqref{e:lip_interior}. As for the second summand, recall that $T_{p^\flat_L} \Gamma$ is contained in the plane $V_L$ and that $|V_L-V_H|\leq C \bmo^{\sfrac{1}{2}} r_L^{1-\delta_1}$. This implies that 
\[
|D g_H (p^\flat_{HL})|\leq C \bmo^{\sfrac{1}{2}} r_L^{1-\delta_1}\, .
\]
In particular we have 
\[
\|Dg_H\|_{L^\infty (\partial D_H \cap B_{2^7 8r_L} (p_{HL}^\flat, V_H))} \leq |D g_H (p^\flat_{HL})| + C \bA r_L
\leq C \bmo^{\sfrac{1}{2}} r_L^{1-\delta_1}\, .
\]
On the other hand,
\[
r_L^{-\alpha} [g_H]_\alpha \leq C r_L^{1-2\alpha} \bA \leq C \bmo^{\sfrac{1}{2}} r_L^{1-2\alpha}\, ,
\]
and thus it suffices to choose $2\alpha<\delta_1$.
\end{proof}

\subsection{Tilted estimate} We follow here \cite[Section 8.5]{DDHM} almost verbatim to establish a suitable comparison between tilted interpolating functions which are defined in different system of coordinates.

\begin{definition}
Four cubes $H,J,L,M\in \mathscr{C}$ make a \emph{distant relation} between $H$ and $L$ if $J,M$ are neighbors (possibly the same cube) with same side length and $H$ and $L$ are descendants respectively of $J$ and $M$. 
\end{definition}

\begin{lemma}[Tilted $L^1$ estimate]\label{lem:tilted_L1}
Under the Assumptions of Theorem \ref{t:cm} the following holds for every quadruple $H,J,L$ and $M$ in $\sS\cup \sW$ which makes a distant relation between $H$ and $L$.
\begin{itemize}
\item If $J\in \mathscr{C}^\natural$, then there is a map $\hat{h}_{LM} : B_{4r_J} (p_{HJ}, V_H) \to V_H^\perp$ such that
\[
\bG_{\hat h_{LM}} = \bG_{h_{LM}} \res \bC_{4r_J} (p_{HJ}, V_H)
\]
and
\begin{equation}\label{e:tilted_L1_int}
\|h_{HJ}- \hat h_{LM}\|_{L^1 (B_{2r_J} (p_{HJ}, V_H))} \leq C \bmo \ell (J)^{5+\beta_1/2}\, .
\end{equation}
\item If both $J$ and $M$ belong to $\mathscr{C}^\flat$, then there is a map $\hat{h}_{LM} : D_{HJ} \cap B_{2^7 4r_J} (p^\flat_{HJ}, V_H) \to V_H^\perp$ such that
\[
\bG_{\hat h_{LM}} = \bG_{h_{LM}} \res \bC_{2^7 4 r_J} (p^\flat_{HJ}, V_H)
\]
and
\begin{equation}\label{e:tilted_L1_b}
\|h_{HJ}- \hat h_{LM}\|_{L^1 (D_{HJ} \cap B_{2^7 2 r_J} (p^\flat_{HJ}, V_H))} \leq C \bmo \ell (J)^{5+\beta_1/2}\, .
\end{equation}
\end{itemize}
\end{lemma}

The proof follows verbatim the arguments given in \cite[Section 8.5]{DDHM}. The only difference is the absence of the ``ambient Riemannian'' manifold which in \cite[Lemma 8.31]{DDHM} is the graph of a function $\Psi$. The case needed for our arguments is the clearly simpler situation in which the linear subspaces $\varpi$ and $\bar\varpi$ in \cite[Lemma 8.31]{DDHM} are given by the trivial subspace $\{0\}$. The proof of this version of the lemma (which is in fact \cite[Lemma 5.6]{DS4}) is even less complicated. However there is a direct way to conclude it directly from the more general statement of \cite[Lemma 8.31]{DDHM}: we can consider $\mathbb R^{2+n}$ as a subspace of $\mathbb R^{2+n+1}$ and apply \cite[Lemma 8.31]{DDHM} to a generic choice of $\varkappa, \bar\varkappa, \pi, \bar \pi$ and the specific choice of $\varpi = \bar\varpi = \{0\}\times \mathbb{R}$ and $\Psi = \bar \Psi : \pi\times \varkappa = \bar\pi\times \bar \varkappa \to \varpi=\bar\varpi$ given by the trivial map $\Psi \equiv 0$.

\section{Final estimates and proof of Theorem \ref{t:cm}}\label{s:cm-final}

\begin{proposition}\label{p:block_estimates}
There is a constant $\omega$ depending upon $\delta_1$ and $\beta_1$ such that,
under the assumptions of Theorem \ref{t:cm}, the following holds for every pair of squares $H, L \in \mathscr{P}^j$ (cf. \eqref{e:def-Pj}).
\begin{itemize}
\item[(a)] $\|u_H\|_{C^{3,\omega} (B_{4r_H} (x_H)}\leq C\bmo^{1/2}$, resp. $\|u_H\|_{C^{3,\omega}(D\cap B_{2^7 4r_H} (x^\flat_H))}\leq C\bmo^{1/2}$, for $H\in \mathscr{C}^\natural$, resp. $H\in \mathscr{C}^\flat$;
\item[(b)] If $H$ and $L$ are neighbors then for any $i\in \{0, 1,2,3\}$, we have
\begin{align}
&\|u_H- u_L\|_{C^i (B_{r_H} (x_H))} \leq C \bmo^{1/2} \ell (H)^{3+\omega-i}
\qquad \mbox{when $H\in \mathscr{C}^\natural$,}\\
&\|u_H- u_L\|_{C^i (D\cap B_{2^7 r_H} (x^\flat_H))} \leq C \bmo^{1/2} \ell (H)^{3+\omega-i} \qquad \mbox{when $H, L\in \mathscr{C}^\flat$;}
\end{align}
\item[(c)] $|D^3 u_H (x_H^\Box) - D^3 u_L (x^\Box_L)|\leq C \bmo^{1/2} |x^\square_H - x^\square_L|^\omega$, where $\square = \phantom{\flat}$ if the corresponding square is a non-boundary square and $\square = \flat$ if it is a boundary square;
\item[(d)] if $H\in \mathscr{C}^\natural$, then $\|u_H - \bp_{V_0}^\perp (p_H)\|_{C^0 (B_{4r_H} (x_H))} \leq C \bmo^{1/2} \ell (H)$ and if $H\in \mathscr{C}^\flat$, then $u_H |_{\partial D \cap B_{2^7 4r_H} (x_H^\flat))} = g$ ;
\item[(e)] $|V_H - T_{(x, u_H (x))} \bG_{u_H} |\leq C\bmo^{1/2} \ell (H)^{1-\delta_1}$ for every $x\in B_{4r_H} (x_H)$, resp. $x\in D\cap B_{2^7 4r_H} (x^\flat_H)$;
\item[(f)] If $H'$ is the square concentric to $H\in \sW_j$ with $\ell (H') = \frac{9}{8} \ell (H)$, then
\begin{equation}
\|\phii_i - u_H\|_{L^1 (H')}\leq C \bmo \ell (H)^{5+\beta_1/2}\qquad \forall i\geq j+1\, . 
\end{equation}
\end{itemize}
\end{proposition}

\subsection{Proof of Proposition \ref{p:block_estimates}}
\begin{proof} 
We follow the proof of \cite[Proposition 8.32]{DDHM} and often we drop here for simplicity the domains where we estimate the norm in.
    \item[(a)]  By \cite[Lemma B.1]{DS3}, it is enough to make the estimates on $h_H$ instead of $u_H$. Fix any square $H \in \sP^j$ and consider the family tree $H=H_i \subset H_{i-1} \subset \cdots \subset H_{N_0}$. We estimate
    \[ \| h_H\|_{C^{3, \omega}} \leq \sum_{j=N_0+1}^i \| h_{HH_j} -h_{HH_{j-1}}\|_{C^{3, \omega}} +  \| h_{HH_{N_0}} \|_{C^{3, \omega}} \, .
   \]
   As these are all harmonic functions, by the mean value property, it is enough to estimate the $L^1$ norms. Again using the harmonicity we see that
   \[  \| h_{HH_j} -h_{HH_{j-1}}\|_{L^1(\Omega_j)} 
   \leq \| \etaa \circ f_{HH_j} -\etaa \circ f_{HH_{j-1}}\|_{L^1(\Omega_j)} + C \bmo r_{H_{j-1}}^{5+ \beta_1}, 
   \]
   where $\Omega_j$ either is $B_{7r_{H_j}}(p_{H_j}, V_H)$ if $H_j \in \sC^\natural$ or $D_H \cap B_{2^7 7r_{H_j}}(p^\flat_{H_j}, V_H)$ if $H_j \in \sC^\flat$. Using Theorem \ref{t:strong_Lipschitz}, we see that both $f_{HH_{j}}$ and $f_{HH_{j-1}}$ describe $\supp(T)$ on a large set $K$, thus their average agree on $K$. Together with the oscillation estimate we then deduce
   \begin{align*}
       \| \etaa \circ f_{HH_j} -\etaa \circ f_{HH_{j-1}}\|_{L^1(\Omega_j)} 
       &\leq C  \ell(H_{j-1})^2 \left( \bmo \ell(H_{j-1})^{2-2\delta_1} \right)^{1+ \gamma_1} \bmo^{\sfrac14}\ell(H_{j-1})^{1+\beta_1}\\
        &\leq C \bmo \ell(H_{j-1})^{5+\beta_1} \,.
   \end{align*}
   For $\| h_{HH_{N_0}} \|_{C^{3, \omega}}$ we argue similarly and use Proposition \ref{pr:elliptic_regularization}.
   
   \item[(b)] By \cite[Lemma C.2]{DS3}, we have
   \[ \|D^j(u_H-u_L) \|_{C^0} \leq C Cr_L^{-2-j}\|u_H-u_L \|_{L^1} + Cr_L^{3+ \omega -j}\|D^3(u_H-u_L) \|_{C^\omega}.
   \]
   The second term is already bounded in (a), thus we are left with showing the $L^1$ estimate. To do so, we again use \cite[Lemma B.1]{DS3} to replace $u_L$ and $u_H$ with functions which have the same graph. It is enough to notice that, by Lemma \ref{lem:tilted_L1}
   \[ \|h_H-\hat h_L \|_{L^1} \leq C \bmo^{\sfrac12} \ell(H)^{5+ \delta_1/2}.
   \]
   \item[(c)] Let $H, L \in \mathscr{P}^j$. In case that $|x_H-x_L| \geq 2^{-N_0}$, the statement follows from (a). Otherwise, we can find ancestors $J, M$ such that $H, L$ are in a distant relation where $\ell(J)= \ell(M)$ is comparable to  $|x_H^\Box-x_L^\Box|$. Then we estimate
   \begin{align*}
       |D^3 u_H (x_H^\Box) - D^3 u_L (x^\Box_L)|
       &\leq |D^3 u_H (x_H^\Box) - D^3 u_{HJ} (x^\Box_J)| + |D^3 u_{LM} (x_M^\Box) - D^3 u_L (x^\Box_L)|\\
       &\quad + |D^3 u_{HJ} (x_J^\Box) - D^3 u_{LM} (x^\Box_M)|.
   \end{align*}
   The bound on the last term is already shown in (b), while for the first two we argue similarly as before. Consider the family tree $H \subset H_{i-1} \subset \cdots \subset J$. By the previous arguments, we deduce
   \[ \| u_{HH_i} - u_{HH_{i-1}} \|_{C^3} \leq C \bmo^{\sfrac12} \ell(H_{i-1})^\omega.
   \]
   \item[(d)] The claim is obvious by construction for boundary cubes. For non-boundary cubes, consider that the height bound for $T$ and the Lipschitz regularity for $f_{H}$ give that $$\left\| \mathbf{p}_{V_{H}}^{\perp}\left(p_{H}\right)-\boldsymbol{\eta} \circ f_{H}\right\|_{\infty} \leq C \bmo^{1/4} \ell(H).$$ We also get
$\left\|\mathbf{p}_{V_{H}}^{\perp} \left(p_{H}\right)-\etaa \circ f_H \right\|_{\infty} \leq C \bmo^{1/4} \ell(H)$. On the other hand the Lipschitz
regularity of the tilted $H$ -interpolating function $h_{H}$ and the $L^{1}$ estimate on $h_{H}-\etaa \circ f_H$ easily gives $\left\|\mathbf{p}_{V_{H}}^{\perp}\left(p_{H}\right)-h_{H}\right\|_{\infty} \leq C \bmo^{1/4} \ell(H)$. The
estimate claimed in (d) follows then from \cite[Lemma B.1]{DS4}.
   \item[(e)] follows from the estimates on $D \overline h_{HL}$ of Lemma \ref{lem:tilted_L1}.
   \item[(f)] By definition of $\phii_j$, it is enough to estimate that for $L$ a neighbour square of $H$, we have
   \[ \|u_H -u_L \|_{L^1} \leq C \bmo \ell(H)^{5 +\delta_1/2}.\]
    
\end{proof}

\subsection{Proof of Theorem \ref{t:cm}}
\begin{proof}
    (a) is an immediate consequence of the definition of 
    $\phii_j$ and the fact that $u_L$ satisfies the correct boundary condition (for $L \in \sC^\flat$).
    (b) follows exactly as in the proof of \cite[Theorem 1.17]{DS3} and from Proposition \ref{p:block_estimates}. In fact, we are in the simpler situation where our "ambient manifold" is just $\mathbb R^{n+2}$ and thus, we can choose $\Psi \equiv 0$. (c) and (d) are consequences of (b).
\end{proof}

\subsection{Proof of Corollary \ref{c:cover} and Theorem \ref{t:normal-approx}}
\begin{proof}
    We extend $\phii$ to all of $[-4,4]^2$ changing the $C^{3, \omega}$-norm only by geometric constant and call this extension $\tilde \phii$. Then consider
    \[ \tilde T := T + Q \cdot \bG_{\tilde \phii |_{[-4,4]^2 \setminus D}}. \]
    Then as $\partial \mathcal M = \Gamma$, so $\partial \tilde T = 0$. We cannot directly apply the corresponding interior paper, \cite[ Corollary 2.2]{DS4}, to $\tilde T$ because the latter is not area-minimizing. However, the argument given in \cite[Proof of Corollary 2.2]{DS4} does not use the area-minimizing assumption. It uses only the height estimates of Proposition \ref{p:tiltingPlanes} (which can be trivially extended to $\tilde{T}$ since the portion added to $T$ is regular) and the constancy theorem (which is valid in our case, since $\tilde{T}$ has no boundary).
    
    As for the existence and estimates on the normal approximation, we also can follow the same argument as in \cite[Section 6.2]{DS4} substituting the current $\tilde{T}$ to the current $T$ in there and the map $\tilde \phii$ to the map $\phii$ in there. First of all notice that the extension is done locally on each square and the ones surrounding it, and thus, even though the union of the squares in our $\sW$ and the set $\mathbf{\Delta}$ does not cover $[-4,4]^2$, this does not prevent us from applying the same procedure. Next, the construction algorithm and the estimates performed in \cite[Section 6.2]{DS4} depend only on the following two facts:
    \begin{itemize}
        \item[(a)] The map $\phii$ in \cite[Section 6.2]{DS4} has, on every $L\in \sW$, the same control on the $C^{3,\omega}$ norm that we have for the map $\tilde\phii$ (up to a constant).
        \item[(b)] For each square $L\in \sW$ (which in the case of \cite[Section 6.2]{DS4} corresponds to an interior square for us) we have a Lipschitz approximation $f_L$ of the current $T\res \bC_{8r_L} (p_L, V_L)$, which in turn coincides with the current $T$ on a set $K_L\times V_L^\perp$, where $|B_{8r_L}\setminus K_L|$ is small and the Lipschitz constant and the height of $f_L$ are both suitably small too. This is literally the case with the very same estimates for our interior squares, because $\tilde T\res \bC_{8r_L} (p_L, V_L) = T \res \bC_{8r_L} (p_L, V_L)$. In the case of boundary squares, we apply Theorem \ref{t:strong_Lipschitz} and we extend the corresponding $f_L$ to a map $\tilde{F}_L$ on the whole disk $B_{2^7 8r_L} (p^\flat_L, V_L)$ by setting it equal to $Q$ copies of the graph of $\tilde \phii$ outside of the domain $D_L\cap B_{2^7 8r_L} (p^\flat_L, V_L)$. We then notice that such extension satisfies the same estimates on the Lipschitz constant and the height. Moreover, over the new region, by construction the extension coincides with the current $T$. Hence, if we denote by $\tilde{K}_L$ the complement of the projection on $V_L$ of the difference set $\supp (\tilde T) \Delta \supp (\bG_L (f_L))$, then 
        \[
        B_{2^7 8r_L} (p^\flat_L, V_L) \setminus \tilde{K}_L = (B_{2^7 8r_L} (p^\flat_L, V_L) \cap D_L) \setminus K_L\, .
        \]
        In particular $|B_{2^7 8r_L} (p^\flat_L, V_L) \setminus \tilde{K}_L|$ has the desired estimate.
    \end{itemize}
    Finally, observe the following. By the construction of \cite[Section 6.2]{DS4} we have a specific description of the set $\mathcal{K}$ consistsing of those points $p$ in the center manifold for which we know that the slice $\langle T, \mathbf{p}, p \rangle$ coincides with the slice of the multivalued approximation, namely $\sum_i \a{F_i (p)}$. First of all, $\mathcal{K}$ contains $\Phii (\mathbf{\Delta})$. Secondly, for every Whitney region $\cL$ corresponding to some square $L\in \sW$, $\mathcal{K}\cap \cL$ is defined in the following fashion. First of all, we denote by $\mathscr{D} (L)$ the family of squares $M\in \sW$ which have nonempty intersection with $L$ (i.e. its neighbors), hence we consider in each $\bC_M := \bC_{8r_M} (p_M, V_M)$, resp. $\bC_M:= \bC_{2^7 8 r_M} (p_M, V_M)$, the corresponding Lipschitz approximation $f_L$ and define
    \[
    \mathcal{K}\cap \cL := \bigcap_{M\in \mathscr{D} (L)} \mathbf{p} (\supp (T)\cap {\rm gr}\, (f_M))\, .
    \]
    Since for boundary cubes $\Gamma\cap \bC_M\subset \supp (T)\cap {\rm gr}\, (f_M)$, we conclude that $\Gamma\cap \cL \subset \mathcal{K}$. On the other hand every point of $\Gamma \cap \cM$ which does not belong to some Whitney region is necessarily contained in the contact set $\Phii (\mathbf{\Delta})$. Thus we conclude that $\Gamma \subset \mathcal{K}$. Observe, moreover, that by construction the map $N$ vanishes identically on the contact set, while we also know that for each $f_M$ as above $f_M$ coincides with the function $g_M$ on $\mathbf{p}_{V_M} (\Gamma)$. In particular this implies that $N$ vanishes identically on the intersection of $\Gamma$ with any Whitney region.
\end{proof}

\subsection{Proof of Proposition \ref{p:additional}} 
\eqref{e:compare-with-average-2} is an ovious consequence of \eqref{e:compare-with-average} since on the complement of the squares $L\in \sW^e$ the two functions $\phii$ and $f$ coincide.

We now turn to \eqref{e:compare-with-average}
Observe next that, by Proposition \ref{p:block_estimates}(f), it suffices to show the claim for the function $u_H$ in place of $\phii$. Observe also that we already know from the above argument that, if we replace $u_H$ with the tilted interpolating function $h_H$ and $f$ with the Lipschitz approximation $f_H = f_{HH}$, the estimate holds, as it is in fact just a special case of \eqref{e:L1_boundary} and \eqref{e:L1_interior}. Fix now a point $x\in H$ and the corresponding point let $y (x) := \bp_{V_H} (u_H (x))$ be the corresponding projection on the plane $V_H$. We can use \cite[(5.4)]{DS4} (where we identify the manifold $\cM$ in there with the affine plane $V_H + \phii (p)$) to compute
\[
|\etaa \circ f (x) - u_H (x)|\leq 
C |\etaa\circ f_H (y) - h_H (y)| + C |V_H - V_0| \Lip (f) \bh (T, \bB_H)\, .
\]
In particular we conclude
\[
|\etaa \circ f (x) - u_H (x)|\leq C |\etaa \circ f_H (y) - h_H (y)| + C \bmo^{1/2} \ell (H)^{1-\delta_1} \bmo^{\gamma_2} \ell (L)^{\gamma_2} \bmo^{1/4} \ell (L)^{1+\beta_1}\, .
\]
Observing that $x\mapsto y (x)$ is a Lipschitz function with Lipschitz constant bounded by $|D\phii|$, i.e. by $C \bmo^{1/2}$ and integrating in $x$, we easily conclude the claimed estimate.

\section{Local lower bounds for the Dirichlet energy and the $L^2$ norm of $N$}\label{s:stopping-scales}

As in \cite[Section 3]{DS4} the aim of this section is to conclude suitable lower bounds for $\int |DN|^2$ and $|N|$ over regions of the center manifold which are close (and sizable) enough to some Whitney region $\mathcal{L}$. Depending on the reason why the refinement was stopped, we will either bound $|N|$ from below in terms of $\ell (L)^{1+\beta_1}$ or we will bound $\int |DN|^2$ from below in terms of the excess of the current in $\bB_L$

\subsection{Lower bound on $|N|$} We start with the following conclusion.

\begin{proposition}[Separation because of the height]\label{p:separation_height}
If $L\in \sW^h$ then $L$ is necessarily an interior square. Moreover, there is constant
$\tilde C>0$ depending on $M_0$ such that whenever $(C_h)^4 \geq \tilde C C_e^\natural$ and $\eps_{CM}>0$ is small enough, then every $L \in \sW^h$ fulfills
    \begin{enumerate}
        \item[(S1)] $\Theta(T, p) \leq Q - \frac12$ for all $p \in \bB_{16r_L}(p_L)$,
        \item[(S2)] $L \cap H = \emptyset$ for all $H \in \sW^n$ with $\ell(H) \leq \frac12 \ell(L)$,
        \item[(S3)] $\cG(N(x), Q \a{\etaa \circ N(x)}) \geq \frac14 C_h \bmo^{\sfrac14} \ell(L)^{1+ \beta_1}$ for all $x \in \Phii(B_{2 \sqrt2 \ell(L)}(x_L))$.
    \end{enumerate}
\end{proposition}

\begin{proof}
    We only have to prove that $L \in \sC^\natural$ as the rest follows from the interior theory in \cite[Section 3]{DS4}. We show that any boundary square $H$ which did not stop because of the excess, also did not stop because of the height. Fix such an $H \in \sC^\flat \setminus \sW^e$. Then we know that its parent $M \in \sC^\flat \cap \sS$ satisfies
    \[ \bE(T, \bB_M^\flat) \leq C_e^\flat \bmo \ell(H)^{2-2\delta_1}  \]
    and we want to show that
    \[ \bh(T, \bB_H^\flat) \leq C_h \bmo^{\sfrac14} \ell(H)^{1+\beta_1} . \]
    To do so, we apply the height bound of Lemma \ref{l:hardtSimonHeight} to a suitable rotated current $\tilde T := O_\sharp T$, where $O$ is a rotation which maps $V_0$ onto $V_H$. Notice that the proof of this lemma is based on the first variation and thus on the minimality of $T$. As $\tilde T$ is area minimizing (with respect to the tilted boundary $O (\Gamma)$), we can directly deduce
    \begin{align*}
        \bh(T, \bB_H^\flat) \leq \bh(T, \bC_{2^7 64 r_H}(p_H^\flat, V_H))
        &\leq Cr_H \big(\bE(T,\bC_{2^7 80 r_H}(p_H^\flat, V_H) + \bA r_H \big)^{\sfrac 12}\\
        &\leq Cr_H \big(\bE(T,\bB_M^\flat) + C |V_M-V_H|^2 + \bA r_H \big)^{\sfrac 12}\\
        &\leq C \bmo^{\sfrac12} r_H^{\sfrac32}\\
        &\leq C_h \bmo^{\sfrac14} \ell(H)^{1+\beta_1},
    \end{align*}
    where we also used Proposition \ref{p:tiltingPlanes} and the sufficient small choice of $\eps_{CM}$.
\end{proof}

A simple corollary of the above proposition is that if a square stopped because of the neighbor condition, then this originated from a larger nearby square which stopped because of the excess.

\begin{corollary}\label{c:chain_of_influence}
    For every $H \in \sW^n$, there is a chain of squares $L_0, L_1,  \dots, L_j = H$ such that
    \begin{enumerate}
        \item[(a)] $L_i \in \sW^n$ for all $i>0$ and $L_0 \in \sW^e$,
        \item[(b)] they are all neighbors, i.e. $L_i \cap L_{i-1} \neq \emptyset$ and $\ell(L_i) = \frac12 \ell(L_{i-1})$.
    \end{enumerate}
    In particular, $H \subset B_{3\sqrt2 \ell(L_0)}(x_{L_0}, V_0)$.
\end{corollary}

Accordingly, we can collect all the squares $H$ which have such a chain relating $H$ to a specific square $L \in \sW^e$. The latter square is not necessarily unique, but it will be convenient to fix a consistent choice of $L$. 

\begin{definition}[Domains of influence]
    First, let us fix an ordering $\{J_i\}_{i \in \N}$ of $\sW^e$ such that the side length is non-increasing. For $J_0$, we define its domain of influence by
    \[ \sW^n(J_0) := \{ H \in \sW^n: \textnormal{there is a chain as in Corollary } \ref{c:chain_of_influence} \textnormal{ with } L_0=J_0 \textnormal{ and } L_j=H \}.
    \]
    Inductively, we define for $k>0$ the domain of influence $ \sW^n(J_k)$ of $J_k$ by all $H \in \sW^n  \setminus \bigcup_{i<k} \sW^n(J_i)$ which have a chain as in Corollary \ref{c:chain_of_influence} with $L_0=J_k$ and $L_j=H$. As it is easy to check using Corollary \ref{c:chain_of_influence} we have  $\sW^n=\mathring{\bigcup}_{k\in\N}\sW^n(J_k)$.
\end{definition}

\subsection{Lower bound on the Dirichlet energy} Having handled the case of ``height stopped'' squares we turn to squares which were stopped because they exceed the excess bound.

\begin{proposition}(Splitting)\label{p:splitting}
There are constants $C_1 (\delta_1), C_2 (M_0, \delta_1)$, $C_3 (M_0, \delta_1)$ such that, if $M_0 \geq C_1 ( \delta_1)$, $C^\natural_e \geq C_2 (M_0, \delta_1)$, $C^\flat_e \geq C_3 (M_0, \delta_1)$, if
the hypotheses of Theorem~\ref{t:normal-approx} hold and if $\eps_{CM}$ is chosen sufficiently small,
then the following holds. If $L\in \sW^e$, $q\in V_0$ with $\dist (L, q) \leq 4\sqrt{2} \,\ell (L)$, $B_{\ell (L)/4} (q, V_0)\subset D$ and $\Omega = \Phii (B_{\ell (L)/4} (q, V_0))$, then (with $C, C_4 = C(\beta_1, \delta_1, M_0, N_0, C^\natural_e, C^\flat_e, C_h)$):
\begin{align}
&C_e^\Box \bmo \ell(L)^{4-2\delta_1} \leq \ell (L)^2 \bE (T, \bB^\Box_L) \leq C \int_\Omega |DN|^2\, ,\label{e:split_1}\\
&\int_{\cL} |DN|^2 \leq C \ell (L)^2 \bE (T, \B^\Box_L) \leq C_4 \ell (L)^{-2} \int_\Omega |N|^2\, . \label{e:split_2}
\end{align}
\end{proposition}

Before coming to the proof of the Proposition, let us first observe an important point. Fix $L$ as in the statement of the Proposition and consider its parent $H$ and its ancestor $J$ $6$ generations before. If $L$ is a boundary square, then $H$ and $J$ are both boundary squares. On the other hand, if $L$ is an interior square, since $C_e^\natural$ is chosen much larger than $C_e^\flat$, we can ensure that both $L$ and $J$ are also interior squares. Indeed, when $\bB_L\subset \bB^\flat_J$ and $J\not\in \sW^e$, we have the obvious estimate
\[
\bE (T, \bB_L) \leq 2^{26} \bE (T, \bB^\flat_J)
\leq 2^{26} C^\flat_e \bmo \ell (J)^{2-2\delta_1}
\leq 2^{38} C^\flat_e \bmo \ell (L)^{2-2\delta_1}\, ,
\]
which therefore, by choosing $C^\natural_e \geq 2^{38} C^\flat_e$ implies that $L$ does not satisfy the excess stopping condition. 

Hence we can invoke \cite[Proposition 3.4]{DS4} to cover the case in which $L\in \sW^e \cap \sC^\natural$, since the proof given in \cite[Section 7.3]{DS4} just uses the fact that all squares $L, H$ and $J$ are interior squares (i.e. the repsective balls $\bB_L, \bB_H$, and $\bB_J$ do not intersect the boundary $\Gamma$). We are thus left to handle the case in which $L$ (and therefore also $H$ and $J$) are boundary squares.   

To do so, we need analogues of three lemmas from \cite{DS4}.
\begin{lemma}\label{l:computation_dirichlet}
    Let $B^+ \subset \R^2$ be a half ball centered at the origin and $w \in W^{1,2}(B^+, \cA_Q(\R^n))$ be Dir-minimizing with $w= Q \a{0}$ on $B^+ \cap (\R \times \{0\})$. Denoting $\bar w := w \oplus (-\etaa \circ w) = \sum_i \a{w_i - \etaa \circ w} $ and $u:= \etaa \circ w$, we have
    \begin{align*}
        Q \int_{B^+} | Du - Du(0)|^2 = \int_{B^+} \cG(Dw, Q \a{Du(0)})^2 - \D(\bar w, B^+).
    \end{align*}
\end{lemma}

\begin{proof}
    We extend $w$ in an odd way to all of the ball $B$. Notice that then also the extension of $u$ is harmonic in all of $B$. Now
   the proof is the same as in \cite[Lemma 7.3]{DS4}, but we repeat it here anyway. First notice, that $u$ is a classical harmonic function and in particular, fulfills the mean value property. We use it to deduce
    \begin{equation}\label{e:harmonic_deviation}
    \begin{split}
        Q \int_B |Du-Du(0)|^2 &= Q \int_B \left( |Du|^2 + |Du(0)|^2 - 2 Du \cdot Du(0) \right)\\
        &= Q \int_B |Du|^2 +Q|B||Du(0)|^2 -2Q \left( \int_B Du \right) \cdot Du(0)\\
        &= Q \int_B |Du|^2 -Q|B||Du(0)|^2.
    \end{split}
    \end{equation}
    Similarly we compute
    \begin{equation}\label{e:dir_energy_split}
    \begin{split}
        Q \int_B |Dw|^2 &= \sum_i \int_B |Dw_i|^2
        = \sum_i \int_B \left(|Dw_i - Du(0)|^2 -|Du(0)|^2 + 2Dw_i \cdot Du(0) \right)\\
        &= \int_B \cG(Dw, Q \a{Du(0)})^2 -Q |B||Du(0)|^2 + 2 Q \left( \int_B  \frac 1Q\sum_i Dw_i \right) \cdot Du(0)\\
        &=\int_B \cG(Dw, Q \a{Du(0)})^2 + Q |B||Du(0)|^2 .
    \end{split}
    \end{equation}
    Last we split the Dirichlet energy of $w$ into the average and the average-free part (as already observed in \eqref{e:DirEnergyForMeanTranslated}).
    \begin{equation}\label{e:dir_energy_average}
    \begin{split}
        \int_B |D\bar w|^2 &= \sum_i \int_B |Dw_i - Du|^2
        = \sum_i \int_B \left( |Dw_i|^2 + |Du|^2 -2 Dw_i \cdot Du \right)\\
        &= \int_B |Dw|^2 + Q \int_B |Du|^2 -2Q \int_B \left( \frac 1Q Dw_i \right)\cdot Du\\
        &=\int_B |Dw|^2 - Q \int_B |Du|^2.
    \end{split}
    \end{equation}
    The three identities \eqref{e:harmonic_deviation}, \eqref{e:dir_energy_split}, \eqref{e:dir_energy_average} and dividing everything by $2$ conclude the lemma.
\end{proof}

An other important ingredient is the unique continuation for Dir-minimizers (compare to \cite[Lemma 7.1]{DS4}).

\begin{lemma}[Unique Continuation for Dir-minimizers]\label{l:unique_continuation}
    For every $0< \eta <1$ and $c>0$, there is a $\delta >0$ such that whenever $B_{2r}^+ \subset V_0$ is the half ball and $w: B_{2r}^+ \to \cA_Q(\R^n)$ is Dir-minimizing with $w= Q \a{0}$ on $B_{2r}^+ \cap (\R \times \{0\})$, $\D(w, B^+_{2r})=1$, and $\D(w, B_r^+) \geq c$, then
    \[ \D(w, B_s(q)) \geq \delta \qquad \textnormal{ for every } B_s(q) \subset B^+_{2r} \textnormal{ with } s \geq \eta r.
    \]
\end{lemma}

\begin{proof}
    The qualitative statement (UC) of the proof of \cite[Lemma 7.1]{DS4} applies directly to our situation while the quantitative statement follows from a blow-up argument that goes analogously for us as $ B_s(q) \subset B^+_{2r}$.
\end{proof}

The previous two lemmas imply the following energy decay for Dir-minimizers (compare to \cite[Proposition 7.2]{DS4}) which itself implies the Proposition \ref{p:splitting}. First fix a number $\lambda >0$ such that
\[ (1+\lambda)^4 < 2^{\delta_1}. \]

\begin{proposition}[Decay estimate for Dir-minimizers]
    For any $\eta >0$ there is a $\delta>0$ such that whenever $B_{2r}^+ \subset V_0$ is the half ball and $w: B_{2r}^+ \to \cA_Q(\R^n)$ is Dir-minimizing with $w= Q \a{0}$ on $B_{2r}^+ \cap (\R \times \{0\})$ and satisfies
    \begin{align*}
        \int_{B^+_{(1+\lambda)r}} \cG(Dw, Q \a{D(\etaa \circ w)(0)})^2
        \geq 2^{\delta_1-4} \D(w, B_{2r}^+),
    \end{align*}
    then we have for any $B_s(q) \subset B_{2r}^+$ with $s \geq \eta r$
    \begin{align*}
        \delta \ \D \left(w, B^+_{(1+\lambda)r} \right)
        \leq \D \left(\bar w, B^+_{(1+\lambda)r} \right)
        \leq \frac{1}{\delta r^2} \int_{B_s(q)} |\bar w|^2.
    \end{align*}
    Here we used again the notation $\bar w := w \oplus (-\etaa \circ w) = \sum_i \a{w_i - \etaa \circ w} $.
\end{proposition}
\begin{proof}
    We follow word by word the proof of \cite[Proposition 7.2]{DS4} using Lemma \ref{l:unique_continuation} and Lemma \ref{l:computation_dirichlet} instead of \cite[Lemma 7.1]{DS4} and \cite[Lemma 7.3]{DS4}. We reach the contradicting inequality
    \[ \int_{B^+_{1+\lambda}} |Du-Du(0)|^2 \geq 2^{\delta_1 -4} \int_{B^+_2} |Du|^2
    \]
    which is false as one can see by reflecting such that $u$ stays harmonic and then using the classical decay for harmonic functions.
\end{proof}

%% file: blow-up-1.tex
\section{Frequency function and monotonicity}\label{s:blow-up}

In this section we take a further crucial step towards the proof of Theorem \ref{t:flat-points}. We recall our key Assumption \ref{a:main-local-2} and we add a further one on the smallness of the excess. Before doing that, we observe a corollary of the decay estimate in Theorem \ref{t:decay}.

\begin{corollary}\label{c:cylindrical_excess_decay}
Let $T$ and $\Gamma$ be as in Assumption \ref{a:main-local} and assume that $0\in \Gamma$ is a flat point and that $Q\a{V}$ is the unique tangent cone to $T$ at $0$. Then there is a geometric constant $\kappa>0$ and constants $C$ and $r_0>0$ (depending on $\Gamma$ and $T$) such that
\begin{equation}\label{e:cylindrical_excess_decay}
\bE (T, \bC_{r}) \leq C r^{4\kappa} \qquad \forall r\leq r_0\, .
\end{equation}
\end{corollary}

Thus, upon rescaling the current appropriately, if $0$ is a flat point we can assume, without loss of generality, the following.

\begin{ipotesi}\label{a:main-local-3}
Let $T$ and $\Gamma$ be as in Assumption \ref{a:main-local}.
$0\in \Gamma$ is a flat point, $Q\a{V}$ is the unique tangent cone to $T$ at $0$, we let $n$ be as in \eqref{e:new-n} and assume that \eqref{e:new-wedge} holds.
In addition we assume to have fixed a choice of the parameters so that Theorem \ref{t:cm} and Theorem \ref{t:normal-approx} hold and that
\begin{equation}\label{e:CM-applies}
\bE (T, \bC_{4R_0 \rho}) + \bA^2 \rho^2 \leq \varepsilon_{CM} \rho^{2\kappa}\, \qquad \forall \rho \leq 1\, .
\end{equation}
\end{ipotesi}
Observe that, by \eqref{e:CM-applies}, we conclude that both Theorem \ref{t:cm} and Theorem \ref{t:normal-approx} can be applied to the current $T_{0,\rho}$ whenever $\rho\leq 1$.

\subsection{Intervals of flattening}\label{s:flattening} We start defining a decreasing set of radii $\{t_1> t_2 > \ldots \}\subset (0,1]$, which at the moment can be both finite and infinite: in the first case one $t_N$ will be equal to $0$, while in the second case all $t_k$'s are positive and $t_k\downarrow 0$. 

$t_1$ is defined to be equal to $2$. We then let $\bar{\cM}_1 = \cM_1$ be the center manifold and $\bar N_1 = N_1$ the corresponding normal approximation which results after we apply Theorem \ref{t:cm} and Theorem \ref{t:normal-approx} to the current $T$. Moreover we let $\sW^{(1)}$ be the squares of the Whitney decomposition described in Definition \ref{d:refining_procedure}. We then distinguish two cases:
\begin{itemize}
    \item[(Stop)] There is a square $H\in \sW^{(1)}$ such that 
    \begin{equation}\label{e:large}
    \dist (0, H) \leq 64 \sqrt{2} \ell (H)\, .
    \end{equation}
    \item[(Go)] There is no such square.
\end{itemize}
Notice that every such square $H$ satisfying \eqref{e:large} is a boundary square. In the first case we select an $H$ as in (Stop) which has maximal sidelength and we define $\bar t_2 := 66 \sqrt{2} \ell (H)$ and $t_2 := t_1 \bar t_2 = 132 \sqrt 2 \ell(H)$. Otherwise we define $t_2 = 0$. Observe that 
\begin{equation}\label{e:t2-over-t1}
\frac{t_2}{t_1} \leq 66\sqrt{2} 2^{-N_0}
\end{equation}

Before proceeding further, we record an important consequence of the Whitney decomposition:

\begin{corollary}\label{c:stop=excess}
If (Stop) holds, then the square $H$ of maximal sidelength that satisfies \eqref{e:large} must be an element of $\sW^e$, i.e. it violates the excess condition
\end{corollary} 
\begin{proof}
Observe that if $H$ is an (NN) square, then there is a neighboring square $H'$ of double sidelength which also belongs to $\sW$ and it is easy to see that the latter satisfies \eqref{e:large} too, violating the maximilaity of $H$. Note next that \eqref{e:large} implies that $H$ is a boundary square, and as such it cannot belong to $\sW^h$. 
\end{proof}

In case $t_2>0$ we then apply Theorem \ref{t:cm} and Theorem \ref{t:normal-approx} to $T_{0, t_2}$ and let $\bar{\cM}_2$ and $\bar{N}_2$ be the corresponding objects. The pair $(\cM_2, N_2)$ will be derived by scaling back the objects at scale $t_2$, namely
\begin{align}
\cM_2 &= \left\{t_2 q : q \in \bar{\cM}_2\right\},\\
N_2 (q) &= t_2 \bar N_2 \left(\frac{q}{t_2}\right)\, .
\end{align}
We then apply the procedure above to $\bar\cM_2$ in place of $\bar\cM_1$ and determine $\bar t_3$ analogously, while we set $t_3 := t_2 \bar t_3$.

We proceed inductively and define $\bar{\cM}_k, \cM_k, \bar{N}_k, N_k$, $\bar t_k$, and $t_k:= t_{k-1}\bar t_k$: the procedure stops when one $t_k$ equals $0$, otherwise goes indefinitely. Observe that for every $k$ we have the estimate
\begin{equation}\label{e:t_k-over-t_k-1}
\frac{t_k}{t_{k-1}} \leq 66 \sqrt{2} 2^{-N_0}.    
\end{equation}

\subsection{Frequency function} Observe that the conclusion of Theorem \ref{t:flat-points} is equivalent to $T$ coinciding with $Q\a{\cM_k}$ for some $k$ in a neighborhood of the origin. A simple corollary of the interior regularity is in fact the following

\begin{corollary}
If $N_k\equiv Q \a{0}$ on some nontrivial open subset of $\cM_k$, then $T = Q \a{\cM_k}$ in a neighborhood of $0$ and in particular Theorem \ref{t:flat-points} holds.
\end{corollary}

We next consider a function $d$ which is $C^2$ in the punctured ball $\bB_1 (0)$, whose gradient $\nabla d$ is tangent to $\Gamma$ and such that (i)-(ii)-(iii) of Definition \ref{Def:distance} hold. Likewise we fix the function $\phi: [0, \infty) \to [0, \infty)$ given by
\[
\phi (t) := \left\{\begin{array}{ll}
1, & \mbox{if $t\in [0, \frac{1}{2}]$},\\
(1-2t), & \mbox{if $t\in [\frac{1}{2}, 1]$},\\
0, & \mbox{if $t\geq 1$}\, .
\end{array}\right.
\]
From now on we denote by $D$ the classical Euclidean differentiation of functions, tensors, and vector fields, which for objects defined on the manifold $\cM_k$ will mean that we compute derivatives along the tangents to the manifold. On the other hand we use the notation $\nabla_{\cM_k}$, $D^{\cM_k}$, and ${\rm div}_{\cM_k}$, respectively for the gradient, Levi-Civita connection, and divergence of (respectively), functions, tensors, and vector fields on $\cM_k$ understood as a Riemannian submanifold of the Euclidean space $\mathbb R^{2+n}$.

We then define
\begin{align}
D (r) &:= \int_{\cM_k} \phi \left(\frac{d (x)}{r}\right) |D N_k|^2 (x)\, d\cH^2 (x),
\qquad \mbox{if $r\in (t_{k+1}, t_k]$},\\
H (r) &:= - \int_{\cM_k} \phi' \left(\frac{d (x)}{r}\right) |\nabla_{\cM_k} d (x)|^2 \frac{|N_k (x)|^2}{d (x)}\, d \cH^2 (x),\qquad \mbox{if $r\in (t_{k+1}, t_k]$}.\\
S(r) &:= \int_{\cM_k} \phi \left(\frac{d (x)}{r}\right) |N_k (x)|^2\, d\cH^2 (x).
\end{align}
We are then ready to state our main estimate.

\begin{theorem}\label{t:frequency}
Let $T$ be as in Assumption \ref{a:main-local-3}. Either $T = Q \a{\cM_k}$ in a neighborhood of the origin for some $k$ (and in that case note that $t_{k+1}=0$), or else $H(r)>0$ and $D(r)>0$ for every $r$. In the latter case the function $I (r) := \frac{r D(r)}{H(r)}$ satisfies the following properties for some constants $C$ and $\tau>0$:
\begin{itemize}
    \item[(a)] For all $r>0$, we have
\begin{align}
I (r) &\geq C^{-1},\, \label{e:bound_from_below}
\end{align}
and
\begin{equation}\label{e:bloody-bound-on-the-D}
D(r) \leq C r^{2+\tau}\,.
\end{equation}
 \item[(b)] $I$ is continuous and differentiable on each open interval $(t_{k+1}, t_k)$ and moreover 
\begin{align}
\frac{d}{dr} \left(\log I (r) + C D(r)^\tau - C t_k^{2\tau-2} \frac{S(r)}{D(r)}\right) & \geq - C r^{\tau-1}
\qquad \mbox{for a.e. $r\in ]t_{k+1}, t_k[$.}\label{e:Gronwall}
\end{align}
\item[(c)]
At each $t_k$ the function $I$ has one-sided limits
 \begin{align*}
I (t_k^+) &= \lim_{t\downarrow t_k} I (t),\\
I (t_k^-) &= \lim_{t\uparrow t_k} I (t),\, 
 \end{align*}
 and moreover
 \begin{align}
\sum_k |I (t_k^+)- I (t_k^-)| < \infty. \label{e:jump_estimate}
\end{align}
\end{itemize}
\end{theorem}

We will prove (a) and (b) in Section \ref{s:proof_monotonicity_frequency} while we devote Section \ref{s:proof_jump_estimate} to show (c). An obvious corollary of Theorem \ref{t:frequency} is the following

\begin{corollary}\label{c:frequency}
Let $T$ be as in Assumption \ref{a:main-local-3}. Either $0$ is a regular point, or else $I(r)$ is well defined for every $r$ and the limit
\[
I_0 := \lim_{r\downarrow 0} I (r)
\]
exists, is finite and positive.
\end{corollary}
\begin{proof}
First of all observe that, since $I (r) \geq C^{-1}$, 
\[
f (r) := \log I (r) - C t_k^{2\tau -2} \frac{S(r)}{D(r)} + C D(r)^\tau  + C r^\tau\geq - \log C\, .
\]
We will also see below in Lemma \ref{l:Poincare} that $S(r) \leq C r^2 D(r)$. Hence, since the Lipschitz constant of $\log$ is bounded on $[C^{-1}, \infty[$, we infer
\begin{equation}\label{e:jump-bound-f}
|f (t_j^+)-f (t_j^-)|\leq C |I (t_j^+)-I (t_j^-)| + C (t_j^+)^{\tau}.
\end{equation}
Next we show that the two bounds \eqref{e:jump_estimate} and \eqref{e:Gronwall} imply that $f$ is bounded from above: considering $\rho\in ]0,1[$, we let $k$ the largest number such that $\rho < t_k$ and we can estimate
\begin{align*}
f (1) - f (\rho) &= \int_\rho^{t_k} f' + \sum_{j=1}^{k-1} \int_{t_{j+1}}^{t_j} f'
+ \sum_{j=2}^{k} (f (t_j^+) - f (t_j^-))
\end{align*}
which turns into
\begin{align*}
f (\rho) &\leq f (1) -\int_\rho^{t_k} f' - \sum_{j=1}^{k-1} \int_{t_{j+1}}^{t_j} f' - \sum_j |f (t_j^+)-f (t_j^-)|\\
&\leq f (1) + C \int_0^1 r^{\tau-1}\, dr + C \sum_j  |I (t_j^+)-I (t_j^-)| < \infty\, 
\end{align*}
(note that in the last line we have used \eqref{e:jump-bound-f}).

Next observe that the distributional derivative of $f$ consists of a nonnegative measure (on the union of the open intervals $(t_{k+1}, t_k)$ and a purely atomic Radon measure which has finite mass by \eqref{e:jump_estimate}. We thus conclude that the distributional derivative of $f$ is a Radon measure. Next fix any $\rho \leq 1$ and let $t_k$ be such that $2 t_{k+1} < \rho < 2 t_k$. We then have the bound
\begin{align*}
|Df| (]\rho,1[) &\leq Df (\rho, t_k^-) + \sum_{1\leq j\leq k-1} Df (]t_{j+1}^+, t_j^-[) + \sum_{2\leq j\leq k} |f (t_j^+)-f (t_j^-)|\\
&\leq 2  \sum_{j=1}^\infty |f (t_j^+)-f (t_j^-)| + \|f\|_\infty < \infty\, .
\end{align*}
Hence, letting $\rho$ go to $0$ we discover that $|Df| (]0,1[) < \infty$, that is $f\in BV (]0,1[)$. This in turn implies that $f$ is a function of bounded variation and hence that $\lim_{r\downarrow 0} f(r)$ exists and is finite. Observe, moreover that by \eqref{e:LimitOfS/D} we infer that
$f(r) - \log (I(r))$ converges to $0$ as $r\downarrow 0$. We thus conclude that 
\[
\lim_{r\downarrow 0} e^{f (r)} = \lim_{r\downarrow 0} I (r)
\]
exists, it is finite, and it is positive. 
\end{proof}
 
\section{Proof of Theorem \ref{t:frequency}: Part I}\label{s:proof_monotonicity_frequency}

\subsection{Proof of \eqref{e:bound_from_below}} The claim is simply equivalent to the existence of a constant $C$ such that $H (r)\leq C r D(r)$. The latter is a consequence of a Poincar\'e-type inequality which uses the fact that $N_k$ vanishes identically on the boundary curve $\Gamma$. The proof will be reduced to \cite[Proposition 9.4]{DDHM}. However, in order to make the latter reduction, we employ a device which will be used in several subsequent computations. Having fixed a positive $r$ different from any $t_j$ we let $k$ be such that
$t_{k+1} < r < t_k$ and we define the corresponding rescaled quantities $\bar {D}_k (t_k^{-1} r), \bar H_k (t_k^{-1} r), \bar S_k (t_k^{-1} r)$, and $\bar I_k (t_k^{-1} r)$. More precisely we define the function $d_k (x) := t_k^{-1} d (t_k x)$ and set 
\begin{align}
\bar D_k (\rho) &:= \int_{\bar \cM_k} \phi \left(\frac{d_k (x)}{\rho}\right) |D \bar N_k|^2 (x)\, d\cH^2 (x) \,,\\
\bar H_k (\rho) &:= - \int_{\bar \cM_k} \phi' \left(\frac{d_k (x)}{\rho}\right) |\nabla_{\bar \cM_k} d_k (x)|^2 \frac{|\bar N_k (x)|^2}{d_k (x)}\, d \cH^2 (x)\, ,\\
\bar S_k (\rho) &:= \int_{\bar \cM_k} \phi \left(\frac{d_k (x)}{\rho}\right) |\bar N_k (x)|^2\, d\cH^2 (x).
\end{align}
We then can immediately check the relations
\begin{align}
\bar D_k (t_k^{-1} r) &= t_k^{-2} D (r)\label{e:back-and-forth-1}\, ,\\
\bar H_k (t_k^{-1} r) &= t_k^{-3} H (r)\label{e:back-and-forth-2}\, ,\\
\bar S_k (t_k^{-1} r) &= t_k^{-4} S (r) \label{e:back-and-forth-3}\, ,\\
\bar S_k' (t_k^{-1} r) &= t_k^{-3} S' (r) \label{e:back-and-forth-4}\,,\\
\bar D_k' (t_k^{-1} r) &= t_k^{-1} D' (r)\, .\label{e:back-and-forth-5}
\end{align}

\begin{lemma}\label{l:Poincare}
There is a constant $C$ such that
\begin{align}
H(r) & \leq C r D(r)\,,\\
S' (r) & \leq C r D(r)\,,\\
S (r) & \leq C r^2 D(r)\label{e:LimitOfS/D}\, .
\end{align}
\end{lemma}
\begin{proof}
We observe that the corresponding inequalities for $\bar D_k$, $\bar H_k$, $\bar S_k$, and $\bar S_k'$ 
follow from \cite[Proposition 9.4]{DDHM}, since the center manifold $\bar\cM_k$, the functions $d_k$, and $N_k$ satisfy the assumptions of the Proposition. 
\end{proof}

\subsection{Derivatives of $H$ and $D$} In order to prove \eqref{e:Gronwall} the first step consists in computing the derivatives of $H$ and $D$. In what follows we will use the usual convention of denoting by $O (g)$ any function $f$ of the real variable $r>0$ with the property that $|f(r)|\leq C g(r)$. Moreover, in order to avoid cumbersome notation, for $r\in (t_{k+1}, t_k]$ we will drop the subscript $\mathcal{M}_k$ from the gradient $\nabla_{\mathcal{M}_k}$ on the manifold.

\begin{proposition}\label{p:D'-and-H'}
Under the assumptions of Theorem \ref{t:frequency} we have, for every $r\in (t_{k+1}, t_k]$,
\begin{align}
D'(r) &= - \int \phi' \left(\frac{d(x)}{r}\right) \frac{d(x)}{r^2} |DN|^2,\\
H'(r) &= \left(\frac{1}{r} + O (1)\right) H(r) + 2 E(r),\label{e:derivative-of-H-bar}
\end{align}
and 
\begin{equation}\label{e:E(r)}
E(r) = - \frac{1}{r} \int \phi' \left(\frac{d(x)}{r}\right) \sum_i N_i (x) \cdot (DN_i (x)\nabla d (x))\, .
\end{equation}
\end{proposition}
\begin{proof} The first derivative is a straightforward computation. For the second, we can follow the computations of \cite[Proof of Proposition 9.5]{DDHM} to conclude that
\[
H' (r) = 2 E(r) - \frac{1}{r} \int \phi' \left(\frac{d(x)}{r}\right) \Delta_{\cM_k} d (x) |N|^2 (x)\, ,
\]
where $\Delta_{\mathcal{M}_k}$ is the Laplace-Beltrami operator on the manifold $\mathcal{M}_k$. Noticing that $\phi' \left(\frac{d(x)}{r}\right)$ vanishes unless $C^{-1} r \leq |x| \leq C r$, our claim will follow once we show that
\[
\Delta_{\mathcal{M}_k} d (x) = \frac1{d(x)} + O (1) = \frac{1}{d(x)} (1+ O (d (x)))\, . 
\]
In order to show the latter estimate, we fix a point $x\in \mathcal{M}_k$ and observe first that the second fundamental form of the center manifold $\bar{\mathcal{M}}_k$ is bounded by $C (\bE (T_{0, t_k}, 4 R_0)^{1/2} + \bA t_k)$, which in turn is bounded by $C t_k^{\kappa}$ for some positive $\kappa$. By rescaling, the second fundamental form $A_{\cM_k}$ of $\mathcal{M}_k$ enjoys the bound $\|A_{\cM_k}\|_\infty \leq C t_k^{\kappa-1}$. On the other hand, recalling that $|x|^{-1} |D d- D |x|| + |D^2 d - D^2 |x|| \leq C $ it is easy to see that 
\begin{align*}
\left|\Delta_{\mathcal{M}_k} d(x) -\frac1{d(x)}\right| 
&\leq \left|\Delta_{\mathcal{M}_k} |x| - \frac{1}{|x|}\right|  + C + C |x| \|A_{\cM_k}\|_\infty \\
&\leq C + C \|A_{\cM_k}\|_\infty \leq C t_k^\kappa + C \leq C\, .\qedhere
\end{align*}
\end{proof}

\subsection{First variations and approximate identities} We start by recalling that, since $T_{0, t_k}$ is area-minimizing and $\partial T_{0, 2 t_k} \res \bC_{4R_0} = Q \a{\Gamma_k}\res \bC_{4R_0}$, then $\delta T_{0, t_k} (X)=0$ for every $X$ which is tangent to $\Gamma$. In what follows we fix a $C^3$  extension $\tilde\phii_k$ of the function $\bar \phii_k$ to $[-4,4]^2\subset V$ (by increasing the $C^{3,\omega}$ estimate on $\phii_k$ by a constant factor) whose graph is the center manifold $\bar\cM_k$ and we denote by $\mathbf{p}_k$ the orthogonal projection onto the graph of $\tilde\phii_k$ (which is of course defined only in a suitable normal neighborhood of it). We then fix the two relevant vector fields with which we will test the stationarity condition:
\begin{align*}
X_o (p) &:= \phi \left( \frac{d_k (\mathbf{p}_k (p))}{r}\right) (p-\mathbf{p}_k (p)),\\
X_i (p) &:= -Y(\bp_k (p)) := - \frac{1}{2} \phi \left(\frac{d_k (\mathbf{p}_k (p)))}{r}\right) \frac{\nabla d_k^2}{|\nabla d_k|^2} (\mathbf{p}_k (p))
\end{align*}
(note that $\nabla$ means the gradient $\nabla_{\bar \cM_k}$ here).

Note that $X_i$ is tangent to both $\bar \cM_k$ and $\Gamma_k$. Moreover, in \cite[Sections 9.4 and 9.5]{DDHM}, the estimates are done separately on both sides of $\Gamma_k$. Thus, it applies to our situation directly with $\cM^+ = \bar \cM_k$. Note also that the fifth error terms vanish for us as our "ambient manifold" is $\R^{n+2}$. We summarize the statements here
and first define the following function
\[ \varphi_k (p):= \phi \left( \frac{d_k (\bp_k (p))}{r} \right).
\]
We also introduce the rescaled quantity 
\[
\bar E_k (\rho) := - \frac{1}{\rho} \int_{\bar \cM_k} \phi' \left(\frac{d_k (x)}{\rho}\right) \sum_i (\bar N_k)_i (x) \cdot (D(\bar N_k)_i (x)\nabla d_k (x))
\]
and record the corresponding relation with $E$, namely
\begin{equation}\label{e:back-and-forth-7}
\bar E_k (t_k^{-1} r) = t_k^{-2} E (r)\, .
\end{equation}
\begin{proposition}[Outer variations]\label{p:outer_variations}
    Let $A_k$ and $H_{\bar \cM_k}$ denote the second fundamental form and the man curvature of $\bar \cM_k$ respectively. Assume $\frac{t_{k+1}}{t_k} < r < 1$. Then we have
    \begin{align}
        |\bar D_k (r) - \bar E_k (r)| = \left| \int_{\bar \cM_k} \left(\varphi_k |D \bar N_k|^2 + \sum_i ((\bar N_k)_i \otimes D\varphi_k):D (\bar N_k)_i \right)  \right|
        \leq \sum_{j=1}^4 |\textnormal{Err}_j^o|,\label{e:outer-with-bar}
    \end{align}
    with
    \begin{align*}
        \textnormal{Err}_1^o &:=-Q\int_{\bar \cM_k} \varphi \langle H_{\bar \cM_k}, \etaa \circ \bar N_k \rangle,\\
        |\textnormal{Err}_2^o| &\leq C \int_{\bar \cM_k} |\varphi_k||A_k|^2|\bar N_k|^2,\\
        |\textnormal{Err}_3^o| &\leq C \int_\cM \left( |\varphi_k|(|D \bar N_k |^2|\bar N_k ||A_k|+|D\bar N_k|^4)+ |D\varphi_k|(|D \bar N_k |^3|\bar N_k| + |D \bar N_k||\bar N_k |^2|A_k|)\right),\\
        \textnormal{Err}_4^o &:= \delta \bT_{\bar F_k} (X_o) - \delta T_{0, t_k} (X_o) = \delta \bT_{\bar F_k} (X_o).
    \end{align*}
\end{proposition}

For the inner variation, we introduce first a bit more of notation. First of all, we see $D(\bar N_k)_j$ as a map from $T\bar \cM_k$ to $\R^{n+2}$. Denoting the components of $(\bar N_k)_j$ by $(\bar N_k)_j=((\bar N_k)_j^1, \dots, (\bar N_k)_j^{n+2})$ and choosing a vector field $Z$ tangent to $\bar \cM_k$, we write
\[ D (\bar N_k)_j(Z) = (D_Z (\bar N_k)_j^1, \dots, D_Z (\bar N_k)_j^{n+2}). \]
Similarly, we have
\[ D (\bar N_k)_j D^{\bar \cM_k} Y(Z) = D(\bar N_k)_j (D^{\bar \cM_k} Y(Z)) = (D_{D^{\bar \cM_k} Y(Z)} (\bar N_k)_j^1, \dots, D_{D^{\bar \cM_k} Y(Z)} (\bar N_k)_j^{n+2} ). \] 
Thus, for the scalar product $D (\bar N_k)_j:D(\bar N_k)_j D^{\bar \cM_k} Y$, we choose an orthonormal frame $e_1, e_2$ of $T\bar \cM_k$ and express
\[ D(\bar N_k)_j:D (\bar N_k)_j D^{\bar \cM_k} Y = \sum_\ell \langle D_{e_\ell} (\bar N_k)_j, D_{D^{\bar \cM_k} Y(e_\ell)} (\bar N_k)_j \rangle
= \sum_{\ell, i}  D_{e_\ell} (\bar N_k)_j^i D_{D^{\bar \cM_k} Y(e_\ell)} (\bar N_k)_j^i . \]
We further introduce the quantity
\[
G(r) := -r^{-2} \int_{\cM_k} \phi \left( \frac dr \right) \frac{d}{|\nabla d|^2} \sum_j |D(N_k)_j \cdot \nabla d|^2
\]
and its correspoding rescaled version
\[
\bar G_k (\rho) = -\rho^{-2} \int_{\bar \cM_k} \phi \left( \frac{d_k}{\rho} \right) \frac{d_k}{|\nabla d_k|^2} \sum_j |D(\bar N_k)_j \cdot \nabla d_k|^2\, ,
\]
while we record the corresponding relation as in \eqref{e:back-and-forth-1}-\eqref{e:back-and-forth-5}:
\begin{equation}\label{e:back-and-forth-6}
\bar G_k (t_k^{-1} r) = t_k^{-1} G (r)\, .
\end{equation}

\begin{proposition}[Inner variations]\label{p:inner_variations}
Under the above assumptions we have
    \begin{align}
        &\left| \bar D_k'(r) - O(t_k^\kappa)\bar D_k (r) - 2\bar G_k(r) \right| \nonumber\\
        &= \frac 2r \left| \int_{\bar \cM_k} \left( \sum_j D (\bar N_k)_j:D (\bar N_k)_j D^{\bar\cM_k} Y  - \frac12 |D \bar N_k|^2 {\rm div}_{\bar \cM_k} Y \right) \right|
        \leq \frac 2r \sum_{j=1}^4 |\textnormal{Err}_j^i|,\label{e:inner-with-bar}
    \end{align}
    with
    \begin{align*}
        \textnormal{Err}_1^i &:= Q\int_{\bar \cM_k} \left( \langle H_{\bar \cM_k}, \etaa \circ \bar N_k \rangle \dv_{\bar \cM_k} Y + \langle D_Y H_{\bar \cM_k}, \etaa \circ \bar N_k \rangle \right),\\
        |\textnormal{Err}_2^i| &\leq C \int_{\bar \cM_k} |A_k|^2 \left(|DY||\bar N_k|^2 + |Y||\bar N_k||D \bar N_k| \right),\\
        |\textnormal{Err}_3^i| &\leq C \int_{\bar \cM_k} \left(|D\bar N_k|^2|Y||A_k|(|\bar N_k| + |D \bar N_k|) + |DY| (|A||\bar N_k|^2|D \bar N_k|+|D \bar N_k|^4) \right)\\
        \textnormal{Err}_4^i &:= \delta \bT_{\bar F_k} (X_i) - \delta T_{0, t_k} (X_i) = \delta \bT_{\bar F_k} (X_i).
    \end{align*}
\end{proposition}
\begin{proof}
The arguments for the proposition are the same as in \cite[Proposition 9.10]{DDHM} and indeed they are based on the Taylor expansions of \cite[Theorems 4.2 \& 4.3]{DS2}. However some more care is required because the term $O (t_k^{\kappa}) D(r)$ appears in the corresponding inequality (namely \cite[(9.28)]{DDHM} as $O(1) D(r)$. The reason for the improvement is based on the computations
\cite[(9.29)]{DDHM} and \cite[Lemma 9.2]{DDHM}: the improvement follows easily from the fact that:
\begin{itemize}
    \item The curvature of the rescaled boundary $\Gamma_k$ is bounded by $t_k$;
    \item The $C^{3}$ norm of the function
$\bar\phii_k$ (whose graph is the center manifold $\bar \cM_k$) is bounded by $(\bE (T_{0, t_k}, \bC_{4 R_0}) + \|\psi_k\|_{C^{3,\alpha_0}})^{1/2}$,
where $\psi_k$ is the function whose graph describes $\Gamma_k$; we thus have $\|\bar\phii_k\|_{C^3}\leq C t_k^\tau$.
\end{itemize}
\end{proof}

\subsection{Families of subregions for estimating the error terms}\label{s:subregions}
We want to estimate the error terms over the Whitney regions in order to use the separation estimate (Proposition \ref{p:separation_height}) and the splitting before tilting estimates (Proposition \ref{p:splitting}). To achieve this goal we goes along the same lines of \cite[Section 9.6]{DDHM} and apply the arguments of \cite[Section 9.6]{DDHM} to the current $T_{0,t_k}$ that gives rise to the center manifold $\bar{\mathscr{M}}_k$ . Notice that in each error term, there is the cut-off $\phi(d_k/r)$, thus it is enough to consider squares which intersect $\mathscr{B}^+_r := \{x \in V_0 \cap D: d_k (\bar \phii_k(x)) <r\}$. However, to sum the estimates over all squares, we prefer the regions over which we integrate to be disjoint. For this purpose, we define a Besicovitch-type covering.

From now on we fix all the constants from Assumption \ref{ass:hierarchy} and treat them as geometric constants. We are going to consider the Whitney decomposition and the corresponding family $\sW^e, \sW^h, \sW^n$ of squares whose definition is detailed in Section \ref{s:center-manifold}. Note that the construction is not applied to the current $T$ and the boundary $\Gamma$, but rather to the rescaled current $T_{0, t_k}$ and the rescaled boundary $\Gamma_k$. Note that the assumptions for the construction apply for each $k$. For our notation to be more precise we should add the dependence on $k$ of the various families $\sW$, however, since $k$ is fixed at this stage, in order to make our formulas simpler we drop such dependence. 

First we consider all squares which stopped for the excess or the height and which influence some square intersecting $\mathscr B_r^+$.
\begin{definition} We define the family $\cT$ to be
    \begin{align*}
    \cT &:= \left\{ L \in \sW^e \cup \sW^h: L \cap \mathscr B_r^+ \neq \emptyset \right\}\\
    &\qquad\qquad\cup \left\{ L \in \sW^e: \textnormal{ there is an } L' \in \sW^n(L) \textnormal{ such that } L' \cap \mathscr B_r^+ \neq \emptyset\right\}.
    \end{align*}
\end{definition}
Notice that because in a chain of squares in $\sW^n$, the sidelengths always double, we have for each $L\in \cT$
\[ \textnormal{sep}(L, \mathscr B_r^+) := \inf\{ |x-y|: x \in L, y\in \mathscr B_r^+\} 
\leq 3 \sqrt 2 \ell(L) .\]

To each such square $L \in \cT$, we associate a ball $B(L)$ which we call \emph{satellite ball}. Preferably this ball is contained in the square and with radius comparable to the sidelength. However, as not every square in $\cT$ is contained in $D$, we choose instead a nearby ball. Moreover we want that the concentric ball with twice the radius to be contained in $\mathscr B_r^+$. Notice that because of the intervals of flattening \eqref{e:large}, the largest square $L$ contributing to the center manifold and intersecting $\mathscr B_r^+$ satisfies $\ell(L) \leq \frac{1}{64\sqrt 2} r$.
\begin{itemize}
    \item If $B_{\ell(L)/2}(x_L) \subset \mathscr B_r^+$, we define $B(L):= B_{\ell(L)/4}(x_L)$. 
    \item If $B_{\ell(L)/2}(x_L) \nsubseteq \mathscr B_r^+$, we choose a point $y \in \partial \mathscr B_r^+$ minimizing the distance to $L$. Notice that the size length of the squares in the domain of influence of $L$ vary by a factor $2$, we have $|x_L -y| \leq 4 \sqrt2\ell(L)$.
    The center of the satellite ball we want to be a point inside $\mathscr B_r^+$ and close to $y$ (and thus close to $x_L$). Indeed, first notice that by the regularity assumption on $\Gamma_k$, $\bar \phii_k$ (Theorem \ref{t:cm}) and $d_k$ (Definition \ref{Def:distance}) there is a $C^1$-diffeomorphism $\Psi_r:\bar{B}_r^+ \to \bar{\mathscr  B}_r^+$ with $\|\Psi_r- \text{Id} \| \leq C \bmo^{\sfrac12}$. Moreover, we define for any $\ell < \frac r2$ the  vectorfield $n_\ell: \partial B_r^+ \to B_r^+$ describing $\partial \{y \in B_r^+: \dist(y, \partial B_r^+)>\ell \}$ by
    \begin{align*}
        n_\ell(x_1,x_2) :=
        \begin{cases}
            (x_1, \ell), &\textnormal{if } |x_1|<r-\ell,\ x_2=0\,, \\
            (r-\ell)(x_1, x_2), &\textnormal{if } x_2>\ell\,,\\
            (r-\ell, \ell), &\textnormal{if } \ell-r <x_1<r,\ x_2 \leq \ell\,,\\
            (-r+\ell, \ell), &\textnormal{if } -r<x_1<-r+\ell,\ x_2 \leq \ell.
        \end{cases}
    \end{align*}
    Notice that if $\eps_{CM}$ is small enough, we have for any $\ell < \frac r2$
    \[ B_{\ell/2}\big( \Psi_r(n_\ell(x)) \big) \subset
    \Psi_r\big( B_\ell(n_\ell(x)) \big) \subset \mathscr B_r^+.\]
    Thus for the $y \in B_{\ell(L)/2}(x_L) \cap \partial \mathscr B_r^+$, we define
    \[ q_L := \Psi_r \big( n_{\ell(L)/2}(\Psi_r^{-1}(y)) \big)\]
    and observe that
    \[ B(L) := B_{\ell(L)/4}(q_L) \subset \mathscr{B}_r^+\,. \]
\end{itemize}
By construction and the estimates on $d_k$, we have if $\eps_{CM}$ is small enough,
\begin{align*}
    |q_L -x_L| \leq 5\sqrt 2 \ell(L)
    \qquad \textnormal{ and thus } \qquad
    \dist(q_L, L) \leq 4\sqrt 2 \ell(L).
\end{align*}

From this family $\cT$, we now choose a maximal subfamily $\mathscr T$ for which the satellite balls are disjoint. Denote by $S:= \sup\{\ell(L): L \in \cT\}$. We define $\mathscr T_1 \subset \{L \in \cT: \frac12 S \leq \ell(L) \leq S\}$ to be a maximal subfamily for which the associated satellite balls are pairwise disjoint. We inductively define $\mathscr T_{k+1} \subset \{L \in \cT: 2^{-k-1} S \leq \ell(L) \leq 2^{-k}S\}$ to be a maximal subfamily such that all the satellite balls $B(L')$ with $L' \in \mathscr T_1 \cup \cdots \cup \mathscr T_k$ are pairwise disjoint. Finally we define $\mathscr T$ to be the union of all the $\mathscr T_k$. As we want to cover all of $\mathscr B_r^+$, we associate to each square in $L \in \mathscr T$ the nearby squares of $\cT$ whose satellite balls intersect $B(L)$ and the domain of influence $\sW^n(L)$. Indeed, by a standard covering argument, notice that if $H \in \cT$, then there is at least one square $L \in \mathscr T$ such that $\dist(H,L) \leq 20 \sqrt 2 \ell(L)$. We fix an arbitrary choice to partition $\cT$ into families $\cT(L)$ such that $L\in\mathscr T$, for any $H \in \cT(L)$ we have $\ell(H) \leq 2 \ell(L)$ and $\dist(H,L) \leq 20 \sqrt 2 \ell(L)$. Now we add the rest of $\mathscr B_r^+$ and define
\[ \sW(L) := \bigcup_{H \in \cT(L)} \sW^n(H) \cup \{H\}. \]
The associated Whitney regions will be called $\cU(L) \subset \cM$,
\[ \cU(L) := \bigcup_{H \in \sW(L)} \bar\Phii_k(H)\, , \]
where the map $\bar\Phii_k$ is the parametrization of the center manifold induced by $\bar\phii_k$, namely $\bar\Phii_k (x) = (x, \bar\phii_k (x))$.

For simplicity of notation, we enumerate $\mathscr T = \{L_i\}_i$ and denote
\begin{align*}
    \cB_r^+ &:= \bar \Phii_k (\mathscr B_r^+) = \bar \cM_k \cap \{d_k <r\} ,\\
    \cU_i &:= \cU(L_i) \cap \cB_r^+ ,\\
    \cB^i &:= \bar\Phii_k(B(L_i)),\\
    \ell_i &:= \ell(L_i).
\end{align*}

Notice that by construction, every satellite ball $B(L_i)$ has distance at least $\ell_i /4$ to $\partial \mathscr B_r^+$. In particular, there is a geometric constant $c>0$ such that
\begin{align*}
    c \frac{\ell_i}{r} \leq \inf_{\bp_k^{-1}(\cB^i)} \varphi_k = \inf_{\cB^i} \varphi_k. 
\end{align*}
As in \cite[Section 9.6.2]{DDHM}, we conclude that there is  a geometric constant $C>0$ such that
\begin{align}
    \sup_{\bp_k^{-1}(\cU_i)} \varphi_k = \sup_{\cU_i} \varphi_k 
    &\leq C \inf_{\bp_k^{-1}(\cU_i)} \varphi_k = C\inf_{\cU_i} \varphi_k, \label{e:sup_inf_phi}\\
    \sum_{H \in \sW(L_i)} \ell(H)^2 &\leq C \ell_i^2. \label{e:sum_length}
\end{align}
Applying the estimates of Theorem \ref{t:normal-approx} and Corollary \ref{c:cover}(ii) in each square of $\sW(L_i)$ and summing over them yields
\begin{align}
    \Lip (\bar N_k |_{\cU_i}) &\leq C \bmo^{\gamma_2} \ell_i^{\gamma_2}\,, \label{e:N_Lip_U} \\
    \|\bar N_k\|_{C^0(\cU_i)} + \sup_{\supp(T)\cap \bp^{-1}(\cU_i)} |\bp^\perp| &\leq C \bmo^{\sfrac{1}{4}} \ell_i^{1+\beta_1}\,,\\
    \|\bT_{\bar F_k} - T_{0, t_k}\| (\p_k^{-1} (\cU_i)) &\leq C \bmo^{1+\gamma_2} \ell_i^{4+\gamma_2} \,,\\
    \int_{\cU_i} |D \bar N_k|^2 &\leq C \bmo \,\ell_i^{4-2\delta_1}\,, \label{e:N_Dir_energy_U}\\
    \int_{\cU_i} |\etaa\circ \bar N_k| &\leq 
    C \bmo \ell_i^{4+\sfrac{\gamma_2}{2}} + C \int_{\cU_i} |\bar N_k|^{2+\gamma_2}\, .\label{e:etaN_U} 
\end{align}

On the other hand, we can use the the Separation Proposition \ref{p:separation_height}, the Splitting Proposition \ref{p:splitting} and the estimates \eqref{e:sup_inf_phi}, \eqref{e:sum_length} to deduce estimates on the normal approximation as stated in the next lemma.

\begin{lemma}
    Assume the assumption \ref{intorno_proiezione} holds. Then there is a geometric constant $C_0$ \footnote{Here and in the sequel we call a constant geometric if it depends only on $n, Q, N_0, M_0, C_e^\flat, C_e^\natural, C_h$ which we fixed.} such that
    \begin{align}
        \bmo \sum_i \left( \ell_i^{4+2\beta_1} ~ \inf_{\cB^i} \varphi_k \right) &\leq C_0 \bar D_k (r), \label{e:sum_ell_varphi}\\
        \bmo \sum_i \ell_i^{4+\beta_1} &\leq C_0 \int_{\cB_r^+} |D\bar N_k |^2
        \leq C_0(\bar D_k (r) + r\bar D_k '(r)).\label{e:sum_ell}
    \end{align}
    Moreover, we have
    \begin{align}\label{e:sup_ell}
        \bmo \sup_i \ell_i \leq C_0 (r\bar D_k (r))^{1/(5 + \beta_1)}
        \quad \textnormal{ and } \quad
        \bmo \sup_i \left( \ell_i \inf_{\cB^i} \varphi_k \right) \leq C_0 \bar D_k (r)^{1/(4 + \beta_1)} ,
    \end{align}
    and 
    \begin{equation}\label{e:estimate-for-D}
    \bar D_k (r) \leq C_0 \bmo r^{4-2\delta_1} \leq C_0 t_k^{2\kappa} r^{4-2\delta_1}\, .
    \end{equation}
\end{lemma}

\begin{proof}
    The proof goes completely analogous to the one of \cite[Lemma 9.13]{DDHM} and we summarize it here. Fix an $L_i \in \mathscr T$. If $L_i \in \sW^h$, it is an interior square and we can use Proposition \ref{p:separation_height} to deduce
    \begin{align}\label{e:N2_in_Bi}
        \int_{\cB^i} |\bar N_k|^2 \geq c_0 \bmo^{\sfrac 12} \ell^{4+2\beta_1}_i.
    \end{align}
    On the other hand, if $L_i \in \sW^e$, then $L_i$ can be either a boundary square or an interior square. However the satellite ball does not intersect the boundary and also we can apply Proposition \ref{p:splitting} in both situations. Thus, we have
    \begin{align}
        \int_{\cB^i} |D\bar N_k|^2 &\geq c_0 \bmo \ell_i^{4-2\delta_1},\label{e:DN2_in_Bi}\\
        \int_{\cB^i} \varphi |D \bar N_k|^2 &\geq c_0 \bmo \ell_i^{4-2\delta_1} ~ \inf_{\cB^i} \varphi_k. \label{e:phi_DN2_in_Bi}
    \end{align}
    Summing over all squares and using \eqref{e:N2_in_Bi}, \eqref{e:DN2_in_Bi} and \eqref{e:phi_DN2_in_Bi}, we conclude
    \begin{align*}
        \bmo \sum_i \ell_i^{4+2\beta_1} ~ \inf_{\cB^i} \varphi_k &\leq C_0 \int_{\cB^+_r} \left( |\bar N_k|^2 + \varphi_k |D\bar N_k|^2 \right),\\
        \bmo \sum_i \ell_i^{4+2\beta_1} \leq C_0 \int_{\cB^+_r} &\left( |\bar N_k|^2 + |D \bar N_k|^2 \right)
        \leq C_0 \int_{\cB^+_r} |D \bar N_k|^2 ,
    \end{align*}
    where we used the Poincar\'e inequality and the fact that $\bar N_k$ vanishes on $\Gamma_k$.
    We conclude by noticing that, as $\phi'=-2$ in $[\frac12, 1]$, we have
    \begin{align*}
        \int_{\{r/2 <d_k <r\}\cap \bar \cM_k} |D \bar N_k |^2 &\leq r \bar D_k'(r)\,,\\
        \int_{\{d_k <r/2\}\cap \bar \cM_k} |D \bar N_k|^2 &\leq \bar D_k (r)\,.
    \end{align*}
    \eqref{e:estimate-for-D} is a consequence of \eqref{e:N_Dir_energy_U}.
\end{proof}

We end this section with estimating the error terms (compare with \cite[Proposition 9.14]{DDHM}). 
\begin{proposition}\label{p:error_estimates}
    There are constants $C, \tau >0$ such that
    \begin{align}
        |\textnormal{Err}_1^o| + |\textnormal{Err}_3^o| + |\textnormal{Err}_4^o| &\leq C \bar D_k (r)^{1+\tau},\label{e:bound-outer-1} \\
        |\textnormal{Err}_2^o| &\leq C t_k^{2\kappa} \bar S_k (r) \leq C t_k^{2\kappa} r^2 \bar D_k (r)\label{e:bound-outer-2}
    \end{align}
    and
    \begin{align}
        |\textnormal{Err}_1^i| + |\textnormal{Err}_3^i| + |\textnormal{Err}_4^i| &\leq C \bar D_k (r)^{\tau} \big( \bar D_k (r) + r\bar D_k '(r) \big), \label{e:bound-inner-1}\\
        |\textnormal{Err}_2^i| &\leq C t_k^{2\kappa} r \bar D_k (r).\label{e:bound-inner-2}
    \end{align}
\end{proposition}

\begin{proof}
    The detailed estimates can be found in the proof of \cite[Proposition 9.14]{DDHM}. Notice that as there it is done for either side of the boundary separately, and as we have the same estimates on $N$, it applies directly to our situation. The idea is as follows. First we notice that
    \[ |Y(p)| \leq \varphi(p) d_k (\bp_k (p))
    \quad \textnormal{ and } \quad
    |DY(p)| \leq C \mathbf{1}_{\cB_r^+}(\bp_k (p)).\]
    Then because of the Theorem \ref{t:cm}, both the second fundamental form and the mean curvature of $\bar \cM_k$ are bounded (and their derivatives) are bouned by $C t_k^\kappa$. The remaining terms in the errors can be split into the regions $\cU_j$ and then be estimated by powers of $\bmo$ and $\ell_j$ using \eqref{e:N_Lip_U} - \eqref{e:etaN_U}. Choosing $\tau \ll \delta_1$ and recalling that $\delta_1 \leq \beta_1 \leq \gamma_1/8$, we see that the powers are higher than what we need for \eqref{e:sum_ell_varphi} and \eqref{e:sum_ell}. Thus with \eqref{e:sup_ell} we gain the additional $\bar D_k (r)^\tau$.
    
    The only relevant difference in the estimates of \cite[Proposition 9.14]{DDHM} is in the terms $\textnormal{Err}_2^i$ and $\textnormal{E}_2^o$, where our estimates have an improved factor $C t_k^{2\kappa}$ in the right hand side. But this follows easily from the fact that in our case we take advantage of $\|A_k\|_\infty \leq C t_k^\kappa$, while in \cite[Proposition 9.14]{DDHM} the second fundamental form of the center manifold is only known to be bounded by a constant. 
\end{proof}

\subsection{Proof \eqref{e:bloody-bound-on-the-D} and \eqref{e:Gronwall}} In order to prove \eqref{e:bloody-bound-on-the-D} we exploit \eqref{e:back-and-forth-1} and \eqref{e:estimate-for-D}: we assume $t_{k+1} < r < t_k$ and estimate
\[
D (r) = t_k^2 \bar{D}_k (t_k^{-1} r) \leq C t_k^{2+2\kappa} (t_k^{-1} r)^{4-2\delta_1} \leq C r^{2+2\kappa}\, .
\]
In order to prove \eqref{e:Gronwall} we follow the computations of \cite[Section 9.1]{DDHM}, but in our setting some additional complications are created by the fact that we need to scale back our estimates for the rescaled quantities $\bar D_k$, $\bar H_k$, $\bar S_k$, $\bar G_k$, and $\bar{S}_k$. First of all we recall \eqref{e:derivative-of-H-bar}:
\begin{equation}\label{e:first-piece}
H' (r) = r^{-1} H (r) + 2 E(r) + O (1) H (r)\, .
\end{equation}
Next we combine \eqref{e:outer-with-bar}, \eqref{e:bound-outer-1}, and \eqref{e:bound-outer-2} to get 
\begin{equation}\label{e:bound-outer-3}
|\bar D_k (t_k^{-1} r) - \bar E_k (t_k^{-1} r)| \leq C \bar D_k (t_k^{-1} r)^{1+\tau} + C t_k^{2\tau} \bar S_k (t_k^{-1} r). 
\end{equation}
We next can use \eqref{e:back-and-forth-1}, \eqref{e:back-and-forth-3}, and \eqref{e:back-and-forth-7} to conclude
\begin{equation}\label{e:little-step}
|D (r) - E (r)|\leq C D (r) (t_k^{-2} D (r))^\tau + C t_k^{2\tau -2} S (r)\, .
\end{equation}
Next recall that $D(r) \leq C r^{2+2\kappa}$. Since $r\leq t_k$ we can write
\[
t_k^{-2} D(r) \leq C t_k^{-2} r^2 D(r)^{1-\sfrac{2}{(2+2\kappa)}} \leq C D (r)^{1-\sfrac{1}{(1+\kappa)}}\, .
\]
Thus, at the prize of choosing $\tau$ smaller, we can translate \eqref{e:little-step} into
\begin{equation}\label{e:second-piece}
|D(r) - E(r)| \leq C D(r)^{1+\tau} + C t_k^{2\tau -2} S(r)\, .
\end{equation}
The final ingredient is derived by first combining \eqref{e:inner-with-bar}, \eqref{e:bound-inner-1}, and \eqref{e:bound-inner-2} to get
\begin{align}
& |\bar D_k'(t_k^{-1} r) + O(t_k^{2\kappa}) \bar D_k (t_k^{-1} r) - \bar G_k ( t_k^{-1} r)|\nonumber\\
\leq & \frac{C}{t_k^{-1} r}\bar D_k (t_k^{-1} r)^\tau \left(\bar D_k (t_k^{-1} r) + t_k^{-1} r \bar D_k' (t_k^{-1} r)\right)
+ C t_k^{2\kappa} \bar D_k (t_k^{-1}r)\, ,\label{e:bound-inner-3}
\end{align}
which in turn, using \eqref{e:back-and-forth-1}, \eqref{e:back-and-forth-6}, and \eqref{e:back-and-forth-7} becomes
\begin{equation}\label{e:little-step-2}
|D' (r) + O (t_k^{2\kappa-1}) D (r) - 2 G (r)|\leq C (t_k^{-2} D (r))^\tau (r^{-1} D (r) + D' (r)) + C t_k^{2\kappa-1} D (r)\, .
\end{equation}
But then, arguing as for \eqref{e:second-piece} we can achieve
\begin{equation}\label{e:third-piece}
|D' (r) - 2 G(r)|\leq C t_k^{2\kappa-1} D(r) + C D(r)^\tau (r^{-1} D(r) + D' (r))\, .
\end{equation}

We are now ready to estimate $\frac{d}{dr} \log I(r)$. We start by writing
\begin{equation}\label{e:parziale-1}
\frac{d}{dr} \log I(r) = \frac{1}{r} + \frac{D'(r)}{D(r)} - \frac{H' (r)}{H(r)}\, .
\end{equation}
Hence, using \eqref{e:first-piece} we write
\begin{equation}\label{e:parziale-2}
\frac{d}{dr} \log I (r) \geq - C + \frac{D'(r)}{D(r)} - \frac{2E(r)}{H (r)}\, .
\end{equation}
Next recall \eqref{e:bloody-bound-on-the-D}
while Lemma \ref{l:Poincare} implies that for $\sigma\in]0,1[$ we have 
\begin{equation}\label{e:bound-on-the-S}
t_k^{2\sigma-2} S (r) \leq C r^2 t_k^{2\sigma-2} D (r) \leq C r^{2\sigma} D(r) \, .
\end{equation}
In combination with the last two bounds, \eqref{e:second-piece} becomes (after possibly choosing a new positive $\tau$)
\begin{equation}\label{e:fourt-piece}
|D (r) - E(r)|\leq C r^{\tau} D (r)\, ,
\end{equation}
which in turn implies 
\begin{equation}\label{e:fifth-piece}
\frac{D(r)}{2} \leq E (r) \leq 2 D(r)\, ,
\end{equation}
provided $r\le r_0$ is sufficiently small with $r_0>0$ depending only on $C$ and $\tau$.

By \eqref{e:fifth-piece} we can turn \eqref{e:second-piece} into
\begin{equation}\label{e:sixth-piece}
|E (r)^{-1} - D(r)^{-1}|\leq C D (r)^{\tau-1} + C t_k^{2\tau-2} \frac{S(r)}{D(r)^2}\, .
\end{equation}
Inserting the latter into \eqref{e:parziale-2} (and considering that $D' (r) \geq 0$) we then get
\begin{equation}\label{e:parziale-3}
\frac{d}{dr} \log (I(r)) \geq \frac{D'(r)}{E(r)} - \frac{2 E(r)}{H(r)} - C \frac{D' (r)}{D (r)^{1-\tau}} - 
C t_k^{2\kappa-2} \frac{S (r) D' (r)}{D(r)^2}-C\, .
\end{equation}
We can finally insert \eqref{e:third-piece} to achieve
\begin{align}
\frac{d}{dr} \log (I(r)) \geq & \frac{2G (r)}{E (r)} - \frac{2 E(r)}{H(r)} - C \frac{D(r)}{E(r)} \left(\frac{D(r)^\tau}{r} + 
\frac{D'(r)}{D(r)^{1-\tau}} + t_k^{2\tau-2}\right) \nonumber\\
& - C \frac{D' (r)}{D (r)^{1-\tau}} -
C t_k^{2\kappa-2} \frac{S (r) D' (r)}{D(r)^2}-C\, .
\end{align}
Next note that:
\begin{itemize}
    \item $G(r) H (r) \geq E(r)^2$, by Cauchy-Schwarz;
    \item $\frac{D(r)}{E(r)}\leq C$;
    \item $D (r) \leq C r^{2+2\kappa}$.
    \item We can rewrite $- \frac{S(r) D'(r)}{D(r)^2} = \frac{d}{dr} \frac{S(r)}{D(r)} - \frac{S'(r)}{D(r)}$, and it is easy to see that $S'$ is positive.
\end{itemize}
So, after possibly choosing $\tau$ smaller, yet positive, we achieve
\begin{equation}\label{e:parziale-4}
\frac{d}{dr} \left(\log I (r) + C D (r)^\tau - C t_k^{2\tau-2} \frac{S(r)}{D(r)}\right) \geq - C r^{\tau-1}\, .
\end{equation}

%% file: blow-up-2.tex
\section{Proof of Theorem \ref{t:frequency}: Part II}\label{s:proof_jump_estimate}

This section is devoted to prove \eqref{e:jump_estimate}. We observe that, by the continuity of the functions
\begin{align*}
t\mapsto H (N_k, t) \qquad \mbox{and} \qquad t \mapsto D (N_k, t)
\end{align*}
we have
\begin{align*}
I (t_k^+) = \frac{t_k D (N_{k-1}, t_k)}{H (N_{k-1}, t_k)} \qquad \mbox{and} \qquad
I (t_k^-) = \frac{t_k D (N_{k}, t_k)}{H (N_{k}, t_k)}\, .
\end{align*}

In order to simplify our notation we use the shortcut $\bE (T, r)$ for $\bE (T, \bB_r)$. We will show the following two propositions
\begin{proposition}\label{p:change-of-center-manifolds}
    There is a constant $C$ independent of $k$ such that, if $\varepsilon_{CM}$ is small enough then
    \begin{align}
    C^{-1} t_k^2  \bE (T, 6t_k) & \leq D (N_{k-1}, t_k) \leq C t_k^2 \bE (T, 6t_k) \label{e:D_Nk-1}\\
    C^{-1} t_k^2 \bE (T, 6t_k) & \leq D (N_k, t_k) \leq C t_k^2 \bE (T, 6t_k) \label{e:D_Nk}\\
    C^{-1} t_k^3 \bE (T, 6t_k) & \leq H (N_{k-1}, t_k) \leq C t_k^3 \bE (T, 6t_k)\label{e:H_Nk-1}\\
    C^{-1} t_k^3  \bE (T, 6t_k)& \leq H (N_k, t_k) \leq C t_k^3 \bE (T, 6t_k)\,.
    \label{e:H_Nk}
    \end{align}
\end{proposition}

\begin{proposition}\label{p:change-of-center-manifolds2}
    There is a positive exponent $\tau_1$ independent of $k$ such that, if $\varepsilon_{CM}$ is small enough then
    \begin{align}
    |D (N_{k-1}, t_k) - &D (N_k, t_k)| \leq C t_k^2 \bE (T, 6t_k)^{1+\tau_1},
    \label{e:difference_D}\\
    |H (N_{k-1}, t_k) - &H (N_k, t_k)| \leq C t_k^3 \bE (T, 6t_k)^{1+\tau_1}. \label{e:difference_H}
    \end{align}
\end{proposition}

Observe that the estimates \eqref{e:D_Nk} (the second one), \eqref{e:H_Nk-1} (the first one), \eqref{e:H_Nk} (the first one), \eqref{e:difference_D}, and \eqref{e:difference_H} imply
\begin{equation}
    |I (t_k^+) - I (t_k^-)| \leq C \bE (T, 6t_k)^{\tau_1} \leq C t_k^{2\kappa \tau_1}\, .
\end{equation}

On the other hand, by the choice of $N_0$ in Assumption \ref{ass:hierarchy}, by \eqref{e:t_k-over-t_k-1}, we get $\frac{t_k}{t_{k-1}} \leq \frac{1}{2}$, which iterated implies $t_k \leq 2^{-k}$. We therefore get 
\begin{equation}
    |I (t_k^+) - I (t_k^-)| \leq C 2^{-2\kappa\tau_1 k}\, ,
\end{equation}
which clearly implies \eqref{e:jump_estimate}.

\begin{proof}[Proof of Proposition \ref{p:change-of-center-manifolds}]
	As the center manifold $\bar \cM_{k-1}$ stopped, and we are close to the boundary, it must have stopped for the excess and thus, there is a square $L \in \sW^e$ such that $c \frac{t_k}{t_{k-1}} \leq \ell(L) \leq C \frac{t_k}{t_{k-1}}$ (recall section \ref{s:flattening}). Looking at its ancestors (as we did in Proposition \ref{p:Whitney}), we notice
	\begin{align}\label{e:excess_m0k-1}
	\bE (T, \rho t_k) = \bE(T_{0, t_{k-1}}, \rho t_k/t_{k-1}) \leq C \bmo (k-1) \left(\rho \frac{t_k}{t_{k-1}}\right)^{2-2\delta_1}\, , 
	\end{align}
	for every $1\leq \rho \leq 5 R_0 \frac{t_{k-1}}{t_k}$ and some geometric constant $C$. Here we denote by $\bmo (k-1)$ and $\bmo (k)$ the two quantities
	\begin{align*}
	\bmo (k-1) &= \bE (T_{0,t_{k-1}}, \bC_{5R_0}) + \|\psi_{k-1}\|^2_{C^{3,\alpha} (]-5R_0, 5R_0[)},\\
	\bmo (k) &= \bE (T_{0, t_k}, \bC_{5R_0}) + \|\psi_k\|^2_{C^{3,\alpha} (]-5R_0, 5R_0[)}\, ,
	\end{align*}
	where $\psi_k$ and $\psi_{k-1}$ are the functions describing the rescaled boundaries $\Gamma_k$ and $\Gamma_{k-1}$. Observe that, since $\psi_k (0) = \psi_{k-1} (0) =0$ and $\psi'_k (0) = \psi'_{k-1} (0) = 0$, it can be readily checked that 
	\[
	\|\psi_k\|^2_{C^{3,\alpha} (]-5R_0, 5R_0[)}\leq 
	\frac{t_k^2}{t_{k-1}^2} \|\psi_{k-1}\|^2_{C^{3,\alpha} (]-5R_0, 5R_0[)}\, ,
	\]
	so that we have
	\begin{align}\label{e:m_0_comparison}
	\bmo (k) \leq \bE (T, \bC_{5R_0 t_k}) + \frac{t_k^2}{t_{k-1}^2} \bmo (k-1)
	\leq C \bmo (k-1) \left(\frac{t_k}{t_{k-1}}\right)^{2-2\delta_1},
	\end{align}
	where we also used \eqref{e:excess_m0k-1}.
	On the other hand, because of the stopping condition we also know that
	\begin{align}\label{e:excess_m0k-1-bis}
	\bE (T, 6 t_k) = \bE (T_{0, t_{k-1}}, 6 t_k/t_{k-1}) \geq C^{-1} \bmo (k-1) \left(\frac{t_k}{t_{k-1}}\right)^{2-2\delta_1}.
	\end{align}
	In particular, we infer by \eqref{e:m_0_comparison} that 
	\begin{equation}\label{e:needed}
	\bE (T, 6 t_k) \geq C^{-1} \bmo (k)\, .
	\end{equation}
	Observe now that for $D (\bar N_k, 1)$ we have the inequality
	\[
	D (\bar N_k, 1) \leq C \bmo (k)
	\]
	by construction of the center manifold (i.e. \eqref{e:global_Dir}). In turn, by rescaling, we can conclude
	\[
	D (N_k, t_k) = t_k^2 D (\bar N_k,1) \leq C t_k^2 \bmo (k) \leq C t_k^2 \bE (T, 6t_k)\, , 
	\]
	namely the first of the two inequalities in \eqref{e:wanted}.
	Then we observe that \eqref{e:D_Nk-1} and \eqref{e:H_Nk-1} follow from the Splitting Proposition \ref{p:splitting} applied to to the current $T_{0,t_k}$ which in turn produces the center manifold $\bar \cM_{k-1}$ and the normal approximation $\bar{N}_{k-1}$ as we are in the situation where the center manifold stopped. 
	Moreover, we recall that by the Poincar\'e inequality (as already observed in \eqref{e:bound_from_below} and proved in Section \ref{s:proof_monotonicity_frequency}), we have for any $r>0$
	\[H(N_k, r) \leq C r D(N_k, r)\,. \]
	Thus \eqref{e:H_Nk} and \eqref{e:D_Nk} follow once we have shown the following inequalities
	\begin{align}
	D (N_k, t_k) \leq C t_k^2 \bE (T, 6 t_k) 
	\leq C t_k^{-1} H (N_k, 6 t_k)\,. \label{e:wanted}
	\end{align}
	
	For the second inequality in \eqref{e:wanted} we adapt the proof of \cite[Proposition 3.7]{DS4} as the only difference to our situation is the cut-off function. We describe here the idea of the argument, the details can be read in \cite[Section 9]{DS4}. 
	Again recall the square $L\in \sW^e$ which stopped in the construction of $\bar \cM_{k-1}$ according to the argument above. By the splitting Proposition \ref{p:splitting}, we then have a nearby ball $B_{\ell/4}(z)$ not intersecting $\Gamma_{0, t_{k-1}}$ such that
	\[ \bmo (k-1) \left(\frac{t_k}{t_{k-1}}\right)^{6-2\delta_1} \leq C \int_{\bar \Phii_{k-1}(B_{\ell/4}(z))} |\bar N_{k-1}|^2 \,. \]
	The argument of \cite[Section 9]{DS4} provides now a similar bound for the ball $B' = 2 \frac{t_{k-1}}{t_k} B_{\ell/4} (z)$, which has radius comparable to $1$, in the center manifold $\bar \cM_k$. More precisely, since
	$\left(\frac{t_{k-1}}{t_k}\right)^4$ is exactly the scaling relating the $L^2$ norm on $B'$ and $B_{\ell/4} (z)$, while $\left(\frac{t_{k-1}}{t_k}\right)^{2-2\delta_1}$ is the scaling factor which makes $\bmo (k)$ and $\bmo (k-1)$ comparable, the corresponding estimate is given by
	\[
	\bmo (k) \leq C \int_{\bar \Phii_{k}(B')} |\bar N_k|^2\, .
	\]
	Applying the rescaling which relates $\bar \cM_{k}$ and $\cM_{k}$, we find a corresponding rescaled ball $B''$ (of radius comparable to $t_k$)  
	\[
	\bmo (k) t_k^4
	\leq C \int_{B''\cap \cM_k} |N_k|^2\, .
	\]
	Using that the center $z$ of the ball can be chosen arbitrarily as long as it is at a distance from $L$ compared to its diameter, we can ensure that $-d (p)^{-1} \phi' (t_k^{-1} d (p)) \geq c t_k^{-1}$ on $B''$ (for some positive geometric constant $c$). We thus get
	\[
	\bmo (k) t_k^3 \leq -\int_{B''\cap \cM_k} |N_k|^2 \frac{\phi' (t_k^{-1} d (p))}{d(p)} \leq C H (N_k, t_k)
	\,. \]
	However $\bE_k (T, 6t_k) \leq C \bmo (k)$, and we have thus completed the proof of the second inequality in \eqref{e:wanted}.
	
\end{proof}    

\begin{proof}[Proof of Proposition \ref{p:change-of-center-manifolds2}]
	Define for $p \in \cM_k$ the map $F_k(p) = \sum_i \a{p+(N_k)_i(p)}$ and for $q \in \cM_{k-1}$ the map $F_{k-1}(q) = \sum_i \a{q+(N_{k-1})_i(q)}$. Moreover denote by $\bE_k := \bE(T, 6t_k)$ and $\bC_k := \bC_{2t_k}(0, V_0)$. In order to compare $N_k$ and $N_{k-1}$, we first apply Theorem \ref{t:normal-approx} to the rescaled currents $T_{0, t_k}$ and $T_{0, t_{k-1}}$ to derive corresponding estimates for the normal approximations $\bar N_k$ and $\bar N_{k-1}$ of the currents on $\bar \cM_k$ and $\bar \cM_{k-1}$. We then scale them back to find corresponding estimates for $N_k$ and $N_{k-1}$. During this process we also observe that, by \eqref{e:excess_m0k-1} and \eqref{e:m_0_comparison}, we have 
	\begin{equation}\label{e:controls-on-the-m_0}
	\bmo (k) + \bmo (k-1) \left(\frac{t_k}{t_{k-1}}\right)^{2-2\delta_1}\leq C \bE_k\, .
	\end{equation}
	Moreover, we will prove later
	\begin{align}
		\|\phii_{k-1}\|_{C^0 (B_{2t_{k}})} &\leq C t_k \bE_k^{\sfrac{1}{2}} , \label{e:phii_k-1_sup}\\
		\|D\phii_{k-1}\|_{C^0 (B_{2t_k})} &\leq C \bE_k^{\sfrac{1}{2}}\, \label{e:Dphi_k}\\
		\|D^2 \phii_{k-1}\|_{C^0 (B_{\textcolor{red}{4}t_k})} &\leq C t_{k-1}^{-1} \bmo (k-1)^{\sfrac{1}{2}}\leq C t_k^{-1} \bE_k^{\sfrac{1}{2}}\, \label{e:second-derivative-good}\\
		\|\phii_k\|_{C^0 (B_{2t_k})} &\leq C t_k \bE_k^{\sfrac{1}{2}}, \label{e:phii_k_sup}\\
		\| D \phii_k\|_{C^0 (B_{\textcolor{red}{5}t_k})} &\leq C \bmo (k)^{\sfrac12} \leq C \bE_k^{\sfrac12}, \label{e:Dphii_k_supnorm}\\
		\| D^2 \phii_k \|_{C^0 (B_{\textcolor{red}{5}t_k})} &\leq C t_k^{-1} \bmo (k)^{\sfrac12} \leq C t_k^{-1} \bE_k^{\sfrac12}\, , \label{e:D2phii_k_sup}\\		
		\|D (\phii_k - \phii_{k-1})\|^2_{L^2 (B_{2t_k})} &\leq C t_k^2 \bE^{1+2\gamma_2}\, .
		\label{e:additional-1000}
	\end{align}
	In particular we get by \eqref{e:controls-on-the-m_0}, \eqref{e:global_Lip}, and \eqref{e:global_Dir} after rescaling back
	\begin{align}
	\Lip(N_k) + \Lip(N_{k-1}) &\leq C \bE_k^{\gamma_2}\,,\\
	\mass(\bT_{F_k} \res \bC_k - \bT_{F_{k-1}}\res \bC_k) 
	&\leq \mass(\bT_{F_k}\res \bC_k - T\res \bC_k) +\mass(T \res \bC_k- \bT_{F_{k-1}}\res \bC_k) \notag \\ 
	&\leq C t_k^2 \bE_k^{1+\gamma_2}\,.
	\end{align}
	Thus, we set $\hat N_k$ to be the $Q$-valued function defined on $\cM_{k-1}$ satisfying
	\[ \bG_{\hat N_k } \res \bC_k = \bT_{F_k} \res \bC_k = \bG_{N_k } \res \bC_k =: S\,, \]
	where with $\bG_{\hat N_k }$ we mean the current associated to the function $p \mapsto p + \hat N_k(p)$.
	By comparing $D(N_k, t_k)$ with $D(\hat N_k, t_k)$ and $H (N_k, t_k)$ with $H (\hat N_k, t_k)$ we make an additional error of size $t_k^2 \bE_k^{1+\gamma_2}$ and size $t_k^3 \bE_k^{1+\gamma_2}$ respectively. We will prove this later. With this aim in mind we change coordinates in the integrals of $D$ and $H$ to flat ones. 
	Denote by $\Phii_k(x):= (x, \phii_k(x))$ and $\Phii_{k-1}(x):= (x, \phii_{k-1}(x))$. We then estimate
	\begin{align*}
	\Big| D(N_k, t_k) - &\int |DN_k|^2(\Phii_k(x)) \phi \big(t_k^{-1} d(\Phii_k(x)) \big) dx \Big|\\
	&\leq C \int_{B_{2t_k}} |DN_k|^2(\Phii_k(x)) \phi \big(t_k^{-1} d(\Phii_k(x)) \big) \left| D \Phii_k(x) - (\Id,0) \right| dx\\
	&\leq C  \|D\phii_{k}\|_{C^0 (B_{2t_k})} \int |DN_k|^2(\Phii_k(x)) \phi \big(t_k^{-1} d (\Phii_k(x)) \big) J \Phii_k(x) dx\\
	&\leq C t_k^2 \bE_k^{\sfrac32} \,, 
	\end{align*}
	where we used \eqref{e:D_Nk} and \eqref{e:Dphii_k_supnorm} for the last inequality.
	Analogous estimates can be employed for $D (\hat N_k, t_k)$, $H (N_k, t_k)$, and $H (\hat N_k, t_k)$. 
	
	Therefore, it is enough to prove
	\begin{align}
	\left|\int |DN_k|^2\phi \big(t_k^{-1} d (\Phii_k(x)) \big) dx -\int |D\hat N_k|^2\phi(t_k^{-1} d (\Phii_{k-1}(x))) dx \right|
	&\leq C t_k^2 \bE_k^{1+\gamma_2}\,, \label{e:D_reparametrized}\\
	\left|\int |N_k|^2\frac{\phi'(t_k^{-1} d(\Phii_{k}(x)))}{d (\Phii_{k}(x))} dx - \int |\hat N_k|^2\frac{\phi'(t_k^{-1} d (\Phii_{k-1}(x)))}{d (\Phii_{k-1}(x))} dx \right|
	&\leq C t_k^3 \bE_k^{1+\gamma_2}\,.\label{e:H_reparametrized}
	\end{align}
	For \eqref{e:D_reparametrized}, notice that $N_k(p) = \sum_i \a{(F_k)_i(p)-p} $. Hence, each component of $N_k$ satisfies
	\[ |D (N_k)_i(\Phii_k(x))| \leq C\ |T_{(F_k)_i(x)} \bT_{F_k} - T_{\Phii_k(x)} \cM_k| \,.\]
	By the Lipschitz bound of $\phii_k$ \eqref{e:Dphii_k_supnorm} and of $F_k$, we thus have
	\begin{align*}
	\int |D N_k|^2 \phi \big(t_k^{-1} d (\Phii_k(x)) \big) &\leq C \int_\bC |\vec{S}(p) - \vec T_{\bp_k(p)} \cM_k|^2 \phi(t_k^{-1} d (\bp_k(p))) d \|S \|(p) + O(t_k^2\bE_k^{1+\gamma_2})\,,\\
	\int |D \hat N_k|^2 \phi(t_k^{-1} d(\Phii_{k-1}(x))) &\leq C \int_\bC |\vec S(p) - \vec T_{\bp_{k-1}(p)} \cM_{k-1}|^2 \phi(t_k^{-1} d (\bp_{k-1}(p))) d \|S \|(p) \\ &\quad+ O(t_k^2 \bE_k^{1+\gamma_2})\,,
	\end{align*}
	where we denoted by $\bp_k$ and $\bp_{k-1}$ the nearest point projection on $\cM_k$ and $\cM_{k-1}$ respectively, while $\bC$ is the vertical cylinder with base $B_{2t_k}$. As we have from Theorem \ref{t:cm} that $\| \phii_k - \phii_{k-1} \|_{C^2} \leq C t_k^{-1} \bE_k^{\sfrac12}$, by the Lipschitz bound of $\phi$, we deduce for any $p \in \supp(S)$ and $ q, q' \in \cM_k$,
	\begin{align*}
	|\phi(t_k^{-1} d (\bp_k(p)))-\phi(t_k^{-1} d (\bp_{k-1}(p))) | &\leq C\bE_k^{1/2}\,, \\
	|T_{q} \cM_k - T_{q'} \cM_k| &\leq Ct_k^{-1} \bE_k^{\sfrac12} |q-q'| \,.
	\end{align*}
	Hence, we have
	\begin{align*}
	\int_\bC |\vec S(p) - \vec T_{\bp_k(p)} \cM_k|^2 |\phi(t_k^{-1}
	&d (\bp_k(p)))-\phi(t_k^{-1} d(\bp_{k-1}(p)) | d \|S \|(p)
	\leq C t_k^2 \bE_k^{\sfrac32} \,,\\
	|T_{\bp_k(p)} \cM_k- T_{\bp_{k-1}(p)} \cM_{k-1}| 
	&\leq C | D \phii_k (\bp_{V_0}(\bp_k(p))) -  D \phii_{k-1} (\bp_{V_0}(\bp_{k-1}(p)))|\\
	&\leq C \bE_k + |D (\phii_k - \phii_{k-1})| (\bp_{V_0} (p))
	\end{align*}
	where we used \eqref{e:global_Lip} in the last inequality. We therefore can conclude
	\begin{align*}
	&\left|\int|DN_k|^2\phi \big(t_k^{-1} d(\Phii_k(x)) \big) dx -\int_{B_{2t_k}} |D\hat N_k|^2\phi(t_k^{-1} d(\Phii_{k-1}(x))) dx \right|\\
	&\quad \leq C t_k^2 \bE_k^{1+ \gamma_2} + C \int_\bC |\vec S(p) - \vec T_{\bp_k(p)} \cM_k |^2|\phi(t_k^{-1} d(\bp_k(p))) - \phi(t_k^{-1} d(\bp_{k-1}(p)))|\ d \|S\|\\
	&\qquad + C \int_\bC \left| |\vec S(p) - \vec T_{\bp_k(p)} \cM_k |^2 - |\vec S(p) - \vec T_{\bp_{k-1}(p)} \cM_{k-1} |^2\right| \phi(t_k^{-1} d(\bp_{k\textcolor{red}{-1}}(p)))\ d \|S\| \\
	&\quad \leq C t_k^2 \bE_k^{1+ \gamma_2} + C \int_\bC |\vec S(p) - \vec T_{\bp_k(p)} \cM_k | |\vec T_{\bp_{k}(p)} \cM_{k} - \vec T_{\bp_{k-1}(p)} \cM_{k-1} | \phi(t_k^{-1} d(\bp_k(p)))\ d \|S\| \\
	&\qquad + C \int_\bC |\vec S(p) - \vec T_{\bp_{k-1}(p)} \cM_{k-1} | |\vec T_{\bp_{k}(p)} \cM_{k} - \vec T_{\bp_{k-1}(p)} \cM_{k-1} | \phi(t_k^{-1} d(\bp_k(p)))\ d \|S\| \\
	&\quad \leq C t_k^2 \bE_k^{1+ \gamma_2} + C t_k \bE_k^{\sfrac12} \left( \int_\bC |\vec T_{\bp_{k}(p)} \cM_{k} - \vec T_{\bp_{k-1}(p)} \cM_{k-1} |^2 \phi(t_k^{-1} d(\bp_k(p)))\ d \|S\| \right)^\frac12\\
	&\quad \leq C t_k^2 \bE_k^{1+ \gamma_2} + C t_k \bE_k^{1/2} \left(\int_{B_{2t_k}} |D \phii_k - D \phii_{k-1}|^2\right)^{\frac{1}{2}}\\
	&\quad \leq C t_k^2 \bE_k^{1+ \gamma_2} \, ,
	\end{align*}
	where we used \eqref{e:additional-1000} for the last inequality.
	\medskip
	
	We finally turn to \eqref{e:H_reparametrized}. For $x \in V_0$, denote by $z_k :=(x, \phii_k(x))$ and $\hat z_k :=(x,\phii_{k-1}(x))$. Then we estimate
	\begin{align*}
	\left| |N_k|^2(z_k) - |\hat N_k|^2(\hat z_k) \right| 
	\leq |N_k|(z_k) \left| |N_k|(z_k) - |\hat N_k|(\hat z_k) \right| 
	+ |\hat N_k|(\hat z)\left| |N_k|(z_k) - |\hat N_k|(\hat z_k) \right| \,.
	\end{align*}
	Moreover, using Cauchy-Schwarz and the fact that the $L^2$ norm of $N_k$ and $\hat N_k$ is bounded by $t_k^2 \bE_k^{1/2}$, we have
	\begin{align}\label{e:normal_appr_squared}
	&\left|\int|N_k|^2\frac{\phi'(t_k^{-1} d(z_k))}{d( z_k)} dx - \int |\hat N_k|^2\frac{\phi'(d ((\hat z_k))}{d (\hat z_k)} dx \right| \notag \\
	&\qquad \leq C t_k\bE_k^{\sfrac12} \left( \int_{B_{2t_k}} \left| |N_k|(z_k) - |\hat N_k|(\hat z_k) \right|^2 dx \right)^\frac12 .
	\end{align}
	If we now define $p_i := (F_k)_i(x)$ and $q_i := (\hat F_k)_i(x) := \hat z_k + (\hat N_k)_i(\hat z_k)$, we have (up to reordering the indices)
	\begin{align*}
	|N_k|(z_k) = \left( \sum_i |p_i - z_k|^2 \right)^\frac12 ,
	\qquad
	|\hat N_k|(\hat z_k) = \left( \sum_i |q_i - \hat z_k|^2 \right)^\frac12.
	\end{align*}
	Now we use the triangle inequality to see
	\begin{align*}
	\left| |N_k|(z_k) - |\hat N_k|(\hat z_k) \right|^2 
	&= \left| \Big( \sum_i |p_i - z_k|^2 \Big)^\frac12 - \Big( \sum_i |q_i - \hat z_k|^2 \Big)^\frac12 \right|^2  \\
	&\leq C \sum_i |p_i - q_{\sigma(i)}|^2  + C |z_k-\hat{z}_k|^2\\
	&=C \cG \Big( \sum_i \a{p_i}, \sum_i \a{q_i} \Big)^2 + C |\phii_k (x)-\phii_{k-1} (x)|^2\, \, ,
	\end{align*}
	for $\sigma$ the permutation realizing the distance $\cG \Big( \sum_i \a{p_i}, \sum_i \a{q_i} \Big)$. 
	
	Note that, since $\phii_k$ and $\phii_{k-1}$ agree on the boundary $\bp_{V_0} (\Gamma)$, we can use \eqref{e:additional-1000} and the Poincar\'e inequality to conclude
	\begin{equation}\label{e:additional100}
	\|\phii_k - \phii_{k-1}\|_{L^2 (B_{2t_k})} \leq C t_k
	\|D\phii_k - D \phii_{k-1}\|_{L^2 (B_{2t_k})} \leq C
	t_k^3 \bE_k^{1/2+\gamma_2}\, .
	\end{equation}

	\begin{figure}[htp]
    \centering
    \includegraphics[width=10cm]{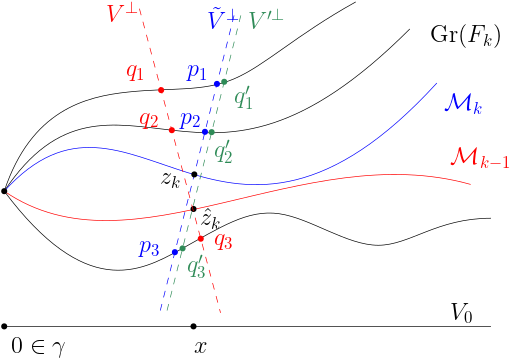}
    \caption{An illustration of how Lemma \ref{l:2d_rotations} is used.}
    \label{fig:whitney}
    \end{figure}
    
	To estimate further we split the distance $\cG \Big( \sum_i \a{p_i}, \sum_i \a{q_i} \Big)$  into a horizontal and vertical part in the following sense. We define $V:= \hat z_k + T_{\hat z_k} \cM_{k-1}$, $\tilde V := z_k + T_{z_k} \cM_k$, $V':= \hat z_k + T_{z_k} \cM_{k}$ and $\sum_i \a{q_i'} := \langle S, \bp_{V'}, 0 \rangle $. Observe that $V$ and $V'$ differ by a rotation, while $V'$ and $\tilde V$ are parallel. We then apply the Lemma \ref{l:2d_rotations} to the shifted situation where $ \hat z_k=0$ and deduce
	\begin{align*}
	\cG \Big( \sum_i \a{q_i}, \sum_i \a{q_i'} \Big) 
	&\leq C \Lip(F_k) \|N_k\|_{C^0} (|V-V_0| +|V'-V_0|)\\
	&\leq C \Lip(F_k) \|N_k\|_{C^0} (|D\phii_k| +|D\phii_{k-1}|)\\
	&\leq C t_k \bE_k^{\sfrac{3}{4}+\gamma_2}\,,
	\end{align*}
	where in the last inequality we used \eqref{e:Dphi_k} and \eqref{e:Dphii_k_supnorm}.
	In order to estimate $\cG \Big( \sum_i \a{p_i}, \sum_i \a{q_i'} \Big)$, 
	we call $f_{\tilde V}: T_{z_k} \cM_k \to \cA_Q(\R^{n})$ the function having the same graph as $F_k$ in $\bC_{2t_k}$. Observe that 
	\[ \left| T_{z_k} \cM_k -V_0 \right| \leq C t_k  \|D^2 \phii_k \|_{C^0} \leq C \bE_k^{\sfrac12}\]
	and by \cite[Proposition 5.2]{DS2}
	\[ \Lip( f_{\tilde V}) \leq C\bE^{\gamma_2}_k\,. \]
	Then we observe that $ \sum_i \a{p_i} = \sum_i \a{{f_{\tilde V}}_i(z_k)}$ and $\sum_i \a{q'_i} = \sum_i \a{ {f_{\tilde V}}_i(\bp_{T_p\cM_k}(\hat z_k))}$. Thus we have
	\begin{align*}
	\cG \Big( \sum_i \a{p_i}, \sum_i \a{q_i'} \Big)
	&\leq \Lip(f_{\tilde V}) |z_k-\bp_{T_{z_k}\cM_k}(\hat z_k)|\\
	&\leq \Lip(f_{\tilde V}) (||\phii_k||_{C^0} +||\phii_{k-1}||_{C^0})
	\leq C t_k \bE_k^{\sfrac12 + \gamma_2}.
	\end{align*}
	Squaring and integrating (and using \eqref{e:additional100}), we deduce 
	\[
	\int_{B_{2t_k}} \left||N_k| (z_k) - |\hat{N}_k| (\hat{z}_k)\right|^2 \leq C t_k^4 \bE_k^{1+2\gamma_2}\, .
	\]
	Inserting in \eqref{e:normal_appr_squared} we conclude
	\[
	\left|\int |N_k|^2\frac{\phi'(t_k^{-1} d (z_k)}{d (z_k)} dx - \int |\hat N_k|^2\frac{\phi'(t_k^{-1} d((\hat z))}{d (\hat z)} dx \right| \leq C t_k^3\bE^{1+\gamma_2}\, .
	\]
	
	\medskip
	
	It remains to prove \eqref{e:phii_k-1_sup}-\eqref{e:additional-1000}.
	
	\eqref{e:Dphii_k_supnorm} and \eqref{e:D2phii_k_sup} follow from Theorem \ref{t:cm} using a simple rescaling and \eqref{e:controls-on-the-m_0}.
	Next, for $\phii_{k-1}$ the estimate on the second derivative derived from Theorem \ref{t:cm} and \eqref{e:controls-on-the-m_0} is favourable, as it gives directly \eqref{e:second-derivative-good}. 
	However the estimate on the first derivative is not, as it would give 
	\begin{equation}\label{e:piccolo_problema-5}
	\|D \phii_{k-1}\|_{C^0 (B_{\textcolor{red}{5}t_k})} \leq C \bmo (k-1)^{\sfrac{1}{2}} 
	\leq C \left(\frac{t_{k-1}}{t_k}\right)^{1-\delta_1} \bE_k^{\sfrac{1}{2}}\, ,
	\end{equation}
	which is not good enough for our purposes.
	
	\underline{Proof of \eqref{e:phii_k-1_sup}, \eqref{e:Dphi_k}, and \eqref{e:phii_k_sup}}
	In order to gain a more favorable estimate for the first derivative (and the $C^0$ norm of $\phii_{k-1})$ we first observe that by Lemma \ref{l:hardtSimonHeight}
	\[
	\bh (T, \bC_{10t_k}(0, V_0)) \leq C \bE_k^{\sfrac{1}{2}} t_k\, .
	\]
	Arguing as in the proof of \eqref{e:H_Nk-1} it is not difficult to see that  
	\begin{equation}\label{e:piccolo_problema-1}
	\int_{\bC_{5t_k}(0, V_0) \cap \cM_{k-1}} |N_{k-1}|^2 \leq C \bE_k t_k^4\, .
	\end{equation}
	Since $\bT_{F_{k-1}}$ coincides with $\supp(T)$ on a large set we can also infer
	\begin{equation}\label{e:piccolo_problema}
	\int_{B_{5t_k}} |\phii_{k-1}|^2 \leq C 
	\bE_k t_k^4\, .
	\end{equation}
	In order to see the latter estimate, consider first a point $p\in \cM_{k-1}$ with the property that the support of $F_{k-1} (p)$ is a subset of the support of $T$. By the height bound we know that $\bh (T, \bC_{10t_k} (0, V_0))\leq C \bE_k^{1/2} t_k$. In particular, if we let $\bp_0^\perp$ be the projection on the orthogonal complement $V_0$, we conclude
	\[
	|\bp_0^\perp \circ F_{k-1}| (p) \leq C\bE_k^{1/2} t_k\, .
	\]
	Consider now that, if $x$ is such that $p= (x, \phii_{k-1} (x))$, since $F_k (p) = 
	\sum_i \a{F_k^i (p)} = \sum_i \a{N_k^i (p) + p}$, we get  
	\begin{align}
	|\phii_{k-1} (x)| &\leq |\bp_0^\perp \circ F_{k-1}| (x, \phii_{k-1} (x)) 
	+ |\bp_0^\perp \circ N_{k-1}| (x, \phii_{k-1} (x))\nonumber\\
	&\leq C \bE_k^{1/2} t_k
	+ |N_{k-1}| (x, \phii_{k-1} (x))\, .\label{e:piccolo_problema-2}
	\end{align}
	Let now $\mathcal{K}$ be the set of such points $p$ (i.e. for which the support of $F_k (p)$ is contained in the support of $T$) and define $K := \bp_0 (\mathcal{K})\cap B_{5t_k}$. Using the bounds \eqref{e:piccolo_problema-1} and \eqref{e:piccolo_problema-2} we easily obtain
	\begin{equation}\label{e:piccolo_problema-3}
	\int_K |\phii_{k-1} (x)|^2 \leq C \bE_k t_k^4\, .
	\end{equation}
	In order to estimate the integral on the remaining portion (i.e. on $B_{5t_k}\setminus K$), we apply \eqref{e:err_regional} to $\bar\cM_{k-1}$, sum over all the stopped squares in $ B_{5t_k}\setminus K$ (which by the stopping condition have side length comparable to $t_k/t_{k-1}$), scale it back to $\cM_{k-1}$ and deduce
	\begin{align}
	|B_{5t_k}\setminus K| &\leq \mathcal{H}^2 (\bB_{6t_k}\cap \cM_{k-1} \setminus \mathcal{K})
	\leq C (\bmo (k-1))^{1+\gamma_2} \left(\frac{t_k}{t_{k-1}}\right)^{4+\gamma_2} t_{k-1}^2\nonumber\\
	&\leq C \left(\frac{t_k}{t_{k-1}}\right)^{2+\gamma_2}\, .\label{e:piccolo_problema-6}
	\end{align}
	Then we observe that, by \eqref{e:piccolo_problema-3} and the classical Chebyshev inequality, there is at least one point $x\in B_{5t_k}$ where $|\phii_{k-1} (x)|\leq C \bE_k^{1/2} t_k$, and we use \eqref{e:piccolo_problema-5} to conclude that for all $y \in B_{5t_k}$ we have
	\begin{equation}\label{e:piccolo_problema-7}
	|\phii_{k-1} (y)|\leq C \bE_k^{\sfrac{1}{2}} t_k + C \bE_k^{\sfrac{1}{2}} \left(\frac{t_{k-1}}{t_k}\right)^{1-\delta_1}  |x-y| \leq C \bE_k^{\sfrac{1}{2}} \left(\frac{t_{k-1}}{t_k}\right)^{1-\delta_1}  t_k \, . 
	\end{equation}
	Putting together \eqref{e:piccolo_problema-3}, \eqref{e:piccolo_problema-6}, and \eqref{e:piccolo_problema-7}, we achieve
	\[
	\int_{B_{5t_k}} |\phii_{k-1}|^2 \leq C \bE_k t_k^4 + C \bE_k \left(\frac{t_{k}}{t_{k-1}}\right)^{2+\gamma_2 - 2 (1-\delta_1)} t_k^4\, . 
	\]
	Since $2+\gamma_2\geq 2-2\delta_1$ and $t_k \leq t_{k-1}$, the latter clearly implies \eqref{e:piccolo_problema}.
	
	We next use Gagliardo-Nirenberg interpolation inequality and from \eqref{e:piccolo_problema-1} and \eqref{e:second-derivative-good} we get \eqref{e:phii_k-1_sup} and \eqref{e:Dphi_k}, namely
	\begin{align*}
	\|\phii_{k-1}\|_{C^0 (B_{2t_{k}})} \leq C t_k \bE_k^{\sfrac{1}{2}} ,\qquad
	\|D\phii_{k-1}\|_{C^0 (B_{2t_k})} \leq C \bE_k^{\sfrac{1}{2}}\, . 
	\end{align*}
	We analogously conclude \eqref{e:phii_k_sup}.
	
	\underline{Proof of \eqref{e:additional-1000}}
	We wish to show that 
	\begin{equation*}
	    \|D (\phii_k - \phii_{k-1})\|^2_{L^2 (B_{2t_k})} 
	    \leq C t_k^2 \bE_k^{1+2\gamma_2}\, .
	\end{equation*}
	We choose a suitable cut-off function $\psi$ which equals $1$ on $B_{2t_k}$ and is compactly supported in $B_{3t_k}$ and write
	\begin{align*}
	    \int_{B_{2t_k}} |D (\phii_k-\phii_{k-1})|^2 \leq    
	    \int_{B_{3t_k}} |D(\phii_k - \phii_{k-1})|^2 \psi
	\end{align*}
	Integrating by parts, we can estimate
	\begin{align*}
    	\int |D(\phii_k - \phii_{k-1})|^2 \psi
	    = \int (\phii_k - \phii_{k-1}) \Delta (\phii_k - \phii_{k-1}) \psi + \int (\phii_k - \phii_{k-1}) \nabla (\phii_k - \phii_{k-1}) \cdot \nabla \psi\, .
	\end{align*}
	We next use that $\|\nabla \psi\| \leq C t_k^{-1}$, \eqref{e:Dphi_k},  \eqref{e:second-derivative-good}, \eqref{e:Dphii_k_supnorm}, and \eqref{e:D2phii_k_sup} to estimate
	\begin{equation}\label{e:additional-1001}
	    \int_{B_{2t_k}} |D (\phii_k - \phii_{k-1})|^2 
    	\leq C \bE_k^{1/2} t_k^{-1} \int_{B_{3t_k}} |\phii_k - \phii_{k-1}|\, .
	\end{equation}
	We next consider the multivalued functions $f_k$ and $f_{k-1}$ on $B_{3t_k}$ and taking values into $\mathcal{A}_Q (\mathbb R^n)$ with the properties that 
	\[
	\bG_{f_k} = \bT_{F_k} \res \bC_{0, 3t_k} \,, \qquad
	\bG_{f_{k-1}} = \bT_{F_{k-1}} \res \bC_{0, 3t_k}\, .
	\]
	Note that the values of $f_k$ and $f_{k-1}$ coincide except for a set of measure at most $t_k^2 \bE_k^{1+\gamma_2}$ (again we use Theorem \ref{t:normal-approx} and sum over the stopped squares). Moreover, because $\Lip (f_k), \Lip (f_{k-1}) \leq C \bE_k^{\gamma_2}$, we immediately draw the conclusion
	\[
	\int_{B_{3t_k}} |\etaa\circ f_k - \etaa \circ f_{k-1}|
	\leq \bE_k^{1+2\gamma_2} t_k^3\, .
	\]
	On the other hand, appealing to Proposition \ref{p:additional} (and rescaling appropriately) we get
	\begin{align*}
	\int_{B_{3t_k}} |\etaa\circ f_k - \phii_k| &\leq C \bE_k^{3/4} t_k^3\,, \\
	\int_{B_{3t_k}} |\etaa\circ f_{k-1} - \phii_{k-1}|
	&\leq C\left( \left(\frac{t_{k-1}}{t_k}\right)^{2-2\delta_1} \bE_k \right)^{3/4} \left(\frac{t_k}{t_{k-1}}\right)^4 t_{k-1}^3\, .
	\end{align*}
	While the first estimate is already suitable for our purposes, the second require some more care. We recall \eqref{e:m_0_comparison}
	to the effect that 
	\[
	\left(\frac{t_{k-1}}{t_k}\right)^{2-2\delta_1} \bE_k \leq
	\left(\frac{t_{k-1}}{t_k}\right)^{2-2\delta_1} \bmo(k) \leq C
	\]
	for a geometric constant $C$. Since $\frac{1}{2-2\delta_1} \geq \frac{3}{4}$, we can then estimate
	\[
	\int_{B_{3t_k}} |\etaa\circ f_{k-1} - \phii_{k-1}|\leq C
	\bE_k^{\frac34} t_k^3 \,.
	\]
	By possible choosing $\gamma_2$ sufficiently small we get 
	\[
	\int_{B_{3t_k}} |\phii_k - \phii_{k-1}|\leq C \bE_k^{1/2+2\gamma_2} t_k^3\, ,
	\]
	which, by \eqref{e:additional-1001}, gives \eqref{e:additional-1000}.
	
\end{proof}

\subsection{Lipschitz estimate using 2d-rotations}
\begin{lemma}\label{l:2d_rotations}
    There is a constant $c>0$ such that the following holds. Let $F: V_0 \to \cA_Q(\R^n)$ be a Lipschitz map with $\Lip(F)<c$, let $V$ and $V'$ be $2$-dimensional subspaces with $|V-V_0| +|V'-V_0|<c$ and denote by $\bp$ and $\bp'$ the orthogonal projection on $V$ and $V'$ respectively. Then for $P:= \langle \bT_F, \bp, 0 \rangle$ and $P':= \langle \bT_F, \bp', 0 \rangle$ it holds
    \begin{equation}\label{e:Lipschitz_2d_rotation}
        \cG(P, P') \leq C \, \Lip(F) \, \| F\|_{C^0} (|V-V_0| +|V'-V_0|)\,.
    \end{equation}
\end{lemma}

\begin{proof}
    We use an argument already observed in more generality in \cite[Lemma D.1]{DS4}. However, we repeat here the parts needed for the previous lemma. First of all, we construct finitely many planes by using 2d-rotations that will allow us to reduce \eqref{e:Lipschitz_2d_rotation} to a one-dimensional situation. Recall the terminology: we say that $R \in \text{SO}(n+2)$ is a 2d-rotation if there are two orthonormal vectors $e_1$, $e_2$ and an angle $\theta$ such that
    \begin{align*}
    \begin{cases}
        R(e_1) &= \cos(\theta)e_1 + \sin(\theta)e_2\,, \\
        R(e_2) &= \cos(\theta)e_1 - \sin(\theta)e_2\,, \\
        R(v) &= v\,, \qquad \text{ for any } v \in \langle e_1, e_2 \rangle^\perp.
    \end{cases}
    \end{align*}
    Now let us denote by $W_1= V \cap V'$. If $\dim(W_1) =2$, then $V=V'$ and there is nothing to prove. Otherwise $\dim(W_1)<2= \dim(V) = \dim(V')$ and we can write
    \[ V= W_1 \oplus \hat V,  \qquad V' = W_1 \oplus \hat V',  \]
    for some subspaces $\hat V$ and $\hat V'$.
    Choose any unit vector $e_1 \in \hat V = V \cap W_1^\perp$ and define
    \[ e'_1 := \frac{\bp'(e_1)}{|\bp'(e_1)|} \in V' \cap W_1^\perp\,.\]
    Moreover, define $R_1$ to be the 2d-rotation mapping $e_1$ onto $ e'_1$ and 
    \begin{align*}
        V_2 &:= R_1(V)\,,\\
        W_2 &:= V_2 \cap V'\,.
    \end{align*}
    Notice that $W_1 \subset V_1$ is invariant under $R_1$, so clearly $W_1 = (W_1 \cap V') \subset (V_2 \cap V') = W_2$. Moreover, $e'_1 \in V_2 \cap V' $, and hence
    \[ W_2 \supset \langle W_1, e'_1 \rangle\,.\]
    As $e'_1 \perp W_1$, we have $\dim(W_2) \geq \dim(W_1)+1$.
    Now, if $\dim(W_2)=2$, then $V_2=R_1(V_1)=V'$ and we define $R_2$ to be the identity. Otherwise $\dim(W_2)=1$ and we can again find a unit vector $e_2 \in  V_2 \cap W_2^\perp$, define
    \[ e'_2 := \frac{\bp'(e_2)}{|\bp'(e_2)|} \in V' \cap W_1^\perp\,,\]
    and define $R_2$ to be the 2d-rotation mapping $e_2$ onto $e'_2$. As before, we denote by $V_3 := R_2(V_2)$ and observe that $W_3:= V_3 \cap V'$ has at least one dimension more than $W_2$. Thus, in both cases we have
    \[ V' = R_2 \circ R_1(V)\,. \]
    Next, denote by $V_1:=V$ and for $j \in \{1, 2,3\}$ the orthogonal projection onto $V_j$ by $\bp_j$ and $P_j := \langle \bT_F, \bp_j, 0 \rangle$. Notice that for $c>0$ small enough, $\supp(P_j)$ is a $Q$-valued point. We claim
    \begin{align*}
        \cG(P_j, P_{j+1}) \leq C \Lip(F) \, \| F\|_{C^0} (|V_j-V_0| +|V_{j+1}-V_0|)
    \end{align*}
    concluding the lemma as $|V_j-V_0| \leq |V-V'| + |V-V_0| \leq 2(|V-V_0| +|V'-V_0|)$ for every $j$.
   Indeed, for each $j$, fix a unit vector $v_j \in V_0$ such that
    \[ \langle e_j, e'_j \rangle \cap V_0 = \{t \cdot v_j: t \in \R \}\,. \]
    Then we can apply the selection principle \cite[Proposition 1.2]{DS1} to the map $F^j(t):= F(t v_j)$ to get a selection
    \[ F^j = \sum_i \a{F_i^j}\]
    for some Lipschitz functions $F_i^j: [-1,1] \to \R^n $ satisfying
    \begin{equation}\label{e:Lip_selection}
        |DF_i^j| \leq |DF| \leq \Lip(F) \qquad \text{a.e.}
    \end{equation}
    We therefore conclude the existence of points $s_1^j, \dots, s_Q^j, s_1^{j+1}, \dots, s_Q^{j+1} \in [-1,1]$ such that
    \begin{align*}
        \cG(P_j, P_{j+1}) &\leq \sum_i \left|F^j_i(s_i^j)-F^j_i(s_i^{j+1}) \right|\\
        &\leq \Lip(F) \sum_i \left|s_i^j-s_i^{j+1} \right|\\
        &\leq \Lip(F) \sum_i \left(|s_i^j| +|s_i^{j+1}| \right) \\
        &\leq Q C \, \Lip(F) \,\| F\|_{C^0} \left(|V_j-V_0| +|V_{j+1}-V_0| \right)\,,
    \end{align*}
    where we also have used \eqref{e:Lip_selection}.
\end{proof}

\section{Blow-up analysis and conclusion}\label{s:blowup_and_conclusion}
In this section we complete the proof of Theorem \ref{t:flat-points}, which in turn completes the proof of Theorem \ref{t:main}. We recall the $I_0$ from Corollary \ref{c:frequency}.
The main point is the following conclusion.

\begin{theorem}\label{t:blow-up}
Let $T$ be as in Assumption \ref{a:main-local-3} and assume that $0$ is not a regular point. Then $I_0 =1$ and for every $\varsigma>0$
\begin{equation}\label{e:slow-decay}
\lim_{r\downarrow 0} \frac{D(r)}{r^{2+\varsigma}}=\infty\, .
\end{equation}
\end{theorem}

The latter is in contradiction with the estimate \eqref{e:bloody-bound-on-the-D} (i.e. $D(r) \leq C r^{2+\tau}$) for some positive constant $\tau$ which depends on the exponent $\alpha$ of Theorem \ref{t:decay}. 

\subsection{Blow-up analysis} As already mentioned, Theorem \ref{t:blow-up} is reached through a suitable ``blow-up'' analysis. First of all, having fixed a sequence of $s_j\downarrow 0$ we define a suitable family of rescalings of the maps $N_k's$. First of all we choose any $k(j)$ with the property that
\begin{equation}\label{e:def_k(j)}
t_{k (j)+1}< s_j \leq t_{k (j)}\, .
\end{equation}
Next we define the exponential map $\ex_k: T_0 \cM_k \to \cM_k$ and we identify each tangent $T_0 \cM_k$ to $\mathbb R^2$ through a suitable rotation of the ambient Euclidean space which maps it onto $\mathbb R^2\times \{0\}$. 
We then consider the rescaled maps
\begin{equation}\label{e:rescaled-maps}
\tilde{N}_j (x) := \frac{N_{k(j)} (\ex_{k (j)} (s_j x))}{D (s_j)^{\sfrac 12}}\, .    
\end{equation}

The main conclusion of our blow-up analysis is the following

\begin{theorem}\label{t:blow-up-2}
Let $T$ be as in Assumption \ref{a:main-local-3} and assume that $0$ is not a regular point. Let $s_j\downarrow 0$ be an arbitrary vanishing sequence of positive radii, let $k(j)$ be an arbitrary choice of integers satisfying \eqref{e:def_k(j)} and let $\tilde{N}_j: B^+_1 \to \Iqs$, where $B^+_1 = B_1\cap \{(x_1, x_2): x_2\geq 0\}$. Then a subsequence, not relabeled, converges strongly in $W^{1,2} (B_1^+)$ to a map $\tilde N_\infty$ satisfying the following conditions:
\begin{itemize}
    \item[(i)] $\tilde{N}_\infty (x_1, 0) = Q \a{0}$ for all $x_1$;
    \item[(ii)] $\tilde{N}_\infty$ is Dir-minimizing;
    \item[(iii)] $\tilde{N}_\infty$ is $I_0$-homogeneous, where $I_0$ is the positive number in Corollary \ref{c:frequency}.
    \item[(iv)] $\etaa\circ \tilde{N}_\infty \equiv 0$;
    \item[(v)] $\int_{B_1^+} |D\tilde{N}_\infty|^2=1$.
\end{itemize}
In particular $I_0=1$.
\end{theorem}

Then the arguments of Theorem \ref{t:classification} apply to $\tilde N_\infty$ and in particular give that $I_0=1$. 

\begin{proof}[Proof of Theorem \ref{t:blow-up-2}]
Observe first that, following the computations of \cite[Section 10.1]{DDHM} we conclude that
\begin{align*}
    e^{-Cs_j} 8^{1+I_0} \leq \frac{H(4s_j)}{H(s_j/2)} \leq e^{Cs_j} 8^{1+4I_0}
\end{align*}
as long as $s_j \leq t_{k(j)}$.
Since $I_0$ exists and is finite, there is a constant $C$ (depending only on $I_0$) such that
\[
D(4 s_j) \leq C D (s_j/2)\,.
\]
On the other hand, arguing as in the proof of Proposition \ref{p:change-of-center-manifolds}, we easily see that 
\[
D (t_{k(j)}) \geq C^{-1} t_{k(j)}^2 \bE (T, 24 t_{k(j)})
\]
(we just need to choose the constant $M_0$ appropriately large to compensate for the larger radius in the right hand side) while $D (4t_{k(j)})\leq C t_{k(j)}^2 \bE(T, 24 t_{k(j)})$. Now, since the geodesic ball $\mathcal{B}_{t_{k(j)}}$ in $\cM_{k(j)}$ contains $\{d<t_{k(j)}/2\}$ while the geodesic ball $\mathcal{B}_{2 t_{k(j)}}\subset \{d < 4 t_{k(j)} \}$, using the fact that the rescaling of the manifolds converge smoothly to the flat plane $V_0$, we easily conclude that
\[
\int_{B_2^+} |D\tilde{N}_j|^2 \leq C \int_{B_1^+} |D\tilde{N}_j|^2\, .
\]
We can then follow the argument of \cite[Section 10.3]{DDHM} to conclude that, up to subsequences, $\tilde{N}_j$ converges strongly in the $W^{1,2} (B_1^+)$ topology to a Dir-minimizing map $\tilde{N}_\infty$. Likewise we can follow the argument of \cite[Section 10.2]{DDHM} to conclude that $\etaa\circ \tilde{N}_\infty$ vanishes identically. Recall that the maps $N_{k(j)}$ vanish identically on $\Gamma$, while the rescalings of the latter converge smoothly to $T_0 \Gamma = \{x_2=0\}$. The strong convergence then implies that $\tilde{N}_\infty = Q \a{0}$ on $\{x_2=0\}\cap B_1$. We have thus proved (i), (ii), (iv), and (v). We can however also see that 
\[
\frac{r \int \phi (r^{-1} |x|) |D\tilde{N}_\infty (x)|^2\, dx}{- \int \phi' ( r^{-1} |x|) |x|^{-1} |\tilde{N}_\infty (x)|^2\, dx} = \lim_{j \to \infty} \frac{r s_j D (r s_j)}{H (r s_j)} = I_0\, ,
\]
which means that the frequency function of $\tilde{N}_\infty$ is constant. This however happens if and only if $\tilde{N}_\infty$ is $I_0$-homogeneous. 

As for the final statement, we invoke Theorem \ref{t:I=1}. 
\end{proof}

Now that we know that $I_0=1$, we can then conclude that by the strong convergence of $\{\tilde N_j\}_j$ in $W^{1,2} (B_1^+)$, we have
\begin{corollary}\label{c:frequency=1}
If $T$ is as in Theorem \ref{t:blow-up-2}, then
\[
\lim_{r\downarrow 0} \frac{D(2r)}{D(r)} = 4\, .
\]
\end{corollary}

\subsection{Proof of \eqref{e:slow-decay} and conclusion} Fix $\varsigma>0$ and consider the sequence of radii $r_k := 2^{-k}$. We know from Corollary \ref{c:frequency=1} that, for $k$ sufficiently large
\[
D (r_k) \geq 2^{-2-\varsigma/2} D (r_{k-1})\, .
\]
In particular we conclude the existence of a $k_0$ such that for every $k \geq k_0$, we have
\[
D (2^{-k}) \geq 2^{-(2+\varsigma/2) (k-k_0)} D (2^{-k_0})\, .
\]
In particular for every $r\leq 2^{-k_0}$ we can write
\[
D (r) \geq \frac{D (2^{-k_0})}{2^{2+\varsigma/2}} r^{2+\varsigma/2}
\]
and since $D (2^{-k_0})>0$, \eqref{e:slow-decay} readily follows.

%% file: convex-boundary-1.bbl
\begin{thebibliography}{10}

\bibitem{Collection}
Some open problems in geometric measure theory and its applications suggested
  by participants of the 1984 {AMS} summer institute.
\newblock In J.~E. Brothers, editor, {\em Geometric measure theory and the
  calculus of variations ({A}rcata, {C}alif., 1984)}, volume~44 of {\em Proc.
  Sympos. Pure Math.}, pages 441--464. Amer. Math. Soc., Providence, RI, 1986.

\bibitem{AllPhD}
W.~K. Allard.
\newblock {On boundary regularity for {P}lateau's problem}.
\newblock {\em Bull. Amer. Math. Soc.}, 75:522--523, 1969.

\bibitem{All}
W.~K. Allard.
\newblock {On the first variation of a varifold}.
\newblock {\em Ann. of Math. (2)}, 95:417--491, 1972.

\bibitem{AllB}
W.~K. Allard.
\newblock {On the first variation of a varifold: boundary behavior}.
\newblock {\em Ann. of Math. (2)}, 101:418--446, 1975.

\bibitem{Almgren}
Frederick~J. Almgren, Jr.
\newblock {\em Almgren's big regularity paper}, volume~1 of {\em World
  Scientific Monograph Series in Mathematics}.
\newblock World Scientific Publishing Co. Inc., River Edge, NJ, 2000.

\bibitem{Alm}
Jr. F.~J. Almgren.
\newblock {\em {Almgren's big regularity paper}}, volume~1 of {\em {World
  Scientific Monograph Series in Mathematics}}.
\newblock World Scientific Publishing Co. Inc., River Edge, NJ, 2000.

\bibitem{Theodora}
T.~{Bourni}.
\newblock {Allard-type boundary regularity for {$C^{1,\alpha}$} boundaries}.
\newblock {\em ArXiv e-prints}, August 2010.

\bibitem{DDHM}
C.~{De Lellis}, G.~{De Philippis}, J.~Hirsch, and A.~Massaccesi.
\newblock On the boundary behavior of mass-minimizing integral currents, 2018.

\bibitem{DNS}
C.~De~Lellis, S.~Nardulli, and S.~Steinbr\"uchel.
\newblock In preparation.

\bibitem{DS1}
C.~{De Lellis} and E.~Spadaro.
\newblock {{$Q$}-valued functions revisited}.
\newblock {\em Mem. Amer. Math. Soc.}, 211(991):vi+79, 2011.

\bibitem{DS3}
C.~{De Lellis} and E.~Spadaro.
\newblock {Regularity of area minimizing currents {I}: gradient {$L^p$}
  estimates}.
\newblock {\em Geom. Funct. Anal.}, 24(6):1831--1884, 2014.

\bibitem{DS2}
C.~{De Lellis} and E.~Spadaro.
\newblock {Multiple valued functions and integral currents}.
\newblock {\em Ann. Sc. Norm. Super. Pisa Cl. Sci. (5)}, 14(4):1239--1269,
  2015.

\bibitem{DS4}
C.~{De Lellis} and E.~Spadaro.
\newblock {Regularity of area minimizing currents {II}: center manifold}.
\newblock {\em Ann. of Math. (2)}, 183(2):499--575, 2016.

\bibitem{DS5}
C.~{De Lellis} and E.~Spadaro.
\newblock {Regularity of area minimizing currents {III}: blow-up}.
\newblock {\em Ann. of Math. (2)}, 183(2):577--617, 2016.

\bibitem{DSS4}
C.~{De Lellis}, E.~{Spadaro}, and L.~{Spolaor}.
\newblock {Regularity theory for $2$-dimensional almost minimal currents III:
  blowup}.
\newblock {\em ArXiv e-prints. To appear in {Jour. of Diff. Geom}}, August
  2015.

\bibitem{DSS3}
C.~{De Lellis}, E.~Spadaro, and L.~Spolaor.
\newblock {Regularity {T}heory for 2-{D}imensional {A}lmost {M}inimal
  {C}urrents {II}: {B}ranched {C}enter {M}anifold}.
\newblock {\em Ann. PDE}, 3(2):3:18, 2017.

\bibitem{DSS1}
C.~{De Lellis}, E.~Spadaro, and L.~Spolaor.
\newblock {Uniqueness of tangent cones for two-dimensional almost-minimizing
  currents}.
\newblock {\em Comm. Pure Appl. Math.}, 70(7):1402--1421, 2017.

\bibitem{DSS2}
C.~De~Lellis, E.~Spadaro, and L.~Spolaor.
\newblock Regularity theory for {$2$}-dimensional almost minimal currents {I}:
  {L}ipschitz approximation.
\newblock {\em Trans. Amer. Math. Soc.}, 370(3):1783--1801, 2018.

\bibitem{Fed}
H.~Federer.
\newblock {\em {Geometric measure theory}}.
\newblock {Die Grundlehren der mathematischen Wissenschaften, Band 153}.
  Springer-Verlag New York Inc., New York, 1969.

\bibitem{HS}
R.~Hardt and L.~Simon.
\newblock {Boundary regularity and embedded solutions for the oriented
  {P}lateau problem}.
\newblock {\em Ann. of Math. (2)}, 110(3):439--486, 1979.

\bibitem{Jonas}
J.~Hirsch.
\newblock {Boundary regularity of {D}irichlet minimizing {$Q$}-valued
  functions}.
\newblock {\em Ann. Sc. Norm. Super. Pisa Cl. Sci. (5)}, 16(4):1353--1407,
  2016.

\bibitem{HM}
Jonas Hirsch and Michele Marini.
\newblock Uniqueness of tangent cones to boundary points of two-dimensional
  almost-minimizing currents.
\newblock {\em arXiv preprint arXiv:1909.13383}, 2019.

\bibitem{Sim}
L.~Simon.
\newblock {\em {Lectures on geometric measure theory}}, volume~3 of {\em
  {Proceedings of the Centre for Mathematical Analysis, Australian National
  University}}.
\newblock Australian National University Centre for Mathematical Analysis,
  Canberra, 1983.

\bibitem{Wh}
B.~White.
\newblock {Tangent cones to two-dimensional area-minimizing integral currents
  are unique}.
\newblock {\em Duke Math. J.}, 50(1):143--160, 1983.

\end{thebibliography}
